\documentclass{article}
\usepackage[T1]{fontenc}
\usepackage[english]{babel}
\usepackage{graphicx}
\usepackage{stmaryrd}
\usepackage{amsmath,amsfonts,amssymb}

\usepackage{ntheorem}
\usepackage{dsfont}
\usepackage{slashed}
\usepackage{hyperref}
\usepackage{tikz}
\usepackage{pgfplots}
\usetikzlibrary{patterns}
\usetikzlibrary{positioning,arrows,arrows.meta}
\usepackage{geometry}
\usepackage{enumitem}
\theoremseparator{.} 
\newtheorem{Th}{Theorem}[section]
\newtheorem{Def}[Th]{Definition}
\newtheorem{Rq}[Th]{Remark}
\newtheorem{Pro}[Th]{Proposition}
\newtheorem{Cor}[Th]{Corollary}
\newtheorem{Lem}[Th]{Lemma}
\newcommand{\R}{\mathbb{R}}
\newcommand{\F}{\widetilde{F}}
\newcommand{\G}{\widetilde{G}}
\newcommand{\Ff}{\overline{F}}
\newcommand{\Sig}{\Sigma}
\newcommand{\Si}{\overline{\Sigma}}
\newcommand{\T}{\mathbb{T}}
\newcommand{\E}{\mathbb{E}}
\newcommand{\Po}{\mathbb{P}}
\newcommand{\Y}{\mathbb{Y}}
\newcommand{\Or}{\mathbb{O}}
\newcommand{\Vv}{\mathbb{V}}
\newcommand{\Sp}{\mathbb{S}}
\newcommand{\V}{\mathbf{k}_1}
\newcommand{\K}{\widehat{\mathbb{P}}_0}
\newenvironment{proof}{\noindent\textit{Proof.~}}{\hfill$\square$\bigbreak} 
\newlength{\plarg}
\title{Sharp asymptotic behavior of solutions of the $3d$ Vlasov-Maxwell system with small data}

\author{L\'eo Bigorgne\footnote{Laboratoire de Math\'ematiques, Univ.\ Paris-Sud, CNRS, Universit\'e Paris-Saclay, 91405 Orsay. E-mail adress : leo.bigorgne@u-psud.fr.}}

\begin{document}

\maketitle
    
\begin{abstract}
We study the asymptotic properties of the small data solutions of the Vlasov-Maxwell system in dimension three. No neutral hypothesis nor compact support assumptions are made on the data. In particular, the initial decay in the velocity variable is optimal. We use vector field methods to obtain sharp pointwise decay estimates in null directions on the electromagnetic field and its derivatives. For the Vlasov field and its derivatives, we obtain, as in \cite{FJS3}, optimal pointwise decay estimates by a vector field method where the commutators are modification of those of the free relativistic transport equation. In order to control high velocities and to deal with non integrable source terms, we make fundamental use of the null structure of the system and of several hierarchies in the commuted equations. 
\end{abstract}

    \tableofcontents
\section{Introduction}

This article is concerned with the asymptotic behavior of small data solutions to the three-dimensional Vlasov-Maxwell system. These equations, used to model collisionless plasma, describe, for one species of particles\footnote{Our results can be extended without any additional difficulty to several species of particles.}, a distribution function $f$ and an electromagnetic field which will be reprensented by a two form $F_{\mu \nu}$. The equations are given by\footnote{We will use all along this paper the Einstein summation convention so that, for instance, $v^i \partial_i f = \sum_{i=1}^3 v^i \partial_i f$ and $\nabla^{\mu} F_{\mu \nu} = \sum_{\mu=0}^3 \nabla^{\mu} F_{\mu \nu}$. The latin indices goes from $1$ to $3$ and the greek indices from $0$ to $3$.}
\begin{eqnarray}\label{VM1}
v^0\partial_t f+v^i \partial_i f +ev^{\mu}{ F_{\mu}}^{ j} \partial_{v^j} f & = & 0, \\ \label{VM2}
\nabla^{\mu} F_{\mu \nu} & = & e J(f)_{\nu} \hspace{2mm} := \hspace{2mm} e\int_{v \in \R^3} \frac{v_{\nu}}{v^0} f dv, \\ \label{VM3}
\nabla^{\mu} {}^* \!  F_{\mu \nu} & = & 0,
\end{eqnarray}
where $v^0=\sqrt{m^2+|v|^2}$, $m>0$ is the mass of the particles and $e \in \R^*$ their charge. For convenience, we will take $m=1$ and $e=1$ for the remainder of this paper. The particle density $f$ is a non-negative\footnote{In this article, the sign of $f$ does not play any role.} function of $(t,x,v) \in \R_+ \times \R^3 \times \R^3$, while the electromagnetic field $F$ and its Hodge dual ${}^* \!  F $ are $2$-forms depending on $(t,x) \in \R_+ \times \R^3$. We can recover the more common form of the Vlasov-Maxwell system using the relations
$$E^i=F_{0i} \hspace{8mm} \text{and} \hspace{8mm} B^i=-{}^* \!  F_{0i},$$
so that the equations can be rewritten as 
\begin{flalign*}
 & \hspace{3cm} \sqrt{1+|v|^2} \partial_t f+v^i \partial_i f + (\sqrt{1+|v|^2} E+v \times B) \cdot \nabla_v f = 0, & \\
& \hspace{3cm} \nabla \cdot E = \int_{v \in \R^3}fdv, \hspace{1.1cm} \partial_t E^j = (\nabla \times B)^j -\int_{v \in \R^3} \frac{v^j}{\sqrt{1+|v|^2}}fdv, & \\
& \hspace{3cm} \nabla \cdot B = 0, \hspace{2,2cm} \partial_t B = - \nabla \times E. &
\end{flalign*} 
We refer to \cite{Glassey} for a detailed introduction to this system.
\subsection{Small data results for the Vlasov-Maxwell system}

The first result on global existence with small data for the Vlasov-Maxwell system in $3d$ was obtained by Glassey-Strauss in \cite{GSt} and then extended to the nearly neutral case in \cite{Sc}. This result required compactly supported data (in $x$ and in $v$) and shows that $\int_v f dv \lesssim \frac{\epsilon}{(1+t)^3}$, which coincides with the linear decay. They also obtain estimates for the electromagnetic field and its derivatives of first order, but they do not control higher order derivatives of the solutions. The result established by Schaeffer in \cite{Sc} allows particles with high velocity but still requires the data to be compactly supported in space\footnote{Note also that when the Vlasov field is not compactly supported (in $v$), the decay estimate obtained in \cite{Sc} on its velocity average contains a loss.}. 

In \cite{dim4}, using vector field methods, we proved optimal decay estimates on small data solutions and their derivatives of the Vlasov-Maxwell system in high dimensions $d \geq 4$ without any compact support assumption on the initial data. We also obtained that similar results hold when the particles are massless ($m=0$) under the additional assumption that $f$ vanishes for small velocities\footnote{Note that there exists initial data violating this condition and such that the system does not admit a local classical solution (see Section $8$ of \cite{dim4}).}. 

A better understanding of the null condition of the system led us in our recent work \cite{massless} to an extension of these results to the massless 3d case. In \cite{ext} we study the asymptotic properties of solutions to the massive Vlasov-Maxwell in the exterior of a light cone for mildly decaying initial data. Due to the strong decay satisfied by the particle density in such a region we will be able to lower the initial decay hypothesis on the electromagnetic field and then avoid any difficulty related to the presence of a non-zero total charge.

The results of this paper establish sharp decay estimates on the small data solutions to the three-dimensional Vlasov-Maxwell system. The hypotheses on the particle density in the variable $v$ are optimal in the sense that we merely suppose $f$ (as well as its derivatives) to be initially integrable in $v$, which is a necessarily condition for the source term of the Maxwell equations to be well defined.

Recently, Wang proved independently in \cite{Wang} a similar result for the $3d$ massive Vlasov-Maxwell system. Using both vector field methods and Fourier analysis, he does not require compact support assumptions on the initial data but strong polynomial decay hypotheses in $(x,v)$ on $f$ and obtained optimal pointwise decay estimates on $\int_v f dv$ and its derivatives.

\subsection{Vector fields and modified vector fields for the Vlasov equations}

The vector field method of Klainerman was first introduced in \cite{Kl85} for the study of nonlinear wave equations. It relies on energy estimates, the algebra $\mathbb{P}$ of the Killing vector fields of the Minkowski space and conformal Killing vector fields, which are used as commutators and multipliers, and weighted functional inequalities now known as Klainerman-Sobolev inequalities.

In \cite{FJS}, the vector field method was adapted to relativistic transport equations and applied to the small data solutions of the Vlasov-Nordstr\"om system in dimensions $d \geq 4$. It provided sharp asymptotics on the solutions and their derivatives. Key to the extension of the method is the fact that even if $Z \in \mathbb{P}$ does not commute with the free transport operator $T:= v^{\mu} \partial_{\mu}$, its complete lift\footnote{The expression of the complete lift of a vector field of the Minkowski space is presented in Definition \ref{defliftcomplete}.} $\widehat{Z}$ does. The case of the dimension $3$, studied in \cite{FJS2}, required to consider modifications of the commutation vector fields of the form $Y=\widehat{Z}+\Phi^{\nu} \partial_{\nu}$, where $\widehat{Z}$ is a complete lift of a Killing field (and thus commute with the free transport operator) while
the coefficients $\Phi$ are constructed by solving a transport equation depending on the solution itself. In \cite{Poisson} (see also \cite{Xianglong}), similar results were proved for the Vlasov-Poisson equations and, again, the three-dimensionsal case required to modify the set of commutation vector fields in order to compensate the worst source terms in the commuted transport equations. Let us also mention \cite{rVP}, where the asymptotic behavior of the spherically symmetric small data solutions of the massless relativistic Vlasov-Poisson system are studied\footnote{Note that the Lorentz boosts cannot be used as commutation vector fields for this system since the Vlasov equation and the Poisson equation have different speed of propagation.}. Vector field methods led to a proof of the stability of the Minkowski spacetime for the Einstein-Vlasov system, obtained independently by \cite{FJS3} and \cite{Lindblad}.

Note that vector field methods can also be used to derive integrated decay for solutions to the the massless Vlasov equation on curved background such as slowly rotating Kerr spacetime (see \cite{ABJ}).

\subsection{Charged electromagnetic field}

In order to present our main result, we introduce in this subsection the pure charge part and the chargeless part of a $2$-form.
\begin{Def}
Let $G$ be a sufficiently regular $2$-form defined on $[0,T[ \times \R^3$. The total charge $Q_G(t)$ of $G$ is defined as
$$ Q_G(t) \hspace{2mm} = \hspace{2mm} \lim_{r \rightarrow + \infty}  \int_{\mathbb{S}_{t,r}} \frac{x^i}{r}G_{0i} d \mathbb{S}_{t,r},$$
where $\mathbb{S}_{t,r}$ is the sphere of radius $r$ of the hypersurface $\{t \} \times \R^3$ which is centered at the origin $x=0$.
\end{Def}
If $(f,F)$ is a sufficiently regular solution to the Vlasov-Maxwell system, $Q_F$ is a conserved quantity. More precisely,
$$ \forall \hspace{0.5mm} t \in [0,T[, \hspace{2cm} Q_F(t)=Q_F(0)= \int_{x \in \R^3} \int_{v \in \R^3} f(0,x,v) dv dx.$$
Note that the derivatives of $F$ are automatically chargeless (see Appendix $C$ of \cite{massless}). The presence of a non-zero charge implies $\int_{\R^3} r|F|^2 dx = +\infty$ and prevents us from propagating strong weighted $L^2$ norms on the electromagnetic field. This leads us to decompose $2$-forms into two parts. For this, let $\chi : \R \rightarrow [0,1]$ be a cut-off function such that
$$ \forall \hspace{0.5mm} s \leq -2, \hspace{3mm} \chi(s) =1 \hspace{1cm} \text{and} \hspace{1cm} \forall \hspace{0.5mm} s \geq -1, \hspace{3mm} \chi(s) =0.$$
\begin{Def}\label{defpure1}
Let $G$ be a sufficiently regular $2$-form with total charge $Q_G$. We define the pure charge part $\overline{G}$ and the chargeless part $\G$ of $G$ as
$$\overline{G}(t,x) := \chi(t-r) \frac{Q_G(t)}{4 \pi r^2} \frac{x_i}{r} dt \wedge dx^i \hspace{1cm} \text{and} \hspace{1cm} \G := G-\overline{G}.$$
\end{Def}
One can then verify that $Q_{\overline{G}}=Q_G$ and $Q_{\G}=0$, so that the hypothesis $\int_{\R^3} r|\G|^2 dx = +\infty$ is consistent. Notice moreover that $G=\G$ in the interior of the light cone.

The study of non linear systems with a presence of charge was initiated by \cite{Shu} in the context of the Maxwell-Klein Gordon equations. The first complete proof of such a result was given by Lindblad and Sterbenz in \cite{LS} and improved later by Yang (see \cite{Yang}). Let us also mention the work of \cite{Bieri}.

\subsection{Statement of the main result}

\begin{Def}
We say that $(f_0,F_0)$ is an initial data set for the Vlasov-Maxwell system if $f_0 : \R^3_x \times \R^3_v \rightarrow \R$ and the $2$-form $F_0$ are both sufficiently regular and satisfy the constraint equations
$$\nabla^i (F_{0})_{i0} =- \int_{v \in \R^3} f_0 dv \hspace{10mm} \text{and} \hspace{10mm} \nabla^i {}^* \! (F_0)_{i0} =0.$$ 
\end{Def}
The main result of this article is the following theorem.

\begin{Th}\label{theorem}
Let $N \geq 11$, $\epsilon >0$, $(f_0,F_0)$ an initial data set for the Vlasov-Maxwell equations \eqref{VM1}-\eqref{VM3}and $(f,F)$ be the unique classical solution to the system arising from $(f_0,F_0)$. If
$$ \sum_{ |\beta|+|\kappa| \leq N+3} \int_{x \in \R^3} \int_{v \in \R^3} (1+|x|)^{2N+3}(1+|v|)^{|\kappa|} \left| \partial_{x}^{\beta} \partial_v^{\kappa} f_0 \right| dv dx  + \sum_{ |\gamma| \leq N+2} \int_{x \in \R^3}  (1+|x|)^{2 |\gamma|+1} \left| \nabla_x^{\gamma} \F_0 \right|^2   dx \leq \epsilon ,$$
then there exists $C>0$, $M \in \mathbb{N}$ and $\epsilon_0>0$ such that, if $\epsilon \leq \epsilon_0$, $(f,F)$ is a global solution to the Vlasov-Maxwell system and verifies the following estimates.
\begin{itemize}
\item Energy bounds for the electromagnetic field $F$ and its chargeless part: $\forall$ $t \in \R_+$,
$$ \sum_{\begin{subarray}{}  Z^{\gamma} \in \mathbb{K}^{|\gamma|} \\ \hspace{1mm} |\gamma| \leq N  \end{subarray}}  \int_{|x| \geq t} \tau_+ \left( | \alpha ( \mathcal{L}_{ Z^{\gamma}}(\F) ) |^2  + | \rho ( \mathcal{L}_{ Z^{\gamma}}(\F) )|^2  +  |\sigma ( \mathcal{L}_{ Z^{\gamma}}(\F) ) |^2 \right)+\tau_- |\underline{\alpha} ( \mathcal{L}_{ Z^{\gamma}}(\F) ) |^2  dx \leq C\epsilon ,$$
$$ \hspace{-0.3cm} \sum_{\begin{subarray}{}  Z^{\gamma} \in \mathbb{K}^{|\gamma|} \\ \hspace{1mm} |\gamma| \leq N  \end{subarray}}  \int_{|x| \leq t} \hspace{-0.2mm} \tau_+ \left(  \left| \alpha \left( \mathcal{L}_{ Z^{\gamma}}(F) \right) \right|^2 \hspace{-0.2mm} + \hspace{-0.2mm} \left| \rho \left( \mathcal{L}_{ Z^{\gamma}}(F) \right) \right|^2 \hspace{-0.2mm} + \hspace{-0.2mm} \left|\sigma \left( \mathcal{L}_{ Z^{\gamma}}(F) \right) \right|^2 \right)+\tau_- \left|\underline{\alpha} \left( \mathcal{L}_{ Z^{\gamma}}(F) \right) \right|^2  dx \leq C\epsilon \log^{2M}(3+t) .$$
\item Pointwise decay estimates for the null components of\footnote{If $|x| \geq t+1$, the logarithmical growth can be removed for the components $\alpha$ and $\underline{\alpha}$.} $\mathcal{L}_{Z^{\gamma}}(F)$: $\forall$ $|\gamma| \leq N-6$, $(t,x) \in \R_+ \times \R^3$,
\begin{flalign*}
& \hspace{0.5cm} |\alpha(\mathcal{L}_{Z^{\gamma}}(F))|(t,x) \hspace{1mm} \lesssim \hspace{1mm} \sqrt{\epsilon}\frac{\log(3+t)}{\tau_+^2} , \hspace{30mm} |\underline{\alpha}(\mathcal{L}_{Z^{\gamma}}(F))|(t,x) \hspace{1mm} \lesssim \hspace{1mm} \sqrt{\epsilon}\frac{\log(3+t)}{\tau_+\tau_-} ,& \\
& \hspace{0.5cm} |\rho(\mathcal{L}_{Z^{\gamma}}(F))|(t,x) \hspace{1mm} \lesssim \hspace{1mm} \sqrt{\epsilon} \frac{\log^2(3+t)}{\tau_+^2}, \hspace{29mm} |\sigma(\mathcal{L}_{Z^{\gamma}}(F))|(t,x) \hspace{1mm} \lesssim \hspace{1mm} \sqrt{\epsilon}\frac{\log^2(3+t)}{\tau_+^2}.&
\end{flalign*}
\item Energy bounds for the Vlasov field: $\forall$ $t \in \R_+$,
$$\sum_{\begin{subarray}{}  \hspace{0.5mm} Y^{\beta} \in \Y^{|\beta|} \\ \hspace{1mm} |\beta| \leq N  \end{subarray}} \int_{x \in \R^3} \int_{v \in \R^3} \left|Y^{\beta} f \right| dv dx \leq C\epsilon.$$
\item Pointwise decay estimates for the velocity averages of $Y^{\beta} f$: $\forall$ $|\beta| \leq N-3$, $(t,x) \in \R_+ \times \R^3$,
$$\int_{ v \in \R^3} \left| Y^{\beta} f \right| dv \lesssim \frac{\epsilon}{\tau_+^2 \tau_-} \hspace{5mm} \text{and} \hspace{5mm} \int_{ v \in \R^3} \left| Y^{\beta} f \right| \frac{dv}{(v^0)^2} \lesssim \epsilon \frac{1}{\tau_+^3} \mathds{1}_{t \geq |x|}+ \epsilon\frac{\log^2(3+t)}{\tau_+^3\tau_-} \mathds{1}_{|x| \geq t}$$.
\end{itemize}
\end{Th}
\begin{Rq}
For the highest derivatives of $f_0$, those of order at least $N-2$, we could save four powers of $|x|$ in the condition on the initial norm and even more for those of order at least $N+1$. We could also avoid any hypothesis on the derivatives of order $N+1$ and $N+2$ of $F_0$ (see Remark \ref{rqgainH}).
\end{Rq}
\begin{Rq}
Assuming more decay on $\F$ and its derivatives at $t=0$, we could use the Morawetz vector field as a multiplier, propagate a stronger energy norm and obtain better decay estimates on its null components in the exterior of the light cone. We could recover the decay rates of the free Maxwell equations (see \cite{CK}) on $\alpha (F)$, $\underline{\alpha} (F)$ and $\sigma (F)$, but not on $\rho(F)$. We cannot obtain a better decay rate than $\tau_+^{-2}$ on $\rho (F)$ because of the presence of the charge. With our approach, we cannot recover the sourceless behavior in the interior region because of the slow decay of $\int_v f dv$.
\end{Rq}

\subsection{Key elements of the proof}
\subsubsection{Modified vector fields}
In \cite{dim4}, we observed that commuting \eqref{VM1} with the complete lift of a Killing vector field gives problematic source terms. More precisely, if $Z \in \mathbb{P}$,
\begin{equation}\label{sourcetermintro}
[T_F, \widehat{Z} ] f= -v^{\mu} {\mathcal{L}_Z(F)_{\mu}}^{ j} \partial_{v^j} f, \hspace{10mm} \text{with} \hspace{3mm} T_F = v^{\mu}\partial_{\mu}+v^{\mu} {F_{\mu}}^{ j} \partial_{v^j}.
\end{equation}
The difficulty comes from the presence of $\partial_v$, which is not part of the commutation vector fields, since in the linear case ($F=0$) $\partial_v f$ essentially behaves as $t\partial_{t,x} f$. However, one can see that the source term has the same form as the non-linearity $v^{\mu} {F_{\mu}}^{ j} \partial_{v^j} f$. In \cite{dim4}, we controlled the error terms by taking advantage of their null structure and the strong decay rates given by high dimensions. Unfortunately, our method does not apply in dimension $3$ since even assuming a full understanding of the null structure of the system, we would face logarithmic divergences. The same problem arises for other Vlasov systems and were solved using modified vector fields in order to cancel the worst source terms in the commutation formula. Let us mention again the works of \cite{FJS2} for the Vlasov-Nordstr\"om system, \cite{Poisson} for the Vlasov-Poisson equations, \cite{FJS3} and \cite{Lindblad} for the Einstein-Vlasov system. We will thus consider vector fields of the form $Y=\widehat{Z}+\Phi^{\nu}\partial_{\nu}$, where the coefficients $\Phi^{\nu}$ are themselves solutions to transport equations, growing logarithmically. As a consequence, we will need to adapt the Klainerman-Sobolev inequalities for velocity averages and the result of Theorem $1.1$ of \cite{dim4} in order to replace the original vector fields by the modified ones.
\subsubsection{The electromagnetic field and the non-zero total charge}
Because of the presence of a non-zero total charge, i.e. $ \lim_{r \rightarrow + \infty} \int_{ \mathbb{S}_{0,r} } \frac{x^i}{r} (F_0)_{0i} d \mathbb{S}_{0,r} \neq 0$, we have, at $t=0$,
$$\int_{\R^3} (1+r) \left| \frac{x^i}{r} F_{0i} \right|^2 dx = \int_{\R^3} (1+r) |\rho(F)|^2 dx= + \infty$$
and we cannot propagate $L^2$ bounds on $\int_{\R^3} (1+t+r) |\rho(F)(t,x)|^2 dx$. However, provided that we can control the flux of the electromagnetic field on the light cone $t=r$, we can propagate weighted $L^2$ norms of $F$ in the interior region. To deal with the exterior of the light cone, recall from Definition \ref{defpure1} the decomposition
\begin{equation}\label{explicit}
F = \F+\Ff, \hspace{1cm} \text{with} \hspace{1cm} \overline{F}(t,x) := \chi(t-r) \frac{Q_F}{4 \pi r^2} dr \wedge dt .
\end{equation}
The hypothesis $\int_{\R^3} (1+|x|) | \F (0,.)| dx < + \infty$ is consistent with the chargelessness of $\F$ and we can then propagate weighted energy norms of $\F$ and bound the flux of $F$ on the light cone. On the other hand, we have at our disposal pointwise estimates on $\overline{F}$ and its derivatives through the explicit formula \eqref{explicit}. These informations will allow us to deduce pointwise decay estimates on the null components of $F$ in both the exterior and the interior regions.

Another problem arises from the source terms of the commuted Maxwell equations, which need to be written with our modified vector fields. This leads us, as \cite{FJS2} and \cite{FJS3}, to rather consider them of the form $Y=\widehat{Z}+\Phi^{i}X_i$, where $X_i=\partial_i+\frac{v^i}{v^0}\partial_t$. The $X_i$ vector fields enjoy a kind of null condition\footnote{Note that they were also used in \cite{dim4} to improve the decay estimate on $\partial \int_v f ds$.} and allow us to avoid a small growth on the electromagnetic field norms which would prevent us to close our energy estimates\footnote{We make similar manipulations to recover the standard decay rate on the modified Klainerman-Sobolev inequalities.}. However, at the top order, a loss of derivative does not allow us to take advantage of them and creates a $t^{\eta}$-loss, with $\eta >0$ a small constant. A key step is to make sure that $\| \left| Y^{\kappa} \Phi \right|^2 Y f \|_{L^1_{x,v}}$, for $|\kappa|=N-1$, does not grow faster than $t^{\eta}$.
\subsubsection{High velocities and null structure of the system}

After commuting the transport equation satisfied by the coefficients $\Phi^i$ and in order to prove energy estimates, we are led to control integrals such as
$$\int_0^t \int_{\R^3} \int_{v \in \R^3}(s+|x|) \left| \mathcal{L}_Z(F)  f \right| dv dx ds.$$
If $f$ vanishes for high velocities, the characteristics of the transport equations have velocities bounded away from $1$. If $f$ is moreover initially compactly supported in space, its spatial support is ultimately disjoint from the light cone and, assuming enough decay on the Maxwell field, one can prove $$|\mathcal{L}_Z(F) f| \lesssim (1+t+r)^{-1}(1+|t-r|)^{-1}| f | \lesssim (1+t+r)^{-2}| f |,$$
so that
\begin{equation}\label{eq:pb1}
\int_0^t \int_{\R^3} \int_{v \in \R^3}(s+|x|) \left| \mathcal{L}_Z(F)  f \right| dv dx ds \lesssim \int_0^t (1+s)^{-1} ds,
 \end{equation}
which is almost uniformly bounded in time\footnote{Dealing with these small growth is the next problem addressed.}. As we do not make any compact support assumption on the initial data, we cannot expect $f$ to vanish for high velocities and certain characteristics of the transport operator ultimately approach those of the Maxwell equations. We circumvent this difficulty by taking advantage of the null structure of the error term given in \eqref{sourcetermintro}, which, in some sense, allows us to transform decay in $|t-r|$ into decay in $t+r$. The key is that certain null components of $v$, $\mathcal{L}_Z(F)$ and $\nabla_v f :=(0,\partial_{v^1} f,\partial_{v^2}f,\partial_{v^3}f)$ behave better than others and we will see in Lemma \ref{nullG} that no product of three bad components appears. More precisely, noting $c \prec d$ if $d$ is expected to behave better than $c$, we have,
$$v^L \prec v^A, \hspace{1mm} v^{\underline{L}}, \hspace{10mm} \underline{\alpha}(\mathcal{L}_Z(F)) \prec \rho (\mathcal{L}_Z(F)) \sim \sigma( \mathcal{L}_Z(F) ) \prec \alpha( \mathcal{L}_Z(F) ) \hspace{10mm} \text{and} \hspace{6mm} \left( \nabla_v f \right)^A \prec \left( \nabla_v f \right)^{r}.$$
In the exterior of the light cone (and for the massless relativistic transport operator), we have $v^A \prec v^{\underline{L}}$ since $v^{\underline{L}}$ permits to integrate along outgoing null cones\footnote{The angular component $v^A$ can, in some sense, merely do half of it since $|v^A| \lesssim \sqrt{v^0 v^{\underline{L}}}$.} and
they are both bounded by $(1+t+r)^{-1}v^0\sum_{z \in \mathbf{k}_1} |z|$, where $\mathbf{k}_1$ is a set of weigths preserved by the free transport operator. In the interior region, the angular components still satisfies the same properties whereas $v^{\underline{L}}$ merely satisfies the inequality
\begin{equation}\label{eq:intro1} 
v^{\underline{L}} \lesssim \frac{|t-r|}{1+t+r}v^0+\frac{v^0}{1+t+r} \sum_{z \in \mathbf{k}} |z| \hspace{10mm} \text{( see Lemma \ref{weights1})}.
\end{equation}
This inequality is crucial for us to close the energy estimates on the electromagnetic field without assuming more initial decay in $v$ on $f$. It gives a decay rate of $(1+t+r)^{-3}$ on $\int_v \frac{v^{\underline{L}}}{v^0} |f| dv$ by only using a Klainerman-Sobolev inequality (Theorem \ref{decayopti} and Proposition \ref{decayopti2} would cost us two powers of $v^0$). As $1 \lesssim v^0 v^{\underline{L}}$ for massive particles, we obtain, combining \eqref{eq:intro1} and Theorem \ref{decayopti}, for $g$ a solution to $v^{\mu} \partial_{\mu} g =0$,
$$ \forall \hspace{0.5mm} t \geq |x|, \hspace{10mm} \int_{v \in \R^3} |g|(t,x,v)dv \lesssim \frac{(1+|t-r|)^k}{(1+t+r)^{3+k}} \sum_{|\beta| \leq 3} \left\| (v^0)^{2k+2}(1+r)^k \widehat{Z}^{\beta}g \right\|_{L^1_{x,v}}(t=0).$$
In the exterior region, the estimate can be improved by removing the factor $(1+|t-r|)^k$ (however one looses one power of $r$ in the initial norm). This remarkable behavior reflects that the particles do not reach the speed of light so that $\int_{v \in \R^3} |g| dv$ enjoys much better decay properties along null rays than along time-like directions and should be compared with solutions to the Klein-Gordon equation (see \cite{Kl93}).

\subsubsection{Hierarchy in the equations}

Because of certains source terms of the commuted transport equation, we cannot avoid a small growth on certain $L^1$ norms as it is suggested by \eqref{eq:pb1}. In order to close the energy estimates, we then consider several hierarchies in the energy norms of the particle density, in the spirit of \cite{LR} for the Einstein equations or \cite{FJS3} for the Einstein-Vlasov system. Let us show how a hierarchy related to the weights $z \in \V$ preserved by the free massive transport operator (which are defined in Subsection \ref{sectionweights}) naturally appears.
\begin{itemize}
\item The worst source terms of the transport equation satisfied by $Yf$ are of the form $(t+r)X_i(F_{\mu \nu})\partial_{t,x} f$.
\item Using the improved decay properties given by $X_i$ (see \eqref{eq:X}), we have
$$ \left| (t+r)X_i(F_{\mu \nu})\partial_{t,x} f \right| \lesssim \sum_{Z \in \mathbb{K}} |\nabla_Z F| \sum_{z \in \V} |z\partial_{t,x} f|.$$
\item Then, we can obtain a good bound on $\| Yf \|_{L^1_{x,v}}$ provided we have a satisfying one on $\| z \partial_{t,x} f \|_{L^1_{x,v}}$. We will then work with energy norms controlling $\| z^{N_0-\beta_P} Y^{\beta} f \|_{L^1_{x,v}}$, where $\beta_P$ is the number of non-translations composing $Y^{\beta}$.
\item At the top order, we will have to deal with terms such as $(t+r)z^{N_0}\partial_{t,x}^{\gamma}(F_{\mu \nu})\partial_{t,x}^{\beta} f$ and we will this time use the extra decay $(1+|t-r|)^{-1}$ given by the translations $\partial_{t,x}^{\gamma}$.
\end{itemize}

\subsection{Structure of the paper}

In Section \ref{sec2} we introduce the notations used in this article. Basic results on the electromagnetic field as well as fundamental relations between the null components of the velocity vector $v$ and the weights preserved by the free transport operator are also presented. Section \ref{sec3} is devoted to the commutation vector fields. The construction and basic properties of the modified vector fields are in particular presented. Section \ref{sec4} contains the energy estimates and the pointwise decay estimates used to control both fields. Section \ref{secpurecharge} is devoted to properties satisfied by the pure charge part of the electromagnetic field. In Section \ref{sec6} we describe the main steps of the proof of Theorem \ref{theorem} and present the bootstrap assumptions. In Section \ref{sec7}, we derive pointwise decay estimates on the solutions and the $\Phi$ coefficients of the modified vector fields using only the bootstrap assumptions. Section \ref{sec8} (respectively Section \ref{sec12}) concerns the improvement of the bootstrap assumptions on the norms of the particle density (respectively the electromagnetic field). A key step consists in improving the estimates on the velocity averages near the light cone (cf. Proposition \ref{Xdecay}). In Section \ref{sec11}, we prove $L^2$ estimates for $\int_v|Y^{\beta}f|dv$ in order to improve the energy estimates on the Maxwell field.

\section{Notations and preliminaries}\label{sec2}

\subsection{Basic notations}

In this paper we work on the $3+1$ dimensional Minkowski spacetime $(\R^{3+1},\eta)$. We will use two sets of coordinates, the Cartesian $(t,x^1,x^2,x^3)$, in which $\eta=diag(-1,1,1,1)$, and null coordinates $(\underline{u},u,\omega_1,\omega_2)$, where
$$\underline{u}=t+r, \hspace{5mm} u=t-r$$
and $(\omega_1,\omega_2)$ are spherical variables, which are spherical coordinates on the spheres $(t,r)=constant$. These coordinates are defined globally on $\R^{3+1}$ apart from the usual degeneration of spherical coordinates and at $r=0$. We will also use the following classical weights,
$$\tau_+:= \sqrt{1+\underline{u}^2} \hspace{8mm} \text{and} \hspace{8mm} \tau_-:= \sqrt{1+u^2}.$$
We denote by $(e_1,e_2)$ an orthonormal basis on the spheres and by $\slashed{\nabla}$ the intrinsic covariant differentiation on the spheres $(t,r)=constant$. Capital Latin indices (such as $A$ or $B$) will always correspond to spherical variables. The null derivatives are defined by
$$L=\partial_t+\partial_r \hspace{3mm} \text{and} \hspace{3mm} \underline{L}=\partial_t-\partial_r, \hspace{3mm} \text{so that} \hspace{3mm} L(\underline{u})=2, \hspace{2mm} L(u)=0, \hspace{2mm} \underline{L}( \underline{u})=0 \hspace{2mm} \text{and} \hspace{2mm} \underline{L}(u)=2.$$
The velocity vector $(v^{\mu})_{0 \leq \mu \leq 3}$ is parametrized by $(v^i)_{1 \leq i \leq 3}$ and $v^0=\sqrt{1+|v|^2}$ since we take the mass to be $1$. We introduce the operator $$T : f \mapsto v^{\mu} \partial_{\mu} f,$$
defined for all sufficiently regular functions $f : [0,T[ \times \R^3_x \times \R^3_v$, and we denote $(0,\partial_{v^1}g, \partial_{v^2}g,\partial_{v^3}g)$ by $\nabla_v g$ so that \eqref{VM1} can be rewritten $$T_F(f) := v^{\mu} \partial_{\mu} f +F \left( v, \nabla_v f \right) =0.$$
We will use the notation $D_1 \lesssim D_2$ for an inequality such as $ D_1 \leq C D_2$, where $C>0$ is a positive constant independent of the solutions but which could depend on $N \in \mathbb{N}$, the maximal order of commutation. Finally we will raise and lower indices using the Minkowski metric $\eta$. For instance, $\nabla^{\mu} = \eta^{\nu \mu} \nabla_{\nu}$ so that $\nabla^{\partial_t}=-\nabla_{\partial_t}$ and $\nabla^{\partial_i}=\nabla_{\partial_i}$ for all $1 \leq i \leq 3$.

\subsection{Basic tools for the study of the electromagnetic field}\label{basicelec}

As we describe the electromagnetic field in geometric form, it will be represented, throughout this article, by a $2$-form. Let $F$ be a $2$-form defined on $[0,T[ \times \R^3_x$. Its Hodge dual ${}^* \! F$ is the $2$-form given by
$${}^* \! F_{\mu \nu} = \frac{1}{2} F^{\lambda \sigma} \varepsilon_{ \lambda \sigma \mu \nu},$$
where $\varepsilon_{ \lambda \sigma \mu \nu}$ are the components of the Levi-Civita symbol. The null decomposition of $F$, introduced by \cite{CK}, is denoted by $(\alpha(F), \underline{\alpha}(F), \rho(F), \sigma (F))$, where
$$\alpha_A(F) = F_{AL}, \hspace{5mm} \underline{\alpha}_A(F)= F_{A \underline{L}}, \hspace{5mm} \rho(F)= \frac{1}{2} F_{L \underline{L} } \hspace{5mm} \text{and} \hspace{5mm} \sigma(F) =F_{12}.$$
Finally, the energy-momentum tensor of $F$ is
$$T[F]_{\mu \nu} :=   F_{\mu \beta} {F_{\nu}}^{\beta}- \frac{1}{4}\eta_{\mu \nu} F_{\rho \sigma} F^{\rho \sigma}.$$
Note that $T[F]_{\mu \nu}$ is symmetric and traceless, i.e. $T[F]_{\mu \nu}=T[F]_{\nu \mu}$ and ${T[F]_{\mu}}^{\mu}=0$. This last point is specific to the dimension $3$ and engenders additional difficulties in the analysis of the Maxwell equations in high dimension (see Section $3.3.2$ in \cite{dim4} for more details).

We have the following alternative form of the Maxwell equations (for a proof, see \cite{CK} or Lemmas $2.2$ and $D.3$ of \cite{massless}).

\begin{Lem}\label{maxwellbis}
Let $G$ be a $2$-form and $J$ be a $1$-form both sufficiently regular and such that
\begin{eqnarray}
\nonumber \nabla^{\mu} G_{\mu \nu} & = & J_{\nu} \\ \nonumber
\nabla^{\mu} {}^* \!  G_{\mu \nu} & = & 0.
\end{eqnarray}
Then,
$$\nabla_{[ \lambda} G_{\mu \nu ]} = 0 \hspace{4mm} \text{and} \hspace{4mm} \nabla_{[ \lambda} {}^* \! G_{\mu \nu ]} = \varepsilon_{\lambda \mu \nu \kappa} J^{\kappa}.$$
We also have, if $(\alpha, \underline{\alpha}, \rho, \sigma)$ is the null decomposition of $G$,
\begin{eqnarray} 
\nabla_{\underline{L}} \hspace{0.5mm} \rho-\frac{2}{r} \rho+ \slashed{\nabla}^A \underline{\alpha}_A & = & J_{\underline{L}}, \label{eq:nullmax1} \\ 
\nabla_{\underline{L}} \hspace{0.5mm} \sigma-\frac{2}{r} \sigma+ \varepsilon^{AB} \slashed{\nabla}_A \underline{\alpha}_B & = & 0, \label{eq:nullmax2} \\ \nabla_{\underline{L}} \hspace{0.5mm} \alpha_A-\frac{\alpha_A}{r}+\slashed{\nabla}_{e_A} \rho+\varepsilon_{BA} \slashed{\nabla}_{e_B} \sigma &=& J_A. \label{eq:nullmax4}
\end{eqnarray}
\end{Lem}

We can then compute the divergence of the energy momentum tensor of a $2$-form.

\begin{Cor}\label{tensorderiv}
Let $G$ and $J$ be as in the previous lemma. Then, $\nabla^{\mu} T[G]_{\mu \nu}=G_{\nu \lambda} J^{\lambda}$.
\end{Cor}

\begin{proof}
Using the previous lemma, we have
\begin{eqnarray}
\nonumber G_{\mu \rho} \nabla^{\mu} {G_{\nu}}^{\rho}& = & G^{\mu \rho} \nabla_{\mu} G_{\nu \rho} \\
\nonumber & = & \frac{1}{2} G^{\mu \rho} (\nabla_{\mu} G_{\nu \rho}-\nabla_{\rho} G_{\nu \mu}) \\
\nonumber & = & \frac{1}{2} G^{\mu \rho} \nabla_{\nu} G_{\mu \rho} \\
\nonumber & = & \frac{1}{4} \nabla_{\nu} (G^{\mu \rho} G_{\mu \rho}).
\end{eqnarray}

Hence,
$$\nabla^{\mu} T[G]_{\mu \nu} = \nabla^{\mu} (G_{\mu \rho}){G_{\nu}}^{\rho}+\frac{1}{4} \nabla_{\nu} (G^{\mu \rho} G_{\mu \rho})-\frac{1}{4}\eta_{\mu \nu} \nabla^{\mu} (G^{\sigma \rho} G_{\sigma \rho})=G_{\nu \rho} J^{\rho}.$$
\end{proof}

Finally, we recall the values of the null components of the energy-momentum tensor of a $2$-form.
\begin{Lem}\label{tensorcompo}
Let $G$ be $2$-form. We have
$$T[G]_{L L}=|\alpha(G)|^2, \hspace{8mm} T[G]_{\underline{L} \underline{L} }=|\underline{\alpha}(G)|^2 \hspace{8mm} \text{and} \hspace{8mm} T[G]_{L \underline{L}}=|\rho(G)|^2+|\sigma(G)|^2.$$
\end{Lem}

\subsection{Weights preserved by the flow and null components of the velocity vector}\label{sectionweights}

Let $(v^L,v^{\underline{L}},v^A,v^B)$ be the null components of the velocity vector, so that
$$v=v^L L+ v^{\underline{L}} \underline{L}+v^Ae_A, \hspace{8mm} v^L=\frac{v^0+\frac{x_i}{r}v^i}{2} \hspace{8mm} \text{and} \hspace{8mm} v^{\underline{L}}=\frac{v^0-\frac{x_i}{r}v^i}{2}.$$
As in \cite{FJS}, we introduce the following set of weights,
$$ \mathbf{k}_1 := \left\{\frac{v^{\mu}}{v^0} \hspace{1mm} / \hspace{1mm} 0 \leq \mu \leq 3 \right\} \cup \left\{ z_{\mu \nu} \hspace{1mm} / \hspace{1mm} \mu \neq \nu \right\}, \hspace{1cm} \text{with} \hspace{1cm } z_{\mu \nu} := x^{\mu}\frac{v^{\nu}}{v^0}-x^{\nu}\frac{v^{\mu}}{v^0}.$$
Note that 
\begin{equation}\label{weightpreserv}
\forall \hspace{0.5mm} z \in \mathbf{k}_1, \hspace{8mm} T(z)=0.
\end{equation}
Recall that if $\mathbf{k}_0 := \mathbf{k}_1 \cup \{ x^{\mu} v_{\mu} \}$, then $v^{\underline{L}} \lesssim \tau_+^{-1} \sum_{w \in \mathbf{k}_0} |w|$. Unfortunately, $x^{\mu} v_{\mu}$ is not preserved by\footnote{Note however that $x^{\mu} v_{\mu}$ is preserved by $|v| \partial_t+x^i \partial_i$, the massless relativistic transport operator.} $T$ so we will not be able to take advantage of this inequality in this paper. In the following lemma, we try to recover (part of) this extra decay. We also recall inequalities involving other null components of $v$, which will be used all along this paper.
\begin{Lem}\label{weights1}
The following estimates hold,
$$ 1 \leq 4v^0v^{\underline{L}}, \hspace{8mm} |v^A| \lesssim \sqrt{v^Lv^{\underline{L}}}, \hspace{8mm} |v^A| \lesssim \frac{v^0}{\tau_+} \sum_{z \in \mathbf{k}_1} |z|,  \hspace{8mm} \text{and} \hspace{8mm} v^{\underline{L}} \lesssim \frac{\tau_-}{\tau_+} v^0+\frac{v^0}{\tau_+}\sum_{z \in \mathbf{k}_1}|z|.$$
\end{Lem}
\begin{proof}
Note first that, as $v^0= \sqrt{1+|v|^2}$,
$$ 4r^2v^Lv^{\underline{L}} \hspace{2mm}  = \hspace{2mm} r^2+r^2 |v|^2-|x^i|^2|v_i|^2-2\sum_{1 \leq k < l \leq n}x^kx^lv^kv^l \hspace{2mm}  = \hspace{2mm} r^2+\sum_{1 \leq k < l \leq n} |z_{kl}|^2.$$
It gives us the first inequality since $v^L \leq v^0$. For the second one, use also that $rv^A=v^0C_A^{i,j} z_{ij}$, where $C_A^{i,j}$ are bounded functions on the sphere such that $re_A = C_A^{i,j} (x^i \partial_j-x^j \partial_i)$. The third one follows from $|v^A| \leq v^0$ and
$$|v^A| \lesssim \frac{v^0}{r} \sum_{1 \leq i < j \leq 3} |z_{ij}| = \frac{v^0}{tr} \sum_{1 \leq i < j \leq 3} \left| x^i\left( \frac{v^j}{v^0}t-x^j+x^j \right)-x^j\left( \frac{v^i}{v^0}t-x^i+x^i \right) \right| \lesssim \frac{v^0}{t} \sum_{q=1}^3 |z_{0q}|.$$
For the last inequality, note first that $v^{\underline{L}} \leq v^0$, which treats the case $t+|x| \leq 2$. Otherwise, use
$$2tv^{\underline{L}}=tv^0-\frac{x^i}{r}tv_i = tv^0-v^0\frac{x^iz_{0i}}{r}-v^0r=(t-r)v^0-\frac{x^i}{r}z_{0i}v^0 \hspace{5mm} \text{and} \hspace{5mm}  r v^{\underline{L}} =(r-t) v^{\underline{L}}+tv^{\underline{L}}.$$
\end{proof}
\begin{Rq}\label{rqweights1}
Note that $v^{\underline{L}} \lesssim \frac{v^0}{\tau_+} \sum_{z \in \mathbf{k}_1} |z|$ holds in the exterior region. Indeed, if $r \geq t$,
$$v^0(r-t) \leq v^0|x|-|v|t \leq |v^0 x-tv| \leq \sum_{i=1}^3 |v^0x^i-tv^i|= v^0 \sum_{i=1}^3 |z_{0i}|.$$
We also point out that $1 \lesssim v^0 v^{\underline{L}}$ is specific to massive particles.
\end{Rq}
Finally, we consider an ordering on $\mathbf{k}_1$ such that $\mathbf{k}_1 = \{ z_i \hspace{1mm} / \hspace{1mm} 1 \leq i \leq |\mathbf{k}_1| \}$.
\begin{Def}\label{orderk1}
If $ \kappa \in \llbracket 1, |\mathbf{k}_1| \rrbracket^r$, we define $z^{\kappa} := z_{\kappa_1}...z_{\kappa_r}$.
\end{Def}
\subsection{Various subsets of the Minkowski spacetime}\label{secsubsets}

We now introduce several subsets of $\mathbb{R}_+ \times \mathbb{R}^3$ depending on $t \in \mathbb{R}_+$, $r \in \R_+$ or $u \in \mathbb{R}$. Let $\Sigma_t$, $\Sp_{t,r}$, $C_u(t)$ and $V_u(t)$ be defined as

\begin{flalign*}
& \hspace{0.5cm} \Sigma_t := \{t\} \times \mathbb{R}^n, \hspace{5.4cm} C_u(t):= \{(s,y)  \in \mathbb{R}_+ \times \mathbb{R}^3 / \hspace{1mm} s \leq t, \hspace{1mm} s-|y|=u \}, & \\
& \hspace{5mm} \Sp_{t,r}:= \{ (s,y) \in \R_+ \times \R^3 \hspace{1mm} / \hspace{1mm} (s,|y|)=(t,r) \} \hspace{4mm} \text{and} \hspace{4mm} V_u(t) := \{ (s,y) \in \mathbb{R}_+ \times \mathbb{R}^3 / \hspace{1mm} s \leq t, \hspace{1mm} s-|y| \leq u \}. &
\end{flalign*}
The volume form on $C_u(t)$ is given by $dC_u(t)=\sqrt{2}^{-1}r^{2}d\underline{u}d \mathbb{S}^{2}$, where $ d \mathbb{S}^{2}$ is the standard metric on the $2$ dimensional unit sphere.
\vspace{5mm}

\begin{tikzpicture}
\draw [-{Straight Barb[angle'=60,scale=3.5]}] (0,-0.3)--(0,5);
\fill[color=gray!35] (2,0)--(5,3)--(9.8,3)--(9.8,0)--(1,0);
\node[align=center,font=\bfseries, yshift=-2em] (title) 
    at (current bounding box.south)
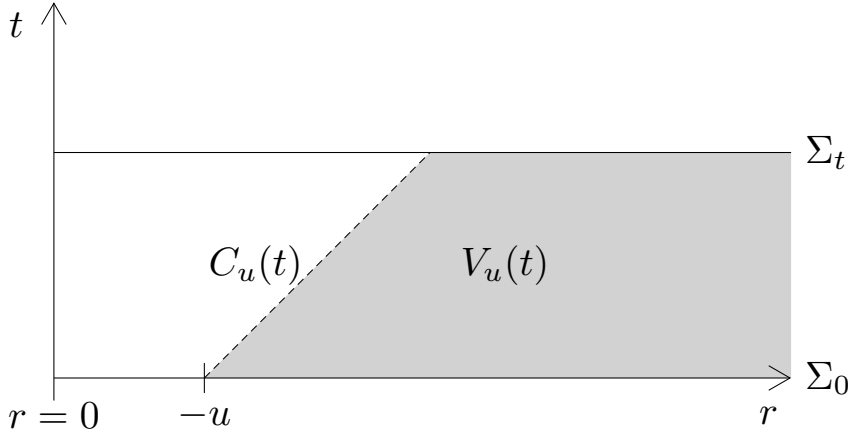
    {The sets $\Sigma_t$, $C_u(t)$ and $V_u(t)$};
\draw (0,3)--(9.8,3) node[scale=1.5,right]{$\Sigma_t$};
\draw (2,0.2)--(2,-0.2);
\draw [-{Straight Barb[angle'=60,scale=3.5]}] (0,0)--(9.8,0) node[scale=1.5,right]{$\Sigma_0$};
\draw[densely dashed] (2,0)--(5,3) node[scale=1.5,left, midway] {$C_u(t)$};
\draw (6,1.5) node[ color=black!100, scale=1.5] {$V_u(t)$}; 
\draw (0,-0.5) node[scale=1.5]{$r=0$};
\draw (2,-0.5) node[scale=1.5]{$-u$};
\draw (-0.5,4.7) node[scale=1.5]{$t$};
\draw (9.5,-0.5) node[scale=1.5]{$r$};   
\end{tikzpicture}

We will use the following subsets, given for $ \underline{u} \in \R_+$, specifically in the proof of Proposition \ref{Phi1},
$$ \underline{V}_{\underline{u}}(t) := \{ (s,y) \in \mathbb{R}_+ \times \mathbb{R}^3 / \hspace{1mm} s \leq t, \hspace{1mm} s+|y| \leq \underline{u} \}.$$
For $b \geq 0$ and $t \in \R_+$, define $\Sig^b_t$ and $\Si^b_t$ as
$$\Sig^b_t:= \{ t \} \times \{ x \in \R^3 \hspace{1mm} / \hspace{1mm} |x| \leq t-b \} \hspace{6mm} \text{and} \hspace{6mm} \Si^b_t:= \{ t \} \times \{ x \in \R^3 \hspace{1mm} / \hspace{1mm} |x| \geq t-b \}.$$
We also introduce a dyadic partition of $\R_+$ by considering the sequence $(t_i)_{i \in \mathbb{N}}$ and the functions $(T_i(t))_{i \in \mathbb{N}}$ defined by
$$t_0=0, \hspace{5mm} t_i = 2^i \hspace{5mm} \text{if} \hspace{5mm} i \geq 1, \hspace{5mm} \text{and} \hspace{5mm} T_{i}(t)= t \mathds{1}_{t \leq t_i}(t)+t_i \mathds{1}_{t > t_i}(t).$$
We then define the truncated cones $C^i_u(t)$ adapted to this partition by
$$C_u^i(t) := \left\{ (s,y) \in \R_+ \times \R^3 \hspace{2mm} / \hspace{2mm} t_i \leq s \leq T_{i+1}(t), \hspace{2mm} s-|y| =u \right\}= \left\{ (s,y) \in C_u(t) \hspace{2mm} / \hspace{2mm} t_i \leq s \leq T_{i+1}(t) \right\}.$$
The following lemma will be used several times during this paper. It depicts that we can foliate $[0,t] \times \R^3$ by $(\Sigma_s)_{0 \leq s \leq t}$, $(C_u(t))_{u \leq t}$ or $(C^i_u(t))_{u \leq t, i \in \mathbb{N}}$.
\begin{Lem}\label{foliationexpli}
Let $t>0$ and $g \in L^1([0,t] \times \R^3)$. Then
$$ \int_0^t \int_{\Sigma_s} g dx ds \hspace{2mm} = \hspace{2mm} \int_{u=-\infty}^t \int_{C_u(t)} g dC_u(t) \frac{du}{\sqrt{2}} \hspace{2mm} = \hspace{2mm} \sum_{i=0}^{+ \infty} \int_{u=-\infty}^t \int_{C^i_u(t)} g dC^i_u(t) \frac{du}{\sqrt{2}}.$$
\end{Lem}
Note that the sum over $i$ is in fact finite. The second foliation will allow us to exploit $t-r$ decay since $\| \tau_-^{-1} \|_{L^{\infty}(C_u(t)}=\tau_-^{-1}$ whereas $\|\tau_-^{-1}\|_{L^{\infty}(\Sigma_s)}=1$. The last foliation will be used to take advantage of time decay on $C_u(t)$ (the problem comes from $\|\tau_+^{-1}\|_{L^{\infty}(C_u(t))} = \tau_-^{-1}$). More precisely, let $0 < \delta < a$ and suppose for instance that, 
$$\forall \hspace{0.5mm} t \in [0,T[, \hspace{6mm} \int_{C_u(t)} g dC_u(t) \leq C (1+t)^{\delta}, \hspace{5mm} \text{so that} \hspace{5mm} \int_{C_u^i(t)} g dC^i_u(t) \leq C (1+T_{i+1}(t))^{\delta} \leq C (1+t_{i+1})^{\delta} .$$ Then,
$$ \int_{C_u(t)} \tau_+^{-a}g dC_u(t) \leq \sum_{i=0}^{+ \infty} \int_{C^i_u(t)} (1+s)^{-a} g dC^i_u(t) \leq  \sum_{i=0}^{+ \infty} (1+t_{i})^{-a} \int_{C^i_u(t)} g dC^i_u(t) \leq 3^aC \sum_{i=0}^{+ \infty} (1+2^{i+1})^{\delta-a}.$$
As $\delta-a <0$, we obtain a bound independent of $T$.
\subsection{An integral estimate}

A proof of the following inequality can be found in the appendix $B$ of \cite{FJS}.

\begin{Lem}\label{intesti}

Let $m \in \mathbb{N}^*$ and let $a$, $b \in \mathbb{R}$, such that $a+b >m$ and $b \neq 1$. Then
$$\exists \hspace{0.5mm} C_{a,b,m} >0, \hspace{0.5mm} \forall \hspace{0.5mm} t \in \mathbb{R}_+, \hspace{1.5cm} \int_0^{+ \infty} \frac{r^{m-1}}{\tau_+^a \tau_-^b}dr \leq C_{a,b,m} \frac{1+t^{b-1}}{1+t^{a+b-m}} .$$
\end{Lem}
\section{Vector fields and modified vector fields}\label{sec3}

For all this section, we consider $F$ a suffciently regular $2$-form.

\subsection{The vector fields of the Poincaré group and their complete lift}

We present in this section the commutation vector fields of the Maxwell equations and those of the relativistic transport operator (we will modified them to study the Vlasov equation). Let $\Po$ be the generators of Poincaré group of the Minkowski spacetime, i.e. the set containing
\begin{flalign*}
& \hspace{1cm} \bullet \text{the translations\footnotemark} \hspace{18mm} \partial_{\mu}, \hspace{2mm} 0 \leq \mu \leq 3, & \\
& \hspace{1cm} \bullet \text{the rotations} \hspace{25mm} \Omega_{ij}=x^i\partial_{j}-x^j \partial_i, \hspace{2mm} 1 \leq i < j \leq 3, & \\
& \hspace{1cm} \bullet \text{the hyperbolic rotations} \hspace{8mm} \Omega_{0k}=t\partial_{k}+x^k \partial_t, \hspace{2mm} 1 \leq k \leq 3. 
\end{flalign*}
\footnotetext{In this article, we will denote $\partial_{x^i}$, for $1 \leq i \leq 3$, by $\partial_{i}$ and sometimes $\partial_t$ by $\partial_0$.}
We also consider $\T:= \{ \partial_{t}, \hspace{1mm} \partial_1, \hspace{1mm} \partial_2, \hspace{1mm} \partial_3\}$ and $\Or := \{ \Omega_{12}, \hspace{1mm} \Omega_{13}, \hspace{1mm} \Omega_{23} \}$, the subsets of $\mathbb{P}$ containing respectively the translations and the rotational vector fields as well as $\mathbb{K}:= \Po \cup \{ S \}$, where $S=x^{\mu} \partial_{\mu}$ is the scaling vector field. The set $\mathbb{K}$ is well known for commuting with the wave and the Maxwell equations (see Subsection \ref{subseccomuMax}). However, to commute the operator $T=v^{\mu} \partial_{\mu}$, one should consider the complete lifts of the elements of $\Po$.
\begin{Def}\label{defliftcomplete}

Let $W=W^{\beta} \partial_{\beta}$ be a vector field. Then, the complete lift $\widehat{W}$ of $W$ is defined by
$$\widehat{W}=W^{\beta} \partial_{\beta}+v^{\gamma} \frac{\partial W^i}{\partial x^{\gamma}} \partial_{v^i}.$$
We then have $\widehat{\partial}_{\mu}=\partial_{\mu}$ for all $0 \leq \mu \leq 3$ and $$\widehat{\Omega}_{ij}=x^i \partial_j-x^j \partial_i+v^i \partial_{v^j}-v^j \partial_{v^i}, \hspace{2mm} \text{for} \hspace{2mm} 1 \leq i < j \leq 3, \hspace{6mm} \text{and} \hspace{6mm} \widehat{\Omega}_{0k} = t\partial_k+x^k \partial_t+v^0 \partial_{v^k}, \hspace{2mm} \text{for} \hspace{2mm} 1 \leq k \leq 3.$$
\end{Def}
One can check that $[T,\widehat{Z}]=0$ for all $Z \in \Po$. Since $[T,S]=T$, we consider
$$\K := \{ \widehat{Z} \hspace{1mm} / \hspace{1mm} Z \in \Po \} \cup \{ S \}$$
and we will, for simplicity, denote by $\widehat{Z}$ an arbitrary vector field of $\K$, even if $S$ is not a complete lift. The weights introduced in Subsection \ref{sectionweights} are, in a certain sense, preserved by the action of $\K$.

\begin{Lem}\label{weights}
Let $z \in \mathbf{k}_1$, $\widehat{Z} \in \K$ and $j \in \mathbb{N}$. Then
$$\widehat{Z}(v^0z) \in v^0 \mathbf{k}_1 \cup \{ 0 \} \hspace{8mm} \text{and} \hspace{8mm} \left| \widehat{Z} (z^j) \right| \leq 3j \sum_{w \in \mathbf{k}_1} |w|^j.$$
\end{Lem}

\begin{proof}
Let us consider for instance $tv^1-x^1v^0$, $x^1v^2-x^2v^1$, $\widehat{\Omega}_{01}$ and $\widehat{\Omega}_{02}$. We have
\begin{eqnarray}
\nonumber \widehat{\Omega}_{01}(x^1v^2-x^2v^1 ) & = & tv^2-x^2v^0, \hspace{2cm} \widehat{\Omega}_{01}(tv^1-x^1v^0) \hspace{2mm} = \hspace{2mm} 0, \\ \nonumber
\widehat{\Omega}_{02}(x^1v^2-x^2v^1 ) & = & x^1v^0-tv^1 \hspace{8mm} \text{and} \hspace{8mm} \widehat{\Omega}_{02}(tv^1-x^1v^0)  \hspace{2mm} = \hspace{2mm} x^2v^1-x^1v^2.
\end{eqnarray}
The other cases are similar. Consequently,
$$\left| \widehat{Z} (z^j) \right| = \left| \widehat{Z} \left(\frac{1}{(v^0)^j}(v^0z)^j \right) \right| \leq j|z|^j+\frac{j}{(v^0)^j}\left| \widehat{Z} \left(v^0z \right) \right| |v^0z|^{j-1} \leq j|z|^j +j\frac{|\widehat{Z}(v^0z)|^j}{(v^0)^j}+j|z|^j,$$
since $|w||z|^{a-1} \leq |w|^a+|z|^a$ when $a \geq 1$.
\end{proof}
The vector fields introduced in this section and the averaging in $v$ almost commute in the following sense (we refer to \cite{FJS} or to Lemma \ref{lift2} below for a proof).
\begin{Lem}\label{lift}
Let $f : [0,T[ \times \mathbb{R}^3_x \times \R^3_v \rightarrow \mathbb{R} $ be a sufficiently regular function. We have, almost everywhere,
$$\forall \hspace{0.5mm} Z \in \mathbb{K}, \hspace{8mm} \left|Z\left( \int_{v \in \R^3 } |f| dv \right) \right| \lesssim \sum_{ \begin{subarray}{} \widehat{Z}^{\beta} \in \K^{|\beta|} \\ \hspace{1mm} |\beta| \leq 1 \end{subarray}} \int_{v \in \R^3 } |\widehat{Z}^{\beta} f | dv .$$
Similar estimates hold for $\int_{v \in \mathbb{R}^3} (v^0)^k |f| dv$. For instance,
$$\left| S\left( \int_{v \in \R^3 } (v^0)^{-2}|f| dv \right) \right| \lesssim \int_{v \in \R^3 } (v^0)^{-2}|Sf| dv.$$
\end{Lem}
The vector spaces engendered by each of the sets defined in this section are actually algebras.
\begin{Lem}
Let $\mathbb{L}$ be either $\mathbb{K}$, $\Po$, $\Or$, $\T$ or $\K$. Then for all $(Z_1,Z_2) \in \mathbb{L}^2$, $[Z_1,Z_2]$ is a linear combinations of vector fields of $\mathbb{L}$. Note also that if $Z_2=\partial \in \T$, then $[Z_1,\partial]$ can be written as a linear combination of translations.
\end{Lem}
We consider an ordering on each of the sets $\mathbb{O}$, $\mathbb{P}$, $\mathbb{K}$ and $\widehat{\mathbb{P}}_0$. We take orderings such that, if $\mathbb{P}= \{ Z^i / \hspace{2mm} 1 \leq i \leq |\mathbb{P}| \}$, then $\mathbb{K}= \{ Z^i / \hspace{2mm} 1 \leq i \leq |\mathbb{K}| \}$, with $Z^{|\mathbb{K}|}=S$, and
$$ \K= \left\{ \widehat{Z}^i / \hspace{2mm} 1 \leq i \leq |\K| \right\}, \hspace{2mm} \text{with} \hspace{2mm} \left( \widehat{Z}^i \right)_{ 1 \leq i \leq |\Po|}=\left( \widehat{Z^i} \right)_{ 1 \leq i \leq |\Po|} \hspace{2mm} \text{and} \hspace{2mm} \widehat{Z}^{|\K|}=S  .$$
If $\mathbb{L}$ denotes $\mathbb{O}$, $\mathbb{P}$, $\mathbb{K}$ or $\widehat{\mathbb{P}}_0$, and  $\beta \in \{1, ..., |\mathbb{L}| \}^r$, with $r \in \mathbb{N}^*$, we will denote the differential operator $\Gamma^{\beta_1}...\Gamma^{\beta_r} \in \mathbb{L}^{|\beta|}$ by $\Gamma^{\beta}$. For a vector field $W$, we denote the Lie derivative with respect to $W$ by $\mathcal{L}_W$ and if $Z^{\gamma} \in \mathbb{K}^{r}$, we will write $\mathcal{L}_{Z^{\gamma}}$ for $\mathcal{L}_{Z^{\gamma_1}}... \mathcal{L}_{Z^{\gamma_r}}$. The following definition will be useful to lighten the notations in the presentation of commutation formulas.

\begin{Def}\label{goodcoeff}
We call good coefficient $c(t,x,v)$ any function $c$ of $(t,x,v)$ such that
$$ \forall \hspace{0.5mm} Q \in \mathbb{N}, \hspace{1mm} \exists \hspace{0.5mm} C_Q >0, \hspace{2mm} \forall \hspace{0.5mm} |\beta| \leq Q, \hspace{2mm} (t,x,v) \in \R_+ \times \R_x^3 \times \R_v^3 \setminus \{ 0 \} \times \{ 0 \} \times \R_v^3,  \hspace{8mm} \left| \widehat{Z}^{\beta} \left( c(t,x,v) \right) \right| \leq C_Q.$$
Similarly, we call good coefficient $c(v)$ any function $c$ such that
$$ \forall \hspace{0.5mm} Q \in \mathbb{N}, \hspace{2mm} \exists \hspace{0.5mm} C_Q >0, \hspace{2mm} \forall \hspace{0.5mm} |\beta| \leq Q, \hspace{2mm} v \in \R^3, \hspace{8mm} \left| \widehat{Z}^{\beta} \left( c(v) \right) \right| \leq C_Q.$$
Finally, we will say that $B$ is a linear combination, with good coefficients $c(v)$, of $(B^i)_{1 \leq i \leq M}$ if there exists good coefficients $(c_i(v))_{1 \leq i \leq M}$ such that $B=c_i B^i$. We define similarly a linear combination with good coefficients $c(t,x,v)$.
\end{Def}

These good coefficients introduced here are to be thought of bounded functions which remain bounded when they are differentiated (by $\K$ derivatives) or multiplied between them. In the remainder of this paper, we will denote by $c(t,x,v)$ (or $c_Z(t,x,v)$, $c_i(t,x,v)$) any such functions. Note that $\widehat{Z}^{\beta} \left( c(t,x,v) \right)$ is not necessarily defined on $\{ 0 \} \times \{ 0 \} \times \R_v^3$ as, for instance, $c(t,x,v)=\frac{x^1}{t+r} \frac{v^2}{v^0}$ satisfies these conditions. Typically, the good coefficients $c(v)$ will be of the form $\widehat{Z}^{\gamma} \left( \frac{v^i}{v^0} \right)$. 

Let us recall, by the following classical result, that the derivatives tangential to the cone behave better than others.

\begin{Lem}\label{goodderiv}
The following relations hold,
$$(t-r)\underline{L}=S-\frac{x^i}{r}\Omega_{0i}, \hspace{3mm} (t+r)L=S+\frac{x^i}{r}\Omega_{0i} \hspace{3mm} \text{and} \hspace{3mm} re_A=\sum_{1 \leq i < j \leq 3} C^{i,j}_A \Omega_{ij},$$
where the $C^{i,j}_A$ are uniformly bounded and depends only on spherical variables. In the same spirit, we have
$$(t-r)\partial_t =\frac{t}{t+r}S-\frac{x^i}{t+r}\Omega_{0i} \hspace{3mm} \text{and} \hspace{3mm} (t-r) \partial_i = \frac{t}{t+r} \Omega_{0i}- \frac{x^i}{t+r}S- \frac{x^j}{t+r} \Omega_{ij}.$$
\end{Lem}

As mentioned in the introduction, we will crucially use the vector fields $(X_i)_{1 \leq i \leq 3}$, defined by
\begin{equation}\label{eq:defXi}
X_i := \partial_i+\frac{v^i}{v^0}\partial_t.
\end{equation}
They provide extra decay in particular cases since
\begin{equation}\label{eq:Xi}
X_i= \frac{1}{t} \left( \Omega_{0i}+z_{0i} \partial_t \right).
\end{equation}
We also have, using Lemma \ref{goodderiv} and $(1+t+r)X_i=X_i+2tX_i+(r-t)X_i$, that there exists good coefficients $c_Z(t,x,v)$ such that
\begin{equation}\label{eq:X}
(1+t+r)X_i=2z_{0i} \partial_t +\sum_{Z \in \mathbb{K}} c_Z(t,x,v) Z.
\end{equation}
By a slight abuse of notation, we will write $\mathcal{L}_{X_i}(F)$ for $\mathcal{L}_{\partial_i}(F)+\frac{v^i}{v^0} \mathcal{L}_{\partial_t}(F)$. We are now interested in the compatibility of these extra decay with the Lie derivative of a $2$-form and its null decomposition.
\begin{Pro}\label{ExtradecayLie}
Let $G$ be a sufficiently regular $2$-form. Then, with $z=t\frac{v^i}{v^0}-x^i$ if $X=X_i$ and $\zeta \in \{ \alpha, \underline{\alpha}, \rho, \sigma \}$, we have
\begin{eqnarray}
\left| \mathcal{L}_{\partial}(G) \right| & \lesssim & \frac{1}{\tau_-} \sum_{Z \in \mathbb{K} } \left|\nabla_Z G \right| \hspace{2mm} \lesssim \hspace{2mm} \frac{1}{\tau_-} \sum_{ |\gamma| \leq 1 } \left| \mathcal{L}_{Z^{\gamma}}(G) \right| , \label{eq:goodlie} \\ 
\left| \mathcal{L}_{X}(G) \right| & \lesssim & \frac{1}{\tau_+} \left( |z| | \nabla_{\partial_t} G|+  \sum_{Z \in \mathbb{K}} \left|\nabla_Z G \right| \right), \label{eq:Xdecay} \\
\tau_-\left| \nabla_{\underline{L}} \zeta  \right|+\tau_+\left| \nabla_L \zeta  \right|+(1+r) \left| \slashed{\nabla} \zeta \right| & \lesssim & \sum_{|\gamma| \leq 1 } \left| \zeta \left( \mathcal{L}_{Z^{\gamma}}(G) \right) \right|, \label{eq:zeta} \\
\left| \zeta \left( \mathcal{L}_{\partial} (G) \right) \right| & \lesssim  & \sum_{|\gamma| \leq 1 } \frac{1}{\tau_-} \left| \zeta \left( \mathcal{L}_{Z^{\gamma}}(G) \right) \right|+\frac{1}{\tau_+} \left|  \mathcal{L}_{Z^{\gamma}}(G) \right|. \label{eq:zeta2}
\end{eqnarray}
\end{Pro}
\begin{proof}
To obtain the first two identities, use Lemma \ref{goodderiv} as well as \eqref{eq:X} and then remark that if $\Gamma$ is a translation or an homogeneous vector field, 
$$ |\nabla_{\Gamma}(G)| \lesssim \left| \mathcal{L}_{\Gamma}(G) \right|+|G|.$$
For \eqref{eq:zeta}, we refer to Lemma $D.2$ of \cite{massless}. Finally, the last inequality comes from \eqref{eq:goodlie} if $2t \leq \max(r,1)$ and from 
$$\partial_i=\frac{\Omega_{0i}}{t}-\frac{x^i}{2t} L-\frac{x^i}{2t} \underline{L} \hspace{1cm} \text{and} \hspace{1cm}  \eqref{eq:zeta} \hspace{6mm} \text{if} \hspace{3mm} 2t \geq \max(r,1).$$
\end{proof}

\begin{Rq}
We do not have, for instance, $\left| \rho \left( \mathcal{L}_{\partial_k} (G) \right) \right| \lesssim \sum_{|\gamma| \leq 1} \tau_-^{-1} \left| \rho \left( \mathcal{L}_{Z^{\gamma}} (G) \right) \right|$, for $1 \leq k \leq 3$.
\end{Rq}
\begin{Rq}
If $G$ solves the Maxwell equations $\nabla^{\mu} G_{\mu \nu} = J_{\nu}$ and $\nabla^{\mu} {}^* \! G_{\mu \nu} =0$, a better estimate can be obtained on $\alpha( \mathcal{L}_{\partial} (G) )$. Indeed, as $|\nabla_{\partial} \alpha | \leq | \nabla_{L} \alpha |+ |\underline{L} \alpha |+ | \slashed{\nabla} \alpha|$, \eqref{eq:zeta} and Lemma \ref{maxwellbis} gives us,
$$\forall \hspace{0.5mm} |x| \geq 1+\frac{t}{2}, \hspace{0.6cm} |\alpha( \mathcal{L}_{\partial} (G) ) |(t,x) \lesssim |J_A|+ \frac{1}{\tau_+} \sum_{|\gamma| \leq 1} \Big( |\alpha ( \mathcal{L}_{Z^{\gamma}} (G) ) |(t,x)+|\sigma ( \mathcal{L}_{Z^{\gamma}} (G) ) |(t,x)+|\rho ( \mathcal{L}_{Z^{\gamma}} (G) ) |(t,x) \Big).$$
We make the choice to work with \eqref{eq:zeta2} since it does not directly require a bound on the source term of the Maxwell equation, which lightens the proof of Theorem \ref{theorem} (otherwise we would have, among others, to consider more bootstrap assumptions).
\end{Rq}
\subsection{Modified vector fields and the first order commutation formula}

We start this section with the following commutation formula and we refer to Lemma $2.8$ of \cite{massless} for a proof\footnote{Note that a similar result is proved in Lemma \ref{calculF} below.}.

\begin{Lem}\label{basiccomuf}
If $\widehat{Z} \in \K \setminus \{ S \}$, then
$$[T_F,\widehat{Z}]( f)  = -\mathcal{L}_{Z}(F)(v,\nabla_v f) \hspace{8mm} \text{and} \hspace{8mm} [T_F,S]( f)  = F(v,\nabla_v f)-\mathcal{L}_{S}(F)(v,\nabla_v f).$$
\end{Lem}
In order to estimate quantities such as $\mathcal{L}_{Z}(F)(v,\nabla_v f)$, we rewrite $\nabla_v f$ in terms of the commutation vector fields (i.e. the elements of $\K$). Schematically, if we neglect the null structure of the system, we have, since $v^0\partial_{v^i}= \widehat{\Omega}_{0i}-t\partial_i-x^i\partial_t$,
\begin{eqnarray}
\nonumber \left| \mathcal{L}_{Z}(F)(v,\nabla_v f) \right| &  \lesssim  & v^0\left| \mathcal{L}_{Z}(F) \right| |\partial_{v} f | \\ \nonumber
& \sim & \tau_+ \left| \mathcal{L}_{Z}(F) \right| |\partial_{t,x} f |+\text{l.o.t.},
\end{eqnarray}
so that the $v$ derivatives engender a $\tau_+$-loss. The modified vector fields, constructed below, will allow us to absorb the worst terms in the commuted equations. 
\begin{Def}\label{defphi}
Let $\Y_0$ be the set of vector fields defined by
$$\Y_0:=\{ \widehat{Z}+\Phi_{\widehat{Z}}^j X_j \hspace{2mm} / \hspace{2mm} \widehat{Z} \in \K \setminus \T \},$$
where $\Phi_{\widehat{Z}}^j : [0,T] \times \R^n_x \times \R^n_v $ are smooth functions which will be specified below and the $X_j$ are defined in \eqref{eq:defXi}. We will denote $\widehat{\Omega}_{0k}+\Phi_{\widehat{\Omega}_{0k}}^j X_j$ by $Y_{0k}$ and, more generally, $\widehat{Z}+\Phi_{\widehat{Z}}^j X_j$ by $Y_{\widehat{Z}}$. We also introduce the sets $$\Y := \Y_0 \cup \T \hspace{8mm} \text{and} \hspace{8mm} \Y_X:= \Y \cup \{ X_1,X_2,X_3 \}.$$
We consider an ordering on $\Y$ and $\Y_X$ compatible with $\K$ in the sense that if $\Y = \{ Y^i \hspace{1mm} / \hspace{1mm} 1 \leq i \leq |\Y| \}$, then $Y^i=\widehat{Z}^i+\Phi^k_{\widehat{Z}^i}X_k$ or $Y^i=\partial_{\mu}=\widehat{Z}^i$. We suppose moreover that $X_j$ is the $(|\Y|+j)^{th}$ element of $\Y_X$. Most of the time, for a vector field $Y \in \Y_0$, we will simply write $Y=\widehat{Z}+\Phi X$. 

Let $\widehat{Z} \in \K \setminus \{S \}$ and $1 \leq k \leq 3$. $\Phi_{\widehat{Z}}^k$ and $ \Phi^k_S$ are defined such as
\begin{equation}\label{defPhicoeff}
\hspace{-1.5mm} T_F(\Phi^k_{\widehat{Z}} )=-t\frac{v^{\mu}}{v^0}\mathcal{L}_Z(F)_{\mu k}, \hspace{7.5mm}  T_F(\Phi^k_S)=t\frac{v^{\mu}}{v^0}\left(F_{\mu k}-\mathcal{L}_S(F)_{\mu k} \right) \hspace{6mm} \text{and} \hspace{6mm} \Phi_{\widehat{Z}}^k(0,.,.)=\Phi_{S}^k(0,.,.)=0.
\end{equation}
\end{Def}
As explained during the introduction, we consider the $X_i$ vector fields rather than translations in view of \eqref{eq:X}. We are then led to compute $[T_F,X_i]$.
\begin{Lem}\label{ComuX}
Let $1 \leq i \leq 3$. We have
$$[T_F,X_i]=-\mathcal{L}_{X_i}(F)(v,\nabla_v  )+\frac{v^{\mu}}{v^0} F_{\mu X_i} \partial_t.$$
\end{Lem}
\begin{proof}
One juste has to notice that
$$[T_F,X_i]=\frac{v^i}{v^0}[T_F,\partial_t]+[T_F,\partial_i]+F\left(v,\nabla_v \left( \frac{v^i}{v^0} \right) \right) \partial_t$$
and $v^{\mu} v^j F_{\mu j} =-v^{\mu} v^0 F_{\mu 0}$, as $F$ is antisymmetric.
\end{proof}
Finally, we study the commutator between the transport operator and these modified vector fields. The following relation,
\begin{equation}\label{eq:vderiv}
\partial_{v^i}=\frac{1}{v^0} \left( Y_{0i}-\Phi^j_{\widehat{\Omega}_{0i}} X_j-t X_i+ z_{0i}\partial_t \right),
\end{equation} 
will be useful to express the $v$ derivatives in terms of the commutation vector fields
\begin{Pro}\label{Comfirst}
Let $Y \in \Y_0 \backslash \{ Y_S \}$. We have, using \eqref{defPhicoeff}
\begin{eqnarray}
\nonumber [T_F,Y] & = & -\frac{v^{\mu}}{v^0}{\mathcal{L}_Z(F)_{\mu}}^j \left(Y_{0j}-\Phi^k_{\widehat{\Omega}_{0j}} X_k+z_{0j}\partial_t \right) -\Phi^j_{\widehat{Z}}\mathcal{L}_{X_j}(F)(v,\nabla_v )+\Phi^j_{\widehat{Z}}\frac{v^{\mu}}{v^0} F_{\mu X_j} \partial_t, \\ \nonumber
[T_F,Y_S] & = & \frac{v^{\mu}}{v^0}\left({F_{\mu}}^j-{\mathcal{L}_S(F)_{\mu}}^j \right) \left(Y_{0j}-\Phi^k_{\widehat{\Omega}_{0j}} X_k+z_{0j}\partial_t \right)-\Phi^j_{S}\mathcal{L}_{X_j}(F)(v,\nabla_v ) +\Phi^j_{S}\frac{v^{\mu}}{v^0} F_{\mu X_j} \partial_t. 
\end{eqnarray}
\end{Pro}
\begin{proof}
We only treat the case $Y \in \Y_0 \setminus \{ Y_S \}$ (the computations are similar for $Y_S$). Using Lemmas \ref{basiccomuf} and \ref{ComuX} as well as \eqref{eq:vderiv}, we have
\begin{eqnarray}
\nonumber [T_F,Y] & = & [T_F,\widehat{Z}]+[T_F,\Phi^j_{\widehat{Z}} X_j] \\ \nonumber 
& = & -\mathcal{L}_Z(F)(v,\nabla_v )+T_F(\Phi^j_{\widehat{Z}} ) X_j +\Phi^j_{\widehat{Z}} [T_F,X_j]. \\ \nonumber
& = & -\mathcal{L}_Z(F)(v,\nabla_v )+T_F(\Phi^j_{\widehat{Z}} )X_j -\Phi^j_{\widehat{Z}}\mathcal{L}_{X_j}(F)(v,\nabla_v  )+\Phi^j_{\widehat{Z}}\frac{v^{\mu}}{v^0} F_{\mu X_j} \partial_t
\\ \nonumber
& = & -\frac{v^{\mu}}{v^0}{\mathcal{L}_Z(F)_{\mu}}^j \left(Y_{0j}-\Phi^k_{\widehat{\Omega}_{0j}} X_k+z_{0j}\partial_t \right)+ \left(t\frac{v^{\mu}}{v^0}{\mathcal{L}_Z(F)_{\mu}}^j+T_F(\Phi^j_{\widehat{Z}} ) \right) X_j \\ \nonumber & & -\Phi^j_{\widehat{Z}}\mathcal{L}_{X_j}(F)(v,\nabla_v  )+\Phi^j_{\widehat{Z}}\frac{v^{\mu}}{v^0} F_{\mu X_j} \partial_t.
\end{eqnarray}
To conclude, recall from \eqref{defPhicoeff} that $t\frac{v^{\mu}}{v^0}{\mathcal{L}_Z(F)_{\mu}}^j+T_F(\Phi^j_{\widehat{Z}} )=0$.
\end{proof}
\begin{Rq}
As we will have $|\Phi| \lesssim \log^2(1+\tau_+)$, a good control on $z_{0j} \partial_t f$ and in view of the improved decay given by $X_j$ (see Proposition \ref{ExtradecayLie}), it holds schematically
$$ \left| [T_F,Y](f) \right| \lesssim \log^2 (1+\tau_+) \left| \mathcal{L}_Z(F)\right| |Y f|,$$
which is much better than $ \left| [T_F,\widehat{Z}](f) \right| \lesssim \tau_+ \left| \mathcal{L}_Z(F) \right| |\partial_{t,x} f|$.
\end{Rq}
Let us introduce some notations for the presentation of the higher order commutation formula.
\begin{Def}
Let $Y^{\beta} \in \Y^{|\beta|}$. We denote by $\beta_T$ the number of translations composing $Y^{\beta}$ and by $\beta_P$ the number of modified vector fields (the elements of $\Y_0$). Note that $\beta_T$ denotes also the number of translations composing $\widehat{Z}^{\beta}$ and $Z^{\beta}$ and $\beta_P$ the number of elements of $\K \setminus \T$ or $\mathbb{K} \setminus \T$. We have
$$|\beta|= \beta_T+\beta_P$$
and, for instance, if $Y^{\beta}=\partial_t Y_1 \partial_3$, $|\beta|=3$, $\beta_T=2$ and $\beta_P=1$. We define similarly $\beta_X$ if $Y^{\beta} \in \Y^{|\beta|}_X$.
\end{Def}
\begin{Def}\label{Pkp}
Let $k=(k_T,k_P) \in \mathbb{N}^2$ and $ p \in \mathbb{N}$. We will denote by $P_{k,p}(\Phi)$ any linear combination of terms such as
$$ \prod_{j=1}^p Y^{\beta_j}(\Phi), \hspace{3mm} \text{with} \hspace{3mm} Y^{\beta_j} \in \Y^{|\beta_j|}, \hspace{3mm} \sum_{j=1}^p |\beta_j| = |k|, \hspace{3mm} \sum_{j=1}^p  \left(\beta_j \right)_P  = k_P$$
and where $\Phi$ denotes any of the $\Phi$ coefficients. Note that $\sum_{j=1}^p \left(\beta_j \right)_T  = k_T$. Finally, if $ \min_{j} |\beta_j| \geq 1$, we will denote $\prod_{j=1}^p Y^{\beta_j}(\Phi)$ by $P_{\beta}(\Phi)$, where $\beta=(\beta_1,...\beta_p)$.
\end{Def}
\begin{Def}
Let $k=(k_T,k_P,k_X) \in \mathbb{N}^3$ and $ p \in \mathbb{N}$. We will denote by $P^X_{k,p}(\Phi)$ any linear combination of terms such as
$$ \prod_{j=1}^p Y^{\beta_j}(\Phi), \hspace{3mm} \text{with} \hspace{3mm} Y^{\beta_j} \in \Y^{|\beta_j|}, \hspace{3mm} \sum_{j=1}^p |\beta_j| = |k|, \hspace{3mm} \sum_{j=1}^p  \left(\beta_j \right)_P  = k_P, \hspace{3mm} \sum_{j=1}^p  \left(\beta_j \right)_X  = k_X \hspace{3mm} \text{and} \hspace{3mm} \min_{1 \leq j \leq p} \left( \beta_j \right)_X \geq 1.$$
We will also denote $ \prod_{j=1}^p Y^{\beta_j}(\Phi)$ by $P^X_{\beta}(\Phi)$.
\end{Def}
\begin{Rq}
For convenience, if $p=0$, we will take $P_{k,p}(\Phi)=1$. Similarly, if $|\beta|=0$, we will take $P_{\beta}(\Phi)=P^X_{\beta}(\Phi)=1$.
\end{Rq}
In view of presenting the higher order commutation formulas, let us gather the source terms in different categories.
\begin{Pro}\label{Comufirst}
Let $Y \in \Y \setminus \T$. In what follows, $0 \leq \nu \leq 3$. The commutator $[T_F,Y]$ can be written as a linear combination, with $c(v)$ coefficients, of terms such as
\begin{itemize}
\item $ \frac{v^{\mu}}{v^0}\mathcal{L}_{Z^{\gamma}}(F)_{\mu \nu} \Gamma $, where $|\gamma| \leq 1$ and $\Gamma \in \Y_0$.
\item $\Phi \frac{v^{\mu}}{v^0}\mathcal{L}_{Z^{\gamma}}(F)_{\mu \nu} \partial_{t,x} $, where $|\gamma| \leq 1$.
\item $z \frac{v^{\mu}}{v^0}\mathcal{L}_{Z^{\gamma}}(F)_{\mu \nu} \partial_{t,x} $, where $|\gamma| \leq 1$ and $z \in \mathbf{k}_1$.
\item $\Phi \mathcal{L}_{X}(F)(v,\nabla_v )$.
\end{itemize}
\end{Pro}
Finally, let us adapt Lemma \ref{lift} to our modified vector fields.
\begin{Lem}\label{lift2}
Let $f : [0,T[ \times \mathbb{R}^3_x \times \R^3_v \rightarrow \mathbb{R} $ be a sufficiently regular function and suppose that for all $|\beta| \leq 1$, $|Y^{\beta} \Phi| \lesssim \log^{\frac{7}{2}}(1+\tau_+)$. Then, we have, almost everywhere,
$$\forall \hspace{0.5mm} Z \in \mathbb{K}, \hspace{8mm} \left|Z\left( \int_{v \in \R^3 } |f| dv \right) \right| \lesssim \sum_{ \begin{subarray}{} Y \in \Y \\ z \in \mathbf{k}_1 \end{subarray}} \int_{v \in \R^3 } \left(  |Yf|+|f| +|X(\Phi)f| + \frac{ \log^7 (1 + \tau_+)}{\tau_+} \left(|z \partial_t f|+|zf| \right) \right) dv .$$
\end{Lem}

\begin{proof}
Consider, for instance, the rotation $\Omega_{12}$. We have by integration by parts, as $\Omega_{12}=\widehat{\Omega}_{12}-v^{1} \partial_{v^2}+v^2 \partial_{v^1}$,
$$  \Omega_{12}\left( \int_{v \in \R^3 } |f| dv \right) = \int_{v \in  \R^3}  \widehat{\Omega}_{12} (|f|) dv -\int_{v \in \R^3}  \left( v^1\partial_{v^2} -v^2 \partial_{v^1}  \right)(|f|) dv= \int_{v \in  \R^3}  \widehat{\Omega}_{12} (|f|) dv.$$
This proves Lemma \ref{lift} for $\Omega_{12}$ since $| \widehat{\Omega}_{12} (|f|) |= | \frac{f}{|f|}\widehat{\Omega}_{12} (f) | \leq |\widehat{\Omega}_{12} (f)|$. On the other hand,
\begin{eqnarray}\label{45:eq}
 \int_{v \in  \R^3}  \widehat{\Omega}_{12} (|f|) dv & = & \int_{v \in \R^3} \left( \widehat{\Omega}_{12} +\Phi_{\widehat{\Omega}_{12}}^k X_k -\Phi_{\widehat{\Omega}_{12}}^k X_k  \right)(|f|) dv  \\  
& = & \int_{v \in \R^3} \frac{f}{|f|} Y_{\widehat{\Omega}_{12}} f dv+\int_{v \in \R^3} X_k \left( \Phi^k_{\widehat{\Omega}_{12}} \right) |f| dv -\int_{v \in \R^3} X_k \left( \Phi^k_{\widehat{\Omega}_{12}} |f| \right) dv \label{eq:lif33}. 
\end{eqnarray}
\eqref{45:eq} implies the result if $t+r \leq 1$. Otherwise, if $t \geq r$, note that by \eqref{eq:Xi},
\begin{eqnarray}
\nonumber \int_{v \in \R^3} X_k \left( \Phi^k_{\widehat{\Omega}_{12}} |f| \right) dv & = & \frac{1}{t}\int_{v \in \R^3} \left( \Omega_{0k}+z_{0k} \partial_t \right)  \left( \Phi^k_{\widehat{\Omega}_{12}} |f| \right) dv \\ \nonumber 
& = & \frac{1}{t}\int_{v \in \R^3} \left( Y_{0k}-v^0\partial_{v^k}-\Phi^q_{\widehat{\Omega}_{0k}}X_q+z_{0k} \partial_t \right)  \left( \Phi^k_{\widehat{\Omega}_{12}} |f| \right)dv \\ \nonumber & = & \frac{1}{t}\int_{v \in \R^3} \left( Y_{0k}+\frac{v_k}{v^0}-\Phi^q_{\widehat{\Omega}_{0k}}X_q+z_{0k} \partial_t \right)  \left( \Phi^k_{\widehat{\Omega}_{12}} |f| \right)dv .
\end{eqnarray}
Consequently, in view of the bounds on $Y^{\beta} \Phi$ for $|\beta| \leq 1$,
$$ \left| \int_{v \in \R^3} X_k \left( \Phi^k_{\widehat{\Omega}_{12}} |f| \right) dv \right|  \lesssim  \sum_{Y \in \Y} \sum_{z \in \mathbf{k}_1} \int_{v \in \R^3} |Yf|+|f|+\frac{|z| \log^7 (1+t)}{t} \left( |\partial_t f|+|f| \right) dv ,$$
and it remains to combine it with \eqref{eq:lif33}. When $t \leq r$, one can use $rX_k=tX_k+(r-t)X_k$ and Lemma \ref{goodderiv}.
\end{proof}

\begin{Rq}\label{lift3}
If moreover $|\Phi| \lesssim \log^2(1+\tau_+)$, one can prove similarly that, for $Z \in \mathbb{K}$, $z \in \mathbf{k}_1$ and $j \in \mathbb{N}^*$,
$$ \left|Z  \left( \int_{v  }  |z^jf| dv  \right)  \right| \lesssim \hspace{1mm} j \hspace{-1mm} \sum_{ \begin{subarray}{} |\xi|+|\beta| \leq 1 \\ \hspace{2.5mm} w \in \mathbf{k}_1 \end{subarray}} \int_{v  } |w^jP^X_{\xi}(\Phi) Y^{\beta} f|+\log^2(3+t)|w^{j-1}f|+\frac{ \log^7(1+\tau_+)|w|^{j+1}}{\tau_+} \left(| \partial_t f|+|f| \right)  dv .$$
To prove this inequality, apply Lemma \ref{lift2} to $z^j f$ and use the two following properties,
$$|Y(z^j)| \leq |\widehat{Z}(z^j)|+|\Phi X(z^j)| \lesssim j \left( \sum_{w \in \mathbf{k}_1} |w|^j+ \log^2(1+\tau_+) |z|^{j-1} \right) \hspace{6mm} \text{and} \hspace{6mm}  \sum_{w \in \mathbf{k}_1} |w||z|^j \lesssim \sum_{w \in \mathbf{k}_1} |w|^{j+1}.$$
It remains to apply Remark \ref{rqweights1} in order to get
$$\forall \hspace{0.5mm} |x| \geq 1+2t, \hspace{1cm} \log^2(1+\tau_+) |z|^{j-1} \lesssim \frac{\log^2(3+r)}{r} \sum_{w \in \V} |w z^{j-1} | \lesssim  \sum_{w \in \V} |w^j|$$
and to note that $\log (1+\tau_+) \lesssim \log (3+t)$ if $|x| \leq 1+2t$.
\end{Rq}

\subsection{Higher order commutation formula}

The following lemma will be useful for upcoming computations.
\begin{Lem}\label{calculF}
Let $G$ be a sufficiently regular $2$-form and $g$ a sufficiently regular function defined respectively on $[0,T[ \times \R^3$ and $[0,T[ \times \R^3_x \times \R^3_v$. Let also $Y=\widehat{Z}+\Phi X \in \Y_0$ and $\nu \in \llbracket 0,3 \rrbracket$. We have, with $n_Z=0$ is $Z \in \mathbb{P}$ and $n_S=-1$,
\begin{eqnarray}
\nonumber Y \left( v^{\mu}G_{\mu \nu} \right) \hspace{-1.5mm} & = & \hspace{-1.5mm} v^{\mu}\mathcal{L}_{Z }(G)_{\mu \nu}+n_Zv^{\mu} G_{\mu \nu} +\Phi v^{\mu}\mathcal{L}_{X }(G)_{\mu \nu}+v^{\mu}G_{\mu [Z,\partial_{\nu}]}, \\
\nonumber Y \left( G \left( v , \nabla_v g \right) \right) \hspace{-1.5mm} & = & \hspace{-1.5mm} \mathcal{L}_Z(G) \left( v , \nabla_v g \right)+2n_ZG \left( v ,\nabla_v g \right)+\Phi \mathcal{L}_X(G) \left( v , \nabla_v g \right)+G \left( v , \nabla_v \widehat{Z} g \right)+c(v)\Phi G \left( v , \nabla_v \partial g \right) .
\end{eqnarray}
For $i \in \llbracket 1,3 \rrbracket$, $ Y \left( v^{\mu} \mathcal{L}_{X_i}(G)_{\mu \nu} \right)$ can be written as a linear combination, with $c(v)$ coefficients, of terms of the form
$$ \Phi^p v^{\mu} \mathcal{L}_{X Z^{\gamma} }(G)_{\mu \theta}, \hspace{3mm} \text{with} \hspace{3mm} 0 \leq \theta \leq 3 \hspace{3mm} \text{and} \hspace{3mm} \max(p,|\gamma|) \leq 1.$$
Finally, $Y \left( \mathcal{L}_{X_i}(G) \left( v , \nabla_v g \right) \right)$ can be written as a linear combination, with $c(v)$ coefficients, of terms of the form
$$ \Phi^p \mathcal{L}_{X Z^{\gamma} }(G) \left( v, \nabla \left( \widehat{Z}^{\kappa} g \right) \right), \hspace{3mm} \text{with} \hspace{3mm} \max(|\gamma|+|\kappa|,p+\kappa_P) \leq 1.$$
\end{Lem}

\begin{proof}
Let $Z_v=\widehat{Z}-Z$ so that $Y=Z+Z_v+\Phi X$. We prove the second and the fourth properties (the first and the third ones are easier). We have
\begin{eqnarray}
\nonumber Y \left( G \left( v , \nabla_v g \right) \right) & = & \mathcal{L}_Z(G) \left( v, \nabla_v g \right)+G \left( [Z,v], \nabla_v g \right)+G \left( v,[Z,\nabla_v g] \right)+G \left( Z_v(v), \nabla_v g\right)+G \left( v, Z_v \left(\nabla_v g \right) \right) \\ \nonumber & &+\Phi \mathcal{L}_{X}(G) \left( v, \nabla_v g \right)+c(v) \Phi G \left( v , \nabla_v \partial g \right).
\end{eqnarray}
Note now that
\begin{itemize}
\item $S_v=0$ and $[S,v]=-v$,
\item $[Z,v]=-Z_v(v)$ if $Z \in \mathbb{P}$.
\end{itemize}
The second identity is then implied by
\begin{itemize}
\item $[\partial, \nabla_v g]=\nabla_v \partial(g)$ and $[S, \nabla_v g ]= \nabla_v S(g)-\nabla_v g$.
\item $[Z, \nabla_v g]+Z_v \left( \nabla_v g \right)= \nabla_v \widehat{Z}(g)$ if $Z \in \mathbb{O}$.
\item $[\Omega_{0i}, \nabla_v g]+(\Omega_{0i})_v \left( \nabla_v g \right)= \nabla_v \widehat{Z}(g)-\frac{v}{v^0} \partial_{v^i}$ and $G(v,v)=0$ as $G$ is a $2$-form.
\end{itemize}
We now prove the fourth identity. We treat the case $Y=\widehat{Z}+\Phi X \in \Y_0 \setminus \{ Y_S \}$ as the computations are similar for $Y_S$. On the one hand, since $[\partial,X_i]=0$ and $X_k= \partial_k+\frac{v^k}{v^0}\partial_t$, one can easily check that $\Phi X_k \left( \mathcal{L}_{X_i}(G) \left( v , \nabla_v g \right) \right)$ gives four terms of the expected form. On the other hand,
$$\widehat{Z} \left( \mathcal{L}_{X_i}(G) \left( v , \nabla_v g \right)  \right)=\widehat{Z} \left( \mathcal{L}_{\partial_i}(G) \left( v , \nabla_v g \right) \right) +\widehat{Z} \left(\frac{v^i}{v^0} \mathcal{L}_{\partial_t}(G) \left( v , \nabla_v g \right)  \right).$$
Applying the second equality of this Lemma to $\mathcal{L}_{\partial}(G)$, $g$ and $\widehat{Z}$ (which is equal to $Y$ when $\Phi=0$), we have
\begin{eqnarray}
\nonumber \widehat{Z} \left(  \mathcal{L}_{\partial_i}(G) \left( v , \nabla_v g \right)  \right) & = & \mathcal{L}_{Z \partial_i}(G) \left( v , \nabla_v g \right) +\mathcal{L}_{\partial_i}(G) \left( v , \nabla_v \widehat{Z} g \right)  \\ \nonumber 
\widehat{Z} \left(\frac{v^i}{v^0} \mathcal{L}_{\partial_t}(G) \left( v , \nabla_v g \right)  \right) & = & \widehat{Z} \left( \frac{v^i}{v^0} \right) \mathcal{L}_{\partial_t}(G) \left( v , \nabla_v g \right)+\frac{v^i}{v^0}\mathcal{L}_{Z \partial_t}(G) \left( v , \nabla_v g \right) +\frac{v^i}{v^0}\mathcal{L}_{\partial_t}(G) \left( v , \nabla_v \widehat{Z} g \right)
\end{eqnarray}
The sum of the last terms of these two identities is of the expected form. The same holds for the sum of the three other terms since 
\begin{eqnarray}
\nonumber [\Omega_{0j},\partial_i]+\frac{v^i}{v^0}[\Omega_{0j},\partial_t]+v^0 \partial_{v^j}\left( \frac{v^i}{v^0} \right) \partial_t \hspace{-2mm} & = & \hspace{-2mm} -\delta_{j}^{i} \partial_t-\frac{v^i}{v^0}\partial_j-\frac{v^i v^j}{(v^0)^2} \partial_t+\delta_{j}^{i} \partial_t= -\frac{v^i}{v^0} X_j=c(v) X_j, \\ \nonumber
[\Omega_{kj},\partial_i]+\frac{v^i}{v^0}[\Omega_{kj},\partial_t]+\left( v^k \partial_{v^j}-v^j \partial_{v^k} \right) \left( \frac{v^i}{v^0} \right) \partial_t \hspace{-2mm} & = & \hspace{-2mm} \delta_{j}^i \partial_k-\delta_k^i \partial_j+\left(\frac{v^k \delta_j^i-v^j \delta_k^i}{v^0} \right)\partial_t= \delta_j^i X_k -\delta_k^i X_j,\\ \nonumber
[S,\partial_i]+\frac{v^i}{v^0}[S,\partial_t] \hspace{-2mm} & = & \hspace{-2mm} - \partial_i-\frac{v^i}{v^0}\partial_t=- X_i .
\end{eqnarray}
\end{proof}

We are now ready to present the higher order commutation formula. To lighten its presentation and facilitate its future usage, we introduce $\mathbb{G}:= \K \cup \Y_0$, on which we consider an ordering. A combination of vector fields of $\mathbb{G}$ will always be denoted by $\Gamma^{\sigma}$ and we will also denote by $\sigma_T$ its number of translations and by $\sigma_P= |\sigma|-\sigma_T$ its number of homogeneous vector fields. In Lemma \ref{GammatoYLem} below, we will express $\Gamma^{\sigma}$ in terms of $\Phi$ coefficients and $\Y$ vector fields. 
 
\begin{Pro}\label{ComuVlasov}
Let $\beta$ be a multi-index. In what follows, $\nu \in \llbracket 0 , 3 \rrbracket$. The commutator $[T_F,Y^{\beta}]$ can be written as a linear combination, with $c(v)$ coefficients, of the following terms.
\begin{itemize}
\item \begin{equation}\label{eq:com1}
 z^d P_{k,p}(\Phi) \frac{v^{\mu}}{v^0}\mathcal{L}_{Z^{\gamma}}(F)_{ \mu \nu} Y^{\sigma}, \tag{type 1-$\beta$}
 \end{equation}
where \hspace{2mm} $z \in \mathbf{k}_1$, \hspace{2mm} $d \in \{ 0,1 \}$, \hspace{2mm} $|\sigma| \geq 1$ \hspace{2mm} $\max ( |\gamma|, |k|+|\gamma|, |k|+|\sigma| ) \leq |\beta|$, \hspace{2mm} $|k|+|\gamma|+|\sigma| \leq |\beta|+1$ \hspace{2mm} and \hspace{2mm} $p+k_P+\sigma_P+d \leq \beta_P$. Note also that, as \hspace{2mm} $|\sigma| \geq 1$, \hspace{2mm} $|k| \leq |\beta|- 1$.
\item \begin{equation}\label{eq:com2}
P_{k,p}(\Phi) \mathcal{L}_{X Z^{\gamma_0}}(F) \left( v, \nabla_v  \Gamma^{\sigma}  \right), \tag{type 2-$\beta$}
\end{equation}
where \hspace{2mm} $|k|+|\gamma_0|+|\sigma| \leq |\beta|-1$, \hspace{2mm} $p+k_P+\sigma_P \leq \beta_P$ \hspace{2mm} and \hspace{2mm} $p \geq 1$. 
\item \begin{equation}\label{eq:com4}
 P_{k,p}(\Phi) \mathcal{L}_{ \partial Z^{\gamma_0}}(F) \left( v,\nabla_v  \Gamma^{\sigma}  \right), \tag{type 3-$\beta$}
 \end{equation}
where \hspace{2mm} $|k|+|\gamma_0|+|\sigma| \leq |\beta|-1$, \hspace{2mm} $p+|\gamma_0| \leq |\beta|-1$ \hspace{2mm} and \hspace{2mm} $p+k_P+\sigma_P \leq \beta_P$.

\end{itemize}
\end{Pro}

\begin{proof}
The result follows from an induction on $|\beta|$, Proposition \ref{Comufirst} (which treats the case $|\beta| =1$) and
$$[T_F,YY^{\beta_0}]=Y[T_F,Y^{\beta_0}]+[T_F,Y]Y^{\beta_0}.$$
Let $ Q \in \mathbb{N}$ and suppose that the commutation formula holds for all $|\beta_0| \leq Q$. We then fix a multi-index $|\beta_0|=Q$, consider $Y \in Y$ and denote the multi-index corresponding to $YY^{\beta_0}$ by $\beta$. Then, $|\beta|=|\beta_0|+1$. 

Suppose first that $Y=\partial$ is a translation so that $\beta_P=(\beta_0)_P$. Then, using Lemma \ref{basiccomuf}, we have
$$ [T_F,\partial]Y^{\beta_0} = -\mathcal{L}_{\partial}(F)(v,\nabla_v Y^{\beta_0}), $$
which is a term of \eqref{eq:com4} as $|\beta_0| = |\beta|-1$ and $(\beta_0)_P=\beta_P$. Using the induction hypothesis, $\partial[T_F,Y^{\beta_0}]$ can be written as a linear combination with good coefficients $c(v)$ of terms of the form\footnote{We do not mention the $c(v)$ coefficients here since $\partial \left( c(v) \right) =0$.}
\begin{itemize}
\item $ \partial \left( z^d P_{k,p}(\Phi) \frac{v^{\mu}}{v^0}\mathcal{L}_{Z^{\gamma}}(F)_{ \mu \nu} Y^{\sigma} \right) $, with $z \in \mathbf{k}_1$, $d \in \{0,1 \}$, $|\sigma| \geq 1$, $\max ( |\gamma|, |k|+|\gamma|, |k|+ |\sigma| ) \leq |\beta_0|$, $|k|+|\gamma|+|\sigma| \leq |\beta_0|+1$ and $p+k_P+\sigma_P+d \leq (\beta_0)_P$. This leads to the sum of the following terms.
\begin{itemize}
\item $\partial(z^d) P_{k,p}(\Phi) \frac{v^{\mu}}{v^0}\mathcal{L}_{Z^{\gamma}}(F)_{ \mu \nu} Y^{\sigma}$, which is of \eqref{eq:com1} since $\partial(z)=0$ or $\frac{v^{\lambda}}{v^0}$.
\item $z^d P_{(k_T+1,k_P),p}(\Phi) \frac{v^{\mu}}{v^0}\mathcal{L}_{Z^{\gamma}}(F)_{ \mu \nu} Y^{\sigma}+z^dP_{k,p}(\Phi) \frac{v^{\mu}}{v^0}\mathcal{L}_{\partial Z^{\gamma}}(F)_{ \mu \nu} Y^{\sigma}+z^dP_{k,p}(\Phi) \frac{v^{\mu}}{v^0}\mathcal{L}_{Z^{\gamma}}(F)_{ \mu \nu} \partial Y^{\sigma},$ which is the sum of terms of \eqref{eq:com1} (as, namely, $k_P$ does not increase and $(\sigma_0)_P=\sigma_P$ if $Y^{\sigma_0}=\partial Y^{\sigma}$).
\end{itemize}
\item $\partial \left( P_{k,p}(\Phi) \mathcal{L}_{ \partial Z^{\gamma_0}}(F) \left( v,\nabla_v \Gamma^{\sigma} \right) \right)$, with $|k|+|\gamma_0|+|\sigma| \leq |\beta_0|-1$, $p+|\gamma_0| \leq |\beta_0|-1$ and $p+k_P+\sigma_P \leq (\beta_0)_P$. We then obtain
$$ P_{(k_T+1,k_P),p}(\Phi) \mathcal{L}_{ \partial Z^{\gamma_0}}(F) \hspace{-0.2mm} \left( v,\nabla_v \Gamma^{\sigma} \right), \hspace{2.3mm} P_{k,p}(\Phi)\mathcal{L}_{ \partial \partial Z^{\gamma_0}}(F) \hspace{-0.2mm} \left( v,\nabla_v \Gamma^{\sigma} \right) \hspace{2.3mm} \text{and} \hspace{2.3mm} P_{k,p}(\Phi) \mathcal{L}_{ \partial Z^{\gamma_0}}(F) \hspace{-0.2mm} \left( v,\nabla_v \partial \Gamma^{\sigma}  \right),$$
which are all of \eqref{eq:com4} since $|k|+|\gamma_0|+|\sigma|+1 \leq |\beta_0|=|\beta|-1$, $p+|\gamma_0|+1 \leq |\beta|-1$ and, if $\Gamma^{\overline{\sigma}} = \partial \Gamma^{\sigma}$, $p+k_P+\overline{\sigma}_P=p+k_P+\sigma_P \leq \left( \beta_0 \right)_P = \beta_P$.
\item $\partial \left( P_{k,p}(\Phi) \mathcal{L}_{ X Z^{\gamma_0}}(F) \left( v,\nabla_v \Gamma^{\sigma} \right) \right)$, with $|k|+|\gamma_0|+|\sigma| \leq |\beta_0|-1$, $p+k_P+\sigma_P \leq (\beta_0)_P$ and $p \geq1$. We then obtain, as $[\partial,X]=0$,
$$ P_{(k_T+1,k_P),p}(\Phi) \mathcal{L}_{ X Z^{\gamma_0}}(F) \hspace{-0.2mm} \left( v,\nabla_v \Gamma^{\sigma} \right) \hspace{-0.3mm} , \hspace{1.7mm} P_{k,p}(\Phi)\mathcal{L}_{X \partial Z^{\gamma_0}}(F) \hspace{-0.2mm} \left( v,\nabla_v \Gamma^{\sigma} \right) \hspace{1.7mm} \text{and} \hspace{1.7mm} P_{k,p}(\Phi)\mathcal{L}_{ X Z^{\gamma_0}}(F) \hspace{-0.2mm} \left( v,\nabla_v \partial \Gamma^{\sigma}  \right) \hspace{-0.3mm} ,$$
which are all of \eqref{eq:com2} since, for instance, $|k|+|\gamma_0|+|\sigma|+1 \leq |\beta_0| = |\beta|-1$.
\end{itemize}
We now suppose that $Y \in \Y \setminus \T$, so that $\beta_P = (\beta_0)_P+1$. We will write schematically that $Y=\widehat{Z}+\Phi X$. Using Proposition \ref{Comufirst}, we have that $[T_F,Y]Y^{\beta_0}$ can be written as a linear combination, with $c(v)$ coefficients, of the following terms.
\begin{itemize}
\item $ \frac{v^{\mu}}{v^0}\mathcal{L}_{Z^{\gamma}}(F)_{\mu \nu} \Gamma Y^{\beta_0} $, where $|\gamma| \leq 1$ and $\Gamma \in \Y$, which is of \eqref{eq:com1}.
\item $\Phi^{1-d}z^d \frac{v^{\mu}}{v^0}\mathcal{L}_{Z^{\gamma}}(F)_{\mu \nu} \partial Y^{\beta_0}$, where $|\gamma| \leq 1$, $d \in \{0,1 \}$ and $z \in \mathbf{k}_1$, which is of \eqref{eq:com1} since, if $\xi$ is the multi-index corresponding to $\partial Y^{\beta_0}$, $\xi_P = (\beta_0)_P < \beta_P$.
\item $ \Phi \mathcal{L}_{X}(F)(v,\nabla_v Y^{\beta_0} )$, which is of \eqref{eq:com2} since $|\beta_0| \leq |\beta|-1$ and $1+(\beta_0)_P \leq \beta_P$.
\end{itemize}
It then remains to compute $Y[T_F,Y^{\beta_0}]$. Using the induction hypothesis, it can be written as a linear combination of terms of the form
\begin{itemize}
\item $ Y \left(c(v) z^d P_{k,p}(\Phi) \frac{v^{\mu}}{v^0}\mathcal{L}_{Z^{\gamma}}(F)_{\mu \nu} Y^{\sigma} \right),$ with $z \in \mathbf{k}_1$, $d \in \{0,1 \}$, $|\sigma| \geq 1$, $\max ( |\gamma|,|k|+|\gamma|, |k|+ |\sigma| ) \leq |\beta_0|$, $|k|+|\gamma|+|\sigma| \leq |\beta_0|+1$ and $p+k_P+\sigma_P+d \leq (\beta_0)_P$. It leads to the following error terms.
\begin{itemize}
\item $ Y\left( \frac{c(v)}{v^0} \right) z^dP_{k,p}(\Phi) v^{\mu}\mathcal{L}_{Z^{\gamma}}(F)_{\mu \nu} Y^{\sigma} $, which is of \eqref{eq:com1} since $Y\left( \frac{c(v)}{v^0}  \right) = \widehat{Z} \left( \frac{c(v)}{v^0}  \right) = \frac{c_0(v)}{v^0} $. 
\item $c(v)Y \left( z^d \right) P_{k,p}(\Phi) \frac{v^{\mu}}{v^0}\mathcal{L}_{Z^{\gamma}}(F) Y^{\sigma}$, which is a linear combination of terms of \eqref{eq:com1} since, by Lemma \ref{weights},
$$Y(z)=\widehat{Z}(z)+\Phi^i_{\widehat{Z}} X_i(z)=c_0(v)z+z'+\Phi^i_{\widehat{Z}}c_i(v), \hspace{2mm} \text{where} \hspace{2mm} z' \in \mathbf{k}_1, \hspace{2mm} \text{and} \hspace{2mm} p+1+k_P+\sigma_P+1 \leq  \beta_P.$$
\item $c(v)z^d P_{(k_T,k_P+1),p}(\Phi) \frac{v^{\mu}}{v^0}\mathcal{L}_{Z^{\gamma}}(F)_{ \mu \nu} Y^{\sigma}+c(v)z^d P_{k,p}(\Phi) \frac{v^{\mu}}{v^0}\mathcal{L}_{Z^{\gamma}}(F)_{ \mu \nu} YY^{\sigma}$, which is the sum of terms of \eqref{eq:com1}, since $p+k_P+\sigma_P+d+1 \leq (\beta_0)_P+1 = \beta_P$.
\item $c(v)z^dP_{k,p+p_0}(\Phi) \frac{v^{\mu}}{v^0}\mathcal{L}_{Z^{\xi}Z^{\gamma}}(F)_{ \mu \theta} Y^{\sigma}$, with $\max (p_0 ,|\xi| ) \leq 1$, which is given by the first identity of Lemma \ref{calculF}.  These terms are of \eqref{eq:com1} since $|k|+|\gamma|+|\xi|+|\sigma| \leq |\beta_0|+2 = |\beta|+1$ and $|\gamma|+|\xi| \leq |\beta|$.
\end{itemize}
For the remaining terms, we suppose for simplicity that $c(v)=1$, as we have just see that $Y \left(c(v) \right)$ is a good coefficient.
\item $ Y \Big( P_{k,p}(\Phi) \mathcal{L}_{ X Z^{\gamma_0}}(F) \left( v , \nabla_v \Gamma^{\sigma} \right) \Big) $, with $|k|+|\gamma_0|+|\sigma| \leq |\beta_0|-1$, $p+k_P+\sigma_P \leq (\beta_0)_P$ and $p \geq 1$. It gives us
$$P_{(k_T,k_P+1),p}(\Phi) \mathcal{L}_{X Z^{\gamma_0}}(F) \left( v , \nabla_v \Gamma^{\sigma} \right), $$ which is of \eqref{eq:com2} since, $p+k_P+1+\sigma_P \leq (\beta_0)_P+1=\beta_P$. We also obtain, using the fourth identity of Lemma \ref{calculF}, 
$$c(v)P_{k,p+p_0}(\Phi)\mathcal{L}_{X Z^{\delta} Z^{\gamma_0}} (F) \left( v , \nabla_v \widehat{Z}^{\xi}\Gamma^{\sigma}  \right), \hspace{3mm} \text{with} \hspace{3mm} \max(|\delta|+|\xi|,p_0+ \xi_P) \leq 1.$$
They are all of \eqref{eq:com2} since $|k|+|\gamma_0|+|\delta|+|\sigma|+|\xi| \leq |\beta_0|=|\beta|-1$, $p+p_0+k_P+\sigma_P+\xi_P \leq (\beta_0)_P+1=\beta_P$ and $p+p_0 \geq p \geq 1$.
\item $ Y \Big(P_{k,p}(\Phi) \mathcal{L}_{\partial Z^{\gamma_0}}(F) \left( v , \nabla_v \Gamma^{\sigma} \right) \Big) $, with $|k|+|\gamma_0|+|\sigma| \leq |\beta_0|-1$, $p+|\gamma_0| \leq |\beta_0|-1$ and $p+k_P+\sigma_P \leq (\beta_0)_P$. We obtain
\begin{itemize}
\item $P_{(k_T,k_P+1),p}(\Phi) \mathcal{L}_{\partial Z^{\gamma_0}}(F) \left( v , \nabla_v \Gamma^{\sigma} \right) $, clearly of \eqref{eq:com4},
\end{itemize} 
and, using the second identity of Lemma \ref{calculF},
\begin{itemize}
\item $ P_{k,p+1}(\Phi)\mathcal{L}_{X \partial Z^{\gamma_0}} (F) \left( v , \nabla_v \Gamma^{\sigma}  \right)$, which is of \eqref{eq:com2}, and 
$$c(v)P_{k,p+p_0}(\Phi)\mathcal{L}_{Z^{\delta} \partial Z^{\gamma_0}} (F) \left( v , \nabla_v \widehat{Z}^{\xi}\Gamma^{\sigma}  \right), \hspace{3mm} \text{with} \hspace{3mm} |\delta|+|\xi| \leq 1, \hspace{3mm} p_0+|\delta| \leq 1 \hspace{3mm} \text{and} \hspace{3mm} p_0+ \xi_P \leq 1.$$
As $p+p_0+|\gamma_0|+|\delta| \leq p+|\gamma_0|+1 \leq |\beta|-1$, $p+p_0+k_P+\sigma_P+\xi_P \leq (\beta_0)_P+1=\beta_P$ and, if $|\delta|=1$, $[Z^{\delta}, \partial ] \in \T \cup \{ 0 \}$, we can conclude that these terms are of \eqref{eq:com4}. 
\end{itemize} 
\end{itemize}
\end{proof}

\begin{Rq}\label{rqjustifnorm}
To deal with the weight $\tau_+$ in the terms of \eqref{eq:com2} and \eqref{eq:com4} (hidden by the $v$ derivatives), we will take advantage of the extra decay given by the $X$ vector fields or the translations $\partial_{\mu}$ through Proposition \ref{ExtradecayLie}. To deal with the terms of \eqref{eq:com1}, when $d=1$, we will need to control the $L^1$ norm of $\sum_{w \in \mathbf{k}_1} |w|^{q+1}P_{k,p}(\Phi)Y^{\sigma}f$, with $k_P+\sigma_P < \beta_P$, in order to control $\||z|^q Y^{\beta}f\|_{L^1_{x,v}}$.
\end{Rq}
As we will need to bound norms such as $\| P_{\xi}(\Phi) Y^{\beta} f \|_{L^1_{x,v}}$, we will apply Proposition \ref{ComuVlasov} to $\Phi$ and we then need to compute the derivatives of $T_F(\Phi)$. This is the purpose of the next proposition. 
\begin{Pro}\label{sourcePhi}
Let $Y^{\beta} \in \Y^{|\beta|}$ and $Z^{\gamma_1} \in \mathbb{K}^{|\gamma_1|}$ (we will apply the result for $|\gamma_1| \leq 1$). Then, $$ Y^{\beta} \left( t \frac{v^{\mu}}{v^0} \mathcal{L}_{Z^{\gamma_1}}(F)_{\mu \zeta} \right)$$ can be written as a linear combination, with $c(v)$ coefficients, of the following terms, with $0 \leq \theta, \nu \leq 3$ and $p \leq |\beta|$.
 \begin{equation}\label{equa1}
 \hspace{-0.5cm} x^{\theta} \frac{v^{\mu}}{v^0} \mathcal{L}_{Z^{\gamma} Z^{\gamma_1}}(F)_{\mu \nu}, \hspace{18mm} \text{where} \hspace{6mm} |\gamma| \leq |\beta| \hspace{17mm} \text{and} \hspace{6mm} \gamma_T=\beta_T. \tag{family $\beta-1$}
 \end{equation}
\begin{equation}\label{equa1bis}
\hspace{-0.3cm} P_{k,p}(\Phi)\frac{v^{\mu}}{v^0} \mathcal{L}_{Z^{\gamma} Z^{\gamma_1}}(F)_{\mu \nu}, \hspace{10mm} \text{where} \hspace{5mm} |k|+|\gamma| \leq |\beta|-1 \hspace{6mm} \text{and} \hspace{6mm} k_P \leq \beta_P. \tag{family $\beta-2$}
 \end{equation}
 \begin{equation}\label{equa2}
 x^{\theta}P_{k,p}(\Phi) \frac{v^{\mu}}{v^0} \mathcal{L}_{X Z^{\gamma} Z^{\gamma_1} }(F)_{\mu \nu}, \hspace{6mm} \text{where} \hspace{5mm} |k|+|\gamma| \leq |\beta|-1 \hspace{6mm} \text{and} \hspace{6mm} k_P < \beta_P. \tag{family $\beta-3$}
 \end{equation}
\end{Pro}

\begin{proof}
Let us prove this by induction on $|\beta|$. The result holds for $|\beta|=0$. We then consider $Y^{\beta}=YY^{\beta_0} \in \Y^{|\beta|}$ and we suppose that the Proposition holds for $\beta_0$. Suppose first that $Y= \partial$, so that $\beta_P=(\beta_0)_P$. Using the induction hypothesis, $\partial Y^{\beta_0} \left( t \frac{v^{\mu}}{v^0} \mathcal{L}_{Z^{\gamma_1}}(F)_{\mu \nu} \right)$ can be written as a linear combination, with good coefficients $c(v)$, of the following terms.
\begin{itemize}
\item $ \partial (x^{\theta}) \frac{v^{\mu}}{v^0} \mathcal{L}_{Z^{\gamma} Z^{\gamma_1}}(F)_{\mu \nu}$, with $|\gamma| \leq |\beta_0| < |\beta|$, which is part of \eqref{equa1bis}.
\item $x^{\theta} \frac{v^{\mu}}{v^0} \mathcal{L}_{\partial Z^{\gamma} Z^{\gamma_1}}(F)_{\mu \nu}$, with $1+|\gamma| \leq 1+|\beta_0|=|\beta|$. Denoting $\partial Z^{\gamma}$ by $Z^{\xi}$, we have $\xi_T=1+\gamma_T=1+(\beta_0)_T=\beta_T$ and this term is part of \eqref{equa1}.
\item $ P_{(k_T+1,k_P),p}(\Phi)\frac{v^{\mu}}{v^0} \mathcal{L}_{ Z^{\gamma} Z^{\gamma_1}}(F)_{\mu \nu}$, with $|k|+1+|\gamma| \leq |\beta|-1+1=|\beta|-1$ and $k_P \leq (\beta_0)_P = \beta_P$, which is part of \eqref{equa1bis}.
\item $ P_{k,p}(\Phi)\frac{v^{\mu}}{v^0} \mathcal{L}_{\partial Z^{\gamma} Z^{\gamma_1}}(F)_{\mu \nu}$, with $|k|+|\gamma|+1 \leq |\beta_0|-1+1=|\beta|-1$ and $k_P \leq (\beta_0)_P = \beta_P$, which is part of \eqref{equa1bis}.
\item $\partial(x^{\theta}) P_{k,p}(\Phi) \frac{v^{\mu}}{v^0} \mathcal{L}_{X Z^{\gamma} Z^{\gamma_1} }(F)_{\mu \nu}$, with $|k|+|\gamma| \leq |\beta_0|-1 \leq |\beta|-2$ and $k_P < (\beta_0)_P=\beta_P$, which is then equal to $0$ or part of \eqref{equa1bis}.
\item $x^{\theta} P_{(k_T+1,k_P),p}(\Phi) \frac{v^{\mu}}{v^0} \mathcal{L}_{X Z^{\gamma} Z^{\gamma_1} }(F)_{\mu \nu}$, with $|k|+1+|\gamma| \leq |\beta_0|-1+1=|\beta|-1$ and $k_P < (\beta_0)_P=\beta_P$, which is then part of \eqref{equa2}.
\item $x^{\theta}P_{k,p}(\Phi) \frac{v^{\mu}}{v^0} \mathcal{L}_{\partial X Z^{\gamma} Z^{\gamma_1} }(F)_{\mu \nu}$, with $|k|+|\gamma|+1 \leq |\beta|-1$ and $k_P < \beta_P$, which is part of \eqref{equa2}, as $[\partial, X ]=0$.
\end{itemize}
Suppose now that $Y=\widehat{Z}+\Phi X \in \Y_0$. We then have $\beta_P=(\beta_0)_P+1$ and $(\beta_0)_T=\beta_T$. In the following, we will skip the case where $Y$ hits $c(v)(v^0)^{-1}$ and we suppose for simplicty that $c(v)=1$. Note however that this case is straightforward since
$$ Y\left( \frac{c(v)}{v^0} \right)= \widehat{Z} \left( \frac{c(v)}{v^0} \right)= \frac{\widehat{Z}(c(v))}{v^0}+c(v) \widehat{Z} \left( \frac{1}{v^0} \right) = \frac{c_1(v)}{v^0} .$$ Using again the induction hypothesis, $Y Y^{\beta_0} \left( t \frac{v^{\mu}}{v^0} \mathcal{L}_{Z^{\gamma_1}}(F)_{\mu \zeta} \right)$ can be written as a linear combination of the following terms.

\begin{itemize}
\item $ Y (x^{\theta}) \frac{v^{\mu}}{v^0} \mathcal{L}_{Z^{\gamma} Z^{\gamma_1}}(F)_{\mu \nu}$, with $|\gamma| \leq |\beta_0| < |\beta|$ and $\gamma_T=(\beta_0)_T=\beta_T$. As, schematically (with $\delta=0$ or $\delta=1$),
\begin{equation}\label{eq:53}
Y(x^{\theta})=\widehat{Z}(x^{\theta})+\Phi X(x^{\theta})=\delta x^{\kappa}+c(v)\Phi,
\end{equation}
This leads to terms of \eqref{equa1} and \eqref{equa1bis}.
\item $x^{\theta} \frac{1}{v^0} Y \left( v^{\mu} \mathcal{L}_{ Z^{\gamma} Z^{\gamma_1}}(F)_{\mu \nu} \right)$, with $|\gamma| \leq |\beta_0|$ and $\gamma_T=(\beta_0)_T=\beta_T$. Using the first identity of Lemma \ref{calculF}, we have that $Y \left( v^{\mu} \mathcal{L}_{ Z^{\gamma} Z^{\gamma_1}}(F)_{\mu \theta} \right)$ is a linear combination of terms such as
$$v^{\mu}\mathcal{L}_{ Z^{\gamma_0} Z^{\gamma} Z^{\gamma_1}}(F)_{\mu \lambda} , \hspace{3mm} \text{with} \hspace{3mm} |\gamma_0| \leq 1, \hspace{3mm} (\gamma_0)_T=0, \hspace{3mm} \text{and} \hspace{3mm} 0 \leq \lambda \leq 3,$$
leading to terms of \eqref{equa1}, and
$$\Phi v^{\mu}\mathcal{L}_{ X Z^{\gamma} Z^{\gamma_1}}(F)_{\mu \nu},$$
giving terms of \eqref{equa2}, as $|\gamma| \leq |\beta_0|=|\beta|-1$.
\item $\frac{1}{v^0} Y \left( P_{k,p}(\Phi) \right) v^{\mu} \mathcal{L}_{ Z^{\gamma} Z^{\gamma_1}}(F)_{\mu \nu} $, with $|k|+|\gamma| \leq |\beta_0|-1$ and $k_P \leq \beta_P$. We obtain terms of \eqref{equa1bis}, since
$$Y \left( P_{k,p}(\Phi) \right)=P_{(k_T,k_P+1),p}(\Phi), \hspace{3mm} |k|+1+|\gamma| \leq |\beta|-1 \hspace{3mm} \text{and} \hspace{3mm} k_P+1 \leq (\beta_0)_P+1 = \beta_P .$$
\item $\frac{1}{v^0} P_{k,p}(\Phi) Y \left( v^{\mu} \mathcal{L}_{ Z^{\gamma} Z^{\gamma_1}}(F)_{\mu \nu} \right)$, with $|k|+|\gamma| \leq |\beta_0|-1$ and $k_P \leq (\beta_0)_P$. Using the first identity of Lemma \ref{calculF}, we have that $Y \left( v^{\mu} \mathcal{L}_{ Z^{\gamma} Z^{\gamma_1}}(F)_{\mu \nu} \right)$ is a linear combination of terms of the form
$$c(v) \Phi^r v^{\mu}\mathcal{L}_{ Z^{\gamma_0} Z^{\gamma} Z^{\gamma_1}}(F)_{\mu \lambda} , \hspace{6mm} \text{with} \hspace{6mm} \max(r,|\gamma_0|) \leq 1 \hspace{6mm} \text{and} \hspace{6mm} 0 \leq \lambda \leq 3.$$
We then obtain terms of \eqref{equa1bis}, as $|k|+|\gamma|+|\gamma_0| \leq |\beta_0|=|\beta|-1$ and $k_P \leq \beta_P$.
\item $Y\left(x^{\theta} \right)P_{k,p}(\Phi) \frac{v^{\mu}}{v^0} \mathcal{L}_{XZ^{\gamma} Z^{\gamma_1} }(F)_{\mu \nu}$, with $|k|+|\gamma| \leq |\beta_0|-1$ and $k_P < (\beta_0)_P$, which, using \eqref{eq:53}, gives terms of \eqref{equa1bis} and \eqref{equa2}.
\item $ x^{\theta}P_{(k_T,k_P+1),p}(\Phi)  \frac{v^{\mu}}{v^0} \mathcal{L}_{X Z^{\gamma} Z^{\gamma_1} }(F)_{\mu \nu}$, with $|k|+1+|\gamma| \leq |\beta_0|-1+1=|\beta|-1$ and $k_P+1 < (\beta_0)_P+1=\beta_P$, which is part of \eqref{equa2}.
\item $x^{\theta}P_{k,p}(\Phi)\frac{1}{v^0} Y \left( v^{\mu} \mathcal{L}_{ X Z^{\gamma} Z^{\gamma_1} }(F)_{\mu \nu} \right)$, with $|k|+|\gamma| \leq |\beta_0|-1$ and $k_P < (\beta_0)_P$. By the third point of Lemma \ref{calculF}, we can write $Y \left( v^{\mu} \mathcal{L}_{X Z^{\gamma} Z^{\gamma_1} }(F)_{\mu \nu} \right)$ as a linear combination of terms such as
$$c(v) \Phi^r v^{\mu} \mathcal{L}_{ X Z^{\gamma_0} Z^{\gamma} Z^{\gamma_1} }(F)_{\mu \lambda}, \hspace{3mm} \text{with} \hspace{3mm} \max(r,|\gamma_0|) \leq 1 \hspace{3mm} \text{and} \hspace{3mm} 0 \leq \lambda \leq 3.$$
It gives us terms of \eqref{equa2}, as $|k|+|\gamma_0|+|\gamma| \leq |\beta_0|=|\beta|-1$ and $k_P < \beta_P$.
\end{itemize}
\end{proof}

The worst terms are those of \eqref{equa1}. They do not appear in the source term of $T_F \left( P^X_{\zeta}(\Phi) \right)$, which explains why our estimate on $\| P^X_{\zeta}(\Phi) Y^{\beta} f \|_{L^1_{x,v}}$ will be better than the one on $\| P_{\xi}(\Phi) Y^{\beta} f \|_{L^1_{x,v}}$.

\begin{Pro}\label{sourceXPhi}
Let $Y^{\overline{\beta}} \in \Y_X^{|\overline{\beta}|}$, with $\overline{\beta}_X \geq 1$, $Z^{\gamma_1} \in \mathbb{K}^{|\gamma_1|}$ and $\beta$ be a multi-index associated to $\Y$ such that $\beta_P=\overline{\beta}_P$ and $\beta_T=\overline{\beta}_T+\overline{\beta}_X$. Then, $ Y^{\overline{\beta}} \left( t \frac{v^{\mu}}{v^0} \mathcal{L}_{Z^{\gamma_1}}(F)_{\mu \zeta} \right)$ can be written as a linear combination of terms of \eqref{equa1bis}, \eqref{equa2} and,
\begin{equation}\label{equa2bis}
\text{if} \hspace{6mm} \beta_P=0, \hspace{6mm} x^{\theta} \frac{v^{\mu}}{v^0} \mathcal{L}_{X Z^{\gamma} Z^{\gamma_1} }(F)_{\mu \nu}, \hspace{6mm} \text{where} \hspace{5mm} |\gamma| \leq |\beta|-1. \tag{family $\beta-3-bis$}
\end{equation}
\end{Pro}
\begin{proof}
The proof is similar to the previous one. The difference comes from the fact a $X$ vector field necessarily have to hit a term of the first family, giving either a term of the second family or of the third-bis family, where we we do not have the condition $k_P < \beta_P$ since $k_P$ and $\beta_P$ could be both equal to $0$.
\end{proof}

\subsection{The null structure of $G(v,\nabla_v g)$}
In this subsection, we consider $G$, a $2$-form defined on $[0,T[ \times \R^3$, and $g$, a function defined on $[0,T[ \times \R^3_x \times \R^3_v$, both sufficiently regular. We investigate in this subsection the null structure of $G(v,\nabla_v g)$ in view of studying the error terms obtained in Proposition \ref{ComuVlasov}. Let us denote by $(\alpha, \underline{\alpha}, \rho, \sigma)$ the null decomposition of $G$. Then, expressing $G \left( v, \nabla_v g \right)$ in null coordinates, we obtain a linear combination of the following terms.
\begin{itemize}
\item The terms with the radial component of $\nabla_v g$ (remark that $\left( \nabla_v g \right)^L =- \left( \nabla_v g \right)^{\underline{L}}=\left( \nabla_v g \right)^r$),
\begin{equation}\label{eq:radi}
v^L \rho \left( \nabla_v g \right)^{\underline{L}}, \hspace{8mm} v^{\underline{L}} \rho \left( \nabla_v g \right)^{L}, \hspace{8mm} v^A \alpha_A \left( \nabla_v g \right)^{L} \hspace{8mm} \text{and} \hspace{8mm} v^A \underline{\alpha}_A \left( \nabla_v g \right)^{\underline{L}}.
\end{equation}
\item The terms with an angular component of $\nabla g$,
\begin{equation}\label{eq:angu}
\varepsilon_{BA} v^B \sigma \left( \nabla_v g \right)^{A}, \hspace{12mm} v^{L} \alpha_A \left( \nabla_v g \right)^{A} \hspace{8mm} \text{and} \hspace{8mm} v^{\underline{L}} \underline{\alpha}_A \left( \nabla_v g \right)^{A}.
\end{equation}
\end{itemize}
We are then led to bound the null components of $\nabla_v g$. A naive estimate, using $v^0\partial_{v^k}= Y_k-\Phi X-t\partial_k-x^k \partial_t$, gives 
\begin{equation}\label{naive2}
\left| \left( \nabla_v g \right)^{L} \right|, \hspace{1mm} \left| \left( \nabla_v g \right)^{\underline{L}} \right|, \hspace{1mm} \left| \left( \nabla_v g \right)^{A} \right| \leq \left| \nabla_v g \right| \lesssim \frac{\tau_++|\Phi|}{v^0} |\nabla_{t,x} g |+\frac{1}{v^0}\sum_{Y \in \Y} |Y g|.
\end{equation}
With these inequalities, using our schematic notations $c \prec d$ if $d$ is expected to behave better than $c$, we have $v^L \rho \left( \nabla_v g \right)^{\underline{L}} \prec \varepsilon_{BA} v^B \sigma \left( \nabla_v g \right)^{A}$, since $v^L \prec v^B$ and $\rho \sim \sigma$. The purpose of the following result is to improve \eqref{naive2} for the radial component in order to have a better control on $v^L \rho \left( \nabla_v g \right)^{\underline{L}}$.
\begin{Lem}\label{vradial}
Let $g$ be a sufficiently regular function, $z \in \mathbf{k}_1$ and $j \in \mathbb{N}^*$. We have
$$ \left| \left( \nabla_v g \right)^{r} \right| \lesssim \frac{\tau_-+|\Phi|}{v^0} |\nabla_{t,x} g |+\frac{1}{v^0}\sum_{Y \in \Y} |Y g| \hspace{10mm} \text{and} \hspace{10mm} \left| \left( \nabla_v z^j \right)^r \right| \lesssim \frac{\tau_-}{v^0}|z|^{j-1}+\frac{1}{v^0} \sum_{w \in \mathbf{k}_1}  |w |^j.$$
\end{Lem}
\begin{proof}
We have
$$( \nabla_v g )^r=\frac{x^i}{r}\partial_{v^i} g \hspace{8mm} \text{and} \hspace{8mm} \frac{x^i}{rv^0}(t\partial_i+x_i\partial_t)=\frac{1}{v^0}(t\partial_r+r\partial_t)=\frac{1}{v^0}(S+(r-t)\underline{L}),$$ 
so that, using $\partial_{v^i}=\frac{1}{v^0}(\widehat{\Omega}_{0i}-t\partial_i-x_i\partial_t)$,
\begin{equation}\label{equ:proof}
( \nabla_v g )^r = \frac{x^i }{rv^0}\widehat{\Omega}_{0i} \left(g \right)-\frac{1}{v^0}S \left( g \right) +\frac{t-r}{v^0}\underline{L} \left( g \right).
\end{equation}
To prove the first inequality, it only remains to write schematically that $\widehat{\Omega}_{0i}=Y_{0i}-\Phi X$, $S=Y_S-\Phi X$ and to use the triangle inequality. To complete the proof of the second inequality, apply \eqref{equ:proof} to $g=z^j$, recall from Lemma \ref{weights} that $ \left| \widehat{Z} \left( z^j \right) \right| \lesssim \sum_{z \in \mathbf{k}_1} |w|^j$ and use that $\left| \underline{L} \left( z^j \right) \right| \lesssim |z|^{j-1}$.
\end{proof}
For the terms containing an angular component, note that they are also composed by either $\alpha$, the better null component of the electromagnetic field, $v^A$ or $v^{\underline{L}}$. The following lemma is fundamental for us to estimate the energy norms of the Vlasov field.
\begin{Lem}\label{nullG}
We can bound $\left| G(v, \nabla_v g ) \right|$ either by
$$  \left( |\rho|+|\underline{\alpha}| \right) \left( \sum_{ Y \in \Y} |Y(g)| \hspace{-0.2mm}+ \hspace{-0.2mm}  \left( \tau_-+|\Phi|+\sum_{w \in \mathbf{k}_1} |w| \right) |\nabla_{t,x} g | \hspace{-0.5mm} \right)+ \left(|\alpha|+\sqrt{\frac{v^{\underline{L}}}{v^0}}|\sigma| \right) \hspace{-1mm} \left( \sum_{ Y \in \Y} |Y(g)| \hspace{-0.2mm} + \hspace{-0.2mm} (\tau_++|\Phi|) |\nabla_{t,x} g | \right)$$
or by
$$ \left(  |\alpha|+|\rho|+\sqrt{\frac{v^{\underline{L}}}{v^0} }|\sigma|+\sqrt{\frac{v^{\underline{L}}}{v^0} }|\underline{\alpha}| \right) \left( \sum_{ Y \in \Y} |Y(g)|+ \left( \tau_++|\Phi| \right) |\nabla_{t,x} g | \right)$$
\end{Lem}
\begin{proof}
The proof consists in bounding the terms given in \eqref{eq:radi} and \eqref{eq:angu}. By Lemma \ref{vradial} and $|v^A| \lesssim \sqrt{v^0v^{\underline{L}}}$, one has
$$ \left| v^L \rho \left( \nabla_v g \right)^{\underline{L}}-v^{\underline{L}} \rho \left( \nabla_v g \right)^L+v^A \underline{\alpha}_A \left( \nabla_v g \right)^{\underline{L}} \right| \lesssim \left( |\rho|+\sqrt{\frac{v^{\underline{L}}}{v^0}}|\underline{\alpha}| \right)\left( \sum_{ Y \in \Y} |Y(g)|+ \left( \tau_-+|\Phi| \right) |\nabla_{t,x} g | \right) .$$
As $v^0 \partial_{v^i} = Y_i-\Phi X-x^i \partial_t-t \partial_i$ and $|v^B | \lesssim \sqrt{v^0 v^{\underline{L}}}$, we obtain
$$ \left| v^L \alpha_A \left( \nabla_v g \right)^A+v^A \alpha_A \left( \nabla_v g \right)^{L}+v^B \sigma_{BA} \left( \nabla_v g \right)^A \right| \lesssim \left(|\alpha|+\sqrt{\frac{v^{\underline{L}}}{v^0}}|\sigma| \right) \left( \sum_{ Y \in \Y} |Y(g)|+ (\tau_++\Phi|) |\nabla_{t,x} g | \right).$$
Finally, using $v^0 \partial_{v^i} = Y_i-\Phi X-x^i \partial_t-t \partial_i$ and Lemma \ref{weights1} (for the first inequality), we get
\begin{eqnarray}
\nonumber \left| v^{\underline{L}} \underline{\alpha}_A \left( \nabla_v g \right)^A \right| & \lesssim & |\underline{\alpha}| \left( \sum_{ Y \in \Y} |Y(g)|+ \left( \tau_-+|\Phi|+\sum_{w \in \mathbf{k}_1} |w| \right) |\nabla_{t,x} g | \right) \\ \nonumber
\left| v^{\underline{L}} \underline{\alpha}_A \left( \nabla_v g \right)^A \right| & \lesssim & \sqrt{\frac{v^{\underline{L}}}{v^0}} |\underline{\alpha}| \left( \sum_{ Y \in \Y} |Y(g)|+ \left( \tau_++|\Phi| \right) |\nabla_{t,x} g | \right).
\end{eqnarray}
\end{proof}
\begin{Rq}
The second inequality will be used in extremal cases of the hierarchies considered, where we will not be able to take advantage of the weights $w \in \mathbf{k}_1$ in front of $|\nabla_{t,x} g|$ and where the terms $\sum_{Y \in \Y_0} |Y g |$ will force us to estimate a weight $z \in \V$ by $\tau_+$ (see Proposition \ref{ComuPkp} below).
\end{Rq}
\subsection{Source term of $T_F(z^jP_{\xi}(\Phi) Y^{\beta}f)$}

In view of Remark \ref{rqjustifnorm}, we will consider hierarchised energy norms controling, for $Q$ a fixed integer, $\| z^{Q-\xi_P-\beta_P} P_{\xi}(\Phi) Y^{\beta} f \|_{L^1_{x,v}}$. In order to estimate them, we compute in this subsection the source term of $T_F(z^jP_{\xi}(\Phi) Y^{\beta}f)$. We start by the following technical result.
\begin{Lem}\label{GammatoYLem}
Let $h : [0,T[ \times \R^3_x \times \R^3_v \rightarrow \R$ be a sufficiently regular function and $\Gamma^{\sigma} \in \mathbb{G}^{|\sigma|}$. Then,
\begin{eqnarray}
\nonumber  \Gamma^{\sigma} h & = & \sum_{\begin{subarray}{} \hspace{1mm} |g|+|\overline{\sigma}| \leq |\sigma| \\ \hspace{2mm} |g| \leq |\sigma|-1 \\ r+g_P+\overline{\sigma}_P \leq \sigma_P \end{subarray} }  c^{g,r}_{\overline{\sigma}}(v) P_{g,r}(\Phi) Y^{\overline{\sigma}} h ,\\
\nonumber 
\left| \partial_{v^i} \left( \Gamma^{\sigma} h \right) \right| & \lesssim & \sum_{\delta=0}^1 \sum_{\begin{subarray}{} \hspace{3.5mm} |g|+|\overline{\sigma}| \leq |\sigma|+1 \\ \hspace{9mm} |g| \leq |\sigma| \\ r+g_P+\overline{\sigma}_P+\delta \leq \sigma_P+1 \end{subarray} } \tau_+^\delta \left| P_{g,r}(\Phi) Y^{\overline{\sigma}} h \right|.
\end{eqnarray}
\begin{proof}
The first formula can be proved by induction on $|\sigma|$, using that $\widehat{Z}=Y-\Phi X$ for each $\widehat{Z}$ composing $\Gamma^{\sigma}$. The inequality then follows using $v^0 \partial_{v^i}=Y_i-\Phi X-t \partial_i-x^i \partial_t$.
\end{proof}
\end{Lem}
\begin{Pro}\label{ComuPkp}
Let $N \in \mathbb{N}$ and $N_0 \geq N$. Consider $\zeta^0$ and $\beta$ multi-indices such that $|\zeta^0|+|\beta| \leq N$ and $|\zeta^0| \leq N-1$. Let also $z \in \mathbf{k}_1$ and $j \leq N_0-\zeta^0_P-\beta_P$. Then, $T_F(z^jP_{\zeta^0}(\Phi) Y^{\beta} f)$ can be bounded by a linear combination of the following terms, where $|\gamma|+|\zeta| \leq |\zeta^0|+|\beta|$.
\begin{itemize}
\item \begin{equation}\label{eq:cat0} 
\left|  F \left(v, \nabla_v \left( z^j \right) \right) P_{\zeta^0}(\Phi) Y^{\beta} f  \right|. \tag{category $0$}
\end{equation}
\item \begin{equation}\label{eq:cat1} 
\left( \left| \nabla_{Z^{\gamma}} F \right|+\frac{\tau_+}{\tau_-} \left| \alpha \left( \mathcal{L}_{Z^{\gamma}}(F) \right) \right| +\frac{\tau_+}{\tau_-} \sqrt{\frac{v^{\underline{L}}}{v^0}} \left| \sigma \left( \mathcal{L}_{Z^{\gamma}}(F) \right) \right| \right)\left| \Phi \right|^n \left|w^i P_{\zeta}(\Phi) Y^{\kappa} f  \right|, \tag{category $1$}
\end{equation}
where \hspace{2mm} $n \leq 2N$, \hspace{2mm} $w \in \mathbf{k}_1$, \hspace{2mm} $|\zeta|+|\gamma|+|\kappa| \leq |\zeta^0|+|\beta|+1$, \hspace{2mm} $i \leq N_0 -\zeta_P-\kappa_P$, \hspace{2mm} $\max( |\gamma|, |\zeta|+|\kappa|) \leq |\zeta^0|+|\beta|$ \hspace{2mm} and \hspace{2mm} $|\zeta| \leq  N-1$. 
\item \begin{equation}\label{eq:cat3}
\hspace{-10mm} \frac{\tau_+}{\tau_-} |\rho \left( \mathcal{L}_{ Z^{\gamma}}(F) \right) | \left|z^{j-1} P_{\zeta}(\Phi) Y^{\sigma} f \right| \hspace{5mm} \text{and} \hspace{5mm} \frac{\tau_+}{\tau_-}  \sqrt{\frac{v^{\underline{L}}}{v^0}}\left| \underline{\alpha} \left( \mathcal{L}_{ Z^{\gamma}}(F)  \right) \right| \left| z^i P_{\zeta}(\Phi) Y^{\kappa} f \right|,  \tag{category $2$}
\end{equation}
where \hspace{2mm} $|\zeta|+|\gamma|+|\kappa| \leq |\zeta^0|+|\beta|+1$, \hspace{2mm} $j-1$, $i=N_0-\zeta_P-\kappa_P$, \hspace{2mm} $\max( |\gamma|, |\zeta|+|\kappa|) \leq |\zeta^0|+|\beta|$ \hspace{2mm} and \hspace{2mm} $|\zeta| \leq N-1$. Morevover, we have $i \leq j$.
\item \begin{equation}\label{eq:cat4}
\tau_+ \left| \frac{v^{\mu}}{v^0} \mathcal{L}_{Z^{\gamma}}(F)_{\mu \theta} z^j P_{\zeta}(\Phi) Y^{\beta} f \right|, \tag{category $3$}
\end{equation}
with \hspace{2mm}$ |\zeta| < |\zeta^0|$, \hspace{2mm} $\zeta_T+\gamma_T = \zeta^0_T$, \hspace{2mm} $\zeta_P \leq \zeta^0_P$, and \hspace{2mm} $|\zeta|+|\gamma| \leq |\zeta^0|+1$. This implies $j \leq N_0-\zeta_P-\beta_P$.
\end{itemize}
Note that the terms of \eqref{eq:cat3} only appears when $j=N_0-k_P-\beta_P$ and the ones of \eqref{eq:cat4} when $|\zeta^0| \geq 1$.
\end{Pro}
\begin{proof}
The first thing to remark is that $$T_F(z^jP_{\zeta^0}(\Phi) Y^{\beta} f)=F \left(v, \nabla_v \left( z^j \right) \right) P_{\zeta^0}(\Phi) Y^{\beta} f +z^jT_F(P_{\zeta^0}(\Phi))Y^{\beta} f+z^jP_{\zeta^0}(\Phi) T_F(Y^{\beta} f ).$$ 
We immediately obtain the terms of \eqref{eq:cat0}. Let us then consider $z^jP_{\zeta^0}(\Phi) T_F(Y^{\beta} f )$. Using Proposition \ref{ComuVlasov}, it can be written as a linear combination of terms of \eqref{eq:com1}, \eqref{eq:com2} or \eqref{eq:com4} (applied to $f$), multiplied by $z^jP_{\zeta^0}(\Phi)$. Consequently, $|z^jP_{\zeta^0}(\Phi) T_F(Y^{\beta} f )|$ can be bounded by a linear combination of 
\begin{itemize}
\item $|z|^j\left| w^d Z^{\gamma}(F_{\mu \nu}) \right| \left| P_{k,p}(\Phi)P_{\zeta^0}(\Phi) Y^{\kappa} f \right|$, with $d \in \{0,1 \}$, $w \in \mathbf{k}_1$, $|\sigma| \geq 1$, $\max( |\gamma|, |k|+|\gamma|, |k|+|\kappa|,|k|+1 ) \leq |\beta|$, $|k|+|\gamma|+|\kappa| \leq |\beta|+1$ and $p+k_P+\kappa_P+d \leq \beta_P$. Now, note that
$$ \exists \hspace{0.5mm} n, \hspace{0.5mm} \zeta \hspace{3mm} \text{such that} \hspace{3mm} P_{k,p}(\Phi) P_{\zeta^0}(\Phi) = \Phi^n P_{\zeta}(\Phi), \hspace{3mm} n \leq |\beta|, \hspace{3mm} \zeta_T=k_T+\zeta^0_T  \hspace{3mm} \text{and} \hspace{3mm} \zeta_P=k_P+\zeta^0_P.$$
Consequently, $|\zeta|=|k|+|\zeta^0| \leq |\zeta^0|+|\beta|-1 \leq N-1$, \hspace{2mm} $|\zeta|+|\gamma| =|k|+|\zeta^0|+|\gamma| \leq |\zeta^0|+|\beta|$, $$|\zeta|+|\kappa|=|k|+|\zeta^0|+|\kappa| \leq |\zeta^0| + |\beta| \hspace{3mm} \text{and} \hspace{3mm} |\zeta|+|\gamma|+|\kappa| \leq |k|+|\zeta^0|+|\gamma|+|\kappa| \leq |\zeta^0|+|\beta|+1.$$ Since $$k_P+\kappa_P+d \leq \beta_P \hspace{3mm} \text{and} \hspace{3mm} \zeta_P=k_P+ \zeta^0 _P, \hspace{3mm} \text{we have} \hspace{3mm} j+d \leq N_0-\zeta_P-\kappa_P.$$ Finally, as $|z^j w^d| \leq |z|^{j+d}+|w|^{j+d}$, we obtain terms of \eqref{eq:cat1}.
\item $|z|^j\left| P_{k,p}(\Phi) \mathcal{L}_{ X Z^{\gamma_0}}(F)\left( v, \nabla_v \left( \Gamma^{\sigma} f \right) \right) P_{\zeta^0}(\Phi) \right|$, with $|k|+|\gamma_0|+|\sigma| \leq |\beta|-1$, $p+k_P+\sigma_P \leq \beta_P$ and $p \geq 1$. Then, apply Lemma \ref{GammatoYLem} in order to get
$$\left| \nabla_v \left( \Gamma^{\sigma} f \right) \right| \lesssim \sum_{\delta=0}^1 \sum_{\begin{subarray}{} \hspace{3mm} |g|+|\overline{\sigma}| \leq |\sigma|+1 \\ \hspace{5mm} |g| \leq |\sigma| \\ r+g_P+\overline{\sigma}_P+\delta \leq \sigma_P+1 \end{subarray} } \tau_+^\delta \left| P_{g,r}(\Phi) Y^{\overline{\sigma}} f \right|.$$
Fix parameters $(\delta, g , r, \overline{\sigma})$ as in the right hand side of the previous inequality and consider first the case $\delta=0$. Then, $|z|^j\left| \mathcal{L}_{ X Z^{\gamma_0}}(F) \right| \left| P_{k,p}(\Phi) P_{g,r}(\Phi)P_{\zeta^0}(\Phi) Y^{\overline{\sigma}} f \right|$ can be bounded by terms such as
$$ |z|^j\left| Z^{\gamma}(F_{\mu \nu}) \right| \left|\Phi^n P_{\zeta}(\Phi) Y^{\overline{\sigma}} f \right| \hspace{-0.3mm} , \hspace{1.9mm} \text{with} \hspace{1.9mm} |\gamma| \leq |\gamma_0|+1, \hspace{1.9mm} n \leq p+r, \hspace{2mm} \zeta_T=k_T+g_T+ \zeta^0_T  , \hspace{1.9mm} \zeta_P=k_P+g_P+\zeta^0_P   .$$
We then have $n \leq 2|\beta|$, $|\zeta|+|\gamma|+|\overline{\sigma}| \leq |k|+|g|+|\zeta^0|+|\gamma_0|+1+|\overline{\sigma}| \leq |\zeta^0|+|\beta|+1$, $|\zeta|+|\overline{\sigma}| \leq |\zeta^0|+ |\beta|$ and $|\zeta| \leq |\zeta^0|+|\beta|-1$. As $$\zeta_P+\overline{\sigma}_P =k_P+g_P+\zeta^0_P+\overline{\sigma}_P \leq k_P+\sigma_P+1+\zeta^0_P \leq \zeta^0_P+\beta_P,$$ we have $j \leq N_0-\zeta_P-\overline{\sigma}_P$. If $\delta=1$, use the inequality \eqref{eq:Xdecay} of Proposition \ref{ExtradecayLie} to compensate the weight $\tau_+$. The only difference is that it brings a weight $w \in \mathbf{k}_1$. To handle it, use $|z^j w | \leq |z|^{j+1}+|w|^{j+1}$ and
$$\zeta_P+\overline{\sigma}_P =k_P+g_P+\zeta^0_P+\overline{\sigma}_P \leq k_P+\sigma_P+1-\delta+\zeta^0_P \leq \zeta^0_P+\beta_P-1,$$
so that $j+1 \leq N_0-\zeta_P-\beta_P$. In both cases, we then have terms of \eqref{eq:cat1}.
\item $|z|^j\left| P_{k,p}(\Phi) \mathcal{L}_{ \partial Z^{\gamma_0}}(F)\left( v, \nabla_v \left( \Gamma^{\sigma^0} f \right) \right) P_{\zeta_0}(\Phi) \right|$, with $|k|+|\gamma_0|+|\sigma^0| \leq |\beta|-1$, $p+|\gamma_0| \leq |\beta|-1$ and $p+k_P+\sigma^0_P \leq \beta_P$, which arises from a term of \eqref{eq:com4}. Applying Lemma \ref{GammatoYLem}, we can schematically suppose that
$$ \Gamma^{\sigma^0} = c(v) \Phi^r P_{\chi}(\Phi) Y^{\kappa} \hspace{3mm} \text{with} \hspace{3mm} |\chi|+|\kappa| \leq |\sigma^0|, \hspace{3mm} |\chi| \leq |\sigma^0|-1 \hspace{3mm} \text{and} \hspace{3mm} r+r_{\chi} +\chi_P+\kappa_P \leq \sigma^0_P,$$ where $r_{\chi}$ is the number of $\Phi$ coefficients in $P_{\chi}(\Phi)$. As $Y \left( c(v) \right)$ is a good coefficient, $c(v)$ does not play any role in what follows and we then suppose for simplicity that $c(v)=1$. We suppose moreover, in order to not have a weight in excess, that
\begin{equation}\label{condihyp}
j+k_P+\chi_P+\kappa_P < N_0-\zeta^0_P
\end{equation} and we will treat the remaining cases below. Using the first inequality of Lemma \ref{nullG} and denoting by $(\alpha, \underline{\alpha}, \rho, \sigma)$ the null decomposition of $\mathcal{L}_{\partial Z^{\gamma_0}}(F)$, we can bound the quantity considered here by the sum of the three following terms
\begin{equation}\label{eq:unus}
|z|^j\left| P_{k,p}(\Phi) P_{\zeta_0}(\Phi) \right| \left( |\alpha|+|\rho|+\sqrt{\frac{v^{\underline{L}}}{v^0}}|\sigma|+|\underline{\alpha}| \right) \sum_{ Y \in \Y_0} \left| Y \left( \Phi^r P_{\chi}(\Phi) Y^{\kappa} f \right) \right|,
\end{equation}
\begin{equation}\label{eq:duo}
|z|^j\left| P_{k,p}(\Phi) P_{\zeta_0}(\Phi) \right| \left( |\rho|+ |\underline{\alpha}| \right) \left( \tau_-  +|\Phi|+ \hspace{-1mm} \sum_{w \in \mathbf{k}_1} |w| \right) \left| \nabla_{t,x} \left( \Phi^r P_{\chi}(\Phi) Y^{\kappa} f \right)   \right|,
\end{equation}
\begin{equation}\label{eq:tres}
|z|^j  \left| P_{k,p}(\Phi) P_{\zeta_0}(\Phi) \right|  \left(\tau_++|\Phi| \right)\left(|\alpha|+\sqrt{\frac{v^{\underline{L}}}{v^0}} |\sigma| \right) \left| \nabla_{t,x} \left( \Phi^r P_{\chi}(\Phi) Y^{\kappa} f \right)   \right|.
\end{equation}
Let us start by \eqref{eq:unus}. We have schematically, for $Y \in \Y_0$, $Y^{\kappa^1}=Y^{\kappa}$ and $Y^{\kappa^2}=Y Y^{\kappa}$,
$$P_{k,p}(\Phi) P_{\zeta^0}(\Phi) Y \left( \Phi^r P_{\chi}(\Phi) Y^{\kappa} f \right) = \Phi^{n_1}P_{\zeta^1}(\Phi) Y^{\kappa^1} f+\Phi^{n_2}P_{\zeta^2}(\Phi) Y^{\kappa^2} f,$$ $$\text{with} \hspace{3mm} |n_i| \leq p+r, \hspace{3mm} |\zeta^i|=|k|+|\zeta^0|+|\chi|+\delta_1^{i} \hspace{3mm} \text{and} \hspace{3mm} \zeta^i_P=k_P+\zeta^0_P+\chi_P+\delta_{1}^{i}.$$
We have, according to \eqref{condihyp}, $$j+\zeta^i_P+\kappa^i_P = \zeta^0_P+j +k_P+\chi_P+\kappa_P+1   \leq N_0.$$ Consequently, as
\begin{equation}\label{bound45}
  |\alpha|+|\rho|+\sqrt{\frac{v^{\underline{L}}}{v^0}}|\sigma|+ |\underline{\alpha}|  \lesssim  \left| \mathcal{L}_{ \partial Z^{\gamma}}(F) \right| \lesssim \sum_{|\gamma| \leq |\gamma_0|+1} \left|\nabla_{ Z^{\gamma}} F \right| \hspace{3mm} \text{and} \hspace{3mm} |\zeta^i|+|\gamma|+|\kappa^i| \leq |\beta|+|\zeta^0|+1,
 \end{equation}
we obtain terms of \eqref{eq:cat1} (the other conditions are easy to check).

Let us focus now on \eqref{eq:duo} and \eqref{eq:tres}. Defining $Y^{\kappa^3}=Y^{\kappa}$ and $Y^{\kappa^4}= \partial Y^{\kappa}$, we have schematically
$$P_{k,p}(\Phi) P_{\zeta^0}(\Phi) \partial \left( \Phi^r P_{\chi}(\Phi) Y^{\kappa} f \right)= \Phi^{n_3}P_{\zeta^3}(\Phi) Y^{\sigma^3} f+\Phi^{n_4}P_{\zeta^4}(\Phi) Y^{\kappa^4} f,$$ $$\text{with} \hspace{3mm} |n_i| \leq p+r \leq 2|\beta|-2, \hspace{3mm} |\zeta^i|=|k|+|\zeta^0|+|\chi|+\delta_{i}^{3} \hspace{3mm} \text{and} \hspace{3mm} \zeta^i_P=k_P+\zeta^0_P+\chi_P.$$
This time, one obtains $j +1 \leq N_0-\zeta^i_P-\kappa^i_P $. As, by inequality \eqref{eq:zeta2} of Proposition \ref{ExtradecayLie},
$$\left( |\rho|+ |\underline{\alpha}| \right) \lesssim \frac{1}{\tau_-}\sum_{|\gamma| \leq |\gamma_0|+1} \left| \nabla_{Z^{\gamma}} F \right|, \hspace{5mm} |\alpha| \lesssim \sum_{|\gamma| \leq |\gamma_0|+1}\frac{1}{\tau_-} |\alpha (\mathcal{L}_{Z^{\gamma}}(F))|+ \frac{1}{\tau_+}\left| \nabla_{Z^{\gamma}} F \right| ,$$
$$ |\sigma| \lesssim \sum_{|\gamma| \leq |\gamma_0|+1}\frac{1}{\tau_-} |\sigma (\mathcal{L}_{Z^{\gamma}}(F))|+ \frac{1}{\tau_+}\left| \nabla_{Z^{\gamma}} F \right| \hspace{5mm} \text{and} \hspace{5mm} |z^j w | \leq |z|^{j+1}+|w|^{j+1},$$  \eqref{eq:duo} and \eqref{eq:tres} also give us terms of \eqref{eq:cat1}.
\item We now treat the remaining terms arising from those of \eqref{eq:com4}, for which $$j+k_P+\chi_P+\kappa_P=N_0-\zeta^0_P.$$ This equality can only occur if $j=N_0-\zeta^0_P-\beta_P$ and $k_P+\chi_P+\kappa_P=\beta_P$. It implies $p+r+r_{\chi}=0$ and we then have to study terms of the form
$$|z|^j\left|  \mathcal{L}_{ \partial Z^{\gamma_0}}(F)\left( v, \nabla_v \left( Y^{\kappa} f \right) \right) P_{\zeta^0}(\Phi) \right|, \hspace{2mm} \text{with} \hspace{2mm} |\gamma_0|+|\kappa| \leq |\beta|-1.$$
Using the second inequality of Lemma \ref{nullG}, and denoting again the null decomposition of $\mathcal{L}_{\partial Z^{\gamma_0}}(F)$ by $(\alpha, \underline{\alpha}, \rho, \sigma)$, we can bound it by quantities such as
$$\left| \Phi \right| \left|  \mathcal{L}_{ \partial Z^{\gamma_0}}(F) \right| \left| z^j P_{\zeta^0}(\Phi) \partial Y^{\kappa} f \right|, \hspace{3mm} \text{leading to terms of} \hspace{3mm} \eqref{eq:cat1},   $$
\begin{equation}\label{3:eq}
 |\rho| \left| P_{\zeta^0}(\Phi) \right| \left( \tau_+|z|^{j-1}\left| Y Y^{\sigma} f \right|+\tau_- |z|^j \left| \partial Y^{\kappa} f \right| \right), \hspace{3mm} \text{with} \hspace{3mm} Y \in \Y_0, \hspace{3mm} \text{and}  
 \end{equation}
\begin{equation}\label{2:eq}
 \left( |\alpha|+\sqrt{\frac{v^{\underline{L}}}{v^0}}|\sigma|+\sqrt{\frac{v^{\underline{L}}}{v^0}} |\underline{\alpha}| \right) \left| P_{\zeta^0}(\Phi) \right| \left( \tau_+|z|^{j-1}\left| Y Y^{\kappa} f \right|+\tau_+ |z|^j \left| \partial Y^{\kappa} f \right| \right), \hspace{3mm} \text{with} \hspace{3mm} Y \in \Y_0.
 \end{equation}
If $Y Y^{\kappa}=Y^{\chi^1}$ and $\partial Y^{\kappa}=Y^{\chi^2}$, we have $$|\zeta^0|+|\chi^i| \leq |k|+|\beta|, \hspace{8mm} j-1 = N_0-\zeta^0_P-\chi^1_P \hspace{8mm} \text{and} \hspace{8mm} j = N_0-\zeta^0_P-\chi_P^2.$$
Thus, \eqref{3:eq} and \eqref{2:eq} give terms of \eqref{eq:cat1} and \eqref{eq:cat3} since we have, according to inequality \eqref{eq:zeta2} of Proposition \ref{ExtradecayLie} and for $\varphi \in \{\alpha, \underline{\alpha}, \rho, \sigma \}$,
$$  |\varphi| \lesssim \sum_{|\gamma| \leq |\gamma_0|+1} \tau_-^{-1} \left| \varphi \left( \mathcal{L}_{Z^{\gamma}} (F) \right) \right|+\tau_+^{-1} \left| \nabla_{Z^{\gamma}} F \right| .$$
\end{itemize}
It then remains to bound $T_F(P_{\zeta^0}(\Phi))z^jY^{\beta}f$. If $|\zeta^0| \geq 1$, there exists $ 1 \leq p \leq |\zeta^0|$ and $\left( \xi^i \right)_{1 \leq i \leq p}$ such that
$$P_{\zeta^0}(\Phi) = \prod_{i=1}^p Y^{\xi^i} \Phi, \hspace{8mm} \min_{1 \leq i \leq p} |\xi^i| \geq 1, \hspace{8mm} \sum_{i=1}^p |\xi^i|=|k| \hspace{8mm} \text{and} \hspace{8mm} \sum_{i=1}^p (\xi^i)_T=k_T.$$
Then, $T_F(P_{\zeta_0}(\Phi))=\sum_{i=1}^p T_F(Y^{\xi^i} \Phi ) \prod_{j \neq i} Y^{\xi^j} \Phi$ and let us, for instance, bound $T_F(Y^{\xi^1} \Phi) Y^{\beta} f \prod_{j =2}^p Y^{\xi^j} \Phi$. To lighten the notation, we define $\chi$ such that
$$P_{\chi}(\Phi)=\prod_{j =2}^p Y^{\xi^j} \Phi, \hspace{8mm} \text{so that} \hspace{8mm} (\chi_T,\chi_P)=\left(\zeta^0_T-\xi^1_T,\zeta^0_P-\xi^1_P \right).$$
Using Propositions \ref{ComuVlasov} and \ref{sourcePhi} (with $|\gamma_1| \leq 1$), $T_F(Y^{\xi_1} \Phi) P_{\chi}(\Phi) Y^{\beta} f $ can be written as a linear combination of terms of $(type \hspace{1mm} 1-\xi_1)$, $(type \hspace{1mm} 2-\xi_1)$, $(type \hspace{1mm} 3-\xi_1)$ (applied to $\Phi$), $(family \hspace{1mm} 1-\xi_1)$, $(family \hspace{1mm} 2-\xi_1)$ and $(family \hspace{1mm} 3-\xi_1)$, multiplied by $P_{\chi}(\Phi) Y^{\beta} f$. The treatment of the first three type of terms is similar to those which arise from $z^j P_{\zeta^0}(\Phi)T_F(Y^{\beta} f )$, so we only give details for the first one. We then have to bound
\begin{itemize}
\item $|z|^j\left| Z^{\gamma}(F_{\mu \nu}) \right|  \left|w^d P_{k,p}(\Phi) Y^{\kappa} \Phi P_{\chi}(\Phi) Y^{\beta} f \right|$, with $d \in \{0,1 \}$, $w \in \mathbf{k}_1$, $|\kappa| \geq 1$ $\max( |\gamma|, |k|+|\gamma|, |k|+|\kappa| ) \leq |\xi^1|$, $|k|+|\gamma|+|\kappa| \leq |\xi^1|+1$ and $p+k_P+\kappa_P+d \leq \xi^1_P$. Note now that
$$P_{k,p}(\Phi) Y^{\kappa} \Phi P_{\chi}(\Phi)= \Phi^n P_{\zeta}(\Phi), \hspace{3mm} \text{with} \hspace{3mm} n \leq p \leq |\xi^1| , \hspace{3mm} \zeta_T=k_T+\kappa_T+\chi_T \hspace{3mm} \text{and} \hspace{3mm} \zeta_P=k_P+\kappa_P+\chi_P.$$
Note moreover that
$$|\zeta|+|\gamma|+|\beta|=|k|+|\gamma|+|\kappa|+|\chi|+|\beta| \leq |\xi^1|+|\chi|+|\beta|+1 = |\zeta^0|+|\beta|+1, \hspace{3mm} |\zeta|+|\beta| \leq |\zeta^0|+|\beta|$$ and $\zeta_P+\beta_P+d=k_P+\kappa_P+d+\chi_P+\beta_P \leq \xi^1_P+\chi_P+\beta_P= \zeta^0_P+\beta_P$, which proves that this is a term of \eqref{eq:cat1}.
\item  $\tau_+|z|^j \left| \frac{v^{\mu}}{v^0} \mathcal{L}_{Z^{\gamma}}(F)_{\mu \theta}  P_{\chi}(\Phi) Y^{\beta} f \right| $, with $|\gamma| \leq |\xi^1|+1$ and $\gamma_T=\xi^1_T$. It is part of \eqref{eq:cat4} as
$$|\chi| < |k|, \hspace{6mm} \chi_T+\gamma_T=\chi_T+\xi^1_T = \zeta^0_T, \hspace{6mm} \chi_P \leq \zeta^0_P \hspace{6mm} \text{and} \hspace{6mm} |\chi|+|\gamma| \leq |\chi|+|\xi^1|+1 =|\zeta^0|+1.$$
\item $\left| Z^{\gamma}(F_{\mu \nu})\right| \left| z^j P_{k,p}(\Phi) P_{\chi} (\Phi)Y^{\beta} f \right| $, with $|k|+ |\gamma| \leq |\xi^1|-1$, $k_P \leq \xi^1_P$ and $p \leq |\xi^1|$, which is part of \eqref{eq:cat1}. Indeed, we can write
$$P_{k,p}(\Phi) P_{\chi} (\Phi) = \Phi^r P_{\zeta}(\Phi), \hspace{3mm} \text{with} \hspace{3mm} r \leq p \leq |\xi^1|, \hspace{3mm} \left( \zeta_T, \zeta_P \right) = \left( k_T+\chi_T,k_P+\chi_P \right)$$
and we then have $|\zeta|+|\gamma| = |k|+|\gamma|+|\chi| \leq |\xi^1|+|\chi| \leq |\zeta^0|$,
$$|\zeta|+|\gamma|+|\beta| \leq |\xi^1|+|\chi| +|\beta| \leq |\zeta^0|+|\beta| \hspace{3mm} \text{and} \hspace{3mm} \zeta_P+\beta_P \leq \xi^1_P+\chi_P+\beta_P = \zeta^0_P+\beta_P$$
\item $\tau_+ \left| \mathcal{L}_{X Z^{\gamma_0}}(F)\right|  \left|z^j P_{k,p}(\Phi) P_{\chi} (\Phi) Y^{\beta} f \right|$, with $|k|+ |\gamma_0| \leq |\xi_1|-1$, $k_P < \xi^1_P$ and $p \leq |\xi^1|$. By inequality \eqref{eq:Xdecay} of Proposition \ref{ExtradecayLie}
$$\exists \hspace{1mm} w \in \mathbf{k}_1, \hspace{8mm} \tau_+\left| \mathcal{L}_{X Z^{\gamma_0}}(F)\right| \lesssim (1+|w|) \sum_{|\gamma| \leq |\gamma_0|+1} \left| \nabla_{Z^{\gamma}} F \right|.$$ 
Note moreover that $k_P+\chi_P+\beta_P \leq \xi^1_P-1+\chi_P+\beta_P < \zeta^0_P+\beta_P$, as\footnote{Note that this term could appear only if $\xi^1_P \geq 1$.} $k_P < \xi^1_P$. We then have $j+1 \leq N_0-k_P-\chi_P-\beta_P$ and we obtain, using $|z^jw| \leq |z|^{j+1}+|w|^{j+1}$ and writting again $P_{k,p}(\Phi) P_{\chi} (\Phi) = \Phi^r P_{\zeta}(\Phi)$, terms which are in \eqref{eq:cat1} (the other conditions can be checked as previously).
\end{itemize}
\end{proof}

\begin{Rq}\label{hierarchyjustification}
There is three types of terms which bring us to consider a hierarchy on the quantities of the form $z^j P_{\xi}(\Phi) Y^{\beta} f$.
\begin{itemize}
\item Those of \eqref{eq:cat0}, as $\nabla_v \left( z^j \right)$ creates (at least) a $\tau_-$-loss and since $\tau_- F \sim \tau_+^{-1}$.
\item The first ones of \eqref{eq:cat3}. Indeed, we will have $|\rho| \lesssim \tau_+^{- \frac{3}{2}}\tau_-^{-\frac{1}{2}}$, so, using\footnote{We will be able to lose one power of $v^0$ as it is suggested by the energy estimate of Proposition \ref{energyf}.} $1 \lesssim \sqrt{v^0 v^{\underline{L}}}$,
$$\frac{\tau_+}{\tau_-}|\rho| \lesssim \frac{v^0}{\tau_+}+\frac{v^{\underline{L}}}{\tau_-^3}.$$
$v^{\underline{L}} \tau_-^{-3}$ will give an integrable term, as the component $v^{\underline{L}}$ will allow us to use the foliation $(u,C_u(t))$ of $[0,t] \times \R^3_x$. However, $v^0 \tau_+^{-1}$ will create a logarithmical growth.
\item The ones of \eqref{eq:cat4}, because of the $\tau_+$ weight and the fact that even the better component of $\mathcal{L}_{Z^{\gamma}}(F)$ will not have a better decay rate than $\tau_+^{-2}$.
\end{itemize} 
We will then classify them by $|\xi|+|\beta|$ and $j$, as one of these quantities is lowered in each of these terms.
\end{Rq}
\begin{Rq}\label{deuxblocs}
Let $\beta$ and, for $i \in \{1,2\}$, $\zeta^i$ be multi-indices such that $|\zeta^i|+|\beta| \leq N$, $|\zeta^1| \leq N-1$ and $N_0 \geq 2N-1$. We can adapt the previous proposition to $T_F \left( z^j P_{\zeta_1}(\Phi) P_{\zeta_2}(\Phi) Y^{\beta} f \right)$. One just has
\begin{itemize}
\item to add the factor $P_{\zeta_2}(\Phi)$ (or $P_{\zeta_1}(\Phi)$) in the terms of each categories and 
\item to replace conditions such as $j \leq N_0-\zeta_P-\sigma_P$ by $j \leq N_0-\zeta_P - \zeta^2_P-\sigma_P$ (or $j \leq N_0-\zeta_P - \zeta^1_P-\sigma_P$).
\end{itemize}
\end{Rq}

The worst terms are those of \eqref{eq:cat4} as they are responsible for the stronger growth of the top order energy norms. However, as suggested by the following proposition, we will have better estimates on $\| z^j P_{\xi}^X(\Phi) Y^{\beta} \|_{L^1_{x,v}}$.

\begin{Pro}\label{ComuPkpX}
Let $N \in \mathbb{N}$, $z \in \mathbf{k}_1$, $N_0 \geq N$, $\xi^0$, $\beta$ and $j \in \mathbb{N}$ be such that $|\xi^0| \leq N-1$, $|\xi^0|+|\beta| \leq N$ and $j \leq N_0-\xi^0_P-\beta_P$. Then,  $T_F(z^j P^X_{\xi^0}(\Phi) Y^{\beta} f)$ can be bounded by a linear combination of terms of \eqref{eq:cat0}, \eqref{eq:cat1}\hspace{-0.1mm}, \eqref{eq:cat3} and
\begin{equation}\label{eq:cat4bis}
\frac{\tau_+}{\tau_-} \left| \frac{v^{\mu}}{v^0} \mathcal{L}_{ Z^{\gamma}}(F)_{\mu \nu} w^{j} P^X_{\xi}(\Phi) Y^{\beta} f \right|, \tag{category $3-X$}
\end{equation}
with \hspace{2mm} $\xi_X < \xi^0_X$, \hspace{2mm} $\xi_T \leq \xi^0_T$, \hspace{2mm} $\xi_P \leq \xi^0_P$,  \hspace{2mm} $|\xi|+|\gamma|+|\beta| \leq |\xi|+|\beta|+1$, \hspace{2mm} $|\gamma| \leq |\xi|+1$, \hspace{2mm} $w \in \mathbf{k}_1$ \hspace{2mm} and \hspace{2mm} $j = N_0-\zeta_P-\beta_P$.

Note that the terms of \eqref{eq:cat3} only appear when $j=N_0-\xi^0_P-\beta_P$ and those of \eqref{eq:cat4bis} if $j=N_0-\xi^0_P-\beta_P$ and $|\xi^0| \geq 1$.
\end{Pro}
\begin{proof}
Proposition \ref{ComuVlasov} also holds for $Y^{\beta} \in \Y_X$ in view of Lemma \ref{ComuX} and the fact that $X$ can be considered as $c(v) \partial$. Then, one only has to follow the proof of the previous proposition and to apply Proposition \ref{sourceXPhi} where we used Proposition \ref{sourcePhi}. Hence, instead of terms of \eqref{eq:cat4}, we obtain
$$ \tau_+ \left| \frac{v^{\mu}}{v^0} \mathcal{L}_{X Z^{\gamma}} (F)_{\mu \nu} z^j P_{\chi}^X(\Phi) Y^{\beta} f \right|, \hspace{3mm} \text{with} \hspace{3mm} |\gamma| \leq |\xi^1|, \hspace{3mm} \chi_X < \xi^0_X, \hspace{3mm} \chi_T \leq \xi^0_T \hspace{3mm} \text{and} \hspace{3mm} \chi_P \leq \xi^0_P.$$
Apply now the second and then the first inequality of Proposition \ref{ExtradecayLie} to obtain that
$$\tau_+ \left| \frac{v^{\mu}}{v^0} \mathcal{L}_{X Z^{\gamma}} (F)_{\mu \theta} z^j P_{\chi}^X(\Phi) Y^{\beta} f \right| \lesssim  \left|  P_{\chi}^X(\Phi) Y^{\beta} f \right| \sum_{|\delta| \leq |\xi_1|+1} \hspace{-0.6mm} \left( \sum_{w \in \mathbf{k}_1 } \frac{  |w|^{j+1}}{\tau_-}\left| \frac{v^{\mu}}{v^0} \mathcal{L}_{Z^{\delta}}(F)_{\mu \theta} \right|+|z|^j\left| \mathcal{L}_{Z^{\delta}}(F)\right| \hspace{-0.6mm} \right)$$
which leads to terms of \eqref{eq:cat4bis} (if $j=N_0-\chi_P-\beta_P$) and \eqref{eq:cat1} (as $P_{\chi}^X(\Phi)$ can be bounded by a linear combination of $P_{\chi^0}(\Phi)$ with $\chi^0_T = \chi_T+\chi_X$ and $\chi^0_P \leq \chi_P$).
\end{proof}
\begin{Rq}
As we will mostly apply this commutation formula with a lower $N_0$ than for our utilizations of Proposition \ref{ComuPkp} or for $|\xi^0|=0$, we will have to deal with terms of \eqref{eq:cat4bis} only once (for \eqref{Auxenergy}).
\end{Rq}

\subsection{Commutation of the Maxwell equations}\label{subseccomuMax}

We recall the following property (see Lemma $2.8$ of \cite{massless} for a proof).
\begin{Lem}\label{basiccom}
Let $G$ and $M$ be respectively a $2$-form and a $1$-form such that $\nabla^{\mu} G_{\mu \nu}=M_{\nu}$. Then,
$$\forall \hspace{0.5mm} Z \in \mathbb{P}, \hspace{5mm} \nabla^{\mu} \mathcal{L}_{Z}(G)_{\mu \nu} = \mathcal{L}_{Z} (M)_{\nu}  \hspace{10mm} \text{and} \hspace{10mm} \nabla^{\mu} \mathcal{L}_{S}(G)_{\mu \nu} = \mathcal{L}_{S} (M)_{\nu} +2M_{\nu}.$$
If $g$ is a sufficiently regular function such that $\nabla^{\mu} G_{\mu \nu} = J(g)_{\nu}$, then
$$\forall \hspace{0.5mm} Z \in \mathbb{P}, \hspace{5mm}  \nabla^{\mu} \mathcal{L}_{Z}(G)_{\mu \nu} = J(\widehat{Z} g)_{\nu} \hspace{10mm} \text{and} \hspace{10mm}  \nabla^{\mu} \mathcal{L}_{S}(G)_{\mu \nu} = J(Sg)_{\nu}+3J(g)_{\nu}.$$
\end{Lem}

We need to adapt this formula since we will control $Yf$ and not $\widehat{Z}f$. We cannot close the estimates using only the formula
$$ J(\widehat{Z} f)=J(Y f ) -J(\Phi^k_{\widehat{Z}} X_k f )$$ as we will have $\|\Phi \|_{L^{\infty}_{v}} \lesssim \log^2 (\tau_+ )$ and since this small loss would prevent us to close the energy estimates.
\begin{Pro}\label{ComuMax1}
Let $Z \in \mathbb{K}$. Then, for $0 \leq \nu \leq 3$, $\nabla^{\mu} \mathcal{L}_Z(F)_{\mu \nu}$ can be written as a linear combination of the following terms.
\begin{itemize}
\item $\int_v \frac{v_{\nu}}{v^0}( X \Phi )^j Y^{\kappa} f dv$, with $j+|\kappa| \leq 1$.
\item $\frac{1}{\tau_+} \int_v c(t,x,v) z P_{k,p}(\Phi) Y^{\kappa} f dv$, with $z \in \mathbf{k}_1$, $p+|k|+|\kappa| \leq 3$ and $|k|+|\kappa| \leq 1$.
\end{itemize}
\end{Pro}
\begin{Rq}\label{rq11}
We would obtain a similar proposition if $J(f)_{\nu}$ was equal to $\int_v c_{\nu}(v) f dv$, except that we would have to replace $\frac{v_{\nu}}{v^0}$, in the first terms, by certain good coefficients $c(v)$.
\end{Rq}
\begin{proof}
If $Z \in \T$, the result ensues from Lemma \ref{basiccom}. Otherwise, we have, using \eqref{eq:X}
\begin{eqnarray}
\nonumber J(\widehat{Z} f) & = & J(Yf)-J(\Phi^{k} X_{k} f) \\ \nonumber
& = & J(Yf)_{\nu}+J(X_k(\Phi^k) f)_{\nu}-J(X_k(\Phi^k f)) \\ \nonumber
& = & J(Yf)+J(X_k(\Phi^k) f)-\frac{1}{1+t+r} \sum_{k=1}^3 J\left( \left(2z_{0k}\partial_t+\sum_{Z \in \mathbb{K}} c_Z(t,x,v) Z \right)(\Phi^k f) \right) .
\end{eqnarray}

Now, note that $J(z_{0k} \partial_t ( \Phi^k f ))= J(z_{0k} \Phi \partial_t f+z_{0k} \partial_t(\Phi) f)$ and, for $Z \in \mathbb{K} \setminus \T$ (in the computations below, we consider $Z=\Omega_{0i}$, but the other cases are similar), by integration by parts in $v$,
\begin{eqnarray}
\nonumber J\left( Z(\Phi^k f) \right) \hspace{-1.4mm} & = & \hspace{-1.4mm} J\left((Y-v^{0}\partial_{v^i}-\Phi^q X_q)(\Phi^k f) \right) \\ \nonumber & = & \hspace{-1.4mm} J \left(Y(\Phi^k)f+\Phi^k Y(f)-\Phi^q X_q ( \Phi^k ) f +\Phi^q \Phi^k X_q (f) \right) \hspace{-0.3mm} + \hspace{-0.3mm} \left(\int_v \Phi^k f dv \right) dx^{i} \hspace{-0.3mm} - \hspace{-0.3mm} \left( \int_v \Phi^k f \frac{v_i}{v^0}dv \right) dx^{0} \hspace{-0.2mm} ,
\end{eqnarray}
where $dx^{\mu}$ is the differential of $x^{\mu}$.
\end{proof}
We are now ready to establish the higher order commutation formula.
\begin{Pro}\label{ComuMaxN}
Let $R \in \mathbb{N}$ and $Z^{\beta} \in \mathbb{K}^{R}$. Then, for all $0 \leq \nu \leq 3$, $\nabla^{\mu} \mathcal{L}_{Z^{\beta}}(F)_{\mu \nu}$ can be written as a linear combination of terms such as
\begin{equation}\label{eq:comu1}
\int_v \frac{v_{\nu}}{v^0} P^X_{\xi}( \Phi ) Y^{\kappa} f dv, \hspace{3mm} \text{with} \hspace{3mm} |\xi|+|\kappa| \leq R, \tag{type $1-R$}
\end{equation}
\begin{equation}\label{eq:comu2}
\frac{1}{\tau_+}\int_v c(t,x,v) zP_{k,p}(\Phi) Y^{\kappa} f dv, \hspace{3mm} \text{with} \hspace{3mm} p+|k|+|\kappa| \leq 3R \hspace{3mm} \text{and} \hspace{3mm} k+|\kappa| \leq R.\tag{type $2-R$}
\end{equation}
\end{Pro}

\begin{proof}
We will use during the proof the following properties, arising from Lemma \ref{weights} and the definition of the $X_i$ vector field,
\begin{equation}\label{eq:Yz}
\forall \hspace{0.5mm} (Y,z) \in \Y \times \mathbf{k}_1, \hspace{3mm} \exists \hspace{0.5mm} z' \in \mathbf{k}_1, \hspace{3mm} Y(z)=c_1(v)z+z'+c_2(v)\Phi,
\end{equation}
\begin{equation}\label{eq:PX}
P^X_{\xi}(\Phi) = \sum_{\begin{subarray}{} \zeta_T = \xi_T+\xi_X \\ \hspace{2mm} \zeta_P \leq \xi_P \end{subarray}} c^\zeta(v) P_{\zeta}(\Phi).
\end{equation}
Let us suppose that the formula holds for all $|\beta_0| \leq R-1 $, with $R \geq 2$ (for $R-1=1$, see Proposition \ref{ComuMax1}). Let $(Z,Z^{\beta_0}) \in \mathbb{K} \times \mathbb{K}^{|\beta_0|}$ with $|\beta_0|=R-1$ and consider the multi-index $\beta$ such that $Z^{\beta}=Z Z^{\beta_0}$. We fix $\nu \in \llbracket 0,3 \rrbracket$. By the first order commutation formula, Remark \ref{rq11} and the induction hypothesis, $\nabla^{\mu} \mathcal{L}_{Z^{\beta}}(F)_{\mu \nu}$ can be written as a linear combination of the following terms (to lighten the notations, we drop the good coefficients $c(t,x,v)$ in the integrands of the terms given by Proposition \ref{ComuMax1}).
\begin{itemize}
\item $\int_v \frac{v_{\nu}}{v^0} \left( X \Phi \right)^j Y^{\kappa^0} \left(  P_{\xi}^X( \Phi ) Y^{\kappa} f \right) dv$, with $j+|\kappa^0| \leq 1$ and $|\xi|+|\kappa| \leq R-1$. It leads to $\int_v \frac{v_{\nu}}{v^0}  P_{\xi}^X( \Phi ) Y^{\kappa} f  dv$,
$$ \int_v \frac{v_{\nu}}{v^0} X(\Phi) P^X_{\xi}( \Phi ) Y^{\kappa} f dv, \hspace{5mm} \int_v \frac{v_{\nu}}{v^0} Y \left(P^X_{\xi}( \Phi ) \right) Y^{\kappa} f dv \hspace{5mm} \text{and} \hspace{5mm} \int_v \frac{v_{\nu}}{v^0} P^X_{\xi}( \Phi ) Y^{\kappa^0} Y^{\kappa} f dv,$$
which are all of \eqref{eq:comu1} since $Y\left(P^X_{\xi}( \Phi ) \right)=P_{\zeta}^X(\Phi)$, with $|\zeta|=|\xi|+1$, and $|\xi|+1+|\kappa| \leq R$.
\item $\int_v c(v) \left( X \Phi \right)^j Y^{\kappa^0} \left( \frac{z}{\tau_+} c(t,x,v) P_{k,p}(\Phi) Y^{\kappa} f \right) dv$, with $j+|\kappa^0| \leq 1$, $z \in \mathbf{k}_1$, $p+|k|+|\kappa| \leq 3R-3$ and $|k|+|\kappa| \leq R-1$. For simplicity, we suppose $c(v)=1$. As
$$ Y \left( \frac{1}{\tau_+} c(t,x,v) \right) =  \frac{1}{\tau_+} c_1(t,x,v)+ \frac{1}{\tau_+} c_2(t,x,v) \Phi,$$ we obtain, dropping the dependance in $(t,x,v)$ of the good coefficients, the following terms (with the first one corresponding to $j=1$ and the other ones to $j=0$).
$$\frac{1}{\tau_+}\int_v c zP_{(k_T+1,k_P),p+1}(\Phi) Y^{\kappa} f dv, \hspace{5mm} \frac{1}{\tau_+}\int_v (c+c_1) zP_{k,p}(\Phi) Y^{\kappa} f dv, \hspace{5mm} \frac{1}{\tau_+}\int_v c_2 zP_{k,p+1}(\Phi) Y^{\kappa} f dv, $$
$$\frac{1}{\tau_+}\int_v c zP_{(k_T+\kappa^0_T,k_P+\kappa^0_P),p}(\Phi) Y^{\kappa} f dv, \hspace{5mm} \frac{1}{\tau_+}\int_v c Y(z) P_{k,p}(\Phi) Y^{\kappa} f dv, \hspace{5mm} \frac{1}{\tau_+}\int_v c z P_{k,p}(\Phi) Y^{\kappa^0} Y^{\kappa} f dv.$$
It is now easy to check that all these terms are of \eqref{eq:comu2} (for the penultimate term, recall in particular \eqref{eq:Yz}). For instance, for the first one, we have
$$(p+1)+(|k|+1)+|\kappa|=(p+|k|+|\kappa|)+2 \leq 3R-1 \leq 3R  \hspace{3mm} \text{and} \hspace{3mm} (|k|+1)+|\kappa| \leq (|k|+|\kappa|)+1 \leq R.$$ 
\item $\frac{1}{\tau_+} \int_v  zP_{k^0,p^0}(\Phi) Y^{\kappa^0} \left( P_{\xi}^X( \Phi ) Y^{\kappa} f  \right) dv$, with $p^0+|k^0|+|\kappa^0| \leq 3$, $|k^0|+|\kappa^0| \leq 1$ and $|\xi|+|\kappa| \leq R-1$. According to \eqref{eq:PX}, we can suppose without loss of generality that $P_{\xi}^X( \Phi )=c(v) P_{\zeta}( \Phi )$, with $|\zeta| \leq |\xi|$. If $|k^0|=1$, we obtain
$$ \frac{1}{\tau_+} \int_v c(v) zP_{(\zeta_T+k^0_T,\zeta_P+k^0_P),r}( \Phi ) Y^{\kappa} f dv, \hspace{1cm} \text{with} \hspace{1cm} r \leq |\zeta|+p^0,$$
which is of \eqref{eq:comu2} since $$(|\zeta|+p^0)+(|\zeta|+|k^0|)+|\kappa| \leq (p^0+|k^0|)+2(|\xi|+|\kappa|)\leq 2R+1 \leq 3R \hspace{3mm} \text{and} \hspace{3mm} (|\zeta|+|k^0|)+|\kappa| \leq R.$$
If $|k_0|=0$, we obtain, with $r \leq |\zeta|+p^0$ and since $Y^{\kappa_0}(c(v))=c_1(v)$,
$$ \frac{1}{\tau_+} \int_v  (c+c_1)(v) zP_{(\zeta_T,\zeta_P),r}(\Phi ) Y^{\kappa} f  dv, \hspace{5mm}  \frac{1}{\tau_+} \int_v  c(v) zP_{(\zeta_T,\zeta_P),r}( \Phi ) Y^{\kappa^0} Y^{\kappa} f   dv  \hspace{5mm} \text{and} \hspace{5mm}$$
$$\frac{1}{\tau_+} \int_v  c(v) zP_{(\zeta_T+\kappa^0_T,\zeta_P+\kappa^0_P),r}( \Phi ) Y^{\kappa} f   dv, $$
which are of \eqref{eq:comu2} since
$$|\zeta|+1+|\kappa| \leq R \hspace{3mm} \text{and} \hspace{3mm} |\zeta|+p^0+|\zeta|+|\kappa^0|+|\kappa| \leq 3+2R-2 \leq 3R .$$
\item $\frac{1}{\tau_+} \int_v w P_{k^0,p^0}(\Phi) Y^{\kappa^0} \left( \frac{z}{\tau_+} c(t,x,v) P_{k,p}(\Phi) Y^{\kappa} f  \right) dv$, with $(w,z) \in \mathbf{k}_1^2$, $p^0+|k^0|+|\kappa^0| \leq 3$, $|k^0|+|\kappa^0| \leq 1$, $p+|k|+|\kappa| \leq 3R-3$ and $|k|+|\kappa| \leq R-1$.

If $|k^0|=1$, we obtain the term
$$\frac{1}{\tau_+} \int_v  c_0(t,x,v) wP_{k+k^0,p+p^0}(\Phi) Y^{\sigma} f  dv, \hspace{3mm} \text{where} \hspace{3mm} c_0(t,x,v):=c(t,x,v)\frac{z}{\tau_+},$$
which is of \eqref{eq:comu2} since 
$$|k+k^0|+(p+p^0)+|\kappa| \leq (p+|k|+|\kappa|)+(p^0+|k^0|) \leq 3R \hspace{3mm} \text{and} \hspace{3mm} |k+k^0|+|\kappa|=(|k|+|\kappa|)+1 \leq R.$$
If $|k_0|=0$, using that
$$ \frac{z}{\tau_+}c(t,x,v)+Y^{\kappa^0} \left( \frac{z}{\tau_+} c(t,x,v) \right) = c_3(t,x,v)+c_4(t,x,v)\Phi,$$
we obtain the following terms of \eqref{eq:comu2},
$$\frac{1}{\tau_+} \int_v   \left( c_3(t,x,v) P_{k,p+p^0}(\Phi)+c_4(t,x,v) P_{k,p+p^0+1}(\Phi) \right) w Y^{\kappa} f dv, $$
$$\frac{1}{\tau_+} \int_v  c_0(t,x,v)w P_{k,p+p^0}(\Phi) Y^{\kappa^0}Y^{\kappa} f   dv \hspace{6mm} \text{and} \hspace{6mm}  \frac{1}{\tau_+} \int_v  c_0(t,x,v)w P_{(k_T,\kappa^0_T,k_P+\kappa^0_P),p+p^0}(\Phi) Y^{\kappa} f.$$
\end{itemize}
\end{proof}

Recall from the transport equation satisfied by the $\Phi$ coefficients that, in order to estimate $Y^{\gamma} \Phi$, we need to control $\mathcal{L}_{Z^{\beta}}(F)$ with $|\beta|=|\gamma|+1$. Consequently, at the top order, we will rather use the following commutation formula. 

\begin{Pro}\label{CommuFsimple}
Let $Z^{\beta} \in \mathbb{K}^{|\beta|}$. Then,
$$\nabla^{\mu} \mathcal{L}_{Z^{\beta}}(F)_{\mu \nu} = \sum_{\begin{subarray}{} |q|+|\kappa| \leq |\beta| \\ \hspace{1.5mm} |q| \leq |\beta|-1 \\ \hspace{1mm}  p \leq q_X+\kappa_T \end{subarray}} J \left( c^{k,q}_{\kappa}(v) P_{q,p}(\Phi) Y^{\kappa} f \right),$$
where $P_{q,p}(\Phi)$ can contain $\Y_X$, and not merely $\Y$, derivatives of $\Phi$. We then denote by $q_X$ its number of $X$ derivatives.
\end{Pro}
\begin{proof}
Iterating Lemma \ref{basiccom}, we have
\begin{equation}\label{comuiterbasic}
 \nabla^{\mu} \mathcal{L}_{Z^{\beta}}(F)_{\mu \nu} = \sum_{|\gamma| \leq |\beta| } C^{\beta}_{\gamma} J \left( \widehat{Z}^{\gamma} f \right).
\end{equation}
The result then follows from an induction on $|\gamma|$. Indeed, write $\widehat{Z}^{\gamma}=\widehat{Z} \widehat{Z}^{\gamma_0}$ and suppose that
\begin{equation}\label{lifttomodif}
\widehat{Z}^{\gamma_0} f=\sum_{\begin{subarray}{} |q|+|\kappa| \leq |\gamma_0| \\ \hspace{1.5mm} |q| \leq |\gamma_0|-1 \\ \hspace{1mm}  p \leq q_X+\kappa_T \end{subarray}} c^{k,q}_{\kappa}(v) P_{q,p}(\Phi) Y^{\kappa} f  .
\end{equation}
If $\widehat{Z}=\partial \in \T$, then 
$$\widehat{Z}^{\gamma} f = \sum_{\begin{subarray}{} |q|+|\kappa| \leq |\gamma_0| \\ \hspace{1.5mm} |q| \leq |\gamma_0|-1 \\ \hspace{1mm}  p \leq q_X+\kappa_T \end{subarray}} c^{k,q}_{\kappa}(v) P_{(q_T+1,q_P,q_X),p}(\Phi) Y^{\kappa} f+c^{k,q}_{\kappa}(v) P_{q,p}(\Phi) \partial Y^{\kappa} f = \sum_{\begin{subarray}{} |q|+|\kappa| \leq |\gamma| \\ \hspace{1.5mm} |q| \leq |\gamma|-1 \\ \hspace{1mm}  p \leq q_X+\kappa_T \end{subarray}} c^{k,q}_{\kappa}(v) P_{q,p}(\Phi) Y^{\kappa} f.$$
Otherwise $\gamma_P=(\gamma_0)_P+1$ and write $\widehat{Z}=Y-\Phi X$ with $Y \in \Y_0$. Hence, using $X Y^{\kappa} f=c(v) \partial Y^{\kappa} f$,
\begin{eqnarray}
\nonumber \widehat{Z}^{\gamma} f \hspace{-2mm} & = & \hspace{-2mm} \sum_{\begin{subarray}{} |q|+|\kappa| \leq |\gamma_0| \\ \hspace{1.5mm} |q| \leq |\gamma_0|-1 \\ \hspace{1mm}  p \leq q_X+\kappa_T \end{subarray}} \Big( Y \left(c^{k,q}_{\kappa}(v) \right) P_{q,p}(\Phi) Y^{\kappa} f+ c^{k,q}_{\kappa}(v) P_{(q_T,q_P+1,q_X),p}(\Phi) Y^{\kappa} f +c^{k,q}_{\kappa}(v) P_{q,p}(\Phi) YY^{\kappa} f \\ \nonumber 
& & \hspace{-4.5mm} +c^{k,q}_{\kappa}(v) P_{(q_T,q_P,q_X+1),p+1}(\Phi) Y^{\kappa} f  +c^{k,q}_{\kappa}(v) P_{(q_T,q_P,q_X),p+1}(\Phi) c(v) \partial Y^{\kappa} \Big) \lesssim \hspace{-0.2mm} \sum_{\begin{subarray}{} |q|+|\kappa| \leq |\gamma| \\ \hspace{1.5mm} |q| \leq |\gamma|-1 \\ \hspace{1mm}  p \leq q_X+\kappa_T \end{subarray}} c^{k,q}_{\kappa}(v) P_{q,p}(\Phi) Y^{\kappa} f  .
\end{eqnarray}
\end{proof}

\section{Energy and pointwise decay estimates}\label{sec4}

In this section, we recall classical energy estimates for both the electromagnetic field and the Vlasov field and how to obtain pointwise decay estimates from them. For that purpose, we need to prove Klainerman-Sobolev inequalities for velocity averages, similar to Theorem $8$ of \cite{FJS} or Theorem $1.1$ of \cite{dim4}, adapted to modified vector fields.

\subsection{Energy estimates}\label{energy}
For the particle density, we will use the following approximate conservation law.
\begin{Pro}\label{energyf}
Let $H : [0,T[ \times \R^3_x \times \R^3_v \rightarrow \R $ and $g_0 : \R^3_x \times \R^3_v \rightarrow \R$ be two sufficiently regular functions and $F$ a sufficiently regular $2$-form defined on $[0,T[ \times \R^3$. Then, $g$, the unique classical solution of 
\begin{eqnarray}
\nonumber T_F(g)&=&H \\ \nonumber
g(0,.,.)&=&g_0,
\end{eqnarray}
satisfies the following estimate,
$$\forall \hspace{0.5mm} t \in [0,T[, \hspace{1cm} \| g \|_{L^1_{x,v}}(t)+ \sup_{u \in \R} \left\|\frac{v^{\underline{L}}}{v^0} g \right\|_{L^1(C_u(t))L^1_{v}}  \leq 2 \| g_0 \|_{L^1_{x,v}}+2 \int_0^t  \int_{\Sigma_s}  \int_v  |H | \frac{dv}{v^0}dxds.$$
\end{Pro}

\begin{proof}
The estimate follows from the divergence theorem, applied to $\int_v \frac{v^{\mu}}{v^0}|f|dv$ in $[0,t] \times \R^3$ and $V_u(t)$, for all $u \leq t$. We refer to Proposition $3.1$ of \cite{massless} for more details.
\end{proof}

We consider, for the remainder of this section, a $2$-form $G$ and a $1$-form $J$, both defined on $[0,T[ \times \R^3$ and sufficiently regular, such that
\begin{eqnarray}
\nonumber \nabla^{\mu} G_{\mu \nu} & = & J_{\nu} \\ \nonumber
\nabla^{\mu} {}^* \! G_{\mu \nu} & = & 0.
\end{eqnarray}
We denote by $(\alpha,\underline{\alpha},\rho,\sigma)$ the null decomposition of $G$. As $\int_{\Sigma_0} r \rho(G)|(0,x)dx=+\infty$ when the total charge is non-zero, we cannot control norms such as $\left\|\sqrt{\tau_+} \rho \right\|_{L^2(\Sigma_t)}$ and we then separate the study of the electromagnetic field in two parts. 
\begin{itemize}
\item The exterior of the light cone, where we propagate $L^2$ norms on the chargeless part $\F$ of $F$ (introduced, as $\Ff$, in Definition \ref{defpure1}), which has a finite initial weighted energy norm. The pure charge part $\Ff$ is given by an explicit formula, which describes directly its asymptotic behavior. As $F=\F+\Ff$, we are then able to obtain pointwise decay estimates on the null components of $F$.
\item The interior of the light cone, where we can propagate $L^2$ weighted norms of $F$ since we control its flux on $C_0(t)$ with the bounds obtained on $\F$ in the exterior region.
\end{itemize}
We then introduce the following energy norms.

\begin{Def}\label{defMax1}
Let $N \in \mathbb{N}$. We define, for $t \in [0,T[$,
\begin{eqnarray} 
\nonumber \mathcal{E}^0[G](t) & : = & \int_{\Sigma_t}\left( |\alpha|^2+|\underline{\alpha}|^2+2|\rho|^2+2|\sigma|^2 \right)dx +\sup_{u \leq t} \int_{C_u(t)} \left( |\alpha|^2+|\rho|^2+|\sigma|^2 \right) dC_u(t), \\
\nonumber \mathcal{E}_N^0[G](t) & : = & \sum_{\begin{subarray}{l} \hspace{0.5mm} Z^{\gamma} \in \mathbb{K}^{|\gamma|} \\ \hspace{1mm} |\gamma| \leq N \end{subarray}}  \mathcal{E}^0_N[\mathcal{L}_{ Z^{\gamma}}(G)](t), \\
\nonumber \mathcal{E}^{S, u \geq 0}[G](t) & := & \int_{\Sigma^0_t}  \tau_+ \left( |\alpha|^2+|\rho|^2+|\sigma|^2 \right)+\tau_- |\underline{\alpha}| dx+ \sup_{0 \leq u \leq t} \int_{C_u(t)} \tau_+  |\alpha|^2+\tau_-\left( |\rho|^2+|\sigma|^2 \right) d C_u(t) . \\
\nonumber \mathcal{E}_N[G](t) & := &\sum_{\begin{subarray}{l} \hspace{0.5mm} Z^{\gamma} \in \mathbb{K}^{|\gamma|} \\ \hspace{1mm} |\gamma| \leq N \end{subarray}}  \mathcal{E}^{S, u \geq 0}_N[\mathcal{L}_{ Z^{\gamma}}(G)](t) \\
\nonumber \mathcal{E}^{S,u \leq 0}[G](t) & := & \int_{\overline{\Sigma}^{0}_t}  \tau_+ \left( |\alpha|^2+|\rho|^2+|\sigma|^2 \right)+\tau_- |\underline{\alpha}| dx+ \sup_{ u \leq 0} \int_{C_u(t)} \tau_+  |\alpha|^2+\tau_-\left( |\rho|^2+|\sigma|^2 \right) d C_u(t) \\ 
\nonumber \mathcal{E}^{Ext}_N[G](t) & : = & \sum_{\begin{subarray}{l} \hspace{0.5mm} Z^{\gamma} \in \mathbb{K}^{|\gamma|} \\ \hspace{1mm} |\gamma| \leq N \end{subarray}}  \mathcal{E}^{S,u \leq 0}_N[\mathcal{L}_{Z^{\gamma}}(G)](t).
\end{eqnarray}
\end{Def}
The following estimates hold.
\begin{Pro}\label{energyMax1}
Let $\overline{S} := S+ \partial_t \mathds{1}_{u >0}+2 \tau_- \partial_t \mathds{1}_{u \leq 0}$. For all $ t \in [0,T[$,
\begin{eqnarray}
\nonumber \mathcal{E}^0[G](t) & \leq & 2\mathcal{E}^0[G](0) + 8\int_0^t \int_{\Sigma_s} |G_{\mu 0} J^{\mu}| dx ds \\ \nonumber
\mathcal{E}^{S,u \leq 0}[G](t) & \leq & 6\mathcal{E}^{S,u \leq 0}[G](0) + 8\int_0^t \int_{\Si^{0}_s } \left| \overline{S}^{\nu}  G_{\nu \mu } J^{\mu} \right| dx ds \\ \nonumber
\mathcal{E}^{S,u \geq 0}[G](t) & \leq & 3\mathcal{E}^{S,u \leq 0}[\G](t) + 8\int_0^t \int_{\Sig^{0}_s } \left| \overline{S}^{\nu}  G_{\nu \mu } J^{\mu} \right| dx ds.
\end{eqnarray}
\end{Pro}
\begin{proof}
For the first inequality, apply the divergence theorem to $T_{\mu 0}[G]$ in $[0,t] \times \R^3$ and $V_u(t)$, for all $u \leq t$. Let us give more details for the other ones. Denoting $ T[G]$ by $T$ and using Lemma \ref{tensorcompo}, we have, if $u \leq 0$,
\begin{eqnarray}
\nonumber  \nabla^{\mu} \left( \tau_- T_{\mu 0} \right)  & = & \tau_-\nabla^{\mu} T_{\mu 0}-\frac{1}{2}\underline{L} \left( \tau_- \right) T_{L 0}  \\ \nonumber
& = & \tau_-  \nabla^{\mu} T_{\mu 0} -\frac{u}{2\tau_-} \left( \left| \alpha  \right|^2+\left| \rho \right|^2+\left| \sigma  \right|^2 \right) \hspace{2mm} \geq \hspace{2mm}  \tau_-  \nabla^{\mu} T_{\mu 0}. 
\end{eqnarray}
Consequently, applying Corollary \ref{tensorderiv} and the divergence theorem in $V_{u_0}(t)$, for $u_0 \leq 0$, we obtain
\begin{equation}\label{eq:1}
 \int_{\Si^{u_0}_t} \tau_- T_{00}dx + \frac{1}{\sqrt{2}} \int_{C_{u_0}(t)} \tau_-T_{L0}dC_{u_0}(t) \leq \int_{\Si^{u_0}_0} \sqrt{1+r^2} T_{00}dx-\int_0^t \int_{\Si^{u_0}_s } \tau_-   G_{0 \nu} J^{\nu} dx ds.
 \end{equation}
On the other hand, as $\nabla^{\mu} S^{\nu}+\nabla^{\nu} S^{\mu}=2\eta^{\mu \nu}$ and ${T_{\mu}}^{\mu}=0$, we have
\begin{eqnarray}
\nonumber  \nabla^{\mu} \left( T_{\mu \nu} S^{\nu} \right)  & = & \nabla^{\mu} T_{\mu \nu}S^{\nu}+T_{\mu \nu} \nabla^{\mu} S^{\nu}  \\ \nonumber
& = & G_{\nu \lambda} J^{\lambda} S^{\nu} +\frac{1}{2} T_{\mu \nu} \left( \nabla^{\mu} S^{\nu}+\nabla^{\nu} S^{\mu} \right) \\ \nonumber 
& = & G_{\nu \lambda} J^{\lambda} S^{\nu}. 
\end{eqnarray}
Applying again the divergence theorem in $V_{u_0}(t)$, for all $u_0 \leq 0$, we get
\begin{equation}\label{eq:2}
 \int_{\Si^{u_0}_t}  T_{0 \nu} S^{\nu} dx + \frac{1}{\sqrt{2}} \int_{C_{u_0}(t)} T_{L \nu} S^{\nu}dC_{u_0}(t) = \int_{\Si^{u_0}_0}  T_{0 \nu} S^{\nu}  dx-\int_0^t \int_{\Si^{u_0}_s }   G_{\mu \nu} J^{\mu} S^{\nu}  dx ds.
 \end{equation}
Using Lemma \ref{tensorcompo} and $2S=(t+r)L+(t-r) \underline{L}$, notice that
\begin{flalign*}
& \hspace{1.3cm} 4\tau_-T_{00} = \tau_-\left( |\alpha|^2+|\underline{\alpha}|^2+2|\rho|^2+2|\sigma|^2 \right), \hspace{10mm} 4T_{0 \nu} S^{\nu} = (t+r)|\alpha|+(t-r)|\underline{\alpha}|+2t(|\rho|+|\sigma|), & \\
& \hspace{1.3cm} 2 \tau_- T_{L0}= \tau_- \left( |\alpha|^2+|\rho|^2+|\sigma|^2 \right), \hspace{23mm} 2 T_{L \nu} S^{\nu} = (t+r) |\alpha|^2+(t-r)|\rho|^2+(t-r)|\sigma|^2, &
\end{flalign*}
and then add twice \eqref{eq:1} to \eqref{eq:2}. The second estimate then follows and we now turn on the last one. Recall that $\nabla^{\mu} T_{\mu \nu} G=G_{\nu \lambda} J^{\lambda}$ and $\nabla^{\mu} \left(  T_{\mu \nu} S^{\nu} \right)  =  G_{\nu \lambda} J^{\lambda} S^{\nu}$. Hence, by the divergence theorem applied in $[0,t] \times \R^3 \setminus V_{0}(t)$, we obtain
\begin{equation}\label{eq:1bis}
 \int_{\Sig^0_t} \left( T_{00} +T_{0 \nu} S^{\nu} \right)dx = \frac{1}{\sqrt{2}} \int_{C_{0}(t)} \left( T_{L0} +T_{L \nu} S^{\nu}  \right) d C_{0}(t) - \int_0^t \int_{\Sig^{0}_s }    G_{0 \nu} J^{\nu} + S^{\nu} G_{\nu \mu } J^{\mu}  dx ds.
 \end{equation}
By Lemma \ref{tensorcompo}, we have $4T_{00} = \left( |\alpha|^2+|\underline{\alpha}|^2+2|\rho|^2+2|\sigma|^2 \right) $, so that
\begin{equation}\label{eq:1bbis}
 4T_{00} +4T_{0 \nu} S^{\nu} \geq \tau_+|\alpha|^2+\tau_-|\underline{\alpha}|^2+\tau_+|\rho|^2+\tau_+|\sigma|^2 \geq 0 \hspace{1cm} \text{on} \hspace{5mm} \Sig^{0}_t.
\end{equation} 
Consequently, the divergence theorem applied in $ V_{u}(t) \setminus V_0(t)$, for $0 \leq u \leq t$, gives
\begin{equation}\label{eq:2bis}
\frac{1}{\sqrt{2}} \int_{C_{u}(t)}  \left(  T_{L0} +T_{L \nu} S^{\nu} \right) dC_u(t) \leq  \frac{1}{\sqrt{2}} \int_{C_{0}(t)}  \left(  T_{L0} +T_{L \nu} S^{\nu} \right) d C_0(t) - \int_{V_u(t) \setminus V_0(t)} \left( G_{0 \nu} J^{\nu} + S^{\nu} G_{\nu \mu } J^{\mu} \right).
 \end{equation}
Not now that $ T_{L0} +T_{L \nu} S^{\nu} \geq \tau_+|\alpha|^2+\tau_-|\rho|^2+\tau_-|\sigma|^2$ if $u \geq 0$ since
\begin{flalign*}
& \hspace{0.8cm} 2  T_{L0}=  |\alpha|^2+|\rho|^2+|\sigma|^2   \hspace{7mm} \text{and} \hspace{7mm} 2 T_{L \nu} S^{\nu} = (t+r) |\alpha|^2+(t-r)|\rho|^2+(t-r)|\sigma|^2. &
\end{flalign*}
It then remains to take the $\sup$ over all $0 \leq u \leq t$ in \eqref{eq:2bis}, to combine it with \eqref{eq:1bis}, \eqref{eq:1bbis} and to remark that
\begin{eqnarray}
\nonumber 2\int_{C_{0}(t)}  T_{L0} +T_{L \nu} S^{\nu} d C_0(t) & \leq &   \int_{C_{0}(t)}  |\rho|^2+|\sigma|^2 d C_{0}(t)+\int_{C_0(t)} \tau_+ |\alpha|^2   d C_0(t) \\ \nonumber
& \leq & \mathcal{E}^{S, u \leq 0}[\G](t),
\end{eqnarray}
since $G=\G$ on $C_0(t)$.
\end{proof}
\subsection{Pointwise decay estimates}
\subsubsection{Decay estimates for velocity averages}
As the set of our commutation vector fields is not $\K$, we need to modify the following standard Klainerman-Sobolev inequality, which was proved in \cite{FJS} (see Theorem $8$).

\begin{Pro}\label{KSstandard}
Let $g$ be a sufficiently regular function defined on $[0,T[ \times \R^3_x \times \R^3_v$. Then, for all $(t,x) \in [0,T[ \times \R^3$,
$$\forall \hspace{0.5mm} (t,x) \in [0,T[ \times \R^3, \hspace{1cm} \int_{v \in \R^3} |g(t,x,v)| dv \lesssim \frac{1}{\tau_+^2 \tau_-} \sum_{\begin{subarray}{l} \widehat{Z}^{\beta} \in \K^{|\beta|} \\ \hspace{1mm} |\beta| \leq 3 \end{subarray}}\|\widehat{Z}^{\beta} g \|_{L^1_{x,v}}(t).$$
\end{Pro}

We need to rewrite it using the modified vector fields. For the remainder of this section, $g$ will be a sufficiently regular function defined on $[0,T[ \times \R^3_x \times \R^3_v$. We also consider $F$, a regular $2$-form, so that we can consider the $\Phi$ coefficients introduced in Definition \ref{defphi} and we suppose that they satisfy the following pointwise estimates, with $M_1 \geq 7$ a fixed integer. For all $(t,x,v) \in [0,T[ \times \R^3 \times \R^3$,
$$|Y \Phi|(t,x,v) \lesssim \log^{\frac{7}{2}}(1+\tau_+), \hspace{8mm} |\Phi|(t,x,v) \lesssim \log^2(1+\tau_+) \hspace{8mm} \text{and} \hspace{8mm} \sum_{ |\kappa| \leq 3}  |Y^{\kappa} \Phi|(t,x,v) \lesssim \log^{M_1}(1+\tau_+).$$

\begin{Pro}\label{KS1}
For all $(t,x) \in [0,T[ \times \R^3$,
$$\tau_+^2 \tau_- \int_{v \in \R^3} |g(t,x,v)| dv \lesssim  \sum_{ |\xi|+|\beta| \leq 3 } \left\| P^X_{\xi}(\Phi)Y^{\beta} g \right\|_{L^1_{x,v}} \hspace{-0.8mm} (t)+ \sum_{ |\kappa| \leq \min(2+\kappa_T,3)} \sum_{z \in \mathbf{k}_1}\frac{\log^{6M_1}(3+t)}{1+t} \left\| z Y^{\kappa} g \right\|_{L^1_{x,v}} \hspace{-0.8mm} (t) .$$
\end{Pro}
\begin{Rq}
This inequality is suitable for us since we will bound $\left\| P^X_{\xi}(\Phi)Y^{\beta} g \right\|_{L^1_{x,v}} $ without any growth in $t$. Moreover, observe that $Y^{\kappa}$ contains at least a translation if $|\kappa|=3$, which is compatible with our hierarchy on the weights $z \in \V$ (see Remark \ref{rqjustifnorm}).
\end{Rq}
\begin{proof}
Let $(t,x) \in [0,T[ \times \R^n$. Consider first the case $|x| \leq \frac{1+t}{2}$, so that, with $\tau := 1+t$,
$$\forall \hspace{0.5mm} |y| \leq \frac{1}{4}, \hspace{3mm} \tau \leq 10(1+|t-|x+\tau y||).$$ For a sufficiently regular function $h$, we then have, using Lemmas \ref{goodderiv} and then \ref{lift2},
\begin{eqnarray}
\nonumber
\left| \partial_{y^i} \left( \int_v |h|(t,x+\tau y,v) dv \right) \right|  & = &  \left| \tau \partial_i \int_v  |h|(t,x+\tau y,v) dv \right| \\ \nonumber
& \lesssim &   \left| (1+|t-|x+\tau y||) \partial_i \int_v  |h|(t,x+\tau y,v) dv \right| \\ \nonumber & \lesssim &  \sum_{Z \in \mathbb{K}} \left| Z \int_v   |h|(t,x+\tau y,v) dv \right| \\ \nonumber & \lesssim &  \sum_{\begin{subarray}{l} |\xi|+|\beta| \leq 1 \\ \hspace{3mm} p \leq 1  \end{subarray}} \sum_{z \in \V}  \int_v \hspace{-0.5mm}  \left(   |P^X_{\xi}(\Phi) Y^{\beta} h|+\frac{\log^7 (1+\tau_+)}{\tau_+} |z \partial^p_t h|   \right) (t,x+\tau y,v) dv . 
\end{eqnarray}
Using a one dimensional Sobolev inequality, we obtain, for $\delta=\frac{1}{4 \sqrt{3}}$ (so that $|y| \leq \frac{1}{4}$ if $|y^i| \leq \delta$ for all $1 \leq i \leq 3$),
\begin{eqnarray}
\nonumber \int_v |g|(t,x,v) dv   & \lesssim &  \sum_{n=0}^1  \int_{|y^1| \leq \delta} \left| \left(\partial_{y^1} \right)^n \int_v |g|(t,x+\tau(y^1,0,0),v) dv  \right| dy^1 \\ \nonumber
& \lesssim &  \sum_{\begin{subarray}{l}  |\xi|+|\beta| \leq 1 \\ \hspace{3mm} p \leq 1 \\ \hspace{2.5mm} z \in \mathbf{k}_1 \end{subarray}} \int_{|y^1| \leq \delta} \int_v   \left( |P^X_{\xi}(\Phi) Y^{\beta}g|+\frac{\log^7 ( 3+t)}{1+t} |z \partial^p_t g|   \right) (t,x+\tau (y^1,0,0),v) dv dy^1 \hspace{-0.7mm}.
\end{eqnarray}
Repeating the argument for $y^2$ and the functions $\int_v P^X_{\xi}(\Phi) Y^{\beta}g dv$ and $ \int_v z \partial^p_t g dv$, we get, as $|z| \leq 2t$ in the region considered and dropping the dependence in $(t,x+\tau(y^1,y^2,0),v)$ of the functions in the integral,
$$\int_v |g|(t,x,v) dv \lesssim \sum_{\begin{subarray}{l}  |\xi|+|\beta| \leq 2 \\ \hspace{2.5mm} z \in \mathbf{k}_1 \end{subarray}} \sum_{\begin{subarray}{l}  |\zeta|+|\kappa| \leq 2 \\ |\kappa| \leq 1+\kappa_T \end{subarray}} \int_{|y^1| \leq \delta} \int_{|y^2| \leq \delta} \int_v  |P^X_{\xi}(\Phi) Y^{\beta} g|+\frac{\log^{14} (3+t)}{1+t} |z P^X_{\zeta}(\Phi) Y^{\kappa}  g| dv dy^1 dy^2.$$
Repeating again the argument for the variable $y^3$, we finally obtain
$$\int_v |g|(t,x,v) dv \lesssim \sum_{\begin{subarray}{l}  |\xi|+|\beta| \leq 3 \\ \hspace{2.5mm} z \in \mathbf{k}_1 \end{subarray}} \sum_{\begin{subarray}{l}  |\zeta|+|\kappa| \leq 3 \\ |\kappa| \leq 2+\kappa_T \end{subarray}} \int_{|y| \leq \frac{1}{4}}  \int_v  |P^X_{\xi}(\Phi) Y^{\beta} g|+ \frac{\log^{21} (3+t)}{1+t} |z P_{\zeta}^X(\Phi) Y^{\kappa} g|dv(t,x+\tau y)dy.$$
It then remains to remark that $\left| P^X_{\zeta}(\Phi) \right| \lesssim \log^{3M_1}(3+t)$ on the domain of integration and to make the change of variables $z=\tau y$. Note now that one can prove similarly that, for a sufficiently regular function $h$,
\begin{equation}\label{eq:sobsphere}
\int_v |h|(t,r,\theta,\phi)dv \lesssim \sum_{\begin{subarray}{l}  |\xi|+|\beta| \leq 2 \\ \hspace{1mm} z \in \mathbf{k}_1 \end{subarray}} \sum_{  |\kappa| \leq \min(1+\kappa_T,2)}  \int_{\mathbb{S}^2} \int_v  |P_{\xi}^X(\Phi)Y^{\beta}h|+\frac{\log^{14+2M_1} (1+\tau_+)}{\tau_+}|zY^{\kappa} h | dv d\mathbb{S}^2 (t,r).
\end{equation}
Indeed, by a one dimensional Sobolev inequality, we have
$$\int_v |f|(t,r,\theta,\phi,v)dv \lesssim \sum_{r=0}^1 \int_{\omega_1} \left| \left(\partial_{\omega_1} \right)^r \int_v |f|(t,r,\theta+\omega_1,\phi,v)dv \right| d\omega_1.$$
Then, since $\partial_{\omega_1}$ ( and $\partial_{\omega_2}$) can be written as a combination with bounded coefficients of the rotational vector fields $\Omega_{ij}$, we can repeat the previous argument. Finally, let us suppose that $\frac{1+t}{2} \leq |x |$. We have, using again Lemmas \ref{goodderiv} and \ref{lift2},
 \begin{eqnarray}
 \nonumber |x|^2 \tau_- \int_v |g|(t,x,v)dv & = & -|x|^2\int_{|x|}^{+ \infty} \partial_r \left( \tau_- \int_v |g|(t,r,\theta,\phi,v)dv \right)dr \\ \nonumber
 & \lesssim & \int_{|x|}^{+ \infty} \int_v | g|(t,r,\theta,\phi,v)dvr^2dr+\int_{|x|}^{+ \infty}\left| \tau_- \partial_r  \int_v |g|(t,r,\theta,\phi,v)dv \right| r^2 dr \\ \nonumber
 & \leq & \sum_{\begin{subarray}{l} |\xi|+|\beta| \leq 1 \\ \hspace{3mm} p \leq 1 \end{subarray}} \sum_{ w \in \mathbf{k}_1}  \int_{0}^{+ \infty} \int_v \hspace{-0.5mm}  \left(  \hspace{-0.5mm} |P^X_{\xi}(\Phi) Y^{\beta} g|+\frac{\log^7 (3+t)}{1+t} |w \partial^p_t g|   \right)(t,r,\theta,\phi,v) dv r^2 dr.
 \end{eqnarray}
It then remains to apply \eqref{eq:sobsphere} to the functions $P^X_{\xi}(\Phi) Y^{\beta} g$ and $z \partial^p_t g$ and to remark that $|z| \leq 2\tau_+$.
\end{proof}
A similar, but more general, result holds.
\begin{Cor}\label{KS2}
Let $z \in \mathbf{k}_1$ and $j \in \mathbb{N}$. Then, for all $(t,x) \in [0,T[ \times \R^3$,
\begin{eqnarray}
\nonumber  \int_{v \in \R^n} |z|^j|g(t,x,v)| dv & \lesssim & \frac{1}{\tau_+^2 \tau_-} \sum_{w \in \mathbf{k}_1} \Bigg( \sum_{d=0}^{\min(3,j)} \sum_{|\xi|+|\beta| \leq 3-d} \log^{2d}(3+t) \left\| w^{j-d}  P^X_{\xi}(\Phi) Y^{\beta} g \right\|_{L^1_{x,v}} \hspace{-0.8mm} (t) \\ \nonumber 
& & \hspace{4.5cm} +\frac{\log^{6M_1}(3+t)}{1+t} \sum_{ |\kappa| \leq \min(2+\kappa_T,3) } \|w^{j+1} Y^{\kappa} f \|_{L^1_{x,v}} \hspace{-0.8mm} (t) \Bigg) .
\end{eqnarray}
\end{Cor}
\begin{proof}
One only has to follow the proof of Proposition \ref{KS1} and to use Remark \eqref{lift3} instead of Lemma \ref{lift2}).
\end{proof}
A weaker version of this inequality will be used in Subsection \ref{subsecH}.
\begin{Cor}\label{KS3}
Let $z \in \mathbf{k}_1$ and $j \in \mathbb{N}$. Then, for all $(t,x) \in [0,T[ \times \R^3$,
\begin{eqnarray}
\nonumber  \int_{v \in \R^n} |z|^j|g(t,x,v)| dv & \lesssim & \frac{1}{\tau_+^2 \tau_-} \sum_{w \in \mathbf{k}_1} \Bigg( \sum_{d=0}^{\min(3,j)} \sum_{|\beta| \leq 3-d} \log^{2d+M_1}(3+t) \left\| w^{j-d} Y^{\beta} g \right\|_{L^1_{x,v}} \hspace{-0.8mm} (t) \\ \nonumber 
& & \hspace{4.5cm} +\frac{\log^{6M_1}(3+t)}{1+t} \sum_{ |\kappa| \leq \min(2+\kappa_T,3) } \|w^{j+1} Y^{\kappa} f \|_{L^1_{x,v}} \hspace{-0.8mm} (t) \Bigg) .
\end{eqnarray}
\end{Cor}
\begin{proof}
Start by applying Corollary \ref{KS2}. It remains to bound the terms of the form
$$\left\| w^{j-d}  P^X_{\xi}(\Phi) Y^{\beta} g \right\|_{L^1_v L^1(\Sigma_t)}, \hspace{1cm} \text{with} \hspace{1cm} d \leq \min(3,j), \hspace{5mm} |\xi|+|\beta| \leq 3-d \hspace{5mm} \text{and} \hspace{5mm} |\xi| \geq 1.$$
For this, we divide $\Sigma_t$ in two regions, the one where $r \leq 1+2t$ and its complement. As $|P^X_{\xi}(\Phi)| \lesssim \log^{M_1}(1+\tau_+)$ and $\tau_+ \lesssim 1+t$ if $r \leq 1+2t$, we have
$$\left\| w^{j-d}  P^X_{\xi}(\Phi) Y^{\beta} g \right\|_{L^1_v L^1(|y| \leq 2t)} \lesssim \log^{M_1}(3+t) \left\| w^{j-d}  Y^{\beta} g \right\|_{L^1_v L^1(\Sigma_t)}.$$
Now recall from Remark \ref{rqweights1} that $1+r \lesssim \sum_{z_0 \in \V} |z_0|$ and $|P^X_{\xi}(\Phi)|(1+r)^{-1} \lesssim \frac{\log^{M_1}(3+t)}{1+t}$ if $r \geq 1+ 2t$, so that
$$\left\| w^{j-d}  P^X_{\xi}(\Phi) Y^{\beta} g \right\|_{L^1_v L^1(|y| \geq 2t)} \lesssim \frac{\log^{M_1}(3+t)}{1+t} \sum_{z_0 \in \V} \left\| z_0^{j+1}  Y^{\beta} g \right\|_{L^1_v L^1(\Sigma_t)}.$$
The result follows from $|\beta| \leq 2-d \leq 2+\beta_T$.
\end{proof}

We are now interested in adapting Theorem $1.1$ of \cite{dim4} to the modified vector fields.

\begin{Th}\label{decayopti} 
 Suppose that $\sum_{|\kappa| \leq 3} \|Y^{\kappa} \Phi\|_{L^{\infty}_{x,v}} (0) \lesssim 1$. Let $H : [0,T[ \times \R^3_x \times \R^3_v \rightarrow \R$ and $h_0 : \R^3_x \times \R^3_v \rightarrow \R$ be two sufficiently regular functions and $h$ the unique classical solution of
\begin{eqnarray}
\nonumber T_F(h) & = & H \\ \nonumber
h(0,.,.) & = & h_0.
\end{eqnarray}
Consider also $z \in \V$ and $j \in \mathbb{N}$. Then, for all $(t,x) \in [0,T[ \times \R^3$ such that $t \geq |x|$,
\begin{eqnarray}
\nonumber \tau_+^3\int_v |z^jh|(t,x,v)\frac{dv}{(v^0)^2} &  \lesssim  & \sum_{ |\beta| \leq 3 }  \| (1+r)^{|\beta|+j} \partial^{\beta}_{t,x} h \|_{L^1_x L^1_v} (0) \\ \nonumber
& &+\sum_{\begin{subarray}{} |\xi|+|\beta| \leq 3 \\ \hspace{1.5mm} w \in \V  \end{subarray} } \hspace{1.5mm} \sum_{\begin{subarray}{} 0 \leq d \leq 3 \\ \delta \in \{0,1\} \end{subarray} }\frac{ \log^{2d}(3+t)}{\sqrt{1+t}^{\delta}} \int_0^t \int_{\Sigma_s} \int_v \left| T_F \left(w^{j-d+\delta} P^X_{\xi}(\Phi)Y^{\beta} h \right) \right| \frac{dv}{v^0} dx ds,
\end{eqnarray}
where $|\xi|=0$ and $|\beta| \leq \min(2+\beta_T,3)$ if $\delta=1$.
\end{Th}
\begin{proof}
If $|x| \leq \frac{t}{2}$, the result follows from Corollary \ref{KS2} and the energy estimate of Proposition \ref{energyf}. If $\frac{t}{2} \leq |x| \leq t$, we refer to Section $5$ of \cite{dim4}, where Lemma $5.2$ can be rewritten in the same spirit as we rewrite Proposition \ref{KSstandard} with modified vector fields.
\end{proof}
To deal with the exterior, we use the following result.
\begin{Pro}\label{decayopti2} 
For all $(t,x) \in [0,T[ \times \R^3$ such that $|x| \geq t$, we have
$$\int_v |g|(t,x,v) \frac{dv}{(v^0)^2}  \lesssim \frac{1}{\tau_+} \sum_{w \in \V} \int_v |w|| g|(t,x,v) dv.$$
\end{Pro}
\begin{proof}
Let $|x| \geq t$. If $|x| \leq 1$, $\tau_+ \leq 3$ and the estimate holds. Otherwise, $\tau_+ \leq 3|x|$ so, as $\left( x^i-t \frac{v^i}{v^0} \right) \in \mathbf{k}_1$ and
$$\left| x-t\frac{v}{v^0} \right| \geq |x|-t\frac{|v|}{v^0} \geq |x| \frac{(v^0)^2-|v|^2}{v^0(v^0+|v|)} \geq  \frac{|x|}{2(v^0)^2}, \hspace{3mm} \text{we have} \hspace{3mm} \int_v |g|(t,x,v) \frac{dv}{(v^0)^2} \lesssim \frac{1}{|x|} \sum_{w \in \mathbf{k}} \int_v |w||g|(t,x,v)dv.$$
\end{proof}
\begin{Rq}
Using $1 \lesssim v^0 v^{\underline{L}}$ and Lemma \ref{weights1}, we can obtain a similar inequality for the interior of the light cone, at the cost of a $\tau_-$-loss. Note however that because of the presence of the weights $w \in \V$, this estimate, combined with Corollary \ref{KS2}, is slightly weaker than Theorem \ref{decayopti}. During the proof, this difference will lead to a slower decay rate insufficient to close the energy estimates.
\end{Rq}

\subsubsection{Decay estimates for the electromagnetic field}

We start by presenting weighted Sobolev inequalities for general tensor fields. Then we will use them in order to obtain improved decay estimates for the null components of a $2$-form\footnote{Note however that our improved estimates on the components $\alpha$, $\rho$ and $\sigma$ require the $2$-form $G$ to satisfy $\nabla^{\mu} {}^* \! G_{\mu \nu} =0$.}. In order to treat the interior of the light cone (or rather the domain in which $|x| \leq 1+\frac{1}{2}t$), we will use the following result.

\begin{Lem}\label{decayint}
Let $U$ be a smooth tensor field defined on $[0,T[ \times \mathbb{R}^3$. Then,
$$\forall \hspace{0.5mm} t \in [0,T[, \hspace{8mm} \sup_{|x| \leq 1+\frac{t}{2}} |U(t,x)| \lesssim \frac{1}{(1+t)^2} \sum_{|\gamma| \leq 2}  \| \sqrt{\tau_-} \mathcal{L}_{Z^{\gamma}}(U)(t,y) \|_{L^2 \left( |y| \leq 2+\frac{3}{4}t  \right)}.$$
\end{Lem}
\begin{proof}
As $|\mathcal{L}_{Z^{\gamma}}(U)| \lesssim \sum_{|\beta| \leq |\gamma|} \sum_{\mu, \nu} | Z^{\beta} (U_{\mu \nu})|$, we can restrict ourselves to the case of a scalar function. Let $t \in \mathbb{R}_+$ and $|x| \leq 1+ \frac{1}{2}t$. Apply a standard $L^2$ Sobolev inequality to $V: y \mapsto U(t,x+\frac{1+t}{4}y)$ and then make a change of variables to get
$$|U(t,x)|=|V(0)| \lesssim \sum_{|\beta| \leq 2} \| \partial_x^{\beta} V \|_{L^2_y(|y| \leq 1)} \lesssim \left( \frac{1+t}{4} \right)^{-\frac{3}{2}} \sum_{|\beta| \leq 2} \left( \frac{1+t}{4} \right)^{|\beta|} \| \partial_x^{\beta} U(t,.) \|_{L^2_y(|y-x| \leq \frac{1+t}{4})}.$$
Observe now that $|y-x| \leq \frac{1+t}{4}$ implies $|y| \leq 2+\frac{3}{4}t$ and that $1+t \lesssim \tau_-$ on that domain. By Lemma \ref{goodderiv} and since $[Z, \partial] \in \T \cup \{0 \}$, it follows
$$( 1+t )^{|\beta|+\frac{1}{2}} \| \partial_x^{\beta} U(t,.) \|_{L^2_y(|y-x| \leq \frac{1+t}{4})} \hspace{1.5mm} \lesssim \hspace{1.5mm} \| \tau_-^{|\beta|+\frac{1}{2}} \partial_x^{\beta} U(t,.) \|_{L^2_y(|y| \leq 2+\frac{3}{4}t)} \hspace{1.5mm} \lesssim  \hspace{1.5mm} \sum_{|\gamma| \leq |\beta|} \| \sqrt{\tau_-} Z^{\gamma} U(t,.) \|_{L^2_y(|y| \leq 2+\frac{3}{4}t)}.$$ 
\end{proof}
For the remaining region, we have the three following inequalities, coming from Lemma $2.3$ (or rather from its proof for the second estimate) of \cite{CK}. We will use, for a smooth tensor field $V$, the pointwise norm
$$ |V|^2_{\mathbb{O},k} := \sum_{p \leq k} \sum_{\Omega^{\gamma} \in \mathbb{O}^{p}} | \mathcal{L}_{\Omega^{\gamma}}(V)|^2.$$ 
\begin{Lem}\label{Sob}
Let $U$ be a sufficiently regular tensor field defined on $\R^3$. Then, for $t \in \R_+$,
\begin{eqnarray}
\nonumber \forall \hspace{0.5mm} |x| \geq \frac{t}{2}+1, \hspace{10mm} |U(x)| & \lesssim & \frac{1}{|x|\tau_-^{\frac{1}{2}}} \left( \int_{  |y| \geq \frac{t}{2}+1} |U(y)|^2_{\mathbb{O},2}+\tau_-^2|\nabla_{\partial_r} U(y) |^2_{\mathbb{O},1} dy \right)^{\frac{1}{2}}, \\ \nonumber
\forall \hspace{0.5mm} |x| > t, \hspace{10mm} |U(x)| & \lesssim & \frac{1}{|x|\tau_-^{\frac{1}{2}}} \left( \int_{  |y| \geq t} |U(y)|^2_{\mathbb{O},2}+\tau_-^2|\nabla_{\partial_r} U(y) |^2_{\mathbb{O},1} dy \right)^{\frac{1}{2}}, \\ \nonumber
\forall \hspace{0.5mm} x \neq 0, \hspace{10mm} |U(x)| & \lesssim & \frac{1}{|x|^{\frac{3}{2}}} \left( \int_{|y| \geq |x|} |U(y)|^2_{\mathbb{O},2}+|y|^2|\nabla_{\partial_r} U(y) |^2_{\mathbb{O},1} dy \right)^{\frac{1}{2}}.
\end{eqnarray}
\end{Lem}

Recall that $G$ and $J$ satisfy
\begin{eqnarray}
\nonumber \nabla^{\mu} G_{\mu \nu} & =& J_{\nu} \\ \nonumber
\nabla^{\mu} {}^* \! G_{ \mu \nu } & = & 0
\end{eqnarray}
and that $(\alpha, \underline{\alpha}, \rho, \sigma)$ denotes the null decomposition of $G$. Before proving pointwise decay estimates on the components of $G$, we recall the following classical result and we refer, for instance, to Lemma $D.1$ of \cite{massless} for a proof. Concretely, it means that $\mathcal{L}_{\Omega}$, for $\Omega \in \Or$, $\nabla_{\partial_r}$, $\nabla_{\underline{L}}$ and $\nabla_L$ commute with the null decomposition.
\begin{Lem}\label{randrotcom}
Let $\Omega \in \Or$. Then, denoting by $\zeta$ any of the null component $\alpha$, $\underline{\alpha}$, $\rho$ or $\sigma$,
$$ [\mathcal{L}_{\Omega}, \nabla_{\partial_r}] G=0, \hspace{1.2cm} \mathcal{L}_{\Omega}(\zeta(G))= \zeta ( \mathcal{L}_{\Omega}(G) ) \hspace{1.2cm} \text{and} \hspace{1.2cm} \nabla_{\partial_r}(\zeta(G))= \zeta ( \nabla_{\partial_r}(G) ) .$$
Similar results hold for $\mathcal{L}_{\Omega}$ and $\nabla_{\partial_t}$, $\nabla_L$ or $\nabla_{\underline{L}}$. For instance, $\nabla_{L}(\zeta(G))= \zeta ( \nabla_{L}(G) )$.
\end{Lem}
\begin{Pro}\label{decayMaxwell}
We have, for all $(t,x) \in \R_+ \times \R^3$,
\begin{eqnarray}
\nonumber |\rho|(t,x)  , \hspace{2mm} |\sigma|(t,x) & \lesssim & \frac{ \sqrt{\mathcal{E}_2[G](t)+\mathcal{E}_2^{Ext}[G](t)}}{\tau_+^{\frac{3}{2}}\tau_-^{\frac{1}{2}}},  \\ \nonumber
|\alpha|(t,x) & \lesssim & \frac{\sqrt{ \mathcal{E}_2[G](t)+\mathcal{E}_2^{Ext}[G](t)}+\sum_{|\kappa| \leq 1} \|r^{\frac{3}{2}} \mathcal{L}_{Z^{\kappa}}(J)_A\|_{L^2(\Sigma_t)}}{\tau_+^2} \\ \nonumber
 |\underline{\alpha}|(t,x) & \lesssim & \min\left( \frac{\sqrt{\mathcal{E}_2[G](t)+\mathcal{E}_2^{Ext}[G](t)}}{\tau_+ \tau_-}, \frac{\sqrt{\mathcal{E}^0_2[G](t)}}{\tau_+ \tau_-^{\frac{1}{2}}} \right).
\end{eqnarray}
Moreover, if $|x| \geq \max (t,1)$, the term involving $\mathcal{E}_2[G](t)$ on the right hand side of each of these three estimates can be removed.
\end{Pro}

\begin{Rq}
As we will have a small loss on $\mathcal{E}_2[F]$ and not on $\mathcal{E}^0_2[F]$, the second estimate on $\underline{\alpha}$ is here for certain situations, where we will need a decay rate of degree at least $1$ in the $t+r$ direction.
\end{Rq}

\begin{proof}
Let $(t,x) \in [0,T[ \times \R^3$. If $|x| \leq 1+\frac{1}{2}t$, $\tau_- \leq \tau_+ \leq 2+2t$ so the result immediately follows from Lemma \ref{decayint}. We then focus on the case $|x| \geq 1+\frac{t}{2}$. During this proof, $\Omega^{\beta}$ will always denote a combination of rotational vector fields, i.e. $\Omega^{\beta} \in \Or^{|\beta|}$. Let $\zeta$ be either $\alpha$, $ \rho$ or $ \sigma$. As, by Lemma \ref{randrotcom}, $\nabla_{\partial_r}$ and $\mathcal{L}_{\Omega}$ commute with the null decomposition, we have, applying Lemma \ref{Sob},
$$r^3 \tau_- |\zeta|^2 \lesssim \int_{ |y| \geq \frac{t}{2}+1} |\sqrt{r} \zeta |^2_{\mathbb{O},2}+\tau_-^2|\nabla_{\partial_r} (\sqrt{r} \zeta) |_{\mathbb{O},1}^2 dy \lesssim  \sum_{\begin{subarray}{} |\gamma| \leq 2 \\ |\beta| \leq 1 \end{subarray}} \int_{ |y| \geq \frac{t}{2}+1} r| \zeta ( \mathcal{L}_{Z^{\gamma}} (G) |^2+r\tau_-^2|  \zeta ( \mathcal{L}_{\Omega^{\beta}} (\nabla_{\partial_r} G)) |^2 dy.$$
As $\nabla_{\partial_r}$ commute with $\mathcal{L}_{\Omega}$ and since $\nabla_{\partial_r}$ commute with the null decomposition (see Lemma \ref{randrotcom}), we have, using $2\partial_r= L-\underline{L}$ and \eqref{eq:zeta},
\begin{equation}\label{zetaeq2}
 |  \zeta ( \mathcal{L}_{\Omega} (\nabla_{\partial_r} G)) |+|  \zeta ( \nabla_{\partial_r} G) | \hspace{2mm} \lesssim \hspace{2mm} |  \nabla_{\partial_r} \zeta ( \mathcal{L}_{\Omega} (G) |+| \nabla_{\partial_r} \zeta (  G) | \hspace{2mm} \lesssim \hspace{2mm} \frac{1}{\tau_-}\sum_{ |\gamma| \leq 2} | \zeta ( \mathcal{L}_{Z^{\gamma}} (G) |.
 \end{equation}
As $\tau_+ \lesssim r \leq \tau_+$ in the region considered, it finally comes
$$\tau_+^3 \tau_- |\zeta|^2 \lesssim \sum_{|\gamma| \leq 2} \int_{ |y| \geq \frac{t}{2}+1} \tau_+| \zeta ( \mathcal{L}_{Z^{\gamma}} (G) |^2 dx \lesssim \mathcal{E}_2[G](t)+\mathcal{E}^{Ext}_2[G](t).$$
Let us improve now the estimate on $\alpha$. As, by Lemma \ref{basiccom}, $\nabla^{\mu} \mathcal{L}_{\Omega} (G)_{\mu \nu} = \mathcal{L}_{\Omega}(J)_{\nu}$ and $\nabla^{\mu} {}^* \! \mathcal{L}_{\Omega} (G)_{\mu \nu} = 0$ for all $\Omega \in \Or$, we have according to Lemma \ref{maxwellbis} that
$$\forall \hspace{0.5mm} |\beta| \leq 1, \hspace{15mm} \nabla_{\underline{L}} \alpha(\mathcal{L}_{\Omega^{\beta}} (G))_A=\frac{1}{r}\alpha(\mathcal{L}_{\Omega^{\beta}} (G))_A-\slashed{\nabla}_{e_A}\rho (\mathcal{L}_{\Omega^{\beta}} (G)) +\varepsilon_{AB} \slashed{\nabla}_{e_B} \sigma (\mathcal{L}_{\Omega^{\beta}} (G))+\mathcal{L}_{\Omega^{\beta}}(J)_A.$$
Thus, using \eqref{eq:zeta}, we obtain, for all $\Omega \in \Or$,
\begin{equation}\label{alphaeq2}
  |  \alpha ( \nabla_{\partial_r} G) |+|  \alpha ( \mathcal{L}_{\Omega} (\nabla_{\partial_r} G)) | \lesssim \left|J_A \right| +\left| \mathcal{L}_{\Omega} (J)_A \right|+ \frac{1}{r}\sum_{ |\gamma| \leq 2} \left( | \alpha ( \mathcal{L}_{Z^{\gamma}} (G) |+| \rho ( \mathcal{L}_{Z^{\gamma}} (G) |+| \sigma ( \mathcal{L}_{Z^{\gamma}} (G) | \right).
 \end{equation}
Hence, utilizing this time the third inequality of Lemma \ref{Sob} and \eqref{alphaeq2} instead of \eqref{zetaeq2}, we get
$$\tau_+^4 |\alpha|^2 \lesssim r^4 |\alpha|^2 \lesssim \int_{ |y| \geq |x|} |\sqrt{r} \alpha|^2_{\mathbb{O},2}+r^2|\nabla_{\partial_r} ( \sqrt{r} \alpha) |_{\mathbb{O},1}^2 dy \lesssim \mathcal{E}_2[G](t)+\mathcal{E}_2^{Ext}[G](t)+\sum_{|\kappa| \leq 1} \|r^{\frac{3}{2}} \mathcal{L}_{Z^{\kappa}}(J)_A\|^2_{L^2(\Sigma_t)}.$$
Using the same arguments as previously, one has
\begin{eqnarray}
\nonumber \int_{|y| \geq \frac{t}{2}+1}  \left| \underline{\alpha} \right|^2_{\mathbb{O},2}  +\tau_-^2 \left| \nabla_{\partial_r}  \underline{\alpha} \right|_{\mathbb{O},1}^2 dy & \lesssim & \mathcal{E}^0_2[G](t),  \\ \nonumber
 \int_{ |y| \geq \frac{t}{2}+1}  \left| \sqrt{\tau_-} \underline{\alpha} \right|^2_{\mathbb{O},2}  +\tau_-^2 \left| \nabla_{\partial_r} \left( \sqrt{\tau_-} \underline{\alpha} \right) \right|_{\mathbb{O},1}^2 dy & \lesssim &  \mathcal{E}_2[G](t)+\mathcal{E}_2^{Ext}[G](t)
 \end{eqnarray}
and a last application of Lemma \ref{Sob} gives us the result. The estimates for the region $|x| \geq \max (t,1)$ can be obtained similarly, using the second inequality of Lemma \ref{Sob} instead of the first one.
\end{proof}
Losing two derivatives more, one can improve the decay rate of $\rho$ and $\sigma$ near the light cone.
\begin{Pro}\label{Probetteresti}
Let $M \in \mathbb{N}$, $C>0$ and assume that 
\begin{equation}\label{eq:imprdecay} \forall \hspace{0.5mm} (t,x) \in [0,T[ \times \R^3, \hspace{1cm} \sum_{|\gamma| \leq 1} |\mathcal{L}_{Z^{\gamma}}(G)|(t,x)+|J_{\underline{L}}|(t,x) \hspace{2mm} \leq \hspace{2mm} C\frac{\log^M(3+t)}{\tau_+ \tau_-}.
\end{equation}
Then, we have
\begin{equation}\label{eq:imprdecay2} \forall \hspace{0.5mm} (t,x) \in [0,T[ \times \R^3, \hspace{1cm} |\rho|(t,x)+|\sigma|(t,x) \hspace{2mm} \lesssim \hspace{2mm} C\frac{\log^{M+1}(3+t)}{\tau_+^2}.
\end{equation}
\end{Pro}
\begin{proof}
Let $(t,x)=(t,r \omega) \in [0,T[ \times \R^3$. If $r \leq \frac{t+1}{2}$ or $t \leq \frac{r+1}{2}$ the inequalities follow from \eqref{eq:imprdecay} since $\tau_+ \lesssim \tau_-$ in these two cases. We then suppose that $\frac{t+1}{2} \leq r \leq 2t-1$, so that $\tau_+ \leq 10\min(r,t)$. Hence, we obtain from equations \eqref{eq:nullmax1}-\eqref{eq:nullmax2} of Lemma \ref{maxwellbis} and \eqref{eq:zeta} that
\begin{equation}\label{eq:fortheproof3}
|\nabla_{\underline{L}} \hspace{0.5mm} \rho|(t,x)+|\nabla_{\underline{L}} \hspace{0.5mm} \sigma|(t,x) \hspace{2mm} \lesssim \hspace{2mm} |J_{\underline{L}}|(t,x)+ \frac{1}{\tau_+} \sum_{|\gamma| \leq 1} |\mathcal{L}_{Z^{\gamma}}(G)|(t,x).
 \end{equation}
Let $\zeta$ be either $\rho$ or $\sigma$ and $$\varphi ( \underline{u}, u) := \zeta \left( \frac{\underline{u}+u}{2}, \frac{\underline{u}-u}{2} \omega \right), \quad \text{so that, by \eqref{eq:fortheproof3} and \eqref{eq:imprdecay}}, \quad   |\nabla_{\underline{L}} \varphi |(\underline{u},u) \lesssim C\frac{\log^M \left(3+\frac{\underline{u}+u}{2} \right)}{(1+\underline{u})^2(1+|u|)}.$$
\begin{itemize}
\item If $r \geq t$, we then have
\begin{eqnarray}
\nonumber |\zeta|(t,x) & = & |\varphi|(t+r,t-r) \hspace{2mm} \leq \hspace{2mm} \int_{u=-t-r}^{t-r} |\nabla_{\underline{L}} \varphi |(t+r,u) du + |\varphi|(t+r,-t-r) \\ \nonumber
& \lesssim & \int_{u=-t-r}^{t-r} |\nabla_{\underline{L}} \varphi |(t+r,u) du + |\zeta|(0,(t+r)\omega) \\ \nonumber
& \lesssim & C \frac{ \log^M \left(3+t \right)}{(1+t+r)^2} \int_{u=-t-r}^{t-r} \frac{du}{1+|u|} + \frac{C}{(1+t+r)^2} \hspace{2mm} \lesssim \hspace{2mm} C \frac{\log^M \left(3+t \right)}{(1+t+r)^2} \log(1+t+r).
\end{eqnarray}
It then remains to use that $t+r \lesssim 1+t$ in the region studied.
\item If $r \leq t$, we obtain using the previous estimate,
\begin{eqnarray}
\nonumber |\zeta|(t,x) & = & |\varphi|(t+r,t-r) \hspace{2mm} \leq \hspace{2mm} \int_{u=0}^{t-r} |\nabla_{\underline{L}} \varphi |(t+r,u) du + |\varphi|(t+r,0) \\ \nonumber
& \lesssim & \int_{u=0}^{t-r} |\nabla_{\underline{L}} \varphi |(t+r,u) du + |\zeta|\left( \frac{t+r}{2}, \frac{t+r}{2} \right) \\ \nonumber
& \lesssim & C \frac{ \log^M \left(3+t \right)}{(1+t+r)^2} \int_{u=0}^{t-r} \frac{du}{1+|u|} + C \frac{\log^{M+1} \left(3+\frac{t+r}{2} \right)}{(1+t+r)^2} \hspace{2mm} \lesssim \hspace{2mm} C \frac{\log^{M+1} \left(3+t \right)}{(1+t+r)^2} .
\end{eqnarray}
\end{itemize}
This concludes the proof.
\end{proof}

\begin{Rq}
Assuming enough decay on $|F|(t=0)$ and on the spherical components of the source term $J_A$, one could prove similarly that $|\alpha| \lesssim \log^{M+2}(3+t) \frac{\tau_-}{\tau_+^3}$.
\end{Rq}

\section{The pure charge part of the electromagnetic field}\label{secpurecharge}

As we will consider an electromagnetic field with a non-zero total charge, $\int_{\R^3} r|\rho(F)| dx$ will be infinite and we will not be able to apply the results of the previous section to $F$ and its derivatives. As mentioned earlier, we will split $F$ in $\F+\Ff$, where $\F$ and $\overline{F}$ are introduced in Definition \ref{defpure1}. We will then apply the results of the previous section to the chargeless field $\F$, which will allow us to derive pointwise estimates on $F$ since the field $\overline{F}$ is completely determined. More precisely, we will use the following properties of the pure charge part $\Ff$ of $F$.
\begin{Pro}\label{propcharge}
Let $F$ be a $2$-form with a constant total charge $Q_F$ and $\overline{F}$ its pure charge part
$$\Ff(t,x) := \chi(t-r) \frac{Q_F}{4 \pi r^2} \frac{x_i}{r} dt \wedge dx^i.$$
 Then,
\begin{enumerate}
\item $\overline{F}$ is supported in $\cup_{t \geq 0} V_{-1}(t)$ and $\F$ is chargeless.
\item $\rho(\overline{F})(t,x)=-\frac{Q_F}{4 \pi r^2} \chi(t-r)$, \hspace{2mm} $\alpha(\overline{F})=0$, \hspace{2mm} $\underline{\alpha}(\overline{F})=0$ \hspace{2mm} and \hspace{2mm} $\sigma(\overline{F})=0$.
\item $\forall \hspace{0.5mm} Z^{\gamma} \in \mathbb{K}^{|\gamma|}$, \hspace{1mm} $\exists \hspace{0.5mm} C_{\gamma} >0$, \hspace{5mm} $|\mathcal{L}_{Z^{\gamma}} (\overline{F}) | \leq C_{\gamma} |Q_F| \tau_+^{-2}$.
\item $\overline{F}$ satisfies the Maxwell equations $\nabla^{\mu} \overline{F}_{\mu \nu} = \overline{J}_{\nu}$ and $\nabla^{\mu} {}^* \! \Ff_{\mu \nu} =0$, with $\overline{J}$ such that
$$\overline{J}_0(t,x)= \frac{Q_F}{4 \pi r^2} \chi'(t-r) \hspace{5mm} \text{and} \hspace{5mm} \overline{J}_i(t,x) =-\frac{Q_F}{4 \pi r^2} \frac{x_i}{r} \chi'(t-r).$$
$\overline{J}$ is then supported in $\{ (s,y) \in \R_+ \times \R^3 \hspace{1mm} / \hspace{1mm} -2 \leq t-|y| \leq -1 \}$ and its derivatives satisfy
$$ \forall \hspace{0.5mm} Z^{\gamma} \in \mathbb{K}^{|\gamma|}, \hspace{1mm} \exists \hspace{0.5mm} \widetilde{C}_{\gamma} >0, \hspace{10mm} |\mathcal{L}_{Z^{\gamma}} (\overline{J})^L |+\tau_+|\mathcal{L}_{Z^{\gamma}} (\overline{J})^A |+\tau_+^2|\mathcal{L}_{Z^{\gamma}} (\overline{J})^{\underline{L}} | \leq \frac{\widetilde{C}_{\gamma} |Q_F|}{ \tau_+^2}.$$
\end{enumerate}
\end{Pro}
\begin{proof}
The first point follows from the definitions of $\overline{F}$, $\chi$ and
$$ Q_{\F}(t) \hspace{1mm} = \hspace{1mm} Q_F-Q_{\Ff}(t) \hspace{1mm} = \hspace{1mm} Q_F-\lim_{r \rightarrow + \infty} \left(  \int_{\mathbb{S}_{t,r}} \frac{x^i}{r} \Ff_{0i} d \mathbb{S}_{t,r} \right) \hspace{1mm} = \hspace{1mm} Q_F- \frac{Q_F}{4 \pi r^2} \int_{\mathbb{S}_{t,r}} d \mathbb{S}_{t,r} \hspace{1mm} = \hspace{1mm} 0.$$
The second point is straightforward and depicts that $\Ff$ has a vanishing magnetic part and a radial electric part. The third point can be obtained using that,
\begin{itemize}
\item for a $2$-form $G$ and a vector field $\Gamma$, $\mathcal{L}_{\Gamma}(G)_{\mu \nu} = \Gamma(G_{\mu \nu})+\partial_{\mu} (\Gamma^{\lambda} ) G_{\lambda \nu}+\partial_{\nu} ( \Gamma^{\lambda} ) G_{\mu \lambda}$.
\item For all $Z \in \mathbb{K}$, $Z$ is either a translation or a homogeneous vector field. 
\item For a function $\chi_0 : u \mapsto \chi_0(u)$, we have $\Omega_{ij}(\chi_0(u))=0$,
$$ \partial_{t} (\chi_0(u))= \chi_0'(u), \hspace{0.6cm} \partial_{i} (\chi_0(u))= -\frac{x^i}{r}\chi_0'(u), \hspace{0.6cm} S(\chi_0(u))= u \chi_0'(u), \hspace{0.6cm} \Omega_{0i} (\chi(u)) = -\frac{x^i}{r}u \chi_0'(u).$$
\item $1+t \leq \tau_+ \lesssim r$ on the support of $\overline{F}$ and $ |u| \leq \tau_- \leq \sqrt{5}$ on the support of $\chi'$.
\end{itemize}
Consequently, one has
$$\forall \hspace{0.5mm} Z^{\xi} \in \mathbb{K}^{|\xi|}, \hspace{5mm} Z^{\xi} \left( \frac{x^i}{r^3} \chi(t-r) \right) \leq C_{\xi,\chi} \tau_+^{-2} \hspace{1cm} \text{and} \hspace{1cm} \left| \mathcal{L}_{Z^{\gamma}}(\Ff) \right| \lesssim \sum_{|\kappa| \leq |\gamma| } \sum_{\mu=0}^3 \sum_{\nu = 0}^3 \left| Z^{\kappa}(\Ff_{\mu \nu}) \right| \lesssim \frac{C_{\gamma}}{\tau_+^2}.$$
The equations $\nabla^{\mu} {}^* \! \Ff_{\mu \nu}$, equivalent to $\nabla_{[ \lambda} \Ff_{\mu \nu]}=0$ by Proposition \ref{maxwellbis}, follow from $\Ff_{ij}=0$ and that the electric part of $\Ff$ is radial, so that $\nabla_i \Ff_{ 0j}-\nabla_j \Ff_{0i} =0$. The other ones ensue from straightforward computations,
\begin{eqnarray}
\nonumber \nabla^{i} \Ff_{i0} \hspace{-1mm} & = & \hspace{-1mm} -\frac{Q_F}{4 \pi} \partial_i \hspace{-0.2mm} \left( \frac{x^i}{r^3} \chi(t-r) \hspace{-0.2mm} \right) \hspace{0.8mm} =  \hspace{0.8mm} -\frac{Q_F}{4 \pi} \left( \hspace{-0.2mm} \left( \frac{3}{r^3}   -3\frac{x_i x^i}{r^5} \right) \hspace{-0.2mm} \chi(t-r)- \frac{x^i}{r^3} \times \frac{x_i}{r} \chi'(t-r) \hspace{-0.2mm} \right) \hspace{0.8mm} =  \hspace{0.8mm} \frac{Q_F}{4 \pi r^2} \chi'(t-r), \\ \nonumber
\nabla^{\mu} \Ff_{\mu i} \hspace{-1mm} & = & \hspace{-1mm} -\partial_t \Ff_{0i} \hspace{1mm} =  \hspace{1mm} -\frac{Q_F}{4 \pi} \frac{x^i}{r^3} \chi'(t-r).
\end{eqnarray}
For the estimates on the derivatives of $\overline{J}$, we refer to \cite{LS} (equations $(3.52a)-(3.52c)$).
\end{proof}

\section{Bootstrap assumptions and strategy of the proof}\label{sec6}

Let, for the remainder of this article, $N \in \mathbb{N}$ such that $N \geq 11$ and $M \in \mathbb{N}$ which will be fixed during the proof. Let also $0 < \eta < \frac{1}{16}$ and $(f_0,F_0)$ be an initial data set satisfying the assumptions of Theorem \ref{theorem}. By a standard local well-posedness argument, there exists a unique maximal solution $(f,F)$ of the Vlasov-Maxwell system defined on $[0,T^*[$, with $T^* \in \R_+^* \cup \{+ \infty \}$. Let us now introduce the energy norms used for the analysis of the particle density.
\begin{Def}\label{normVlasov}
Let $Q \leq N$, $q \in \mathbb{N}$ and $a = M+1$. For $g$ a sufficiently regular function, we define the following energy norms,
\begin{eqnarray}
\nonumber \E[g](t) & := & \|  g \|_{L^1_{x,v} }(t) +\int_{C_u(t)} \int_v \frac{v^{\underline{L}}}{v^0} \left| g \right| dv dC_u(t), \\ \nonumber
\E^{q}_Q[g](t) &  := & \sum_{\begin{subarray}{l} 1 \leq i \leq 2 \\ \hspace{1mm} z \in \mathbf{k}_1 \end{subarray}} \sum_{ \begin{subarray}{} |\xi^i|+|\beta| \leq Q \\ \hspace{1mm} |\xi^i| \leq Q-1 \end{subarray}}  \sum_{j=0}^{2N-1+q- \xi^1_P-\xi^2_P-\beta_P}  \log^{- (j+ |\xi^1|+|\xi^2|+|\beta|)a}(3+t) \E \left[ z^j P_{\xi^1}(\Phi)P_{\xi^2}(\Phi)Y^{\beta} f \right] \hspace{-1mm} (t), \\ \nonumber
 \overline{\E}_N[g](t) & := & \sum_{\begin{subarray}{l} 1 \leq i \leq 2 \\ \hspace{1mm} z \in \mathbf{k}_1 \end{subarray}} \sum_{ \begin{subarray}{} |\xi^i|+|\beta| \leq Q \\ \hspace{1mm} |\xi^i| \leq Q-1 \end{subarray}}  \sum_{j=0}^{2N-1- \xi^1_P-\xi^2_P-\beta_P} \hspace{-1mm} \log^{-aj}(3+t) \E \left[ z^j P_{\xi^1}(\Phi)P_{\xi^2}(\Phi)Y^{\beta} f \right] \hspace{-1mm} (t), \\ \nonumber
  \E^X_{N-1}[f](t) & := & \sum_{\begin{subarray}{l} 1 \leq i \leq 2 \\ \hspace{1mm} z \in \mathbf{k}_1 \end{subarray}} \sum_{ |\zeta^i|+|\beta| \leq N-1}   \sum_{j=0}^{2N-2-\zeta^1_P-\zeta^2_P-\beta_P} \log^{-2j}(3+t) \E \left[  z^j  P^X_{\zeta^1}(\Phi)P^X_{\zeta^2}(\Phi)Y^{\beta} f \right]\hspace{-1mm} (t), \\ \nonumber
  \E^X_{N}[f](t) & := & \sum_{ z \in \mathbf{k}_1} \sum_{ \begin{subarray}{} |\zeta|+|\beta| \leq N \\ \hspace{1mm} |\zeta| \leq N-1 \end{subarray}}  \sum_{j=0}^{2N-2-\zeta_P-\beta_P} \log^{-2j}(3+t) \E \left[ z^j  P^X_{\zeta}(\Phi)Y^{\beta} f \right]\hspace{-1mm} (t) .
\end{eqnarray}
To understand the presence of the logarithmical weights, see Remark \ref{hierarchyjustification}.
\end{Def}
In order to control the derivatives of the $\Phi$ coefficients and $\overline{\E}_N[f]$ at $t=0$, we prove the following result.
\begin{Pro}\label{Phi0}
Let $|\beta| \leq N-1$ a multi index and $Y^{\beta} \in \Y^{|\beta|}$. Then, at $t=0$,
\begin{eqnarray}
\nonumber \max \left( |Y^{\beta} \Phi |, | \widehat{Z}^{\beta} \Phi | \right) & \lesssim & \frac{1+r^2}{v^0} \sum_{|\gamma| \leq |\beta| -1} \left| \mathcal{L}_{Z^{\gamma}}(F) \right| \\ \nonumber &  \lesssim & \frac{\sqrt{\epsilon}}{v^0 }.
\end{eqnarray}
\end{Pro} 
\begin{proof}
Note that the second inequality ensues from
\begin{equation}\label{decayF00}
\sum_{|\gamma| \leq N-2} \left\| \mathcal{L}_{Z^{\gamma}}(F) \right\|_{L^{\infty}(\Sigma_0)} \lesssim \frac{\sqrt{\epsilon}}{1+r^2},
\end{equation}
which comes from Proposition \ref{decayMaxwell}. Let us now prove the first inequality. Unless the opposite is mentioned explicitly (as in \eqref{equa7}), all functions considered here will be evaluated at $t=0$. As $\Phi(0,.,.)=0$, the result holds for $|\beta|=0$. Let $1 \leq |\beta| \leq N-1$ and suppose that the result holds for all $|\sigma| < |\beta|$. Note that, for instance, 
$$ Y_2 Y_1 \Phi = \widehat{Z}_2 \widehat{Z}_1 \Phi +\Phi X \widehat{Z}_1 \Phi+Y_2(\Phi) X \Phi+\Phi \widehat{Z}_2 X \Phi + \Phi \Phi X X \Phi.$$
More generally, we have,
\begin{equation}\label{equa5}
\left| Y^{\beta} \Phi \right| \lesssim \sum_{\begin{subarray}{} p \leq |k|+|\sigma| \leq |\beta| \\ \hspace{2.5mm} k < |\beta| \end{subarray} } P_{k,p}(\Phi) \widehat{Z}^{\sigma} \Phi.
\end{equation}
Consequently, using the induction hypothesis, we only have to prove the result for $\widehat{Z}^{\beta} \Phi$. Indeed, as $|k| < |\beta|$, by \eqref{decayF00},
\begin{equation}\label{equa3}
|P_{k,p}(\Phi) \widehat{Z}^{\sigma}(\Phi) | \lesssim |\widehat{Z}^{\sigma}(\Phi)| \left|\frac{1+r^{2}}{v^0}\right|^p \sum_{|\gamma| \leq N-2 } \left| \mathcal{L}_{Z^{\gamma}}(F) \right|^p  \lesssim |\widehat{Z}^{\sigma}(\Phi)|.
\end{equation}
Combining \eqref{equa5} and \eqref{equa3}, we would then obtain the inequality on $|Y^{\beta} \Phi|$, if we would have it on $\widehat{Z}^{\sigma} \Phi$ for all $|\sigma| \leq |\beta|$. Let us then prove that the result holds for $\widehat{Z}^{\beta} \Phi$ and suppose, for simplicity, that $\Phi=\Phi^k_{\widehat{Z}}$, with $\widehat{Z} \neq S$. Remark that
$$|\widehat{Z}^{\beta} \Phi | \lesssim \sum_{|\alpha_2|+|\alpha_1|+q \leq |\beta|}(1+|x|)^{|\alpha_1|+q}(v^0)^{|\alpha_2|} |\partial_{v}^{\alpha_2} \partial_x^{\alpha_1} \partial_t^q \Phi|$$
and let us prove by induction on $q$ that
\begin{equation}\label{equa6}
\forall \hspace{0.5mm} |\alpha_2|+|\alpha_1|+q \leq |\beta|, \hspace{8mm} (1+|x|)^{|\alpha_1|+q}(v^0)^{|\alpha_2|} |\partial_{v}^{\alpha_2} \partial_x^{\alpha_1} \partial_t^q \Phi| \lesssim \frac{1+r^2}{v^0} \sum_{|\gamma| \leq |\beta| -1} \left| \mathcal{L}_{Z^{\gamma}}(F) \right|.
\end{equation}
Recall that for $t \in [0,T^*[$,
\begin{equation}\label{equa7}
T_F(\Phi)=v^{\mu} \partial_{\mu} \Phi + F(v,\nabla_v \Phi)=-t\frac{v^{\mu}}{v^0} \mathcal{L}_Z(F)_{\mu k}.
\end{equation}
As $\Phi(0,.,.)=0$ and $v^0\partial_t \Phi=-v^i \partial_i \Phi-F(v,\nabla_v \Phi)$, implying $\partial_t \Phi(0,.,.)=0$, \eqref{equa6} holds for $q \leq 1$. Let $2 \leq q \leq |\beta|$ and suppose that \eqref{equa6} is satisfied for all $q_0 < q$. Let $|\alpha_2|+|\alpha_1| \leq |\beta|-q$. Using the commutation formula given by Lemma \ref{basiccomuf}, we have (at $t=0$),
$$v^0\partial_x^{\alpha_1} \partial_t^q \Phi=-v^i \partial_i\partial_x^{\alpha_1} \partial_t^{q-1} \Phi-\frac{v^{\mu}}{v^0}\mathcal{L}_{\partial_x^{\alpha_1} \partial_t^{q-2} Z}(F)_{\mu k}+\sum_{|\gamma_1|+q_1+|\gamma_2| = |\alpha_1|+q-1} C^1_{\gamma_1, \gamma_2} \mathcal{L}_{\partial^{\gamma_2}}(F)(v,\nabla_v \partial_{x}^{\gamma_1} \partial_t^{q_1} \Phi),$$
Dividing the previous equality by $v^0$, taking the $\partial_v^{\alpha_2}$ derivatives of each side and using Lemma \ref{goodderiv}, we obtain
\begin{eqnarray}
\nonumber |\partial_v^{\alpha_2} \partial_x^{\alpha_1} \partial_t^q \Phi| & \lesssim & \sum_{|\alpha_3| \leq |\alpha_2|} (v^0)^{-|\alpha_2|+|\alpha_3|}|\partial^{\alpha_3}_v\partial_x\partial_x^{\alpha_1} \partial_t^{q-1} \Phi|+\sum_{|\gamma| \leq |\alpha_1|+q-2} \frac{1}{(v^0)^{1+|\alpha_2|}(1+r)^{|\alpha_1|+q-2}}\left|\mathcal{L}_{Z^{\gamma} Z}(F)\right|\\ \nonumber & & +\sum_{\begin{subarray}{} |\gamma_1|+q_1+n = |\alpha_1|+q-1 \\ \hspace{5mm} 1 \leq |\alpha_4| \leq |\alpha_2|+1 \end{subarray}} \sum_{|\gamma_2| \leq n} \frac{1}{(v^0)^{|\alpha_2|-|\alpha_4|+1}(1+r)^n}\left| \mathcal{L}_{Z^{\gamma_2}}(F) \right| |\partial_v^{\alpha_4}\partial_{x}^{\gamma_1} \partial_t^{q_1} \Phi|.
\end{eqnarray}
It then remains to multiply both sides of the inequality by $(v^0)^{|\alpha_2|}(1+r)^{|\alpha_1|+q}$ and
\begin{itemize}
\item To bound $(v^0)^{|\alpha_2|}(1+r)^{|\alpha_1|+q} (v^0)^{-|\alpha_2|+|\alpha_3|}|\partial^{\alpha_3}_v\partial_x\partial_x^{\alpha_1} \partial_t^{q-1} \Phi|$ with the induction hypothesis.
\item To remark that $(v^0)^{|\alpha_2|}(1+r)^{|\alpha_1|+q} \frac{1}{(v^0)^{1+|\alpha_2|}(1+r)^{|\alpha_1|+q-2}}\left|\mathcal{L}_{Z^{\gamma} Z}(F)\right|$ has the desired form.
\item To note that, using $|\gamma_1|+q_1+1 = |\alpha_1|+q-n$ and the induction hypothesis,
\begin{eqnarray}
\nonumber \frac{(v^0)^{|\alpha_2|}(1+r)^{|\alpha_1|+q}}{(v^0)^{|\alpha_2|-|\alpha_4|+1}(1+r)^n} |\partial_v^{\alpha_4}\partial_{x}^{\gamma_1} \partial_t^{q_1} \Phi|\left| \mathcal{L}_{Z^{\gamma_2}}(F) \right|  \hspace{-2mm} & = & \hspace{-2mm} \frac{1+r}{v^0}(v^0)^{|\alpha_4|}(1+r)^{|\gamma_1|+q_1}|\partial_v^{\alpha_4}\partial_{x}^{\gamma_1} \partial_t^{q_1} \Phi|\left| \mathcal{L}_{Z^{\gamma_2}}(F) \right|  \\ \nonumber
& \lesssim & \hspace{-2mm} \frac{1+r}{v^0} \left| \mathcal{L}_{Z^{\gamma_2}}(F) \right| \hspace{-1mm} \sum_{|\zeta| \leq |\alpha_4|+|\gamma_1|+q_1-1} \hspace{-1mm} \frac{(1+r)^2}{v^0} \left| \mathcal{L}_{Z^{\zeta}}(F) \right| \\ \nonumber
& \lesssim &  \hspace{-2mm} \sum_{|\zeta| \leq |\alpha_2|+|\alpha_1|+q-1} \frac{(1+r)^2}{v^0} \left| \mathcal{L}_{Z^{\zeta}}(F) \right| ,
\end{eqnarray}
since $\left| \mathcal{L}_{Z^{\gamma_2}}(F) \right| \lesssim (1+r)^{-2}$, as $|\gamma_2| \leq |\alpha_1|+q-1 \leq |\beta|-1 \leq N-2$. This concludes the proof of the Proposition.
\end{itemize}
\end{proof}
\begin{Cor}\label{coroinit}
There exists $\widetilde{C} >0$ a constant depending only on $N$ such that $ \E^{4}_N[f](0) \leq \widetilde{C} \epsilon = \widetilde{\epsilon}$. Without loss of generality and in order to lighten the notations, we suppose that $\E^{4}_N[f](0) \leq \epsilon$.
\end{Cor}
\begin{proof}
All the functions considered here are evaluated at $t=0$. Consider multi-indices $\xi_1$, $\xi_2$ and $\beta$ such that, for $i \in \{1 , 2 \}$, $\max(|\xi^i|+1,|\xi^i|+|\beta|) \leq N$ and $j \leq 2N+3-\xi^1_P-\xi^2_P-\beta_P$. Then,
$$\left|z^j P_{\xi^1}(\Phi) P_{\xi^2}(\Phi) Y^{\beta} f \right| \leq \left|z^j P_{\xi^1}(\Phi) P_{\xi^2}(\Phi) \widehat{Z}^{\beta} f \right|+ \sum_{\begin{subarray}{} \hspace{1mm} |k|+|\kappa| \leq |\beta| \\ \hspace{1mm} |k| \leq |\beta|-1  \\ p+k_P+\kappa_P < \beta_P \end{subarray} } \left|z^j P_{\xi^1}(\Phi) P_{\xi^2}(\Phi) P_{k,p}(\Phi) \widehat{Z}^{\kappa} f \right|.$$
Using the previous proposition and the assumptions on $f_0$, one gets, with $C_1 >0$ a constant,
$$ \E^{4}_N[f](0) \leq (1+C_1 \sqrt{\epsilon} )\sum_{\begin{subarray}{l} \hspace{0.5mm} \widehat{Z}^{\beta} \in \K^{|\beta|} \\  |\beta| \leq N \end{subarray}} \| z^{2N+3-\beta_P} \widehat{Z}^{\beta} f \|_{L^1_{x,v}}(0) .$$
By similar computations than in Appendix $B$ of \cite{massless}, we can bound the right hand side of the last inequality by $\widetilde{C} \epsilon$ using the smallness hypothesis on $(f_0,F_0)$.
\end{proof}
By a continuity argument and the previous corollary, there exists a largest time $T \in ]0,T^*[$ such that, for all $t \in [0,T[$,
\begin{eqnarray}\label{bootf1}
\E^4_{N-3}[f](t) & \leq & 4\epsilon, \\ 
\E^{0}_{N-1}[f](t) & \leq  & 4\epsilon, \label{bootf2} \\ 
\overline{\E}_{N}[f](t) & \leq & 4\epsilon(1+t)^{\eta}, \label{bootf3} \\
\sum_{|\beta| \leq N-2} \left\| r^{\frac{3}{2}} \int_v \frac{v^A}{v^0} \widehat{Z}^{\beta} f dv \right\|_{L^2(\Sigma_t)} & \leq & \sqrt{\epsilon}, \label{bootL2} \\ 
\mathcal{E}^0_{N}[F](t) & \leq & 4\epsilon, \label{bootF1} \\
\mathcal{E}^{Ext}_N[\F](t) & \leq & 8 \epsilon, \label{bootext} \\
\mathcal{E}_{N-3}[F](t) & \leq & 30\epsilon \log^2(3+t), \label{bootF2} \\
\mathcal{E}_{N-1}[F](t) & \leq & 30\epsilon \log^{2M}(3+t), \label{bootF3} \\ 
\mathcal{E}_{N}[F](t) & \leq & 30\epsilon(1+t)^{\eta}. \label{bootF4}
\end{eqnarray}

The remainder of the proof will then consist in improving our bootstrap assumptions, which will prove that $(f,F)$ is a global solution to the $3d$ massive Vlasov-Maxwell system. The other points of the theorem will be obtained during the proof, which is divided in four main parts.

\begin{enumerate}
\item First, we will obtain pointwise decay estimates on the particle density, the electromagnetic field and then on the derivatives of the $\Phi$ coefficients, using the bootstrap assumptions.
\item Then, we will improve the bootstrap assumptions \eqref{bootf1}, \eqref{bootf2} and \eqref{bootf3} by several applications of the energy estimate of Proposition \ref{energyf} and the commutation formula of Proposition \ref{ComuPkp}. The computations will also lead to optimal pointwise decay estimates on $\int_v |Y^{\beta} f | \frac{dv}{(v^0)^2}$.
\item The next step consists in proving enough decay on the $L^2$ norms of $\int_v |zY^{\beta} f | dv$, which will permit us to improve the bootstrap assumption \eqref{bootL2}.
\item Finally, we will improve the bootstrap assumptions \eqref{bootF1}-\eqref{bootF4} by using the energy estimates of Proposition \ref{energyMax1}.
\end{enumerate}
\section{Immediate consequences of the bootstrap assumptions}\label{sec7}

In this section, we prove pointwise estimates on the Maxwell field, the $\Phi$ coefficients and the Vlasov field. We start with the electromagnetic field.

\begin{Pro}\label{decayF}
We have, for all $|\gamma| \leq N-3$ and $(t,x) \in [0,T[ \times \R^3$,
\begin{flalign*}
& \hspace{0.5cm} |\alpha(\mathcal{L}_{Z^{\gamma}}(F))|(t,x) \hspace{1mm} \lesssim \hspace{1mm} \sqrt{\epsilon}\frac{\log^M(3+t)}{\tau_+^2}, \hspace{17mm} |\underline{\alpha}(\mathcal{L}_{Z^{\gamma}}(F))|(t,x) \hspace{1mm} \lesssim \hspace{1mm} \sqrt{\epsilon}\min \left( \frac{1}{\tau_+\tau_-^{\frac{1}{2}}}, \frac{\log^M(3+t)}{\tau_+\tau_-} \right) \hspace{-0.5mm} ,& \\
& \hspace{0.5cm} |\sigma(\mathcal{L}_{Z^{\gamma}}(F))|(t,x) \hspace{1mm} \lesssim \hspace{1mm} \sqrt{\epsilon} \frac{\log^M(3+t)}{\tau_+^{\frac{3}{2}}\tau_-^{\frac{1}{2}}}, \hspace{17mm} |\rho(\mathcal{L}_{Z^{\gamma}}(F))|(t,x) \hspace{1mm} \lesssim \hspace{1mm} \sqrt{\epsilon}\frac{\log^M(3+t)}{\tau_+^{\frac{3}{2}}\tau_-^{\frac{1}{2}}}.&
\end{flalign*}
Moreover, if $|x| \geq t$,
\begin{flalign*}
& \hspace{0.5cm} |\alpha(\mathcal{L}_{Z^{\gamma}}(F))|(t,x) \hspace{1mm} \lesssim \hspace{1mm} \frac{\sqrt{\epsilon}}{\tau_+^2}, \hspace{35mm} |\underline{\alpha}(\mathcal{L}_{Z^{\gamma}}(F))|(t,x)  \hspace{1mm} \lesssim \hspace{1mm} \frac{\sqrt{\epsilon}}{\tau_+\tau_-}, & \\
& \hspace{0.5cm} |\sigma(\mathcal{L}_{Z^{\gamma}}(F))|(t,x) \hspace{1mm} \lesssim \hspace{1mm} \frac{\sqrt{\epsilon}}{\tau_+^{\frac{3}{2}}\tau_-^{\frac{1}{2}}}, \hspace{31mm} |\rho(\mathcal{L}_{Z^{\gamma}}(F))|(t,x) \hspace{1mm} \lesssim \hspace{1mm} \frac{\sqrt{\epsilon}}{\tau_+^{\frac{3}{2}}\tau_-^{\frac{1}{2}}}. &
\end{flalign*}
We also have
$$ \forall \hspace{0.5mm} (t,x) \in [0,T[ \times \R^3, \hspace{1.5cm} \sum_{|\kappa| \leq N} \left|  \mathcal{L}_{Z^{\kappa}}(\Ff)  \right|(t,x) \hspace{1mm} \lesssim \hspace{1mm} \frac{\epsilon}{\tau_+^2}.$$
\end{Pro}

\begin{Rq}\label{lowderiv}
If $|\gamma| \leq N-5$, we can replace the $\log^M(3+t)$-loss in the interior of the lightcone by a $\log(3+t)$-loss (for this, use the bootstrap assumption \eqref{bootF2} instead of \eqref{bootF3} in the proof below).
\end{Rq}
\begin{Rq}\label{decayoftheo}
Applying Proposition \ref{Probetteresti} and using the estimate \eqref{decayf} proved below, we can also improve the decay rates of the components $\rho$ and $\sigma$ near the light cone. We have, for all $|\gamma| \leq N-6$,
\begin{eqnarray}
 \forall \hspace{0.5mm}  (t,x) \in [0,T[ \times \R^3, \hspace{1cm} |\rho(\mathcal{L}_{Z^{\gamma}}(F))|(t,x)+|\sigma (\mathcal{L}_{Z^{\gamma}}(F))|(t,x) & \lesssim & \sqrt{\epsilon} \frac{\log^2(3+t)}{\tau_+^2} . \label{kevatal} 
\end{eqnarray}
\end{Rq}

\begin{proof}
The last estimate, concerning $\Ff$, ensues from Proposition \ref{propcharge} and $|Q_F| \leq \| f_0 \|_{L^1_{x,v}} \leq \epsilon$. The estimate $\tau_+ \sqrt{\tau_-}|\underline{\alpha}| \lesssim \sqrt{\epsilon}$ follows from Proposition \ref{decayMaxwell} and the bootstrap assumption \eqref{bootF1}. Note that the other estimates hold with $F$ replaced by $\F$ since $\mathcal{E}_{N-1}[F]=\mathcal{E}_{N-1}[\F]$ and according to Proposition \ref{decayMaxwell} and the bootstrap assumptions \eqref{bootext}, \eqref{bootF3} and \eqref{bootL2}. It then remains to use $F=\F+\Ff$ and the estimates obtained on $\Ff$ and $\F$.
\end{proof}
\begin{Rq}\label{justif}
Even if the pointwise decay estimates \eqref{kevatal}, which correspond to the ones written in Theorem \ref{theorem}, are stronger than the ones given by Proposition \ref{decayF} (or Remark \ref{lowderiv}) in the region located near the light cone, we will not work with them for two reasons.
\begin{enumerate}
\item Using these stronger decay rates do not simplify the proof. We compensate the lack of decay in $t+r$ of the estimates given by Proposition \ref{decayF} for the components $\rho$ and $\sigma$ by taking advantage of the inequality\footnote{We are able to use this inequality in the energy estimates as the degree in $v$ of the source terms of $T_F(Y^{\beta} f)$ is $0$ whereas the one of $v^{\mu} \partial_{\mu}Y^{\beta}f$ is equal to $1$.} $1 \lesssim \sqrt{v^0 v^{\underline{L}}}$ and the good properties of $v^{\underline{L}}$.
\item Compared to the estimates given by Remark \ref{lowderiv}, \eqref{kevatal} requires to control one derivative more of the electromagnetic field in $L^2$. Working with them would then force us to take $N \geq 12$.
\end{enumerate}
\end{Rq}
We now turn on the $\Phi$ coefficients and start by the following lemma.
\begin{Lem}\label{LemPhi}
Let $G$, $G_1$, $G_2 : [0,T[ \times \R^3_x \times \R^3_v \rightarrow \R$ and $\varphi_0 : \R^3_x \times \R^3_v \rightarrow \R$ be four sufficiently regular functions such that $|G| \leq G_1+G_2$. Let $\varphi$, $\widetilde{\varphi}$, $\varphi_1$ and $\varphi_2$ be such that
$$ T_F( \varphi ) =G, \hspace{5mm} \varphi(0,.,.)=\varphi_0, \hspace{12mm} T_F(\widetilde{\varphi})=0, \hspace{5mm} \widetilde{\varphi}(0,.,.)=\varphi_0$$
and, for $i \in \{1,2 \}$,
$$  T_F(\varphi_i)=G_i, \hspace{5mm} \varphi_i(0,.,.)=0.$$
Then, on $[0,T[ \times \R^3_x \times \R^3_v$,
$$|\varphi| \leq |\widetilde{\varphi}|+|\varphi_1|+|\varphi_2|.$$
\end{Lem}
\begin{proof}
Denoting by $X(s,t,x,v)$ and $V(s,t,x,v)$ the characteristics of the transport operator, we have by Duhamel's formula,
\begin{eqnarray}
\nonumber |\varphi|(t,x,v) & = & \left| \widetilde{\varphi}(t,x,v)+\int_0^t \frac{G}{v^0} \left( s,X(s,t,x,v),V(s,t,x,v) \right) ds \right| \\ \nonumber
& \leq & |\widetilde{\varphi}|(t,x,v)+\int_0^t \frac{G_1+G_2}{v^0} \left( s,X(s,t,x,v),V(s,t,x,v) \right) ds \\ \nonumber
& = & |\widetilde{\varphi}|(t,x,v)+|\varphi_1|(t,x,v)+|\varphi_2|(t,x,v).
\end{eqnarray}
\end{proof}
\begin{Pro}\label{Phi1}
We have, $\forall \hspace{0.5mm} (t,x,v) \in [0,T[ \times \R^3_x \times \R^3_v$
$$ |\Phi|(t,x,v) \lesssim \sqrt{\epsilon} \log^2 (1+\tau_+), \hspace{3mm} |\partial_{t,x} \Phi| (t,x,v) \lesssim \sqrt{\epsilon} \log^{\frac{3}{2}}(1+\tau_+) \hspace{3mm} \text{and} \hspace{3mm} |Y \Phi|(t,x,v) \lesssim \sqrt{\epsilon} \log^{\frac{7}{2}} (1+\tau_+).$$
\end{Pro}
\begin{proof}
We will obtain this result through the previous Lemma and by parameterizing the characteristics of the operator $T_F$ by $t$ or by $u$. Let us start by $\Phi$ and recall that, schematically, $T_F(\Phi)=-t\frac{v^{\mu}}{v^0} \mathcal{L}_Z(F)_{\mu k}$. Denoting by $(\alpha, \underline{\alpha}, \rho, \sigma)$ the null decomposition of $\mathcal{L}_Z(F)$ and using $|v^A| \lesssim \sqrt{v^0 v^{\underline{L}}}$ (see Lemma \ref{weights1}), we have
\begin{eqnarray}
\nonumber \left| \frac{v^{\mu}}{v^0} \mathcal{L}_Z(F)_{\mu k} \right| & \lesssim & \frac{v^L+|v^A|}{v^0}|\alpha|+\frac{v^L+v^{\underline{L}}}{v^0}|\rho|+\frac{|v^A|}{v^0}|\sigma|+\frac{v^{\underline{L}}+|v^A|}{v^0}|\underline{\alpha}| \\ \nonumber
& \lesssim & |\alpha|+|\rho|+|\sigma|+\sqrt{\frac{v^{\underline{L}}}{v^0}}|\underline{\alpha}|.
\end{eqnarray}
Using the pointwise estimates given by Remark \ref{lowderiv} as well as the inequalities $1 \lesssim \sqrt{v^0 v^{\underline{L}}}$, which comes from Lemma \ref{weights1}, and $2ab \leq a^2+b^2$, we get
\begin{equation}\label{eq:firstphiesti}
 \tau_+\left| \frac{v^{\mu}}{v^0} \mathcal{L}_Z(F)_{\mu k} \right| \hspace{2mm} \lesssim \hspace{2mm} \sqrt{\frac{\epsilon v^0 v^{\underline{L}}}{\tau_+ \tau_-}}\log(3+t)+v^{\underline{L}}\frac{\sqrt{\epsilon}}{ \tau_-} \log(3+t) \hspace{2mm} \lesssim \hspace{2mm} \frac{v^0 \sqrt{\epsilon}}{\tau_+}\log(3+t)+\frac{v^{\underline{L}}\sqrt{\epsilon}}{\tau_-} \log(3+t).
\end{equation}
Consider now the functions $\varphi_1$ and $\varphi_2$ such that
$$T_F(\varphi_1) = \frac{v^0 \sqrt{\epsilon}}{\tau_+}\log(3+t), \hspace{8mm}  T_F(\varphi_2) = \frac{v^{\underline{L}}\sqrt{\epsilon}}{\tau_-}\log(3+t) \hspace{6mm} \text{and} \hspace{6mm} \varphi_1(0,.,.)=\varphi_2(0,.,.)=0.$$
According to Lemma \ref{LemPhi}, we have $|\Phi| \lesssim |\varphi_1|+|\varphi_2|$. In order to estimate $\varphi_1$, we will parametrize the characteristics of the operator $T_F$ by $t$. More precisely, let $(X_{s,y,v}(t),V_{s,y,v}(t))$ be the value in $t$ of the characteristic which is equal to $(y,v)$ in $t=s$, with $s < T$. Dropping the indices $s$, $y$ and $w$, we have
$$ \frac{dX^i}{dt}(t) = \frac{V^i(t)}{V^0(t)} \hspace{1.5cm} \text{and} \hspace{1.5cm} \frac{d V^i}{dt}(t) = \frac{V^{\mu}(t)}{V^0(t)} {F_{ \mu}}^{ i}(t,X(t)).$$
Duhamel's formula gives
$$ |\varphi_1|(s,y,v) \hspace{2mm} \lesssim \hspace{2mm} \sqrt{\epsilon} \int_0^s \frac{\log(3+t) }{\tau_+ \big( t, X_{t,y,v}(t) \big)} ds \hspace{2mm} \leq \hspace{2mm} \sqrt{\epsilon} \int_0^s \frac{\log(3+t)}{1+t} ds \hspace{2mm} \leq \hspace{2mm} \sqrt{\epsilon} \log^2 (3+s).$$
For $\varphi_2$, we parameterize the characteristics of $T_F$ by\footnote{Note that $T_F=2v^{\underline{L}} \partial_u +2v^L \partial_{\underline{u}}+v^A e_A+F(v,\nabla_v)$} $u$. For a point $(s,y) \in [0,T[ \times \R^3$, we will write its coordinates in the null frame as $(z,\underline{z},\omega_1, \omega_2)$. Let $(\underline{U}_{z,\underline{z},\omega_1, \omega_2,v}(u), \Omega^1_{z,\underline{z},\omega_1, \omega_2,v}(u),\Omega^2_{z,\underline{z},\omega_1, \omega_2,v}(u),V_{z,\underline{z},\omega_1, \omega_2,v}(u))$ be the value in $u$ of the characteristic which is equal to $(s,y,v)=(z,\underline{z},\omega_1,\omega_2,v)$ in $u=z$. Dropping the indices $z$, $\underline{z}$, $\omega_1$, $\omega_2$ and $v$, we have
$$\frac{d \underline{U}}{du}(u)= \frac{V^L(u)}{V^{\underline{L}}(u)}, \hspace{1.5cm} \frac{d \Omega^A}{du}(u) = \frac{V^A(u)}{2V^{\underline{L}}(u)} \hspace{1.5cm} \text{and} \hspace{1.5cm} \frac{d V^i}{du}(u) = \frac{V^{\mu}(u)}{2V^{\underline{L}}(u)} {F_{\mu}}^{ i}(u,\underline{U}(u),\Omega (u)).$$
Note that $u \mapsto \frac{1}{2}(u+\underline{U}(u))$ vanishes in a unique $z_0$ such that $ -\underline{z} \leq z_0 \leq z$, i.e. the characteristic reaches the hypersurface $\Sigma_0$ once and only once, at $u=z_0$. This can be noticed on the following picture, representing a possible trajectory of $(u,\underline{U}(u))$, which has to be in the backward light cone of $(z,\underline{z})$ by finite time of propagation,

\vspace{5mm}

\begin{tikzpicture}
\draw [-{Straight Barb[angle'=60,scale=3.5]}] (0,-0.3)--(0,4);
\fill[color=gray!35] (1,0)--(4,3)--(7,0)--(1,0);
\node[align=center,font=\bfseries, yshift=-2em] (title) 
    at (current bounding box.south)
    {\hspace{1cm} The trajectory of $\left (u,\underline{U} \left(u \right) \right)$ for $u \leq z$.};
\draw [-{Straight Barb[angle'=60,scale=3.5]}] (0,0)--(9,0) node[scale=1.5,right]{$\Sigma_0$};
\fill (4,3)  circle[radius=2pt];
\fill (1,0)  circle[radius=2pt];
\fill (7,0)  circle[radius=2pt];
\draw (4.6,3) node[scale=1]{$(z,\underline{z})$};
\draw (1,-0.3) node[scale=1]{$(0,z)$};
\draw (7,-0.3) node[scale=1]{$(0,\underline{z})$};
\draw (0,-0.5) node[scale=1.5]{$r=0$};
\draw (-0.5,3.7) node[scale=1.5]{$t$};
\draw (8.7,-0.5) node[scale=1.5]{$r$};  
\draw[scale=1,domain=sqrt(3)+1:4,smooth,variable=\x,black]  plot (\x,{0.5*(\x-1)*(\x-1)-1.5});
\end{tikzpicture}

or by noticing that 
$$g(u):=u+\underline{U}(u) \hspace{12mm} \text{satisfies} \hspace{12mm} g'(u) \hspace{1mm} \geq \hspace{1mm} 1+ \frac{V^L\left( u \right)}{V^{\underline{L}}\left( u \right)} \hspace{1mm} \geq \hspace{1mm} 1$$
so that $g$ vanishes in $z_0$ such that $-\underline{z}=z-(z+\underline{z}) \leq z_0 \leq z$. Similarly, one can prove (or observe) that $\sup_{z_0 \leq u \leq z } \underline{U}(u) \leq \underline{z}$.  It then comes that
\begin{equation}\label{eq:varphi2}
 |\varphi_2|(s,y,v)  \hspace{2mm} \lesssim \hspace{2mm} \sqrt{\epsilon} \int_{z_0}^z \frac{\log \left( 3+\underline{U} \left(u \right) \right)}{ \tau_-(u,\underline{U} \left( u \right))} du \hspace{2mm} \lesssim \hspace{2mm} \sqrt{\epsilon} \log (3+ \underline{z} ) \int_{-\underline{z}}^z \frac{1}{1+|u|} d u \hspace{2mm} \lesssim \hspace{2mm} \sqrt{\epsilon} \log^2(1+\underline{z}),
 \end{equation}
which allows us to deduce that $|\Phi|(s,y,v) \lesssim \sqrt{\epsilon} \log^2 (3+s+|y|)$. We prove the other estimates by the continuity method. Let $0 < T_0 < T$ and $ \underline{u} > 0$ be the largest time and null ingoing coordinate such that
\begin{equation}\label{bootPhi}
|\nabla_{t,x} \Phi| (t,x,v) \leq C \sqrt{\epsilon} \log^{\frac{3}{2}} (1+\tau_+) \hspace{8mm} \text{and} \hspace{8mm} \sum_{Y \in \Y_0} |Y \Phi| (t,x,v) \leq C \sqrt{\epsilon} \log^{\frac{7}{2}} (1+\tau_+)
\end{equation}
hold for all $(t,x,v) \in \underline{V}_{\underline{u}}(T_0) \times \R^3_v$ and where the constant $C>0$ will be specified below. The goal now is to improve the estimates of \eqref{bootPhi}. Using the commutation formula of Lemma \ref{basiccomuf} and the definition of $\Phi$, we have (in the case where $\Phi$ is not associated to the scaling vector field), for $\partial \in \T$,
$$T_F \left( \partial \Phi \right) = - \mathcal{L}_{\partial}(F)(v,\nabla_v \Phi)-\partial \left( t\frac{v^{\mu}}{v^0}\mathcal{L}_{Z}(F)_{\mu k} \right).$$
With $\delta = \partial (t) \in \{0,1 \}$, one has
$$ \partial \left( t\frac{v^{\mu}}{v^0}\mathcal{L}_{Z}(F)_{\mu k} \right) = \delta \frac{v^{\mu}}{v^0}\mathcal{L}_{Z}(F)_{\mu k}+t\frac{v^{\mu}}{v^0}\mathcal{L}_{\partial Z}(F)_{\mu k}.$$ 
Using successively the inequality \eqref{eq:zeta2}, the pointwise decay estimates\footnote{Note that we use the estimate $|\underline{\alpha}| \lesssim \sqrt{\epsilon} \tau_+^{-1}\tau_-^{-\frac{1}{2}}$ here in order to obtain a decay rate of $\tau_+^{-1}$ in the $t+r$ direction.} given by Remark \ref{lowderiv} and the inequalities $1 \lesssim \sqrt{v^0 v^{\underline{L}}}$, $2ab \leq a^2+b^2$, we get
\begin{eqnarray}
\nonumber  t\frac{v^{\mu}}{v^0}\mathcal{L}_{\partial Z}(F)_{\mu k} & \lesssim & \tau_+ \left( |\alpha ( \mathcal{L}_{\partial Z}(F) ) |+|\rho (\mathcal{L}_{\partial Z}(F))|+| \sigma ( \mathcal{L}_{\partial Z}(F) ) |+ \sqrt{\frac{v^{\underline{L}}}{v^0}} |\underline{\alpha} ( \mathcal{L}_{\partial Z}(F) ) | \right) \\ \nonumber
& \lesssim & \frac{\tau_+}{\tau_-} \sqrt{v^0 v^{\underline{L}}} \sum_{|\beta|\leq 2} \Big(\frac{\tau_-}{\tau_+} \left| \mathcal{L}_{Z^{\beta}}(F) \right|+ |\alpha ( \mathcal{L}_{Z^{\beta}}(F) ) |+|\rho (\mathcal{L}_{ Z^{\beta} }(F))|+| \sigma ( \mathcal{L}_{Z^{\beta}}(F) ) | \\ \nonumber
& & + \sqrt{\frac{v^{\underline{L}}}{v^0}} |\underline{\alpha} ( \mathcal{L}_{ Z^{\beta}}(F) ) | \Big) \\
& \lesssim &  \sqrt{v^0 v^{\underline{L}}}  \frac{\sqrt{\epsilon} \log(3+t)}{\tau_+^{\frac{1}{2}} \tau_-^{\frac{3}{2}}}+ v^{\underline{L}} \frac{\tau_+}{\tau_-} \frac{\sqrt{\epsilon} }{\tau_+ \tau_-^{\frac{1}{2}}} \hspace{2mm} \lesssim \hspace{2mm} \sqrt{\epsilon} \frac{ v^0 }{\tau_+} \log^{\frac{1}{2}}(3+t) + \sqrt{\epsilon} \frac{ v^{\underline{L}} }{\tau_-^{\frac{3}{2}}} \log^{\frac{3}{2}}(3+t). \label{bootPhiT1}
\end{eqnarray}
Similarly,
\begin{eqnarray}
\nonumber \frac{v^{\mu}}{v^0}\mathcal{L}_{ Z}(F)_{\mu k} & \lesssim &  \left( |\alpha ( \mathcal{L}_{ Z}(F) ) |+|\rho (\mathcal{L}_{ Z}(F))|+| \sigma ( \mathcal{L}_{ Z}(F) ) |+ \sqrt{\frac{v^{\underline{L}}}{v^0}} |\underline{\alpha} ( \mathcal{L}_{ Z}(F) ) | \right) \\ 
& \lesssim &   \frac{\sqrt{\epsilon} \log(3+t)}{\tau_+^{\frac{3}{2}} \tau_-^{\frac{1}{2}}}+ v^{\underline{L}} \frac{\sqrt{\epsilon} }{\tau_+ \tau_-^{\frac{1}{2}}} \hspace{2mm} \lesssim \hspace{2mm}  \sqrt{\epsilon} \frac{ v^0 }{\tau_+^{\frac{5}{4}}} + \sqrt{\epsilon} \frac{ v^{\underline{L}} }{\tau_-^{\frac{5}{4}}} . \label{bootPhiT2}
\end{eqnarray}
Expressing $\mathcal{L}_{\partial}(F)(v,\nabla_v \Phi)$ in null components, denoting by $(\alpha, \underline{\alpha}, \rho, \sigma)$ the null decomposition of $\mathcal{L}_{\partial}(F)$ and using the inequalities $|v^A| \lesssim \sqrt{v^0v^{\underline{L}}}$, $1 \lesssim \sqrt{v^0 v^{\underline{L}}}$ (see Lemma \ref{weights1}), one has
\begin{equation}\label{eq:referlater}
 \left|\mathcal{L}_{\partial}(F)(v,\nabla_v \Phi) \right|  \lesssim  \sqrt{v^0v^{\underline{L}}}|\rho | \left| \left( \nabla_v \Phi \right)^r\right|+\left(\sqrt{v^0v^{\underline{L}}} |\alpha  |+v^{\underline{L}}|\underline{\alpha}  |+v^{\underline{L}} |\sigma  | \right) \left| \nabla_v \Phi \right|.
 \end{equation}
Using Lemma \ref{vradial}, $v^0 \partial_{v^i}=Y_i-\Phi X-t\partial_i-x^i \partial_t$ and the bootstrap assumption on the $\Phi$ coefficients \eqref{bootPhi}, we obtain
\begin{eqnarray}
\nonumber \left| \left( \nabla_v \Phi \right)^r\right| & \lesssim & \sum_{Y \in \Y_0} |Y \Phi |+ |\Phi| |X( \Phi ) |+ \tau_- |\nabla_{t,x} \Phi | \hspace{2mm} \lesssim \hspace{2mm} C \sqrt{\epsilon} \log^{\frac{7}{2}} (1+\tau_+)+C\sqrt{\epsilon} \tau_- \log^{\frac{3}{2}} (1+\tau_+) , \\ \nonumber
 \left| \nabla_v \Phi \right| & \lesssim & \sum_{Y \in \Y_0} |Y \Phi |+ |\Phi| |X( \Phi ) |+ \tau_+ |\nabla_{t,x} \Phi | \hspace{2mm} \lesssim \hspace{2mm} C\sqrt{\epsilon} \log^{\frac{7}{2}} (1+\tau_+)+C\sqrt{\epsilon} \tau_+ \log^{\frac{3}{2}} (1+\tau_+).
 \end{eqnarray} 
We then deduce, by \eqref{eq:zeta2} and the pointwise estimates given by Remark \ref{lowderiv},
\begin{eqnarray}
\nonumber \sqrt{v^0v^{\underline{L}}}|\rho | \left| \left( \nabla_v \Phi \right)^r\right|+\sqrt{v^0v^{\underline{L}}} |\alpha  |\left| \nabla_v \Phi \right| & \lesssim &  C\epsilon \frac{\sqrt{v^0 v^{\underline{L}}}}{\tau_+ \tau_- }  \log^{\frac{5}{2}} (1+\tau_+) \hspace{2mm} \lesssim \hspace{2mm} C \epsilon \frac{v^0}{\tau_+^{\frac{3}{2}}}+ C \epsilon \frac{v^{\underline{L}}}{\tau_-^2}, \\ \nonumber
\left( v^{\underline{L}}|\underline{\alpha}  |+v^{\underline{L}} |\sigma  | \right) \left| \nabla_v \Phi \right| & \lesssim & C \epsilon \frac{v^{\underline{L}}}{\tau_-^{\frac{3}{2}}} \log^{\frac{3}{2}}(1+\tau_+).
\end{eqnarray}
Combining these two last estimates with \eqref{bootPhiT1} and \eqref{bootPhiT2}, we get
$$ \left| T_F \left( \partial \Phi \right) \right| \lesssim (\sqrt{\epsilon}+C\epsilon) \frac{v^0}{\tau_+} \log^{\frac{1}{2}}(1+\tau_+)+ (\sqrt{\epsilon}+C\epsilon) \frac{v^{\underline{L}}}{\tau_-^{\frac{5}{4}}} \log^{\frac{3}{2}}(1+\tau_+).$$
We then split $\partial \Phi$ in three functions $\widetilde{\psi}+\psi_1+\psi_2$ such that $\psi_1(0,.,.)=\psi_2(0,.,.)=0$, $\widetilde{\psi}(0,.,.)=\partial \Phi (0,.,.)$,
$$ T_F(\psi_1)=(\sqrt{\epsilon}+C\epsilon) \frac{v^0}{\tau_+} \log^{\frac{1}{2}}(1+\tau_+), \hspace{10mm} T_F(\psi_2)=(\sqrt{\epsilon}+C\epsilon) \frac{v^{\underline{L}}}{\tau_-^{\frac{5}{4}}} \log^{\frac{3}{2}}(1+\tau_+) \hspace{10mm} \text{and} \hspace{10mm} T_F(\widetilde{\psi})=0  .$$
According to Proposition \ref{Phi0}, we have $\|\widetilde{\psi} \|_{L^{\infty}_{t,x,v}} = \| \partial \Phi (0,.,.) \|_{L^{\infty}_{x,v}} \lesssim \sqrt{\epsilon}$. Fix now $(s,y,v) \in \underline{V}_{\underline{u}}(T_0) \times \R^3_v$ and let $(z,\underline{z}, \omega_1, \omega_2)$ be the coordinates of $(s,y)$ in the null frame. Keeping the notations used previously in this proof, we have
\begin{eqnarray}
\nonumber  |\psi_1|(s,y,v) & \lesssim & (\sqrt{\epsilon}+C\epsilon) \int_0^s \frac{ \log^{\frac{1}{2}} (1+\tau_+(t,X(t)))}{\tau_+ (t,X(t))} dt \\ 
& \lesssim & (\sqrt{\epsilon}+C\epsilon) \int_0^{s} \frac{ \log^{\frac{1}{2}} (3+t)}{1+t} dt  \hspace{2mm} \lesssim \hspace{2mm}    (\sqrt{\epsilon}+C\epsilon) \log^{\frac{3}{2}} (3+t), \label{eq:repeat1} \\ \nonumber
|\psi_2|(s,y,v) & \lesssim & (\sqrt{\epsilon}+C\epsilon) \int_{z_0}^z \frac{\log^{\frac{3}{2}} \left( 1+\tau_+(u,\underline{U} \left(u \right)) \right)}{ \tau_-^{\frac{5}{4}}(u,\underline{U} \left( u \right))} du \\  & \lesssim & (\sqrt{\epsilon}+C\epsilon) \log^{\frac{3}{2}} (3+ \underline{z} ) \int_{-\underline{z}}^z \frac{1}{(1+|u|)^{\frac{5}{4}}} d u \hspace{2mm} \lesssim \hspace{2mm} (\sqrt{\epsilon}+C\epsilon) \log^{\frac{3}{2}}(3+\underline{z}). \label{eq:repeat2}
\end{eqnarray}
Thus, there exists $C_1 >0$ such that
$$ \forall \hspace{0.5mm} (s,y,v) \in \underline{V}_{\underline{u}}(T_0) \times \R^3_v, \hspace{1cm} |\nabla_{t,x} \Phi |(s,y,v) \leq C_1(\sqrt{\epsilon}+C \epsilon ) \log^{\frac{3}{2}}(1+\tau_+(s,y))$$ and we can then improve the bootstrap assumption on $\nabla_{t,x} \Phi$ if $C$ is choosen large enough and $\epsilon$ small enough. It remains to study $Y \Phi$ with $Y \in \Y_0$. Using Lemma \ref{Comufirst}, $T_F(Y \Phi)$ can be bounded by a linear combination of terms of the form
$$ \left| \frac{v^{\mu}}{v^0}\mathcal{L}_Z(F)_{\mu k} Y \Phi \right|, \hspace{9mm} \tau_+\left| \frac{v^{\mu}}{v^0}\mathcal{L}_Z(F)_{\mu k} \partial_{t,x} \Phi \right|, \hspace{9mm} \left| \Phi \mathcal{L}_{\partial}(F)(v, \nabla_v \Phi ) \right| \hspace{9mm} \text{and} \hspace{9mm} \left| Y \left( t \frac{v^{\mu}}{v^0} \mathcal{L}_{Z}(F)_{\mu k} \right) \right|.$$
Using the bootstrap assumption \eqref{bootPhi} in order to estimate $|Y \Phi |$ and reasoning as for \eqref{bootPhiT2}, one obtains
$$ \left| \frac{v^{\mu}}{v^0}\mathcal{L}_Z(F)_{\mu k} Y \Phi \right| \hspace{2mm} \lesssim \hspace{2mm} C\epsilon \frac{v^0}{\tau_+^{\frac{5}{4}}}+C\epsilon \frac{v^{\underline{L}}}{\tau_-^{\frac{5}{4}}} .$$
Bounding $|\partial_{t,x} \Phi|$ with the bootstrap assumption \eqref{bootPhi} and using the inequality \eqref{eq:firstphiesti}, it follows
$$\tau_+\left| \frac{v^{\mu}}{v^0}\mathcal{L}_Z(F)_{\mu k} \partial \Phi \right| \hspace{2mm} \lesssim \hspace{2mm} C\epsilon \frac{v^0}{\tau_+}\log^{\frac{5}{2}}(1+\tau_+)+C\epsilon \frac{v^{\underline{L}}}{\tau_-}\log^{\frac{5}{2}}(1+\tau_+).$$
As $|\Phi| \lesssim \sqrt{\epsilon} \log^2(1+\tau_+)$, we get, using the bound obtained on the left hand side of \eqref{eq:referlater},
$$\Phi \mathcal{L}_{\partial}(F)(v, \nabla_v \Phi ) \hspace{2mm} \lesssim \hspace{2mm} C\epsilon \frac{v^0}{\tau_+^{\frac{3}{2}}} \log^2(1+\tau_+)+C\epsilon \frac{v^{\underline{L}}}{\tau_-^{\frac{3}{2}}} \log^{\frac{7}{2}}(1+\tau_+).$$
For the remaining term, one has schematically, by the first equality of Lemma \ref{calculF},
$$ \left| Y \left( t \frac{v^{\mu}}{v^0} \mathcal{L}_{Z}(F)_{\mu k} \right) \right| \hspace{2mm} \lesssim \hspace{2mm} \left(\tau_++|\Phi| \right) \left| \frac{v^{\mu}}{v^0} \mathcal{L}_{Z}(F)_{\mu \theta} \right|+\tau_+ \left| \frac{v^{\mu}}{v^0} \mathcal{L}_{ZZ}(F)_{\mu k} \right|+\tau_+ |\Phi|\left| \frac{v^{\mu}}{v^0} \mathcal{L}_{\partial Z}(F)_{\mu k} \right|. $$
Using $|\Phi| \lesssim \log^2(1+\tau_+) \leq \tau_+$ and following \eqref{eq:firstphiesti}, we get
$$ \left(\tau_++|\Phi| \right) \left| \frac{v^{\mu}}{v^0} \mathcal{L}_{Z}(F)_{\mu \theta} \right|+\tau_+ \left| \frac{v^{\mu}}{v^0} \mathcal{L}_{ZZ}(F)_{\mu k} \right| \hspace{2mm} \lesssim \hspace{2mm} \sqrt{\epsilon} \frac{v^0}{\tau_+} \log (1+\tau_+)+\sqrt{\epsilon} \frac{v^{\underline{L}}}{\tau_-} \log (1+\tau_+).$$
Combining \eqref{bootPhiT1} with $|\Phi| \lesssim \log^2(1+\tau_+)$, we obtain
$$ \tau_+ |\Phi|\left| \frac{v^{\mu}}{v^0} \mathcal{L}_{\partial Z}(F)_{\mu k} \right| \hspace{2mm} \lesssim \hspace{2mm} \sqrt{\epsilon} \frac{v^0}{\tau_+} \log^{\frac{5}{2}} (1+\tau_+)+\sqrt{\epsilon} \frac{v^{\underline{L}}}{\tau_-^{\frac{3}{2}}} \log^{\frac{7}{2}} (1+\tau_+).$$
Consequently, one has
$$\left| T_F ( Y \Phi ) \right| \hspace{2mm} \lesssim \hspace{2mm} (\sqrt{\epsilon}+C \epsilon) \frac{v^0}{\tau_+} \log^{\frac{5}{2}} (1+\tau_+)+(\sqrt{\epsilon}+C \epsilon) \frac{v^{\underline{L}}}{\tau_-^{\frac{5}{4}}} \log^{\frac{7}{2}} (1+\tau_+)+(\sqrt{\epsilon}+C \epsilon) \frac{v^{\underline{L}}}{\tau_-} \log^{\frac{5}{2}} (1+\tau_+).$$
One can then split $Y \Phi$ in three functions $\widetilde{\varsigma}$, $\varsigma_1$ and $\varsigma_2$ defined as $\widetilde{\psi}$, $\psi_1$ and $\psi_2$ previously. We have $\| \widetilde{\varsigma} \|_{L^{\infty}_{t,x,v}} \lesssim \sqrt{\epsilon}$ since $\|Y \Phi \|_{L^{\infty}_{x,v}}(0) \lesssim \sqrt{\epsilon}$ (see Proposition \ref{Phi0}) and we can obtain $|\varsigma_1|+|\varsigma_2| \lesssim ( \sqrt{\epsilon}+C \epsilon ) \log^{\frac{7}{2}} (1+\tau_+)$ by similar computations as those of \eqref{eq:repeat1}, \eqref{eq:repeat2} and \eqref{eq:varphi2}. So, taking $C$ large enough and $\epsilon$ small enough, we can improve the bootstrap assumption on $Y \Phi$ and conclude the proof.
\end{proof}
For the higher order derivatives, we have the following result.
\begin{Pro}
For all $(Q_1,Q_2) \in \llbracket 0, N-4 \rrbracket^2$ satisfying $Q_2 \leq Q_1$, there exists $R(Q_1,Q_2) \in \mathbb{N}$ such that
$$\forall \hspace{0.5mm} |\beta| \leq N-4, \hspace{3mm} (t,x) \in [0,T[ \times \R^3, \hspace{12mm} \left| Y^{\beta} \Phi \right|(t,x) \lesssim \sqrt{\epsilon} \log^{R(|\beta|,\beta_P)} (1+\tau_+ ).$$
Note that $R(Q_1,Q_2)$ is independent of $M$ if $Q_1 \leq N-6$.
\end{Pro}
\begin{proof}
The proof is similar to the previous one and we only sketch it. We process by induction on $Q_1$ and, at $Q_1$ fixed, we make an induction on $Q_2$. Let $|\beta| \leq N-4$ and suppose that the result holds for all $ Q_1 \leq |\beta|$ and $Q_2 \leq \beta_P$ satisfying $Q_1 < |\beta|$ or $Q_2 < \beta_P$. Let $0 < T_0 < T$ and $\underline{u} >0$ be such that
$$ \forall \hspace{0.5mm} (t,x,v) \in \underline{V}_{\underline{u}}(T_0) \times \R^3_v, \hspace{1cm} |Y^{\beta} \Phi|(t,x,v) \leq C \sqrt{\epsilon} \log^{R(|\beta|,\beta_P)}(1+\tau_+),$$
with $C>0$ a constant sufficiently large. We now sketch the improvement of this bootstrap assumption, which will imply the desired result. The source terms of $T_F(Y^{\beta} \Phi)$, given by Propositions \ref{ComuVlasov} and \ref{sourcePhi}, can be gathered in two categories.
\begin{itemize}
\item The ones where there is no $\Phi$ coefficient derived more than $|\beta| -1$ times, which can then be bounded by the induction hypothesis and give logarithmical growths, as in the proof of the previous Proposition. We then choose $R(|\beta|,\beta_P)$ sufficiently large to fit with these growths.
\item The ones where a $\Phi$ coefficient is derived $|\beta|$ times. Note then that they all come from Proposition \ref{ComuVlasov}, when $|\sigma| = |\beta|$ for the quantities of \eqref{eq:com1} and when $|\sigma|=|\beta|-1$ for the other ones. We then focus on the most problematic ones (with a $\tau_+$ or $\tau_-$ weight, which can come from a weight $z \in \V$ for the terms of \eqref{eq:com1}), leading us to integrate along the characteristics of $T_F$ the following expressions.
\end{itemize}
\begin{equation}\label{eq:Phi11}
\tau_+ \left| \frac{v^{\mu}}{v^0} \mathcal{L}_{Z^{\gamma}}(F)_{\mu \nu} Y^{\kappa} \Phi \right|, \hspace{5mm} \text{with} \hspace{5mm} |\gamma| \leq N-3, \hspace{5mm} |\kappa| = |\beta| \hspace{5mm} \text{and} \hspace{5mm} \kappa_P < \beta_P,
\end{equation}
\begin{equation}\label{eq:Phi22}
\left| \Phi^{p} \mathcal{L}_{\partial Z^{\gamma_0}} (F) \left( v, \Gamma^{\kappa} \Phi \right) \right|, \hspace{5mm} \text{with} \hspace{5mm} |\gamma_0| \leq N-4, \hspace{5mm} |\kappa| = |\beta|-1 \hspace{5mm} \text{and} \hspace{5mm} p+\kappa_P \leq \beta_P.
\end{equation}
To deal with \eqref{eq:Phi11}, use the induction hypothesis, as $\kappa_P < \beta_P$. For the other terms, recall from Lemma \ref{GammatoYLem} that we can schematically suppose that 
$$\Gamma^{\kappa} \Phi = P_{q,n}(\Phi) Y^{\zeta} \Phi, \hspace{5mm} \text{with} \hspace{5mm} |q|+|\zeta| \leq |\beta|-1, \hspace{5mm} |q| \leq |\beta|-2 \hspace{5mm} \text{and} \hspace{5mm} n+q_P+\zeta_P = \kappa_P.$$
Expressing \eqref{eq:Phi22} in null coordinates and transforming the $v$ derivatives with Lemma \ref{vradial} or $v^0 \partial_{v^i}=Y_i-\Phi X-x^i \partial_t-t \partial_i$, we obtain the following bad terms,
$$ \left( \tau_- |\rho|+ \tau_+ |\alpha|+\tau_+\sqrt{\frac{v^{\underline{L}}}{v^0}}\left( |\sigma|+|\underline{\alpha}| \right) \right) \Phi^p \partial_{t,x} \left( P_{q,n}(\Phi) Y^{\zeta} \Phi \right).$$
Then, note that there is no derivatives of order $|\beta|$ in $\Phi^p \partial_{t,x} \left( P_{q,n}(\Phi) \right) Y^{\zeta} \Phi$ so that these terms can be handled using the induction hypothesis. It then remains to study the terms related to $  P_{q,n+p}(\Phi) \partial_{t,x} Y^{\zeta} \Phi$. If $\zeta_P < \beta_P$, we can treat them using again the induction hypothesis. Otherwise $p+n=0$ and we can follow the treatment of \eqref{eq:referlater}. Finally, the fact that $R(|\beta|,\beta_P)$ is independent of $M$ if $|\beta| \leq N-6$ follows from Remark \ref{lowderiv} and that we merely need pointwise estimates on the derivatives of $F$ up to order $N-5$ in order to bound $Y^{\xi} \Phi$, with $|\xi| \leq N-6$.
\end{proof}
\begin{Rq}\label{estiPkp}
There exist $(M_1,M_2) \in \mathbb{N}^2$, with $M_1$ independent of $M$, such that, for all $p \leq 3N$ and $(t,x,v) \in [0,T[ \times \R^3 \times \R^3$,
$$ \sum_{|k| \leq N-6} |P_{k,p}(\Phi)|(t,x,v) \lesssim  \log^{M_1}(1+\tau_+) \hspace{10mm} \text{and} \hspace{10mm} \sum_{|k| \leq N-4} |P_{k,p}(\Phi)|(t,x,v) \lesssim  \log^{M_2}(1+\tau_+).$$
\end{Rq}
We are now able to apply the Klainerman-Sobolev inequalities of Proposition \ref{KS1} and Corollary \ref{KS2}. Combined with the bootstrap assumptions \eqref{bootf1}, \eqref{bootf3} and the estimates on the $\Phi$ coefficients, one immediately obtains that, for any $z \in \mathbf{k}_1$, $\max(|\xi|+|\beta|,|\xi|+1) \leq N-6$, $j \leq 2N-\xi_P-\beta_P$,
\begin{equation}\label{decayf}
\forall \hspace{0.5mm} (t,x) \in [0,T[\times \R^3,  \qquad  \int_v |z^jP_{\xi}(\Phi)Y^{\beta} f|(t,x,v) dv \lesssim \epsilon \frac{\log^{(j+|\xi|+|\beta|+3)a}(3+t)}{\tau_+^2\tau_-}.
\end{equation}

\section{Improvement of the bootstrap assumptions \eqref{bootf1}, \eqref{bootf2} and \eqref{bootf3} }\label{sec8}

As the improvement of all the energy bounds concerning $f$ are similar, we unify them as much as possible. Hence, let us consider
\begin{itemize}
\item $Q \in \{ N-3,N-1,N\}$, $n_{N-3}=4$, $n_{N-1}=0$ and $n_N=0$.
\item Multi-indices $\beta^0$, $\xi^0$ and $\xi^2$ such that $\max (|\xi^0|+|\beta^0|, 1+|\xi^0| ) \leq Q$ and $\max (|\xi^2|+|\beta^0|, 1+|\xi^2| ) \leq Q$.
\item A weight $z_0 \in \mathbf{k}_1$ and $q \leq 2N-1+n_Q-\xi^0_P-\xi^2_P-\beta^0_P$.
\end{itemize} 
According to the energy estimate of Propostion \ref{energyf}, Corollary \ref{coroinit} and since $\xi^0$ and $\xi^2$ play a symmetric role, we could improve \eqref{bootf1}-\eqref{bootf3}, for $\epsilon$ small enough, if we prove that
\begin{eqnarray}\label{improvebootf}
\int_0^t \int_{\Sigma_s} \int_v \left| T_F \left( z^q_0 P_{\xi^0}(\Phi) Y^{\beta^0} f \right) P_{\xi^2}(\Phi) \right| \frac{dv}{v^0} dx ds & \lesssim & \epsilon^{\frac{3}{2}}(1+t)^{\eta}  \log^{aq}(3+t) \hspace{3mm} \text{if} \hspace{3mm} Q =N, \\
& \lesssim & \epsilon^{\frac{3}{2}}  \log^{(q+|\xi^0|+|\xi^2|+|\beta^0|)a}(3+t) \hspace{3mm} \text{otherwise}. \label{improvebootf2}
\end{eqnarray}
For that purpose, we will bound the spacetime integral of the terms given by Proposition \ref{ComuPkp}, applied to $z^q_0 P_{\xi^0}(\Phi) Y^{\beta^0} f$. We start, in Subsection \ref{sec81}, by covering the term of \eqref{eq:cat0}. Subsection \ref{sec82} (respectively \ref{ref:L2elec}) is devoted to the study of the expressions of the other categories for which the electromagnetic field is derived less than $N-3$ times (respectively more than $N-2$ times). Finally, we treat the more critical terms in Subsection \ref{sec86}. In Subsection \ref{Ximpro}, we bound $\E_N^X[f]$, $ \E_{N-1}^X[f]$ and we improve the decay estimate of $\int_v (v^0)^{-2} |Y^{\beta} f|dv$ near the light cone.

\subsection{The terms of \eqref{eq:cat0}}\label{sec81}

The purpose of this Subsection is to prove the following proposition.

\begin{Pro}\label{M1}
Let $\xi^1$, $\xi^2$ and $\beta$ such that $\max(1+|\xi^i|,|\xi^i|+|\beta|) \leq N$ for $i \in \{1,2 \}$. Consider also $z \in \mathbf{k}_1$, $r \in \mathbb{N}^*$, $0 \leq \kappa \leq \eta$, $0 < j \leq 2N+3-\xi^1-\xi^2_P-\beta_P$ and suppose that,
$$\forall \hspace{0.5mm} t \in [0,T[, \hspace{8mm} \E \left[  z^j P_{\xi^1}(\Phi)P_{\xi^2}(\Phi) Y^{\beta} f \right](t)+\log^2(3+t)\E \left[  z^{j-1} P_{\xi^1}(\Phi)P_{\xi^2}(\Phi) Y^{\beta} f \right](t) \lesssim \epsilon (1+t)^{\kappa} \log^r(3+t).$$
Then,
$$ \int_0^t \int_{\Sigma_s} \int_v \left| F\left(v,\nabla_v z^j\right) P_{\xi^1}(\Phi)P_{\xi^2}(\Phi) Y^{\beta} f \right| \frac{dv}{v^0} dx ds \lesssim \epsilon^{\frac{3}{2}}(1+t)^{\kappa} \log^{r}(3+t).$$
\end{Pro}
\begin{proof}
To lighten the notations, we denote $P_{\xi^1}(\Phi)P_{\xi^2}(\Phi) Y^{\beta}f$ by $h$ and, for $d \in \{0,1 \}$, $\E \left[  z^{j-d} h \right]$ by $H_{j-d}$, so that
$$H_{j-d}(t) \hspace{1mm} = \hspace{1mm} \| z^{j-d} h \|_{L^1_{x,v}}(t)+\sup_{u \in \R} \int_{C_u(t)} \int_v \frac{ v^{\underline{L}}}{v^0}|z^{j-d} h| dv dC_u(t) \hspace{1mm} \lesssim \hspace{1mm} \epsilon (1+t)^{\kappa} \log^{r-2d}(3+t).$$
Using Lemmas \ref{weights1} and \ref{vradial}, we have
$$ \left| \left( \nabla_v z^j \right)^L \right|, \hspace{1mm} \left| \left( \nabla_v z^j \right)^{\underline{L}} \right|, \hspace{1mm} \frac{|v^A|+v^{\underline{L}}}{v^0}\left| \left( \nabla_v z^j \right)^A \right| \lesssim \frac{\tau_-}{v^0}|z|^{j-1}+\frac{1}{v^0} \sum_{w \in \mathbf{k}_1} |w|^j.$$
Hence, the decomposition of $F\left(v,\nabla_v |z|^j\right)$ in our null frame brings us to control the integral, over $[0,T] \times \R^3_x \times \R^3_v$, of\footnote{The second term comes from $\alpha(F)_Av^L\left(\nabla_v |z|^j \right)^A$.}
$$\left( \tau_-|w|^{j-1}+|w|^j \right)(|\rho(F)|+|\alpha(F)|+|\sigma(F)|+|\underline{\alpha}(F)|)\frac{|h|}{v^0} \hspace{5mm} \text{and} \hspace{5mm} \left( \tau_+|w|^{j-1}+ |w|^j \right) |\alpha(F)|\frac{|h|}{v^0}.$$
According to Remark \ref{lowderiv} and using $1 \lesssim \sqrt{v^0 v^{\underline{L}}}$ (see Lemma \ref{weights1}), we have
$$ \tau_-(|\rho(F)|+|\sigma(F)|+|\underline{\alpha}(F)|)+\tau_+|\alpha(F)| \lesssim \sqrt{\epsilon} \frac{\log(3+t)}{\tau_+},  \hspace{6mm}  |\rho(F)|+|\sigma(F)|+|\underline{\alpha}(F)|+|\alpha(F)| \lesssim \sqrt{\epsilon} \frac{v^0}{\tau_+^{\frac{3}{2}}}+\sqrt{\epsilon}\frac{v^{\underline{L}}}{\tau_-^{\frac{3}{2}}}.$$ 
The result is then implied by the following two estimates,
\begin{eqnarray}
\nonumber \int_0^t \int_{\Sigma_s} \int_v \sqrt{\epsilon} |h|\left( \frac{|w|^{j-1}}{1+s}\log(3+s)+\frac{|w|^j}{(1+s)^{\frac{3}{2}}} \right) dvdxds \hspace{-0.5mm} & \lesssim & \hspace{-0.5mm} \sqrt{\epsilon} \int_0^t \frac{\log(3+s)}{1+s} H_{j-1}(s)ds+\int_0^t \frac{H_{j}(s)}{(1+s)^{\frac{3}{2}}} ds \\ \nonumber
& \lesssim & \hspace{-0.5mm} \epsilon^{\frac{3}{2}} \int_0^t \frac{\log^{r-1}(3+t)}{(1+s)^{1-\kappa}} +\frac{\log^{r}(3+t)}{(1+s)^{\frac{5}{4}-\kappa}} ds \\ \nonumber
 & \lesssim & \epsilon^{\frac{3}{2}}(1+t)^{\kappa} \log^r(3+t), \\
\nonumber \int_0^t \int_{\Sigma_s} \frac{\sqrt{\epsilon}}{\tau_-^{\frac{3}{2}}} \int_v \frac{v^{\underline{L}}}{v^0} \left| w^j h \right| dvdxds \hspace{-0.5mm} & = & \hspace{-0.5mm} \int_{u=-\infty}^t \frac{\sqrt{\epsilon}}{\tau_-^{\frac{3}{2}}} \int_{C_u(t)} \int_v \frac{v^{\underline{L}}}{v^0} \left| w^j h \right| dv dC_u(t) du \\ \nonumber
& \lesssim & \hspace{-0.5mm} \sqrt{\epsilon} H_j(t) \int_{u=-\infty}^{+ \infty} \frac{du}{\tau_-^{\frac{3}{2}}} \hspace{2mm} \lesssim \hspace{2mm} \epsilon^{\frac{3}{2}} (1+t)^{\kappa} \log^r (3+t).
\end{eqnarray}
\end{proof}

\subsection{Bounds on several spacetime integrals}\label{sec82}

We estimate in this subsection the spacetime integral of the source terms of \eqref{eq:cat1}-\eqref{eq:cat4} of $T_F(z_0^q P_{\xi^0}(\Phi) Y^{\beta^0}f )$, multiplied by $(v^0)^{-1} P_{\xi^2}(\Phi)$, where the electromagnetic field is derived less than $N-3$ times. We then fix, for the remainder of the subsection,
\begin{itemize}
\item multi-indices $\gamma$, $\beta$ and $\xi^1$ such that $$|\gamma| \leq N-3, \hspace{3mm} |\xi^1|+ |\gamma| + |\beta| \leq Q+1, \hspace{3mm} |\beta| \leq |\beta^0|, \hspace{3mm} |\xi^1|+|\beta| \leq |\xi^0|+|\beta^0| \leq Q \hspace{3mm} \text{and} \hspace{3mm} |\xi^1| \leq Q-1.$$
\item $n \leq 2N$, \hspace{2mm} $z \in \mathbf{k}_1$ \hspace{2mm} and \hspace{2mm} $j \in \mathbb{N}$ \hspace{2mm} such that \hspace{2mm} $j \leq 2N-1+n_Q-\xi^1_P-\xi^2_P-\beta_P$.
\item We will make more restrictive hypotheses for the study of the terms of \eqref{eq:cat3} and \eqref{eq:cat4}. For instance, for the last ones, we will take $|\xi^1| < |\xi^0|$ and $j=q$. This has to do with their properties described in Proposition \ref{ComuPkp}.
\end{itemize}
Note that $|\xi^2|+|\beta| \leq Q$. To lighten the notations, we introduce
$$h := z^j P_{\xi^1}(\Phi) P_{\xi^2}(\Phi) Y^{\beta} f.$$
We start by treating the terms of \eqref{eq:cat1}.
\begin{Pro}\label{M2}
Under the bootstrap assumptions \eqref{bootf1}-\eqref{bootf3}, we have,
$$I_1:=\int_0^t \int_{\Sigma_s} \int_v  |\Phi|^n \left( \left| \nabla_{Z^{\gamma}} F \right|+\frac{\tau_+}{\tau_-} \left| \alpha \left( \mathcal{L}_{Z^{\gamma}}(F) \right)\right|+\frac{\tau_+}{\tau_-}\sqrt{\frac{v^{\underline{L}}}{v^0}} \left| \sigma \left( \mathcal{L}_{Z^{\gamma}}(F) \right) \right| \right) \left| h \right| \frac{dv}{v^0} dx ds \hspace{2mm} \lesssim \hspace{2mm} \epsilon^{\frac{3}{2}}.$$
\end{Pro}
\begin{proof}
According to Propositions \ref{Phi1}, \ref{decayF} and $1 \lesssim \sqrt{v^0 v^{\underline{L}}}$, we have
\begin{eqnarray}
\nonumber \left| \Phi \right|^n \left| \nabla_{Z^{\gamma}} F \right|+\left| \Phi \right|^n\frac{\tau_+}{\tau_-} \left| \alpha \left( \mathcal{L}_{Z^{\gamma}}(F) \right)\right|+\left| \Phi \right|^n \frac{\tau_+}{\tau_-}\sqrt{\frac{v^{\underline{L}}}{v^0}} \left| \sigma \left( \mathcal{L}_{Z^{\gamma}}(F) \right) \right| \hspace{-1.5mm} & \lesssim  & \hspace{-1.5mm} \sqrt{\epsilon}\log^{4N+M}(3+t)\left( \frac{\sqrt{v^0 v^{\underline{L}}}}{\tau_+\tau_-}+ \frac{v^{\underline{L}}}{\tau_+^{\frac{1}{2}}\tau_-^{\frac{3}{2}}} \right)  \\ \nonumber
& \lesssim & \hspace{-1.5mm} \sqrt{\epsilon} \frac{v^0}{\tau_+^{\frac{5}{4}}}+\sqrt{\epsilon} \frac{v^{\underline{L}}}{\tau_+^{\frac{1}{4}}\tau_-^{\frac{3}{2}}}.
\end{eqnarray}
Then,
\begin{eqnarray}
\nonumber I_1 & \lesssim  & \int_0^t \int_{\Sigma_s} \frac{\sqrt{\epsilon}}{\tau_+^{\frac{5}{4}}} \int_v  |h| dv dx ds + \int_0^t \int_{\Sigma_s} \frac{\sqrt{\epsilon}}{\tau_+^{\frac{1}{4}}\tau_-^{\frac{3}{2}}} \int_v \frac{ v^{\underline{L}}}{v^0} |h| dv dx ds \\ \nonumber
& \lesssim & \sqrt{\epsilon} \int_0^t \frac{\E[h](s)}{(1+s)^{\frac{5}{4}}}ds+\sqrt{\epsilon} \int_{u=-\infty}^t \int_{C_u(t)} \frac{1}{\tau_+^{\frac{1}{4}}\tau_-^{\frac{3}{2}}}  \int_v \frac{ v^{\underline{L}}}{v^0} |h| dv d C_u(t) du.
\end{eqnarray}
Recall now the definition of $(t_i)_{i \in \mathbb{N}}$, $(T_i(t))_{i \in \mathbb{N}}$ and $C_u^i(t)$ from Subsection \ref{secsubsets}. By the bootstrap assumption \eqref{bootf3} and $2\eta < \frac{1}{8}$, we have 
$$\E[h](s) \lesssim \epsilon(1+s)^{\frac{1}{8}} \hspace{5mm} \text{and} \hspace{5mm} \sup_{u \in \R} \int_{C_u^i(t)} \int_v v^0 v^{\underline{L}} |h| dv dC_u^i(t) \lesssim  \epsilon (1+T_{i+1}(t))^{2 \eta} \lesssim \epsilon (1+t_{i+1})^{ \frac{1}{8}},$$
so that, using also\footnote{Note that the sum over $i$ is actually finite as $C^i_u(t) = \varnothing$ for $i \geq \log_2(1+t)$.} $1+t_{i+1} \leq 2(1+t_i) $ and Lemma \ref{foliationexpli}, 
\begin{eqnarray}
\nonumber \sqrt{\epsilon} \int_0^t \frac{\E[h](s)}{(1+s)^{\frac{5}{4}}} & \lesssim & \epsilon^{\frac{3}{2}} \int_0^{+ \infty} \frac{ds}{(1+s)^{\frac{9}{8}}} \hspace{2mm} \lesssim \hspace{2mm} \epsilon^{\frac{3}{2}}, \\
\nonumber \sqrt{\epsilon} \int_{u=-\infty}^t \int_{C_u(t)} \frac{1}{\tau_+^{\frac{1}{4}}\tau_-^{\frac{3}{2}}}  \int_v \frac{ v^{\underline{L}}}{v^0} |h| dv d C_u(t) du \hspace{-1mm} & = & \hspace{-1mm} \sqrt{\epsilon}  \int_{u=-\infty}^t \sum_{i=0}^{+ \infty} \int_{C^i_u(t)} \frac{1}{\tau_+^{\frac{1}{4}}\tau_-^{\frac{3}{2}}}  \int_v \frac{ v^{\underline{L}}}{v^0} |h| dv d C^i_u(t) du \\ \nonumber
& \lesssim &  \sqrt{\epsilon}  \int_{u=-\infty}^t \frac{1}{\tau_-^{\frac{3}{2}}} \sum_{i=0}^{+ \infty}\frac{1}{(1+t_i)^{\frac{1}{4}}} \int_{C^i_u(t)}  \int_v \frac{ v^{\underline{L}}}{v^0} |h| dv d C^i_u(t) du \\ \nonumber
& \lesssim &   \epsilon^{\frac{3}{2}} \int_{u=-\infty}^t \frac{du}{\tau_-^{\frac{3}{2}}} \sum_{i=0}^{+ \infty}\frac{(1+t_{i+1})^{\frac{1}{8}}}{(1+t_{i+1})^{\frac{1}{4}}}  \\ \nonumber
& \lesssim &  \epsilon^{\frac{3}{2}} \int_{u = - \infty}^{+\infty} \frac{du}{\tau_-^{\frac{3}{2}}} \sum_{i=0}^{+ \infty} 2^{-\frac{i}{8}}  \hspace{2mm} \lesssim \hspace{2mm} \epsilon^{\frac{3}{2}}.
\end{eqnarray}
\end{proof}
We now start to bound the problematic terms.
\begin{Pro}\label{M3}
We study here the terms of \eqref{eq:cat3}. If, for $\kappa \geq 0$ and $r \in \mathbb{N}$,
$$
\E[h](t)= \left\|  h \right\|_{L^1_{x,v}}(t)+\sup_{u \in \R} \int_{C_u(t)} \int_v \frac{ v^{\underline{L}}}{v^0} |h| dv dC_u(t) \lesssim \epsilon (1+s)^{\kappa} \log^r(3+t), \hspace{3mm} \text{then}
$$
\vspace{-5mm}
\begin{flalign*}
& \hspace{1cm} I^1_3:=\int_{0}^t \int_{\Sigma_s}  \frac{\tau_+}{\tau_-} \left| \underline{\alpha} \left( \mathcal{L}_{Z^{\gamma}} ( F) \right) \right| \int_v \sqrt{\frac{v^{\underline{L}}}{v^0}} \left| h \right| \frac{dv}{v^0} dx ds \lesssim \epsilon^{\frac{3}{2}}(1+s)^{\kappa} \log^r(3+t) \hspace{3mm} \text{and} & \\
& \hspace{1cm} I^2_3:=\int_{0}^t \int_{\Sigma_s}  \frac{\tau_+}{\tau_-} \left| \rho \left( \mathcal{L}_{Z^{\gamma}} ( F) \right) \right|\int_v  \left| h \right| \frac{dv}{v^0} dx ds \lesssim \epsilon^{\frac{3}{2}}(1+s)^{\kappa} \log^{r+a}(3+t)  .&
\end{flalign*}
\end{Pro}
\begin{Rq}
The extra $\log^{a}(3+t)$-growth on $I^2_3$, compared to $I^1_3$, will not avoid us to close the energy estimates in view of the hierarchies in the energy norms. Indeed, we have $j=q-1$ (in $I^2_3$) according to the properties of the terms of \eqref{eq:cat3} (in $I_3^1$, we merely have $j \leq q$).
\end{Rq}
\begin{proof}
Recall first from Lemma \ref{weights1} that $1+|v^A|  \lesssim \sqrt{v^0 v^{\underline{L}}}$. Then, using Proposition \ref{decayF} and the inequality $2CD \leq C^2+D^2$, one obtains
$$ \sqrt{\frac{v^{\underline{L}}}{v^0}} \frac{\tau_+}{\tau_-}\left|\underline{\alpha} \left( \mathcal{L}_{ Z^{\gamma}}(F)  \right) \right| \lesssim  \sqrt{\epsilon} \frac{v^{\underline{L}}}{\tau_-^{\frac{3}{2}}} \hspace{8mm} \text{and} \hspace{8mm} \frac{\tau_+}{\tau_-} \left| \rho \left( \mathcal{L}_{Z^{\gamma}}(F) \right) \right| \lesssim \sqrt{\epsilon} \log^M(3+t) \frac{v^0}{\tau_+}+\sqrt{\epsilon} \log^M(3+t) \frac{v^{\underline{L}} }{ \tau_-^3} .$$
We then have, as $a = M+1$,
\begin{eqnarray}
\nonumber I^1_3 &  \lesssim  & \int_{u = -\infty}^t \frac{\sqrt{\epsilon}}{\tau_-^{\frac{3}{2}}} \int_{C_u(t)}  \int_v \frac{ v^{\underline{L}}}{v^0}  |h| dv dC_u(t) du  \hspace{2mm} \lesssim \hspace{2mm}  \epsilon^{\frac{3}{2}} \E[h](t) \int_{u=-\infty}^{+\infty} \frac{du}{\tau_-^{\frac{3}{2}}} \hspace{2mm} \lesssim \hspace{2mm} \epsilon^{\frac{3}{2}} (1+s)^{\kappa} \log^r(3+t) , \\
\nonumber I^2_3 & \lesssim &  \sqrt{\epsilon}  \int_0^t \int_{\Sigma_s} \frac{\log^M(3+s)}{\tau_+} \int_v  |h| dv dx ds +\sqrt{\epsilon} \log^M (3+t) \int_{u = -\infty}^t \frac{\sqrt{\epsilon}}{\tau_-^{\frac{3}{2}}} \int_{C_u(t)}  \int_v \frac{ v^{\underline{L}}}{v^0}  |h| dv dC_u(t) du  \\ \nonumber 
& \lesssim & \sqrt{\epsilon}  \int_0^t \frac{\log^{r+M}(3+s)}{(1+s)^{1- \kappa}}  ds+ \epsilon^{\frac{3}{2}} (1+t)^{\kappa} \log^{r+M}(3+t) \\ \nonumber
& \lesssim & \epsilon^{\frac{3}{2}}(1+t)^{\kappa} \log^{r+M+1}(3+t) \hspace{2mm} = \hspace{2mm} \epsilon^{\frac{3}{2}}(1+t)^{\kappa} \log^{r+a}(3+t) .
\end{eqnarray}
 \end{proof}
We finally end this subsection by the following estimate.
\begin{Pro}\label{MM3}
We suppose here that $\max( |\xi^1|+|\beta|, |\xi^1|+1) \leq N-1$. Then, 
\begin{eqnarray}
\nonumber I_4 \hspace{2mm} := \hspace{2mm} \int_0^t \int_{\Sigma_s} \tau_+ \int_v  \left| \frac{v^{\mu}}{v^0} \mathcal{L}_{Z^{\gamma}}(F)_{\mu \nu} \right| \left| h \right| \frac{dv}{v^0} dx ds & \lesssim & \epsilon^{\frac{3}{2}} \log^{(1+j+|\xi^1|+|\xi^2|+|\beta|)a}(3+t) \hspace{1cm} \text{if} \hspace{3mm} |\xi^2| \leq N-2, \\
\nonumber & \lesssim & \epsilon^{\frac{3}{2}} (1+t)^{\frac{3}{4} \eta} \hspace{4cm} \text{otherwise}.
\end{eqnarray} 
\end{Pro}
\begin{Rq}\label{rq:i4}
To understand the extra hypothesis made in this proposition, recall from the properties of the terms of \eqref{eq:cat4} that we can assume $|\xi^1| < |\xi^0|$, $\beta=\beta^0$ and $j=q$. We then have
$$1+j+|\xi^1|+|\xi^2|+|\beta| \leq q+|\xi^0|+|\xi^2|+|\beta^0|.$$
\end{Rq}
\begin{proof}
Let us denote by $(\alpha, \underline{\alpha}, \rho, \sigma)$ the null decomposition of $\mathcal{L}_{Z^{\gamma}}(F)$. Using $1+|v^A| \leq \sqrt{v^0 v^{\underline{L}}}$ and Proposition \ref{decayF}, we have
\begin{eqnarray}
\nonumber \tau_+ \left| \frac{v^{\mu}}{(v^0)^2} \mathcal{L}_{Z^{\gamma}}(F)_{\mu \nu} \right| \hspace{-1mm} & \lesssim & \hspace{-1mm} \tau_+ \sqrt{\frac{v^{\underline{L}}}{v^0}} \left( |\alpha|+|\rho|+|\sigma| \right)+\tau_+\frac{v^{\underline{L}}}{v^0}|\underline{\alpha}| \\ \nonumber
& \lesssim & \hspace{-1mm} \sqrt{\epsilon} \sqrt{\frac{v^{\underline{L}}}{v^0}} \frac{\log^M (3+t)}{\sqrt{\tau_+ \tau_-}}+\sqrt{\epsilon} \frac{v^{\underline{L}}}{v^0} \frac{\log^M (3+t)}{\tau_-} \hspace{1.2mm} \lesssim \hspace{1.2mm} \sqrt{\epsilon} \frac{\log^M (3+t)}{\tau_+}+\sqrt{\epsilon} \frac{v^{\underline{L}}}{v^0} \frac{\log^M (3+t)}{\tau_-}.
\end{eqnarray}
As $\tau_- \sim \tau_+$ away from the light cone (for, say\footnote{If $(s,y)$ is in one of these regions of $[0,t] \times \R^3$, we have $|y| \geq 2s$ or $|y| \leq \frac{s}{2}$.}, $u \leq -t$ and $ u \geq \frac{t}{2}$), we finally obtain that
\begin{eqnarray}
\nonumber I_4 \hspace{-1mm} & = & \hspace{-1mm} \sqrt{\epsilon} \int_0^t \frac{\log^M(3+s)}{1+s} \int_{ \Sigma_s } \int_v |h| dv dx ds + \sqrt{\epsilon} \log^M (3+t) \int_{u=-t}^{\frac{t}{2}} \frac{1}{\tau_-} \int_{C_u(t)} \int_v \frac{ v^{\underline{L}}}{v^0} |h| dv dC_u(t) du \\ \nonumber
& \lesssim & \hspace{-1mm} \sqrt{\epsilon} \log^M(3+t) \sup_{[0,t]} \E[h] \int_0^t \hspace{-0.5mm} \frac{ds}{1+s}+\sqrt{\epsilon}\log^M(3+t) \E[h](t) \int_{u=-t}^t \hspace{-0.5mm} \frac{du}{\tau_-} \hspace{1.2mm} \lesssim \hspace{1.2mm} \sqrt{\epsilon} \log^{a}(3+t) \sup_{[0,t]} \E[h] .
\end{eqnarray}
If $|\xi^2| \leq N-2$, the bootstrap assumption \eqref{bootf1} or \eqref{bootf2} gives
$$ \sup_{[0,t]} \E[h] \leq \epsilon \log^{(j+|\xi^1|+|\xi^2|+|\beta|)a}(3+t)$$
and we can conclude the proof in that case. If $|\xi^2|=N-1$, we have $j \leq 2N-1-\xi^1_P-\xi^2_P-\beta_P$ since this case appears only if $Q=N$. Let $(i_1,i_2) \in \mathbb{N}^2$ be such that $$i_1+i_2=2j, \hspace{1cm} i_1 \leq 2N-1-2 \xi^1_P-\beta_P \hspace{1cm} \text{and} \hspace{1cm} i_2 \leq 2N-1-2 \xi^2_P-\beta_P.$$ Using the bootstrap assumptions \eqref{bootf2} and \eqref{bootf3}, we have
\begin{eqnarray}
\nonumber  \E[h](t) & = & \int_{\Sigma_t} \int_v \left| z^j P_{\xi^1}(\Phi) P_{\xi^2}(\Phi) Y^{\beta} f \right| dv dx \\ \nonumber
& \lesssim & \left| \int_{\Sigma_t} \int_v \left| z^{i_1} P_{\xi^1}(\Phi)^2 Y^{\beta} f \right| dv dx \int_{\Sigma_t} \int_v \left| z^{i_2}  P_{\xi^2}(\Phi)^2 Y^{\beta} f \right| dv dx \right|^{\frac{1}{2}} \\ \nonumber
& \lesssim & \left|  \log^{(i_1+2|\xi^1|+|\beta|)a}(3+t) \E_{N-1}^0[f](t) \log^{a i_2}(3+t) \overline{\E}_N[f](t) \right|^{\frac{1}{2}} \hspace{2mm} \lesssim \hspace{2mm} \epsilon (1+t)^{\frac{3}{4} \eta},
\end{eqnarray}
which ends the proof.
\end{proof}
Note now that Propositions \ref{ComuPkp}, \ref{M1}, \ref{M2}, \ref{M3} and \ref{MM3} imply \eqref{improvebootf2} for $Q=N-3$, so that $\E^4_{N-3}[f] \leq 3 \epsilon$ on $[0,T[$. 
\subsection{Completion of the bounds on the spacetime integrals}\label{ref:L2elec}

In this subsection, we bound the spacetime integrals considered previously when the electromagnetic field is differentiated too many times to be estimated pointwise. For this, we make crucial use of the pointwise decay estimates on the velocity averages of $ \left| z^j P_{\zeta}(\Phi) Y^{\beta} f \right|$ which are given by \eqref{decayf}. The terms studied here appear only if $|\xi^0|+|\beta^0| \geq N-2$ since otherwise the electromagnetic field would be differentiated at most $N-3$ times. We then fix, for the remainder of the subsection, $Q \in \{N-1,N \}$,
\begin{itemize}
\item multi-indices $\gamma$, $\beta$ and $\xi^1$ such that \hspace{1mm} $N-2 \leq |\gamma| \leq N$, $$ |\gamma|+|\xi^1| \leq Q, \hspace{3mm} |\xi^1|+ |\gamma| + |\beta| \leq Q+1, \hspace{3mm} |\beta| \leq |\beta^0|, \hspace{3mm} |\xi^1|+|\beta| \leq |\xi^0|+|\beta^0| \leq Q \hspace{3mm} \text{and} \hspace{3mm} |\xi^1|  \leq Q-1.$$
\item $n \leq 2N$, \hspace{2mm} $z \in \mathbf{k}_1$ \hspace{2mm} and \hspace{2mm} $j \in \mathbb{N}$ \hspace{2mm} such that \hspace{2mm} $j \leq 2N-1-\xi^1_P-\xi^2_P-\beta_P$.
\item Consistently with Proposition \ref{ComuPkp}, we will, in certain cases, make more assumptions on $\xi^1$ or $j$, such as $j \leq q$ for the terms of \eqref{eq:cat3}.
\end{itemize}
Note that $|\xi^2|+|\beta| \leq Q$ and that there exists $i_1$ and $i_2$ such as
$$i_1+i_2=2j, \hspace{5mm} i_1 \leq 2N-1-2\xi^1_P-\beta_P \hspace{5mm} \text{and} \hspace{5mm} i_2 \leq 2N-1-2\xi^2_P-\beta_P.$$
To lighten the notations, we introduce
$$h := z^j P_{\xi^1}(\Phi) P_{\xi^2}(\Phi) Y^{\beta} f, \hspace{10mm} h_1 := z^{i_1} P_{\xi^1}(\Phi)^2 Y^{\beta} f \hspace{10mm} \text{and} \hspace{10mm} h_2 := z^{i_2} P_{\xi^2}(\Phi)^2 Y^{\beta} f,$$
so that $\left| h \right| = \sqrt{| h_1 h_2 |}$. As $|\gamma| \geq N-2$, we have $|\xi^1| \leq 2 \leq N-7$ and $2|\xi^1| + |\beta| \leq 5 \leq N-6$. Thus, by Lemma \ref{weights1} and \eqref{decayf}, we have, for all $(t,x) \in [0,T[ \times \R^3$,
\begin{equation}\label{eq:h1}
 \tau_+^3 \int_v |h_1| \frac{dv}{(v^0)^2}+\tau_+^2 \tau_- \int_v |h_1| dv \hspace{1mm} \lesssim \hspace{1mm} \int_v \left(\tau_+^3 \frac{v^{\underline{L}}}{v^0}+\tau_+^2 \tau_- \right) |h_1| dv \hspace{1mm} \lesssim \hspace{1mm} \epsilon \log^{(4+i_1+2|\xi^1|+|\beta|)a}(3+t).
\end{equation}
Using Remark \ref{rqweights1}, we have,
\begin{equation}\label{eq:h11}
\forall \hspace{0.5mm} |x| \geq t, \hspace{1cm} \tau_+^3 \tau_- \int_v |h_1| \frac{dv}{(v^0)^2} \hspace{1mm} \lesssim \hspace{1mm}  \tau_+^3 \tau_- \int_v \frac{v^{\underline{L}}}{v^0} |h_1| dv \hspace{1mm} \lesssim \hspace{1mm} \epsilon \log^{(4+i_1+2|\xi^1|+|\beta|)a}(3+t).
\end{equation}
\begin{Pro}\label{M21}
The following estimates hold, 
$$I^1_1  :=  \int_0^t  \int_{\Sigma_s}  \int_v  |\Phi|^n\left| \nabla_{Z^{\gamma}} F \right|  \left| h \right| \frac{dv}{v^0} dx ds  \lesssim  \epsilon^{\frac{3}{2}}, \hspace{6mm} I^2_1  :=  \int_0^t  \int_{\Sigma_s}  \int_v  |\Phi|^n \frac{\tau_+}{\tau_-}\sqrt{\frac{v^{\underline{L}}}{v^0}} \left| \sigma \left( \mathcal{L}_{Z^{\gamma}}(F) \right) \right|  \left| h \right| \frac{dv}{v^0} dx ds  \lesssim  \epsilon^{\frac{3}{2}}$$
$$ \text{and} \hspace{8mm} I^3_1 := \int_0^t \int_{\Sigma_s} \int_v |\Phi|^n \frac{\tau_+}{\tau_-} \left| \alpha \left( \mathcal{L}_{Z^{\gamma}}(F) \right) \right| \left| h \right| \frac{dv}{v^0} dx ds  \lesssim  \epsilon^{\frac{3}{2}}.$$
\end{Pro}
\begin{proof}
Using the Cauchy-Schwarz inequality twice (in $x$ and then in $v$), $ \| \nabla_{Z^{\gamma}} F \|^2_{L^2(\Sigma_t)} \lesssim  \mathcal{E}_N^0[F](t) \leq 4\epsilon $, $|\Phi| \lesssim \sqrt{\epsilon} \log^2(1+\tau_+)$, $\overline{\E}_N[f](t) \lesssim \epsilon (1+t)^{\eta}$ and \eqref{eq:h1}, we have
\begin{eqnarray}
\nonumber I^1_1 & \lesssim & \int_0^t \| \nabla_{Z^{\gamma}} F \|_{L^2(\Sigma_s)} \left\| \int_v |\Phi|^{n} |h| \frac{dv}{v^0} \right\|_{L^2(\Sigma_s)}ds \\ \nonumber
& \lesssim & \sqrt{\epsilon} \int_0^t  \left\| \log^{8N}(1+\tau_+)  \int_v |h_1|  \frac{dv}{(v^0)^2} \int_v |h_2| dv  \right\|_{L^1(\Sigma_s)}^{\frac{1}{2}}ds \\ \nonumber
& \lesssim & \sqrt{\epsilon} \int_0^t  \left\| \log^{8N}(1+\tau_+)  \int_v |h_1|  \frac{dv}{(v^0)^2} \right\|_{L^{\infty}(\Sigma_s)}^{\frac{1}{2}} \sqrt{\E[h_2](s)}ds \\ \nonumber
& \lesssim & \epsilon \int_0^t \frac{\log^{4N+3Na}(3+s)}{(1+s)^{\frac{3}{2}}} \log^{ai_2}(3+s) \overline{\E}_N[f](s) ds \hspace{2mm} \lesssim \hspace{2mm} \epsilon^{\frac{3}{2}}. 
\end{eqnarray}
For the second one, recall from the bootstrap assumptions \eqref{bootF1} and \eqref{bootf3} that for all $t \in [0,T[$ and $i \in \mathbb{N}$,
$$ \int_{C^i_u(t)} \hspace{-0.5mm} |\sigma|^2 dC^i_u(t) \leq \mathcal{E}_N^0[F](t_{i+1}(t)) \lesssim \epsilon \hspace{3.2mm} \text{and} \hspace{3.2mm} \sup_{u \in \R} \int_{C_u^i(t)} \int_v \frac{ v^{\underline{L}}}{v^0} \left| h_2 \right| dv dC^i_u(t)  \lesssim \E[h_2](T_{i+1}(t)) \lesssim \epsilon (1+t_{i+1})^{2\eta}.$$
Hence, using this time a null foliation, one has
\begin{eqnarray}
\nonumber I_1^2 & \lesssim & \sum_{i=0}^{+ \infty}\int_{u=-\infty}^t \frac{1}{\tau_-} \left| \int_{C^i_u(t)} |\sigma \left( \mathcal{L}_{Z^{\gamma}} ( F) \right)|^2 dC^i_u(t) \int_{C_u^i(t)} \tau_+^{2} \left| \int_v |\Phi|^n\sqrt{\frac{ v^{\underline{L}}}{v^0}}  |h|  \frac{dv}{v^0} \right|^2 dC_u^i(t) \right|^{\frac{1}{2}} du \\ \nonumber
& \lesssim & \sqrt{\epsilon} \sum_{i=0}^{+ \infty} \int_{u=-\infty}^t \frac{1}{\tau_-} \left| \int_{C_u^i(t)} \tau_+^{2} \log^{8N}(1+\tau_+) \int_v \left| h_1 \right| \frac{dv}{(v^0)^2} \int_v \frac{ v^{\underline{L}}}{v^0} \left| h_2 \right| dv dC_u^i(t) \right|^{\frac{1}{2}} du \\ \nonumber
& \lesssim & \sqrt{\epsilon} \sum_{i=0}^{+ \infty }\int_{u=-\infty}^t \frac{1}{\tau_-} \left| \int_{C_u^i(t)} \frac{1}{\tau_+^{\frac{3}{4}}} \int_v \frac{ v^{\underline{L}}}{v^0} \left| h_2 \right| dv dC_u^i(t) \right|^{\frac{1}{2}} du \\ \nonumber
& \lesssim & \epsilon^{\frac{3}{2}} \int_{u=-\infty}^{+ \infty} \frac{du}{\tau_-^{\frac{9}{8}}}  \sum_{i=0}^{+ \infty} \frac{(1+t_{i+1})^{\eta}}{(1+t_i)^{\frac{1}{4}}} \hspace{2mm} \lesssim \hspace{2mm} \epsilon^{\frac{3}{2}}. 
\end{eqnarray}
For the last one, use first that $F=\F+\Ff$ to get
$$I^3_1 = I_1^{\F}+I_1^{\Ff} := \int_0^t \int_{\Sigma_s} \int_v |\Phi|^n \frac{\tau_+}{\tau_-} \left| \alpha \left( \mathcal{L}_{Z^{\gamma}}(\F) \right) \right| \left| h \right| \frac{dv}{v^0} dx ds+\int_0^t \int_{\Sigma_s} \int_v |\Phi|^n \frac{\tau_+}{\tau_-} \left| \alpha \left( \mathcal{L}_{Z^{\gamma}}(\Ff) \right) \right| \left| h \right| \frac{dv}{v^0} dx ds.$$
By Proposition \ref{decayF}, we have $|\mathcal{L}_{Z^{\gamma}}(\Ff)| \lesssim \epsilon \tau_+^{-2}$. Hence, using $|\Phi| \lesssim \log^2(1+\tau_+)$ and $1 \lesssim \sqrt{v^0 v^{\underline{L}}}$, we have
$$ |\Phi|^n \frac{\tau_+}{\tau_-} \left| \alpha \left( \mathcal{L}_{Z^{\gamma}}(\Ff) \right) \right| \lesssim  \frac{ \epsilon \sqrt{v^0v^{\underline{L}}} }{\sqrt{v^0}\tau_+^{\frac{3}{4}} \tau_-} \leq \epsilon \frac{v^0}{\tau_+^{\frac{5}{4}}} + \epsilon \frac{ v^{\underline{L}} }{\tau_+^{\frac{1}{4}} \tau_-^2}$$
and we can bound $I_1^{\Ff}$ by $\epsilon^{\frac{3}{2}}$ as $I_1$ in Proposition \ref{M2}. For $I_1^{\F}$, remark first that, by the bootstrap assumptions \eqref{bootext}, \eqref{bootF4} and since $F=\F$ in the interior of the light cone, 
$$ \int_{C^i_u(t)} \tau_+\left| \alpha \left( \mathcal{L}_{Z^{\gamma}}(\F) \right) \right|^2 dC_u^i(t) \lesssim \mathcal{E}_N[F](T_{i+1}(t))+\mathcal{E}^{Ext}_{N}[\F](T_{i+1}(t)) \lesssim \epsilon (1+t_{i+1})^{\eta}.$$
It then comes, using $1 \lesssim \sqrt{v^0 v^{\underline{L}}}$, $16 \eta < 1$ and $\int_v |\Phi|^n |h_1| dv \lesssim \epsilon \tau_+^{-\frac{3}{2}}\tau_-^{-1}$, that
\begin{eqnarray}
\nonumber I_1^{\F} & \lesssim & \sum_{i=0}^{+ \infty}\int_{u=-\infty}^t \frac{1}{\tau_-} \left| \int_{C^i_u(t)} \tau_+ \left| \alpha \left( \mathcal{L}_{Z^{\gamma}}(F) \right) \right|^2 dC_u^i(t) \int_{C_u^i(t)} \tau_+ \left| \int_v |\Phi|^n \sqrt{\frac{ v^{\underline{L}}}{v^0}}  |h|  dv \right|^2 dC_u^i(t) \right|^{\frac{1}{2}} du \\ \nonumber
& \lesssim & \sqrt{\epsilon} \sum_{i=0}^{+ \infty} (1+t_{i+1})^{\eta}\int_{u=-\infty}^t \frac{1}{\tau_-} \left| \int_{C_u^i(t)} \tau_+ \int_v |\Phi|^n \left| h_1 \right| dv \int_v \frac{ v^{\underline{L}}}{v^0} \left| h_2 \right| dv dC_u^i(t) \right|^{\frac{1}{2}} du \\ \nonumber
& \lesssim & \sqrt{\epsilon} \sum_{i=0}^{+ \infty }\frac{(1+t_{i+1})^{2\eta}}{(1+t_i)^{\frac{1}{4}}} \int_{u=-\infty}^{+\infty} \frac{1}{\tau_-^{\frac{3}{2}}} \hspace{2mm} \lesssim \hspace{2mm} \epsilon^{\frac{3}{2}} \sum_{i=0}^{+ \infty} 2^{-\frac{i}{4}(1-8 \eta)} \hspace{2mm} \lesssim \hspace{2mm} \epsilon^{\frac{3}{2}}. 
\end{eqnarray}
\end{proof} 
We now turn on the problematic terms.
\begin{Pro}\label{M32}
If $|\xi_2| \leq N-2$, we have
\begin{eqnarray}
\nonumber I^1_3 & = & \int_{0}^t \int_{\Sigma_s}  \frac{\tau_+}{\tau_-} \left| \underline{\alpha} \left( \mathcal{L}_{Z^{\gamma}} ( F) \right) \right| \int_v \sqrt{\frac{v^{\underline{L}}}{v^0}}  |h| \frac{dv}{v^0} dx ds \hspace{2mm} \lesssim \hspace{2mm} \epsilon^{\frac{3}{2}} \log^{(3+j+|\xi_1|+|\xi_2|+|\beta|)a} (3+t) \hspace{3mm} \text{and} \\ \nonumber
I^2_3 & = & \int_{0}^t \int_{\Sigma_s}  \frac{\tau_+}{\tau_-} \left| \rho \left( \mathcal{L}_{Z^{\gamma}} ( F) \right) \right|\int_v |h| \frac{dv}{v^0} dx ds \hspace{2mm} \lesssim \hspace{2mm} \epsilon^{\frac{3}{2}} \log^{(2+j+|\xi_1|+|\xi_2|+|\beta|)a} (3+t).
\end{eqnarray}
Otherwise, $|\xi^2|=N-1$ and $I^1_3 +I^2_3  \lesssim \epsilon^{\frac{3}{2}} (1+t)^{\frac{3}{4} \eta}$.
\end{Pro}
\begin{Rq}\label{nopb}
Note that these estimates are sufficient to improve the bootstrap assumptions \eqref{bootf2} and \eqref{bootf3}. Indeed,
\begin{itemize}
\item the case $|\xi^2|=N-1$ concerns only the study of $\overline{\E}_N[f]$.
\item Even if the bound on $I^2_3+I^1_3$, when $|\xi^2| \leq N-2$ could seem to possess a factor $\log^{3a}(3+t)$ in excess, one has to keep in mind that $|\gamma| \geq N-2$, so $|\xi^1|+|\beta| \leq 3$ and $|\xi^0|+|\beta^0| \geq N-2$. Moreover, by the properties of the terms of \eqref{eq:cat3}, $j \leq q$. We then have, as $N \geq 8$,
$$j+3+|\xi^1|+|\xi^2|+|\beta| \leq q+|\xi^0|+|\xi^2|+|\beta^0|.$$
\end{itemize}
\end{Rq}
\begin{proof}
Throughout this proof, we will use \eqref{eq:h1} and the bootstrap assumption \eqref{bootF1}, which implies
$$ \left\| \underline{\alpha} \left( \mathcal{L}_{Z^{\gamma}} ( F) \right) \right\|_{L^2 \left( \Sigma_t \right) } + \sup_{u \in \R} \left\| \rho \left( \mathcal{L}_{Z^{\gamma}} ( F) \right) \right\|_{L^2 \left( C_u(t) \right) } \lesssim \sqrt{\mathcal{E}_{N}^0[F](t)}  \lesssim \epsilon^{\frac{1}{2}}.$$
Applying the Cauchy-Schwarz inequality twice (in $(t,x)$ and then in $v$), we get
\begin{eqnarray}
\nonumber I^1_3 & \lesssim & \left| \int_{0}^t \frac{\left\| \underline{\alpha} \left( \mathcal{L}_{Z^{\gamma}} ( F) \right) \right\|_{L^2 \left( \Sigma_s \right) }}{1+s} ds   \int_{u=- \infty}^t \int_{C_u(t)} \frac{\tau_+^3}{\tau_-^2} \left| \int_v  \sqrt{\frac{v^{\underline{L}}}{v^0}} \left| h\right| \frac{dv}{v^0} \right|^2 dC_u(t) du \right|^{\frac{1}{2}} \\ \nonumber
& \lesssim & \epsilon^{\frac{1}{2}}\log^{\frac{1}{2}}(1+t) \left|   \int_{u=- \infty}^t \frac{1}{\tau_-^2} \int_{C_u(t)} \int_v  \frac{v^{\underline{L}}}{v^0} \left|  h_2 \right| dv dC_u(t) du \right|^{\frac{1}{2}} \sup_{u \in \R} \left\| \tau_+^3 \int_v |h_1| \frac{dv}{(v^0)^2} \right\|_{L^{\infty} \left( C_u(t) \right) }^{\frac{1}{2}} \\ \nonumber
& \lesssim & \epsilon \log^{\frac{1}{2}+\frac{a}{2} \left(4+i_1+ 2|\xi|^1+|\beta| \right)}(3+t) \sqrt{\E[h_2](t)}.
\end{eqnarray}
Using $1 \lesssim \sqrt{v^0 v^{\underline{L}}}$ and the Cauchy-Schwarz inequality (this time in $(\underline{u},\omega_1,\omega_2)$ and then in $v$), we obtain
\begin{eqnarray}
\nonumber I^2_3 & \lesssim &  \int_{u=-\infty}^t \left\| \rho \left( \mathcal{L}_{Z^{\gamma}} ( F) \right) \right\|_{L^2 \left( C_u(t) \right) }   \left| \int_{C_u(t)} \frac{\tau_+^2}{\tau_-^2} \left| \int_v  \sqrt{\frac{v^{\underline{L}}}{v^0}} \left| h \right| dv \right|^2 dC_u(t) \right|^{\frac{1}{2}} du \\ \nonumber
& \lesssim & \epsilon^{\frac{1}{2}}  \int_{u = - \infty}^t \frac{1}{\tau_-^{\frac{3}{2}}} \left\| \tau_+^2 \tau_- \int_v  |h_1| dv \right\|_{L^{\infty} \left( C_u(t) \right) }^{\frac{1}{2}} \left| \int_{C_u(t)} \int_v  \frac{v^{\underline{L}}}{v^0} \left|  h_2 \right| dv dC_u(t) \right|^{\frac{1}{2}} du  \\ \nonumber
& \lesssim & \epsilon \log^{\frac{a}{2} \left(4+i_1 +2|\xi|^1+|\beta| \right)} \sqrt{\E[h_2](t)}.
\end{eqnarray}
It then remains to remark that, by the bootstrap assumptions \eqref{bootf2} and \eqref{bootf3},
\begin{itemize}
\item $\E[h_2](t) \leq \log^{ (i_2+2|\xi_2|+|\beta|)a}(3+t) \E^0_{N-1}[f](t) \lesssim \epsilon \log^{ (i_2+2|\xi_2|+|\beta|)a}(3+t)$, if $|\xi_2| \leq N-2$, or
\item $\E[h_2](t) \leq \log^{a i_2}(3+t) \overline{\E}_{N}[f](t) \lesssim \epsilon (1+t)^{\eta} \log^{a i_2}(3+t)$, if $|\xi_2| = N-1$.
\end{itemize}
\end{proof}
Let us move now on the expressions of \eqref{eq:cat4}. The ones where $|\gamma|=N$ are the more critical terms and will be treated later.
\begin{Pro}\label{M41}
Suppose that $N-2 \leq |\gamma| \leq N-1$. Then, if $|\xi_2| \leq N-2$,
\begin{flalign*}
& \hspace{1cm} I_4=\int_0^t \int_{\Sigma_s} \int_v \tau_+ \left| \frac{v^{\mu}}{v^0} \mathcal{L}_{Z^{\gamma}}(F)_{\mu \nu} \right| \left| h \right| \frac{dv}{v^0} dx ds \lesssim \epsilon^{\frac{3}{2}} \log^{(3+j+|\xi_1|+|\xi_2|+|\beta|)a}(3+t) &
\end{flalign*}
and $I_4 \lesssim \epsilon^{\frac{3}{2}} (1+t)^{\frac{3}{4} \eta}$ otherwise.
\end{Pro}
For similar reasons as those given in Remark \ref{nopb}, these bounds are sufficient to close the energy estimates on $\overline{\E}_N[f]$ and $\E^0_{N-1}[f]$.

\begin{proof}
Denoting by $(\alpha, \underline{\alpha}, \rho, \sigma )$ the null decomposition of $\mathcal{L}_{Z^{\gamma}}(\F)$ and using $|v^A| \lesssim \sqrt{v^0v^{\underline{L}}}$, we have
\begin{eqnarray}
\nonumber \left| \frac{v^{\mu}}{v^0} \mathcal{L}_{Z^{\gamma}}(F)_{\mu \nu} \right| & \lesssim & |\alpha (\mathcal{L}_{ Z^{\gamma}}(F))|+|\sigma (\mathcal{L}_{ Z^{\gamma}}(F)) |+|\rho (\mathcal{L}_{ Z^{\gamma}}(F)) |+\sqrt{\frac{v^{\underline{L}}}{v^0}}|\underline{\alpha} (\mathcal{L}_{ Z^{\gamma}}(F)) | \\ \nonumber
& \lesssim & |\alpha|+|\rho|+|\sigma|+\sqrt{\frac{ v^{\underline{L}}}{v^0}}|\underline{\alpha}|+\left| \mathcal{L}_{Z^{\gamma}}( \Ff ) \right|.
\end{eqnarray}
and we can then bound $I_4$ by $I_{\alpha,\sigma,\rho}+I_{\underline{\alpha}}+I_{\Ff}$ (these quantities will be clearly defined below). Note now that 
\begin{equation}\label{eq:alphasigma}
 \left\| \sqrt{\tau_+} |\alpha| +\sqrt{\tau_+} |\rho|+\sqrt{\tau_+} |\sigma| \right\|^2_{L^2(\Sigma_s)}+\left\| \sqrt{\tau_-} |\underline{\alpha}| \right\|^2_{L^2(\Sigma_s)} \hspace{1mm} \lesssim  \hspace{1mm} \mathcal{E}^{Ext}_{N}[\F](s)+ \mathcal{E}_{N-1}[F](s)  \hspace{1mm} \lesssim  \hspace{1mm} \epsilon \log^{2M}(3+s).
\end{equation} 
Then, using the Cauchy-Schwarz inequality twice (in $(t,x)$ and then in $v$), the estimates \eqref{eq:h1} and \eqref{eq:h11} as well as $a = M+1$, we get
\begin{eqnarray}
\nonumber I_{\underline{\alpha}} & := & \int_0^t \int_{\Sigma_s} \tau_+ |\underline{\alpha}|  \int_v \sqrt{\frac{ v^{\underline{L}}}{v^0}} |h| \frac{dv}{v^0} dx ds \\ \nonumber
& \lesssim & \left| \int_0^t \frac{\|\sqrt{\tau_-} |\underline{\alpha}| \|^2_{L^2(\Sigma_s)}}{1+s}ds \int_{u=-\infty}^t  \int_{C_u(t)} \frac{\tau_+^2(1+s)}{\tau_-} \left| \int_v \sqrt{ \frac{ v^{\underline{L}}}{v^0}} |h|\frac{ dv}{v^0} \right|^2 dC_u(t) du \right|^{\frac{1}{2}} \\ \nonumber
& \lesssim & \sqrt{\epsilon} \log^{M+\frac{1}{2}}(3+t) \left| \int_{u=-\infty}^t \frac{1}{\tau_-} \left\| \tau_+^2(1+s) \int_v |h_1| \frac{dv}{(v^0)^2} \right\|_{L^{\infty}(C_u(t))} \int_{C_u(t)} \int_v \frac{ v^{\underline{L}}}{v^0} |h_2| dv  dC_u(t) du \right|^{\frac{1}{2}} \\ \nonumber
& \lesssim & \epsilon^{\frac{3}{2}} \log^{-\frac{1}{2}+\frac{a}{2}(5+i_1+2|\xi^1|+|\beta|)}(3+t) \sqrt{\E[h_2](t)} \left| \int_{u=-\infty}^0 \frac{du}{\tau_-^{\frac{3}{2}}}+\int_{u=0}^t \frac{du}{\tau_-}  \right|^{\frac{1}{2}} \\ \nonumber
 & \lesssim & \epsilon^{\frac{3}{2}} \log^{\frac{a}{2}(6+i_1+2|\xi^1|+|\beta|)}(3+t) \sqrt{\E[h_2](t)}.
\end{eqnarray}
Similarly, one has
\begin{eqnarray}
\nonumber I_{\alpha, \rho, \sigma} & := & \int_0^t \int_{\Sigma_s} \tau_+ (|\alpha|+|\rho|+|\sigma|) \int_v  |h| \frac{dv}{v^0} dx ds \\ \nonumber
& \lesssim & \int_0^t \| \sqrt{\tau_+}|\alpha|+\sqrt{\tau_+} |\rho|+\sqrt{\tau_+} |\sigma|  \|_{L^2(\Sigma_s)} \left\| \sqrt{\tau_+} \int_v |h| \frac{dv}{v^0} \right\|_{L^2(\Sigma_s)} ds \\ \nonumber
& \lesssim & \int_0^t \sqrt{\epsilon} \log^M(3+s) \left\| \tau_+ \int_v |h_1|\frac{dv}{(v^0)^2} \right\|_{L^{\infty}(\Sigma_s)}^{\frac{1}{2}} \left\| \int_v  |h_2| dv \right\|_{L^1(\Sigma_s)}^{\frac{1}{2}} ds \\ \nonumber
& \lesssim & \epsilon \log^{\frac{a}{2}(6+i_1+2|\xi^1|+|\beta|)} (3+t)  \left\| \E[h_2] \right\|_{L^{\infty}([0,t])}^{\frac{1}{2}}.
\end{eqnarray}
For the last integral, recall from Propositions \ref{propcharge} and \ref{decayF} that $\Ff(t,x)$ vanishes for all $t-|x| \geq -1$ and that $|\mathcal{L}_{Z^{\gamma}}(\Ff)| \lesssim \epsilon \tau_+^{-2}$. We are then led to bound
\begin{eqnarray}
\nonumber I_{\Ff} & := &  \int_0^t \int_{|x| \geq s+1} \tau_+ |\mathcal{L}_{Z^{\gamma}}(\Ff)|\int_v  |h| \frac{dv}{v^0} dx ds \\ \nonumber
& \lesssim & \int_0^t \frac{\sqrt{\epsilon}}{1+s} \int_{\Sigma_s} \int_v \sqrt{ \left| h_1 h_2 \right| } dv dx ds \hspace{1mm} \lesssim \hspace{1mm} \int_0^t \frac{\sqrt{\epsilon}}{1+s} \left| \int_{\Sigma_s} \int_v  \left| h_1  \right|  dv dx \int_{\Sigma_s} \int_v  \left| h_2  \right|  dv dx \right|^{\frac{1}{2}} ds \\  \nonumber
& \lesssim &  \sqrt{\epsilon} \log(3+t) \left\| \E[h_1] \right\|_{L^{\infty}([0,t])}^{\frac{1}{2}}\left\| \E[h_2] \right\|_{L^{\infty}([0,t])}^{\frac{1}{2}} .
\end{eqnarray}
Thus, as $\left\| \E[h_1] \right\|_{L^{\infty}([0,t])} \lesssim \epsilon \log^{(i_1+2|\xi_1|+|\beta|)a}(3+t)$ and $i_1+i_2=2j$, we have
\begin{itemize}
\item $I_4 \lesssim \epsilon^{\frac{3}{2}} (1+t)^{\frac{3}{4} \eta}$ if $|\xi_2|=N-1$, since $\E[h_2](t) \leq \log^{ai_2}(3+t) \overline{\E}_N[f](t) \leq \epsilon (1+t)^{\eta} \log^{ai_2}(3+t) $, and
\item $I_4 \lesssim \epsilon^{\frac{3}{2}} \log^{(3+j+|\xi_1|+|\xi_2|+|\beta|)a}(3+t)$ otherwise, as $\E[h_2] \leq \log^{(i_2+2|\xi^2|+|\beta|)a}(3+t) \E_{N-1}^0[f](t)$ .
\end{itemize}
\end{proof}
A better pointwise decay estimate on $\int_v |h_1|(v^0)^{-2}dv$ is requiered to bound sufficiently well $I_4$ when $|\gamma|=N$. We will then treat this case below, in the last part of this section. However, note that all the Propositions already proved in this section imply \eqref{improvebootf2}, for $Q=N-1$, and then $\E^0_{N-1}[f] \leq 3\epsilon$ on $[0,T[$.

\subsection{Estimates for $ \E^X_{N-1}[f]$, $ \E^X_{N}[f]$ and obtention of optimal decay near the lightcone for velocity averages}\label{Ximpro}

The purpose of this subsection is to establish that\footnote{Note that we cannot unify these norms because of a lack of weights $z \in \V$. As we will apply Proposition \ref{ComuPkp} with $N_0=2N-1$, we cannot propagate more than $2N-2$ weights and avoid in the same time the problematic terms.} $ \E^X_{N-1}[f]$, $ \E^X_{N}[f] \leq 3\epsilon$ on $[0,T[$ and then to deduce optimal pointwise decay estimates on the velocity averages of the particle density. Remark that, according to the energy estimate of Proposition \ref{energyf}, $ \E^X_{N}[f] \leq 3\epsilon$ follows, if $\epsilon$ is small enough, from
\begin{equation}\label{4:eq} \int_0^t \int_{\Sigma_s} \int_v \left| T_F \left( z^q P^X_{\xi} (\Phi ) Y^{\beta} f \right) \right| \frac{dv}{v^0} dx ds \lesssim \epsilon^{\frac{3}{2}}\log^{2q} (3+t),
\end{equation}
\begin{itemize}
\item for all multi-indices $\beta$ and $\xi$ such that $\max(|\beta|+|\xi|,|\xi|+1) \leq N$ and
\item for all $z \in \mathbf{k}_1$ and $q \in \mathbb{N}$ such that $q \leq 2N-2-\xi_P-\beta_P$.
\end{itemize}
Most of the work has already been done. Indeed, the commutation formula of Proposition \ref{ComuPkpX} (applied with $N_0=2N-1$) leads us to bound only terms of \eqref{eq:cat0} and \eqref{eq:cat1} since $q \leq 2N-2-\xi_P-\beta_P$.  Note that we control quantities of the form
$$ z^j P_{\xi^1}(\Phi) Y^{\beta^1} f, \hspace{5mm} \text{with} \hspace{5mm} |\xi^1|+|\beta^1| \leq N, \hspace{5mm} |\xi^1| \leq N-1 \hspace{5mm} \text{and} \hspace{5mm} j \leq 2N-1-\xi^1_P-\beta^1_P.$$
Consequently, \eqref{4:eq} ensues from Propositions \ref{M1}, \ref{M2} and \ref{M21}. $\E_{N-1}^X[f]$ can be estimated similarly since we also control quantities such as
$$ z^j P_{\xi^1}(\Phi) P_{\xi^2}(\Phi) Y^{\kappa} f, \hspace{5mm} \text{with} \hspace{5mm} \max(|\xi^1|+|\kappa|,|\xi^2|+|\kappa|) \leq N-1 \hspace{5mm} \text{and} \hspace{5mm} j \leq 2N-1-\xi^1_P-\xi^2_P-\kappa_P.$$
Note that \eqref{4:eq} also provides us, through Theorem \ref{decayopti}, that, for all $\max ( |\xi|+|\beta|, 1 +|\xi| ) \leq N-3$,
$$ \forall \hspace{0.5mm} |x| \leq t < T, \hspace{3mm} z \in \V, \hspace{3mm} j \leq 2N-5-\xi_P-\beta_P, \hspace{5mm} \int_v \left|z^j P^X_{\xi} (\Phi) Y^{\beta} f \right| \frac{dv}{(v^0)^2} \lesssim \epsilon \frac{\log^{2j}(3+t)}{\tau_+^3}.$$
For the exterior region, use Proposition \ref{decayopti2} and $ \E^{X}_N[f] \leq 3\epsilon$ to derive, for all $\max ( |\xi|+|\beta|, |\xi|+1 ) \leq N-3$,
$$ \forall \hspace{0.5mm} (t,x) \in V_0(T), \hspace{3mm} z \in \V, \hspace{3mm} j \leq 2N-6-\xi_P-\beta_P, \hspace{5mm} \int_v \left|z^j P^X_{\xi} (\Phi) Y^{\beta} f \right| \frac{dv}{(v^0)^2} \lesssim \epsilon \frac{\log^{2(j+1)}(3+t)}{\tau_+^3\tau_-}.$$
We summerize all these results in the following proposition (the last estimate comes from Corollary \ref{KS2}).
\begin{Pro}\label{Xdecay}
If $\epsilon$ is small enough, then $ \E^{X}_{N-1}[f] \leq 3 \epsilon$ and $ \E^{X}_{N}[f] \leq 3 \epsilon$ hold on $[0,T]$. Moreover, we have, for all $\max ( |\xi|+|\beta|, |\xi|+1 ) \leq N-3$, $z \in \V$ and $j \leq 2N-6 - \xi_P-\beta_P$,
\begin{eqnarray}
\nonumber \forall \hspace{0.5mm} (t,x) \in [0,T[ \times \R^3, \hspace{10mm} \int_v \left|z^j P^X_{\xi} (\Phi) Y^{\beta} f \right| \frac{dv}{(v^0)^2} & \lesssim & \epsilon \frac{\log^{2j}(3+t)}{\tau_+^3} \mathds{1}_{t \geq |x|}+ \epsilon\frac{\log^{2(j+1)}(3+t)}{\tau_+^3\tau_-} \mathds{1}_{|x| \geq t}, \\ \nonumber
\forall \hspace{0.5mm} (t,x) \in [0,T[ \times \R^3, \hspace{10mm} \int_v \left|z^j P^X_{\xi} (\Phi) Y^{\beta} f \right| dv & \lesssim & \epsilon \frac{\log^{2j}(3+t)}{\tau_+^2 \tau_-}.
\end{eqnarray}
\end{Pro}

\subsection{The critical terms}\label{sec86}

We finally bound $I_4$, defined in Proposition \ref{M41}, when $|\gamma|=N$, which concerns only the improvement of the bound of the higher order energy norm $\overline{\E}_N[f]$. We keep the notations introduced in Subsection \ref{ref:L2elec} and we start by precising them. Using the properties of the terms of \eqref{eq:cat4}, we remark that we necessarily have
$$P_{\xi^0}(\Phi)=Y^{\xi^0} \Phi, \hspace{5mm} |\xi^0|=N-1, \hspace{5mm} |\beta^0| \leq 1, \hspace{5mm} |\xi^1|=0, \hspace{5mm} \beta = \beta^0, \hspace{5mm} \gamma_T = \xi^0_T \hspace{5mm} \text{and} \hspace{5mm} j=q.$$ We are then led to prove
$$I_4 = \int_0^t \int_{\Sigma_s} \int_v \tau_+ \left| \frac{v^{\mu}}{v^0} \mathcal{L}_{Z^{\gamma}}(F)_{\mu \nu} \right| \left| z^{q} P_{\xi^2}(\Phi) Y^{\beta^0} f  \right|  \frac{dv}{v^0} dx ds \hspace{2mm} \lesssim \hspace{2mm} \epsilon^{\frac{3}{2}} (1+t)^{\eta} \log^{aq}(3+t).$$ 
If $\gamma_T=\xi^0_T \geq 1$, one can use inequality \eqref{eq:zeta2} of Proposition \ref{ExtradecayLie} and $|v^A| \lesssim \sqrt{v^0 v^{\underline{L}}}$ in order to obtain 
$$ \tau_+ \left| \frac{v^{\mu}}{v^0} \mathcal{L}_{Z^{\gamma}}(F)_{\mu \nu} \right| \lesssim \left(1+ \frac{\sqrt{ v^{\underline{L}}}\tau_+}{\sqrt{v^0}\tau_-} \right) \sum_{|\gamma_0| \leq N} \left| \nabla_{Z^{\gamma_0}} F \right|+ \frac{\tau_+}{\tau_-} \sum_{|\gamma_0| \leq N} \left| \alpha ( \mathcal{L}_{Z^{\gamma_0}}(F)) \right|+\left| \rho ( \mathcal{L}_{Z^{\gamma_0}}(F)) \right|$$ and then split $I_4$ in four parts and bound them by $\epsilon^{\frac{3}{2}}$ or $\epsilon^{\frac{3}{2}} (1+t)^{\frac{3}{4} \eta}$, as $I^1_1$, $I^3_1$, $I^1_3$ and $I^2_3$ in Propositions \ref{M21} and \ref{M32}. Otherwise, $\xi^0_P=N-1$ and $q \leq N-\xi^2_P-\beta^0_P$ so that we take $i_2 \leq 2N-1-2\xi^2_P-\beta^0_P$ and $i_1 \leq 1-\beta^0_P$. Then, we divide $[0,t] \times \R^3$ in two parts, $V_0(t)$ and its complement. Following the proof of Proposition \ref{M41}, one can prove, as $\mathcal{E}_{N}^{Ext}[\F] \lesssim \epsilon$ and $|\mathcal{L}_{Z^{\gamma}}(\Ff)| \lesssim \epsilon \tau_+^{-2}$ on $[0,T[$, that 
$$ \int_0^t \int_{\overline{\Sigma}^0_s} \int_v \tau_+ \left| \frac{v^{\mu}}{v^0} \mathcal{L}_{Z^{\gamma}}(F)_{\mu \nu} \right| \left| z^{q} P_{\xi^2}(\Phi) Y^{\beta^0} f  \right|  \frac{dv}{v^0} dx ds \lesssim \epsilon^{\frac{3}{2}} (1+t)^{\frac{3}{4} \eta}.$$
To lighten the notations, let us denote the null decomposition of $\mathcal{L}_{Z^{\gamma}}(F)$ by $(\alpha, \underline{\alpha}, \rho, \sigma)$. Recall from Lemma \ref{weights1} that $ \tau_+ |v^A| \lesssim  v^0 \sum_{w \in \V} |w|$ and $\tau_+v^{\underline{L}} \lesssim \tau_-v^0+v^0\sum_{w \in \V} |w|$, so that
\begin{eqnarray}
\nonumber \tau_+ \left| \frac{v^{\mu}}{v^0} \mathcal{L}_{Z^{\gamma}}(F)_{\mu \nu} \right| & \lesssim & \tau_+\left(|\alpha|+|\rho|\right)+\tau_+\frac{v^{\underline{L}}}{v^0} |\underline{\alpha}|+ \tau_+ \frac{|v^A|}{v^0}\left( |\sigma|+|\underline{\alpha}| \right) \\ \nonumber
& \lesssim & \left(\tau_+|\alpha|+\tau_+|\rho|+\tau_-|\underline{\alpha}| \right)+\sum_{w \in \V} |w| \left( |\sigma|+|\underline{\alpha}| \right).
\end{eqnarray}
We can then split the remaining part of $I_4$ in two integrals. The one associated to $\sum_{w \in \V} |w|(|\sigma|+|\underline{\alpha}|)$ can be bounded by $\epsilon^{\frac{3}{2}}$ as $I^1_1$ in Proposition \ref{M21} since $i_1+1 \leq 2N-1-\beta_P^0$. For the one associated to $\left(\tau_+|\alpha|+\tau_+|\rho|+\tau_-|\underline{\alpha}| \right)$, $\overline{I}_4$, we have
\begin{eqnarray}
\nonumber \overline{I}_4 & := & \int_0^t \int_{\Sigma^0_s} \left(\tau_+|\alpha|+\tau_+|\rho|+\tau_-|\underline{\alpha}| \right) \int_v \left| z^{q} P_{\xi^2}(\Phi) Y^{\beta^0} f  \right|  \frac{dv}{v^0} dx ds \\ \nonumber
& \lesssim & \int_0^t \sqrt{\mathcal{E}_N[F](s)} \left\| \sqrt{\tau_+} \int_v \left| z^{q} P_{\xi^2}(\Phi) Y^{\beta^0} f  \right|  \frac{dv}{v^0} \right\|_{L^2(\Sig^{0}_s)} ds \\ \nonumber
& \lesssim & \sqrt{\epsilon} \int_0^t \sqrt{\mathcal{E}_N[F](s)} \left\| \tau_+ \int_v \left| z^{i_1} Y^{\beta^0} f  \right|  \frac{dv}{(v^0)^2} \right\|^{\frac{1}{2}}_{L^{\infty}(\Sig^0_s)} \left\| \int_v \left| z^{i_2} P_{\xi^2}(\Phi)^2 Y^{\beta^0} f  \right|  dv \right\|^{\frac{1}{2}}_{L^1(\Sig^0_s)} ds .
\end{eqnarray}
Using the bootstrap assumptions \eqref{bootf3}, \eqref{bootF4} and the pointwise decay estimate on $\int_v \left| z^{i_1} Y^{\beta^0} f  \right|  \frac{dv}{(v^0)^2}$ given in Proposition \ref{Xdecay}, we finally obtain
$$\overline{I}_4 \hspace{2mm} \lesssim \hspace{2mm} \sqrt{\epsilon} \int_0^t (1+s)^{\frac{\eta}{2}} \frac{\sqrt{\epsilon} \log^{i_1}(3+s)}{1+s} \sqrt{\epsilon} (1+s)^{\frac{\eta}{2}} \log^{\frac{a}{2}i_2}(3+s) ds \hspace{2mm} \lesssim \hspace{2mm} \epsilon^{\frac{3}{2}} (1+t)^{\eta} \log^{a q}(3+t),$$ which concludes the improvement of the bootstrap assumption \eqref{bootf3}.
\begin{Rq}
In view of the computations made to estimate $\overline{I}_4$, note that.
\begin{itemize}
\item The use of Theorem \ref{decayopti}, instead of \eqref{decayf} combined with $1 \lesssim v^0 v^{\underline{L}}$ and Lemma \ref{weights1}, was necessary. Indeed, for the case $q=0$, a decay rate of $\log^2 (3+t) \tau_+^{-3}$ on $\int_v \left| Y^{\beta^0} f  \right|  \frac{dv}{(v^0)^2}$ would prevent us from closing energy estimates on $\mathcal{E}_N[F]$ and $\overline{\E}_N[f]$. 
\item Similarly, it was crucial to have a better bound on $\mathcal{E}^{Ext}_{N}[G](t)$ than $\epsilon (1+t)^{\eta}$ as the decay rate given by Proposition \ref{Xdecay} on $\int_v \left| Y^{\beta^0} f  \right|  \frac{dv}{(v^0)^2}$ is weaker, in the $t+r$ direction, outside the light cone.
\end{itemize}
\end{Rq}
Note that Propositions \ref{M2}, \ref{M3}, \ref{MM3}, \ref{M21}, \ref{M32} and \ref{M41} also prove that
\begin{equation}\label{Auxenergy}
\mathbb{A}[f](t) := \sum_{i=1}^2 \sum_{\begin{subarray}{} |\xi^i|+|\beta| \leq N \\ |\xi^i| \leq N-2 \end{subarray}}\sum_{\begin{subarray}{} |\zeta^i|+|\beta| \leq N \\ |\zeta^i| \leq N-1 \end{subarray}} \E \left[ P_{\xi^1}(\Phi) P_{\xi^2}(\Phi) Y^{\beta} f \right] (t)+\E \left[ P_{\zeta^1}^X(\Phi) P_{\zeta^2}^X(\Phi) Y^{\beta} f \right] (t) \lesssim \epsilon (1+t)^{\frac{3}{4} \eta}.
\end{equation}
Indeed, to estimate this energy norm, we do not have to deal with the critical terms of this subsection (as $|\xi^i| \leq N-2$ and according to Proposition \ref{ComuPkpX}). 

\section{$L^2$ decay estimates for the velocity averages of the Vlasov field}\label{sec11}

In view of the commutation formula of Propositions \ref{ComuMaxN} and \ref{CommuFsimple}, we need to prove enough decay on quantities such as $\left\| \sqrt{\tau_-} \int_v  | Y^{\beta} f | dv \right\|_{L^2_x}$, for all $|\beta| \leq N$. Applying Proposition \ref{Xdecay}, we are already able to obtain such estimates if $|\beta| \leq N-3$ (see Proposition \ref{estiL2} below). The aim of this section is then to treat the case of the higher order derivatives. For this, we follow the strategy used in \cite{FJS} (Section $4.5.7$). Before exposing the proceding, let us rewrite the system. Let $I_1$, $I_2$ and $I^q_1$, for $N-5 \leq q \leq N$, be the sets defined as
\begin{flalign*}
& \hspace{0.7cm} I_1 := \left\{ \beta \hspace{2mm} \text{multi-index} \hspace{2mm}  / \hspace{2mm} N-5 \leq |\beta| \leq N  \right\} = \{ \beta^1_{1}, \beta^1_{2}, ..., \beta^1_{|I_1|} \}, \hspace{1.3cm} I^q_1 := \left\{ \beta \in I_1 \hspace{2mm}  / \hspace{2mm} |\beta| = q  \right\}, & \\
& \hspace{0.7cm} I_2 := \left\{ \beta \hspace{2mm} \text{multi-index} \hspace{2mm}  / \hspace{2mm} |\beta| \leq N-5  \right\} =  \{ \beta^2_{1}, \beta^2_{2}, ..., \beta^2_{|I_2|} \}, & 
\end{flalign*}
and $R^1$ and $R^2$ be two vector valued fields, of respective length $|I_1|$ and $|I_2|$, such that
$$ R^1_j=  Y^{\beta^1_{j}}f \hspace{2cm} \text{and} \hspace{2cm}  R^2_j= Y^{\beta^2_{j}}f. $$
We will sometimes abusively write $j \in I_i$ instead of $\beta^i_{j} \in I_i$ (and similarly for $j \in I^k_1$). The goal now is to prove $L^2$ estimates on $\int_v |R^1| dv$. Finally, we denote by $\Vv$ the module over the ring $C^0([0,T[ \times \R^3_x \times \R^3_v)$ engendered by $( \partial_{v^l})_{1 \leq l \leq 3}$. In the following lemma, we apply the commutation formula of Proposition \ref{ComuVlasov} in order to express $T_F(R^1)$ in terms of $R^1$ and $R^2$ and we use Lemma \ref{GammatoYLem} for transforming the vector fields $\Gamma^{\sigma} \in \mathbb{G}^{|\sigma|}$.
\begin{Lem}\label{bilanL2}
There exists two matrix functions $A :[0,T[ \times \R^3 \times \R^3_v \rightarrow \mathfrak M_{|I_1|}(\Vv)$ and $B :[0,T[ \times \R^3 \times \R^3_v \rightarrow  \mathfrak M_{|I_1|,|I_2|}(\Vv)$ such that $T_F(R^1)+AR^1=B R^2$. Furthermore, if $1 \leq i \leq |I_1|$, $A$ and $B$ are such that $T_F(R^1_i)$ is a linear combination, with good coefficients $c(v)$, of the following terms, where $r \in \{1,2 \}$ and $\beta^r_j \in I_r$.
\begin{itemize}
\item \begin{equation}\label{eq:com11}
 z^d P_{k,p}(\Phi) \frac{v^{\mu}}{v^0}\mathcal{L}_{Z^{\gamma}}(F)_{ \mu \nu} R^r_j, \tag{type 1}
 \end{equation}
where \hspace{2mm} $z \in \mathbf{k}_1$, \hspace{2mm} $d \in \{ 0,1 \}$, \hspace{2mm} $\max ( |\gamma|, |k|+|\beta^r_j| ) \leq |\beta^1_i|$, \hspace{2mm} $|k| \leq |\beta^1_i|- 1$, \hspace{2mm} $|k|+|\gamma|+|\beta^r_j| \leq |\beta^1_i|+1$ \hspace{2mm} and \hspace{2mm} $p+k_P+(\beta^r_j)_P+d \leq (\beta^1_i)_P$.
\item \begin{equation}\label{eq:com22}
P_{k,p}(\Phi) \mathcal{L}_{X Z^{\gamma_0}}(F) \Big( v, \nabla_v \left( c(v) P_{q,s}(\Phi) R^r_j \right)  \Big), \tag{type 2}
\end{equation}
where \hspace{2mm} $|k|+|q|+|\gamma_0|+|\beta^r_j| \leq |\beta^1_i|-1$, \hspace{2mm} $|q| \leq |\beta^1_i|-2$,\hspace{2mm} $p+s+k_P+q_P+(\beta^r_j)_P \leq (\beta^1_i)_P$ \hspace{2mm} and \hspace{2mm} $p \geq 1$. 
\item \begin{equation}\label{eq:com44}
 P_{k,p}(\Phi) \mathcal{L}_{ \partial Z^{\gamma_0}}(F) \Big( v, \nabla_v \left( c(v) P_{q,s}(\Phi) R^r_j \right) \Big), \tag{type 3}
 \end{equation}
where \hspace{2mm} $|k|+|q|+|\gamma_0|+|\beta^r_j| \leq |\beta^1_i|-1$, \hspace{2mm} $|q| \leq |\beta^1_i|-2$, \hspace{2mm} $p+s+|\gamma_0| \leq |\beta^1_i|-1$ \hspace{2mm} and \hspace{2mm} $p+s+k_P+q_P+(\beta^r_j)_P \leq (\beta^1_i)_P$.
\end{itemize}
We also impose that $|\beta^2_j| \leq N-6$ on the terms of \eqref{eq:com22}, \eqref{eq:com44} and that $|\beta^1_j| \geq N-4$ on the terms of \eqref{eq:com11}, which is possible since $ \beta \in I_1 \cap I_2$ if $|\beta| = N-5$.
\end{Lem}
\begin{Rq}\label{rqcondiH}
Note that if $\beta^1_i \in I^{N-5}_1$, then $A^q_i =0$ for all $q \in \llbracket 1, |I_1| \rrbracket$. If $1 \leq n \leq 5$ and $\beta^1_i \in I^{N-5+n}_1$, then the terms composing $A_i^q$ are such that $\max(|k|+1,|\gamma|) \leq n$ or $|k|+|q|+|\gamma_0| \leq n-1$.
\end{Rq}
Let us now write $R=H+G$, where $H$ and $G$ are the solutions to
$$\left\{
    \begin{array}{ll}
         T_F(H)+AH=0 \hspace{2mm}, \hspace{2mm} H(0,.,.)=R(0,.,.),\\
        T_F(G)+AG=BR^2 \hspace{2mm}, \hspace{2mm} G(0,.,.)=0.
    \end{array}
\right.$$
The goal now is to prove $L^2$ estimates on the velocity averages of $H$ and $G$. As the derivatives of $F$ and $\Phi$ composing the matrix $A$ are of low order, we will be able to commute the transport equation satisfied by $H$ and to bound the $L^1$ norm of its derivatives of order $3$ by estimating pointwise the electromagnetic field and the $\Phi$ coefficients, as we proceeded in Subsection \ref{sec82}. The required $L^2$ estimates will then follow from Klainerman-Sobolev inequalities. Even if we will be lead to modify the form of the equation defining $G$, the idea is to find a matrix $K$ satisfying $G=KR^2$, such that $\E[KKR^2]$ do not grow too fast, and then to take advantage of the pointwise decay estimates on $\int_v |R^2|dv$ in order to obtain the expected decay rate on $\| \int_v |G| dv \|_{L^2_x}$.
\begin{Rq}
As in \cite{massless}, we keep the $v$ derivatives in the construction of $H$ and $G$. It has the advantage of allowing us to use Lemma \ref{vradial}. If we had already transformed the $v$ derivatives, as in \cite{dim4}, we would have obtained terms such as $x^{\theta} \partial g$ from $\left( \nabla_v g \right)^r$. Indeed, Lemma \ref{vradial} would have led us to derive coefficients such as $\frac{x^k}{|x|}$ and then to deal, for instance, with factors such as $\frac{t^3}{|x|^3}$ (apply three boost to $\frac{x^k}{|x|}$). We would then have to work with an another commutation formula leading to terms such as $x^{\theta} \frac{v^{\mu}}{v^0}\partial(F)_{\mu \nu} H_j$ and would then need at least a decay rate of $\tau_+^{-\frac{3}{2}}$ on $\rho$, in the $t+r$ direction, in order to close the energy estimates on $H$. This could be obtained by assuming more decay on $F$ initially in order to use the Morawetz vector field $ \overline{K}_0$ or $\tau_-^{-b} \overline{K}_0$ as a multiplier. 

However, this creates two technical difficulties compared to what we did in \cite{dim4}. The first one concerns $H$ and will lead us to consider a new hierarchy (see Subsection \ref{subsecH}). The other one concerns $G$ and we will circumvent it by modifying the source term of the transport equation defining it (see Subsecton \ref{subsecG}).
\end{Rq}
\begin{Rq}
In Subsection \ref{subsecG}, we will consider a matrix $D$ such that $T_F(R^2)=DR^2$ and we will need to estimate pointwise and independently of $M$, in order to improve the bootstrap assumption on $\mathcal{E}_{N-1}[F]$, the derivatives of the electromagnetic field of its components. It explains, in view of Remark \ref{lowderiv}, why we take $I_2$ such as $|\beta^2_j| \leq N-5$.
\end{Rq}
\subsection{The homogeneous part}\label{subsecH}
 The purpose of this subsection is to bound $L^1$ norms of components of $H$ and their derivatives. We will then be able to obtain the desired $L^2$ estimates through Klainerman-Sobolev inequalities. For that, we will make use of the hierarchy between the components of $H$ given by $(\beta^1_i)_P$. However, as, for $N-4 \leq q \leq N$ and $\beta_i^1 \in I^q_1$, we need information on $\| \widehat{Z}^{\kappa} H_j \|_{L^1_{x,v}}$, with $\beta^1_j \in I^{q-1}_1$ and $|\kappa|=4$, in order to close the energy estimate on $\widehat{Z}^{\xi} H_i$, with $|\xi|=3$, we will add a new hierarchy in our energy norms. This leads us to define, for $\delta \in \{ 0,1 \}$,
$$ \E_H^{\delta}(t) := \sum_{ z \in \V} \sum_{q=0}^5 \sum_{|\beta| \leq 3+q} \sum_{i \in I^{N-q}_1} \sum_{j=0}^{2N+2+\delta-\beta_P-\beta^1_P} \log^{-j(\delta a+2)}(3+t) \E \left[ z^j \widehat{Z}^{\beta} H_i \right](t).$$
\begin{Lem}\label{Comuhom2}
Let $\widetilde{N} \geq N+3$, $0 \leq q \leq 5$, $i \in I^{N-q}_1$, $|\beta| \leq 3+q$, $ z \in \V$ and $j \leq \widetilde{N}-\beta_P-(\beta^1_i)_P$. Then, $T_F( z^j \widehat{Z}^{\beta} H_i)$ can be bounded by a linear combination of the following terms, where $$p \leq 3N, \hspace{5mm} \max(|k|+1,|\gamma|) \leq 8, \hspace{5mm} |\kappa| \leq |\beta|+1, \hspace{5mm} |\beta^1_l| \leq |\beta^1_i| \hspace{5mm} \text{and} \hspace{5mm} |\kappa|+|\beta^1_l| \leq |\beta^1_i|.$$ 
\begin{itemize}
\item \begin{equation}\label{eq:cat0H} 
\left|  F \left(v, \nabla_v \left( z^j \right) \right) Y^{\beta} H_i  \right|. \tag{category $0-H$}
\end{equation}
\item \begin{equation}\label{eq:cat1H} 
\left| P_{k,p}(\Phi) \right| \left| w^{r}  Y^{\kappa} H_l  \right| \left( \left| \nabla_{Z^{\gamma}} F \right| +\frac{\tau_+}{\tau_-} \left| \alpha \left( \mathcal{L}_{Z^{\gamma}}(F) \right) \right|+\frac{\tau_+}{\tau_-} \sqrt{\frac{v^{\underline{L}}}{v^0}} \left| \sigma \left( \mathcal{L}_{Z^{\gamma}}(F) \right) \right|  \right), \tag{category $1-H$}
\end{equation}
where \hspace{2mm} $w \in \mathbf{k}_1$ \hspace{2mm} and \hspace{2mm} $r \leq \widetilde{N} -k_P-\kappa_P-(\beta^1_l)_P$. 
\item \begin{equation}\label{eq:cat3H}
\hspace{-10mm} \frac{\tau_+}{\tau_-} |\rho \left( \mathcal{L}_{ Z^{\gamma}}(F) \right) | \left|z^{j-1}  Y^{\kappa} H_l \right| \hspace{8mm} \text{and} \hspace{8mm} \frac{\tau_+}{\tau_-}  \sqrt{\frac{v^{\underline{L}}}{v^0}}\left| \underline{\alpha} \left( \mathcal{L}_{ Z^{\gamma}}(F)  \right) \right| \left| z^r  Y^{\kappa} H_l \right|,  \tag{category $2-H$}
\end{equation}
where \hspace{2mm} $j-1$, $r=\widetilde{N}-\kappa_P-(\beta^1_l)_P$ and $r \leq j$.
\end{itemize}
The terms of \eqref{eq:cat3H} can only appear if $j=\widetilde{N}-\beta_P-(\beta^1_i)_P$.
\end{Lem}
\begin{proof}
We merely sketch the proof as it is very similar to previous computations. One can express $T_F(\widehat{Z}^{\beta} H_i)$ using Lemma \ref{bilanL2} and following what we did in the proof of Proposition \ref{ComuVlasov}. It then remains to copy the proof of Proposition \ref{ComuPkp} with $|\zeta_0|=0$, which explains that we do not have terms of \eqref{eq:cat4}. Note that $\max(|k|+1,|\gamma|) \leq 8$ comes from Remark \ref{rqcondiH} and the fact that $|\kappa|$ can be equal to $|\beta|+1$ ensues from the transformation of the $v$ derivative in the terms obtained from those of \eqref{eq:com22} and \eqref{eq:com44}.
\end{proof}
\begin{Rq}\label{rqrqrq}
As $|\gamma| \leq 8 \leq N-3$, we have at our disposal pointwise decay estimates on the electromagnetic field (see Proposition \eqref{decayF}). Similarly, as $|k| \leq 7 \leq N-4$, Remark \ref{estiPkp} gives us $|P_{k,p}(\Phi)| \lesssim \log^{M_2}(1+\tau_+)$.
\end{Rq}
We are now ready to bound $\E^{\delta}_H$ and then to obtain estimates on $\int_v |z^j H_i|dv$.
\begin{Pro}\label{estidecayH}
We have $\E^1_H +\E^0_H \lesssim \epsilon$ on $[0,T[$. Moreover, for $0 \leq q \leq 5$ and $|\beta| \leq q$,
$$\forall \hspace{0.5mm} (t,x) \in [0,T[ \times \R^3, \hspace{3mm} z \in \V, \hspace{3mm} i \in I^{N-q}_1,  \hspace{3mm} j \leq 2N-1-\beta_P-(\beta^1_i)_P, \hspace{5mm} \int_v |z^j Y^{\beta} H_i|  dv \lesssim \epsilon \frac{\log^{2j+M_1}(3+t)}{\tau_+^2 \tau_-}.$$
\end{Pro}
\begin{proof}
In the same spirit as Corollary \ref{coroinit} and in view of commutation formula of Lemma \ref{Comuhom2} (applied with $\widetilde{N} = 2N+3$) as well as the assumptions on $f_0$, there exists $C_H >0$ such that $\E^0_H(0) \leq \E^1_H(0) \leq C_H \epsilon$. We can prove that they both stay bounded by $3 C_H \epsilon$ by the continuity method. As it is very similar to what we did previously, we only sketch the proof. Consider $\delta \in \{ 0 , 1 \}$, $0 \leq r \leq 5$, $i \in I^{N-r}_1$, $|\beta| \leq 3+r$, $z \in \V$ and $j \leq 2N+2+\delta-\beta_P-(\beta^1_i)_P$. The goal is to prove that
$$ \int_0^t \int_{\Sigma_s} \int_v \left| T_F( z^j H_i ) \right| \frac{dv}{v^0}dxds \lesssim \epsilon^{\frac{3}{2}} \log^{j(\delta a+2)}(3+t).$$
According to Lemma \ref{Comuhom2} (still applied with $\widetilde{N}=2N+3$), it is sufficient to obtain, if $\delta=1$, that the integral over $[0,t] \times \R^3_x \times \R^3_v$ of all terms of \eqref{eq:cat0H}-\eqref{eq:cat3H} are bounded by $\epsilon^{\frac{3}{2}} \log^{j( a+2)}(3+t)$. If $\delta=0$, we only have to deal with terms of \eqref{eq:cat0H} and \eqref{eq:cat1H} and to estimate their integrals by $\epsilon^{\frac{3}{2}} \log^{2j}(3+t)$. In view of Remark \ref{rqrqrq}, we only have to apply (or rather follow the computations of) Propositions \ref{M1}, \ref{M2} and \ref{M3}. The pointwise decay estimates then ensue from the Klainerman-Sobolev inequality of Corollary \ref{KS3}. 
\end{proof}
\begin{Rq}
A better decay rate, $\log^{2j}(3+t) \tau_+^{-2} \tau_-^{-1}$, could be proved in the previous proposition by controling a norm analogous to $\E_N^{X}[f]$ but we do not need it to close the energy estimates on $F$.
\end{Rq}
\begin{Rq}\label{rqgainH}
We could avoid any hypothesis on the derivatives of order $N+1$ and $N+2$ of $F^0$ (see Subsection $17.2$ of \cite{FJS3}).
\end{Rq}
\subsection{The inhomogeneous part}\label{subsecG}
As the matrix $B$ in $T_F(G)+AG=BR^2$ contains top order derivatives of the electromagnetic field, we cannot commute the equation and prove $L^1$ estimates on $\widehat{Z} G$. Let us explain schematically how we will obtain an $L^2$ estimate on $\int_v |G|dv$ by recalling how we proceeded in \cite{dim4}. We did not work with modified vector field and the matrices $A$ and $B$ did not hide $v$ derivatives of $G$. Then we introduced $K$ the solution of $T_F(K)+AK+KD=B$ which initially vanishes and where $T_F(R^2)=DR^2$. Thus $G=KR^2$ and we proved $\E[|K|^2 |R^2|] \leq \epsilon$ so that the expected $L^2$ decay estimate followed from
$$\left\| \int_v |G|dv \right\|_{L^2_x} \lesssim \left\| \int_v |R^2| dv \right\|^{\frac{1}{2}}_{L^{\infty}_x} \E[|K|^2 |R^2|]^{\frac{1}{2}}.$$
The goal now is to adapt this process to our situation. There are two obstacles.
\begin{itemize}
\item The $v$ derivatives hidden in the matrix $A$ will then be problematic and we need first to transform them. 
\item The components of the (transformed) matrix $A$ have to decay sufficiently fast. We then need to consider a larger vector valued field than $G$ by including components such as $z^j G_i$ in order to take advantage of the hierarchies in the source terms already used before. 
\end{itemize}
Recall from Definition \ref{orderk1} that we considered an ordering on $\mathbf{k}_1$ and that, if $\kappa$ is a multi-index, we have
$$z^{\kappa}= \prod_{i=1}^{|\kappa|} z_{\kappa_i} \hspace{3mm} \text{and} \hspace{3mm} |z^{\kappa}| \leq \sum_{w \in \mathbf{k}_1} |w|^{|\kappa|}.$$ In this section, we will sometimes have to work with quantities such as $z^{\kappa}$ rather than with $z^j$, where $j \in \mathbb{N}$.
\begin{Def}
Let $I$ and $I^q$, for $N-5 \leq q \leq N$, be the sets
$$ I  :=  \{ (\kappa, \beta ) \hspace{1mm} / \hspace{1mm} N-5 \leq |\beta| \leq N \hspace{2mm} \text{and} \hspace{2mm} |\kappa| \leq N-\beta_P \} = \{ (\kappa_1,\beta_1),...,(\kappa_{|I|}, \beta_{|I|}) \}, \hspace{5mm} I^q := \{ (\kappa, \beta ) \in I \hspace{1mm} / \hspace{1mm} |\beta|=q \}.$$
Define now $L$, the vector valued fields of length $|I|$, such that
$$L_i = z^{\kappa_i} G_j, \hspace{8mm} \text{with} \hspace{8mm} \beta_j^1 = \beta_i, \hspace{8mm} \text{and} \hspace{8mm} [i]_I:=|\kappa_i|.$$
Moreover, for $Y \in \Y$, $1 \leq j \leq |I_1|$ and $1 \leq i \leq |I|$, we define $j_Y$ and $i_Y$ the indices such that
$$R^1_{j_{Y}}=Y Y^{\beta^1_j} f \hspace{8mm} \text{and} \hspace{8mm} L_{i_Y} = z^{\kappa_{i_Y}} G_{j_Y}.$$
\end{Def}
The following result will be useful for transforming the $v$ derivatives.
\begin{Lem}\label{deriG}
Let $Y \in \Y$ and $ \beta^1_i \in I_1 \setminus I^{N}_1$. Then
$$Y G_i = G_{i_Y}+H_{i_Y}-Y H_i.$$
\end{Lem}
\begin{proof}
Recall that $R=H+G$ and remark that $Y R^1_i = Y Y^{\beta^1_i} f= R^1_{i_Y}$.
\end{proof}
We now describe the source terms of the equations satisfied by the components of $L$.
\begin{Pro}\label{bilanG}
There exists $N_1 \in \mathbb{N}^*$, a vector valued field $W$ and three matrix-valued functions $\overline{A} : [0,T[ \times \R^3 \times \R^3 \rightarrow \mathfrak M_{|I|}(\R)$, $\overline{B} : [0,T[ \times \R^3 \times \R^3 \rightarrow \mathfrak M_{|I|,N_1}(\R)$, $\overline{D} : [0,T[ \times \R^3 \times \R^3 \rightarrow \mathfrak M_{N_1}(\R)$ such that
$$T_F(L)+\overline{A}L= \overline{B} W, \hspace{8mm} T_F(W)= \overline{D} W \hspace{8mm} \text{and} \hspace{8mm} \sum_{z \in \V} \int_v  |z^{2} W| dv \lesssim \epsilon \frac{\log^{3N+M_1} (3+t)}{\tau_+^2 \tau_-}.$$
In order to depict these matrices, we use the quantity $[q]_W$, for $1 \leq q \leq N_1$, which will be defined during the construction of $W$ in the proof. $\overline{A}$ and $\overline{B}$ are such that $T_F(L_i)$ can be bounded, for $1 \leq i \leq |I|$, by a linear combination of the following terms, where \hspace{1mm} $|\gamma| \leq 5$, \hspace{1mm} $1 \leq j,q \leq |I|$ \hspace{1mm} and \hspace{1mm} $1 \leq r \leq N_1$.
 \begin{equation}\label{eq:cat0G} 
\big( \tau_- \left( |\rho (F) |+|\sigma(F) |+|\underline{\alpha}(F)| \right)+\tau_+ |\alpha(F)| \big) \left| L_j  \right|, \hspace{5mm} \text{with} \hspace{5mm} [j]_I = [i]_I-1. \tag{category $0-\overline{A}$}
\end{equation}
 \begin{equation}\label{eq:cat1G} 
\log^{M_1}(3+t) \left|   L_j  \right| \left( \left| \nabla_{Z^{\gamma}} F \right|+ \frac{\tau_+}{\tau_-} \left| \alpha \left( \mathcal{L}_{Z^{\gamma}}(F) \right) \right|+ \frac{\tau_+}{\tau_-} \sqrt{\frac{v^{\underline{L}}}{v^0}} \left| \sigma \left( \mathcal{L}_{Z^{\gamma}}(F) \right) \right| \right) .\tag{category $1-\overline{A}$}
\end{equation}
 \begin{equation}\label{eq:cat3G}
 \frac{\tau_+}{\tau_-} |\rho \left( \mathcal{L}_{ Z^{\gamma}}(F) \right) | \left| L_j \right| + \frac{\tau_+}{\tau_-}  \sqrt{\frac{v^{\underline{L}}}{v^0}}\left| \underline{\alpha} \left( \mathcal{L}_{ Z^{\gamma}}(F)  \right) \right| \left| L_q \right|,  \hspace{5mm} \text{with} \hspace{5mm} [j]_I+1, \hspace{1mm} [q]_I \leq [i]_I. \tag{category $2-\overline{A}$}
\end{equation}
 \begin{equation}\label{eq:cat1Y} 
\left| P_{k,p}(\Phi) \right| \left|   W_r  \right| \left( \left| \nabla_{Z^{\zeta}} F \right|+ \frac{\tau_+}{\tau_-} \left| \alpha \left( \mathcal{L}_{Z^{\zeta}}(F) \right) \right|+ \frac{\tau_+}{\tau_-} \sqrt{\frac{v^{\underline{L}}}{v^0}} \left| \sigma \left( \mathcal{L}_{Z^{\zeta}}(F) \right) \right| \right), \tag{category $1-\overline{B}$}
\end{equation}
where \hspace{1mm} $p \leq 2N$, \hspace{1mm} $|k| \leq N-1$ \hspace{1mm} and \hspace{1mm} $|k|+|\zeta| \leq N$. Moreover, if $|k| \geq 6$, there exists $\kappa$ and $\beta$ such that \hspace{1mm} $W_r = z^{\kappa} Y^{\beta} f$, \hspace{1mm} $|k|+|\beta| \leq N$ \hspace{1mm} and \hspace{1mm} $|\kappa| \leq N+1-k_P-\beta_P$.

The matrix $\overline{D}$ is such that, for $1 \leq i \leq N_1$, $T_F(W_i)$ is bounded by a linear combination of the following expressions, where \hspace{1mm} $|\gamma| \leq N-5$ \hspace{1mm} and \hspace{1mm} $1 \leq j,q \leq N_1$.
 \begin{equation}\label{eq:cat0W} 
\left( \tau_- \left( |\rho (F) |+|\sigma(F) |+|\underline{\alpha}(F)| \right)+\tau_+ |\alpha(F)| \right) \left| W_j  \right|, \hspace{5mm} \text{with} \hspace{5mm} [j]_W = [i]_W-1. \tag{category $0-\overline{D}$}
\end{equation}
\begin{equation}\label{eq:cat1W} 
\log^{M_1}(3+t) \left|   W_j  \right| \left( \left| \nabla_{Z^{\gamma}} F \right|+\frac{\tau_+}{\tau_-} \left| \alpha \left( \mathcal{L}_{Z^{\gamma}}(F) \right) \right|+ \frac{\tau_+}{\tau_-} \sqrt{\frac{v^{\underline{L}}}{v^0}} \left| \sigma \left( \mathcal{L}_{Z^{\gamma}}(F) \right) \right| \right). \tag{category $1-\overline{D}$}
\end{equation}
\begin{equation}\label{eq:cat3W}
 \frac{\tau_+}{\tau_-} |\rho \left( \mathcal{L}_{ Z^{\gamma}}(F) \right) | \left| W_j \right| + \frac{\tau_+}{\tau_-}  \sqrt{\frac{v^{\underline{L}}}{v^0}}\left| \underline{\alpha} \left( \mathcal{L}_{ Z^{\gamma}}(F)  \right) \right| \left| W_q \right|,  \hspace{5mm} \text{with} \hspace{5mm} [j]_W+1, \hspace{1mm} [q]_W \leq [i]_W. \tag{category $2-\overline{D}$}
\end{equation}
\end{Pro}
\begin{proof}
The main idea is to transform the $v$ derivatives in $AG$, following the proof of Lemma \ref{nullG}, and then to apply Lemma \ref{deriG} in order to eliminate all derivatives of $G$ in the source term of the equations. We then define $W$ as the vector valued field, and $N_1$ as its length, containing all the following quantities
\begin{itemize}
\item $z^j Y^{\beta} f $, \hspace{1mm} with \hspace{1mm} $z \in \V$, \hspace{1mm} $|\beta| \leq N-5$ \hspace{1mm} and \hspace{1mm} $j \leq N+1- \beta_P$,
\item $z^j\left(H_{i_Y}-Y H_i \right)$, \hspace{1mm} with \hspace{1mm} $z \in \V$, \hspace{1mm} $Y \in \Y$, \hspace{1mm} $\beta_i^1 \in I_1 \setminus I_1^N$ \hspace{1mm} and \hspace{1mm} $j \leq N+3-\left(\beta^1_{i_Y} \right)_P$.
\item $z^j Y^{\beta} H_i$, \hspace{1mm} with \hspace{1mm} $z \in \V$, \hspace{1mm} $|\beta|+|\beta^1_i| \leq N$ \hspace{1mm} and  \hspace{1mm} $ j \leq N+3-\beta_P-(\beta^1_i)_P$.
\end{itemize}
Let us make three remarks. 
\begin{itemize}
\item If $1 \leq i \leq N_0$, we can define, in each of the three cases, $[i]_W:=j$.
\item Including the terms $z^{N+1- \beta_P} Y^{\beta} f$ and $z^{N+1-\left(\beta^1_{i_Y} \right)_P}\left(H_{i_Y}-Y H_i \right)$ in $W$ allows us to avoid any term of category $2$ related to $\overline{B}$.
\item The components such as $z^j Y^{\beta} H_i$ are here in order to obtain an equation of the form $T_F(W)= \overline{D} W$.
\end{itemize} 
The form of the matrix $\overline{D}$ then follows from Proposition \ref{ComuPkp} if $Y_i = z^j Y^{\beta} f $ and from Lemma \ref{Comuhom2}, applied with $\widetilde{N}=N+3$, otherwise (we made an additional operation on the terms of category $0$ which will be more detailed for the matrix $\overline{A}$). Note that we use Remark \ref{estiPkp} to estimate all quantities such as $P_{k,p}(\Phi)$. The decay rate on $\int_v |z^2 W| dv$ follows from Proposition \ref{Xdecay} and \ref{estidecayH}.

We now turn on the construction of the matrices $\overline{A}$ and $\overline{B}$. Consider then $1 \leq i \leq |I|$ and $1 \leq q \leq |I_1|$ so that $L_i=z^{\kappa_i} G_q$ and $|\kappa_i| \leq N-(\beta^1_q)_P$. Observe that
$$T_F(L_i)= T_F(z^{\kappa_i})G_q+z^{\kappa_i} T_F(G_q)= F \left( v, \nabla_v (z^{\kappa_i}) \right)G_q+z^{\kappa_i} T_F(G_q).$$
The first term on the right hand side gives terms of \eqref{eq:cat0G} and \eqref{eq:cat1G} as, following the computations of Proposition \ref{M1}, we have
$$\nabla_v \left( \prod_{r=1}^{|\kappa_i|} z_r \right) = \sum_{p=1}^{|\kappa_i|} \nabla_v(z_p) \prod_{r \neq p} z_r, \hspace{5mm} \left| F(v,\nabla_v z_p) \right| \lesssim  \tau_- \left( |\rho (F) |+|\sigma(F) |+|\underline{\alpha}(F)| \right)+\tau_+ |\alpha(F)|+\sum_{w \in \V} |w F| .$$
The remaining quantity, $z^{\kappa_i} T_F(G_q)=-z^{\kappa_i} A_q^r G_r+z^{\kappa_i} B_q^r R^2_r$, is described in Lemma \ref{bilanL2}. Express the terms given by $z^{\kappa_i} A_q^r G_r$ in null components and transform the $v$ derivatives\footnote{Note that this is possible since $\partial_v G_r$ can only appear if $\beta^1_r \in I_1 \setminus I^N_1$.} of $G_r$ using Lemma \ref{deriG}, so that, schematically (see \eqref{equ:proof}),
\begin{eqnarray}
\nonumber  v^0\left( \nabla_v G_r \right)^r & = & Y G_r+(t-r) \partial G_r = G_{r_Y}+H_{r_Y}-Y H_r+(t-r)(G_{r_{\partial}}+H_{r_{\partial}}-\partial H_r) \hspace{5mm} \text{and} \hspace{5mm} \\
\nonumber v^0 \partial_{v^b} G_r & = & Y_{0b} G_r+x \partial G_r = G_{r_{Y_{0b}}}+H_{r_{Y_{0b}}}-Y_{0b} H_r+x(G_{r_{\partial}}+H_{r_{\partial}}-\partial H_r) .
\end{eqnarray}
 By Remark \ref{rqcondiH}, the $\Phi$ coefficients and the electromagnetic field are both derived less than $5$ times. We then obtain, with similar operations as those made in proof of Proposition \ref{ComuPkp}, the matrix $\overline{A}$ and the columns of the matrix $\overline{B}$ hitting the component of $W$ of the form $z^j\left(H_{l_Y}-Y H_l \right)$. For $z^{\kappa_i} B_q^r R^2_r$, we refer to the proof of Proposition \ref{ComuPkp}, where we already treated such terms.
\end{proof}
To lighten the notations and since there will be no ambiguity, we drop the index $I$ (respectively $W$) of $[i]_I$ for $1 \leq i \leq |I|$ (respectively $[j]_W$ for $1 \leq j \leq N_1$). Let us introduce $K$ the solution of $T_F(K)+\overline{A}K+K\overline{D}=\overline{B}$, such as $K(0,.,.)=0$. Then, $KY= L$ since they are solution of the same system and they both initially vanish. The goal now is to control $\E[|K|^2|Y|]$. As, for $ 1 \leq i \leq |I|$ and $1 \leq j,p \leq N_1$,
\begin{equation}\label{eq:K}
T_F\left( |K^j_i|^2 W_p\right) = |K^j_i |^2\overline{D}^q_pW_q-2\left(\overline{A}^q_i K^j_q +K^q_i \overline{D}^j_q \right) K^j_i W_p+2\overline{B}^j_iK^j_iW_p,
\end{equation}
 we consider $\E_L$, the following hierarchized energy norm,
$$\E_L(t):= \sum_{ \begin{subarray}{} 1 \leq j,p \leq N_1 \\ \hspace{1mm} 1 \leq i \leq |I| \end{subarray}} \log^{-4[i]-2[p]+4[j]}(3+t) \mathbb{E} \left[\left|K_i^j\right|^2 W_p \right](t). $$
The sign in front of $[j]$ is related to the fact that the hierarchy is inversed on the terms coming from $K \overline{D}$. It prevents us to expect a better estimate than $\E_L(t) \lesssim \log^{4N+12}(3+t)$.
\begin{Lem}\label{Inho1}
We have, for $M_0 = 4N+12$ and if $\epsilon$ small enough, $\E_L(t) \lesssim \epsilon \log^{M_0}(3+t)$ for all $t \in [0,T[$.
\end{Lem}
\begin{proof}
We use again the continuity method. Let $T_0 \in [0,T[$ be the largest time such that $\E_L(t) \leq 2 \epsilon \log^{M_0}(3+t)$ for all $t \in [0,T_0[$ and let us prove that, if $\epsilon$ is small enough,
\begin{equation}\label{thegoal}
\forall \hspace{0.5mm} t \in [0,T_0[, \hspace{3mm} \E_L(t) \lesssim  \epsilon^{\frac{3}{2}} \log^{M_0}(3+t).
\end{equation}
As $T_0 >0$ by continuity ($K$ vanishes initially), we would deduce that $T_0=T$. We fix for the remainder of the proof $1 \leq i \leq |I|$ and $1 \leq j, p \leq N_1$. According to the energy estimate of Proposition \ref{energyf}, \eqref{thegoal} would follow if we prove that
\begin{eqnarray}
\nonumber I_{\overline{A}, \overline{D}} & := & \int_0^t \int_{\Sigma_s} \int_v \left| |K^j_i |^2 \overline{D}^q_pW_q-2\left( \overline{A}^k_i K^j_k +K^r_i \overline{D}^j_r \right) K^j_i W_p \right| \frac{dv}{v^0} dx ds \lesssim \epsilon^{\frac{3}{2}} \log^{M_0+4[i]+2[p]-4[j]}(3+t), \\ \nonumber
I_{\overline{B}} & := & \int_0^t \int_{\Sigma_s} \int_v \left| B_i^j \right| \left| K^j_i  W_p \right|  \frac{dv}{v^0} dx ds \lesssim \epsilon^{\frac{3}{2}}.
\end{eqnarray}
Let us start by $I_{\overline{A}, \overline{D}}$ and note that in all the terms given by Proposition \ref{bilanG}, the electromagnetic field is derived less than $N-5$ times so that we can use the pointwise decay estimates given by Remark \ref{lowderiv}. The terms of \eqref{eq:cat1G} and \eqref{eq:cat1W} can be easily handled (as in Proposition \ref{M2}). We then only treat the following cases, where $|\gamma| \leq N-5$ (the other terms are similar).
$$  \left| \overline{D}^j_r \right| = \tau_- \left( |\rho (F) |+|\sigma(F) |+|\underline{\alpha}(F)| \right)+\tau_+ |\alpha(F)|, \hspace{5mm} \text{with} \hspace{5mm} [j]=[r]-1,$$
$$ \left| \overline{A}^k_i \right| \lesssim \frac{\tau_+ \sqrt{v^{\underline{L}}}}{\tau_- \sqrt{v^0}} |\underline{\alpha} (\mathcal{L}_{Z^{\gamma}}(F))|, \hspace{3mm} \text{with} \hspace{3mm} [k] \leq [i], \hspace{5mm} \text{and} \hspace{5mm} \left| \overline{D}^q_p \right| \lesssim \frac{\tau_+}{\tau_-} |\rho (\mathcal{L}_{Z^{\gamma}}(F))|, \hspace{3mm} \text{with} \hspace{3mm} [q] < [p]. $$
Without any summation on the indices $r$, $k$ and $q$, we have, using Remark \ref{lowderiv}, $1 \lesssim \sqrt{v^0 v^{\underline{L}}}$ and the Cauchy-Schwarz inequality several times,
\begin{eqnarray}
\nonumber \int_0^t \int_{\Sigma_s} \int_v  \left| K^r_i \overline{D}^j_r K^j_i W_p \right| \frac{dv}{v^0} dx ds & \lesssim & \sqrt{\epsilon} \int_0^t \frac{\log (3+s)}{1+s}  \left| \E \left[ \left| K^r_i \right|^2 W_p \right] \hspace{-0.5mm} (s) \hspace{0.5mm} \E \left[ \left| K^j_i \right|^2 W_p \right] \hspace{-0.5mm} (s) \right|^{\frac{1}{2}}  ds \\ \nonumber
& \lesssim & \epsilon^{\frac{3}{2}} \log^{2+M_0+4[i]+2[p]-2[r]-2[j]}(3+t) \hspace{2mm} \lesssim \hspace{2mm} \epsilon^{\frac{3}{2}} \log^{M_0+4[i]+2[p]-4[j]}(3+t), \\ \nonumber 
\int_0^t \int_{\Sigma_s} \int_v  \left|\overline{A}^k_i K^j_k K^j_i W_p\right| \frac{dv}{v^0} dx ds & \lesssim & \sqrt{\epsilon}\int_{u=-\infty}^t \frac{\tau_+}{\tau_+\tau_-^{\frac{3}{2}}} \int_{C_u(t)} \int_v \frac{v^{\underline{L}}}{v^0} \left| K^j_k \right| \left| W_p\right|^{\frac{1}{2}} \left| K^j_i \right| \left| W_p\right|^{\frac{1}{2}} dv dC_u(t) du \\ \nonumber
& \lesssim & \sqrt{\epsilon} \left| \E \left[ \left| K^j_k \right|^2 W_p \right] \hspace{-0.5mm} (t) \hspace{0.5mm} \E \left[ \left| K^j_i \right|^2 W_p \right]  \hspace{-0.5mm} (t) \right|^{\frac{1}{2}}  \int_{u=-\infty}^{+\infty} \frac{du}{\tau_-^{\frac{3}{2}}} \\ \nonumber
& \lesssim & \epsilon^{\frac{3}{2}} \log^{M_0+2[k]+2[i]+2[p]-4[j]}(3+t) \hspace{2mm} \lesssim \hspace{2mm} \epsilon^{\frac{3}{2}} \log^{M_0+4[i]+2[p]-4[j]}(3+t), \\ \nonumber 
\nonumber \int_0^t \int_{\Sigma_s} \int_v  \left| K^j_i \right|^2 \left| \overline{D}^q_pW_q \right| \frac{dv}{v^0} dx ds & \lesssim & \sqrt{\epsilon} \int_0^t \int_{\Sigma_s} \int_v \log (3+s) \frac{\sqrt{v^{\underline{L}}v^0}}{\tau_+^{\frac{1}{2}} \tau_-^{\frac{3}{2}}}   \left| K^j_i \right|^2 \left| W_q \right| \frac{dv}{v^0} dx  ds \\ \nonumber
& \lesssim & \sqrt{\epsilon}  \left( \int_0^t \frac{\log(3+s)}{1+s} ds+\log(3+t) \int_{-\infty}^{+\infty} \frac{du}{\tau_-^3} \right) \sup_{[0,t]} \E \left[ \left| K^j_i \right|^2 W_q \right] \\ \nonumber
& \lesssim & \epsilon^{\frac{3}{2}} \log^{2+M_0+4[i]+2[q]-4[j]}(3+t) \hspace{2mm} \lesssim \hspace{2mm} \epsilon^{\frac{3}{2}} \log^{M_0+4[i]+2[p]-4[j]}(3+t).
\end{eqnarray}
It remains to study $I_{\overline{B}}$. The form of $\overline{B}^j_i$ is given by Propoposition \ref{bilanG} and the computations are close to the ones of Proposition \ref{M21}. We then only consider the following two cases,
$$ \left| \overline{B}^j_i K_i^j W_p \right| \lesssim  \log^{M_1}(1+\tau_+)\frac{\tau_+ \sqrt{v^{\underline{L}}}}{\tau_-\sqrt{v^0}} \left| \sigma (\mathcal{L}_{Z^{\zeta}} (F) ) \right| \left|K_i^j \right| |W_p|, \hspace{5mm} \text{with}  \hspace{5mm} |\zeta| \leq N \hspace{5mm} \text{and} \hspace{5mm} $$
$$  \left| \overline{B}^j_i K_i^j W_p \right| \lesssim \left|\Phi^r P_{\xi}(\Phi)|| \nabla_{ Z^{\gamma}} F  \right| \left| K_i^j W_p \right|, \hspace{5mm} \text{with} \hspace{5mm} r \leq 2N, \hspace{6mm} |\xi|+|\gamma| \leq N \hspace{5mm} \text{and} \hspace{5mm} 6 \leq |\xi| \leq N-1.$$
In the first case, using the Cauchy-Schwarz inequality twice (in $(t,x)$ and then in $v$), we get
\begin{eqnarray}
\nonumber I_{\overline{B}} & \lesssim &  \int_{u=-\infty}^t \left\| \sigma (\mathcal{L}_{Z^{\zeta}} (F) ) \right\|_{L^2(C_u(t))} \left| \int_{C_u(t)} \log^{2M_1}(1+\tau_+)\frac{\tau_+^2}{\tau_-^2}  \left| \int_v \sqrt{ \frac{v^{\underline{L}}}{v^0}} \left| K^j_i W_p \right| \frac{dv}{v^0} \right|^2 dC_u(t) \right|^{\frac{1}{2}} du \\ \nonumber
& \lesssim & \sqrt{\epsilon} \sum_{q=0}^{+ \infty} \int_{u=-\infty}^t  \frac{1}{\tau_-^{\frac{5}{4}}}  \left\| \tau_+^{\frac{11}{4}} \int_v  \left| W_p \right| \frac{dv}{(v^0)^2} \right\|^{\frac{1}{2}}_{L^{\infty}(C^q_u(t))}  \left\| \int_v \frac{v^{\underline{L}}}{v^0} \left| K^j_i \right|^2 \left| W_p \right| dv \right\|_{L^{1}(C^q_u(t))}^{\frac{1}{2}}  du  \\ \nonumber
& \lesssim & \epsilon^{\frac{3}{2}} \int_{-\infty}^{+\infty} \frac{du}{\tau_-^{\frac{5}{4}}} \sum_{q=0}^{+\infty} \frac{\log^{\frac{M_0+4[i]+2[p]+3N+M_1}{2}}(3+t_{q+1})}{(1+t_{q})^{\frac{1}{8}}} \hspace{2mm} \lesssim \hspace{2mm} \epsilon^{\frac{3}{2}},
\end{eqnarray}
using the bootstrap assumption on $\E_L$ and $\int_v |W_p| \frac{dv}{(v^0)^2} \lesssim \int_v |W_p| \frac{v^{\underline{L}}}{v^0} dv \lesssim \epsilon \log^{3N+M_1}(3+t) \tau_+^{-3}$, which comes from Proposition \ref{bilanG} and Lemma \ref{weights}. For the remaining case, we have $|\gamma| \leq N-6$ and we can then use the pointwise decay estimates on the electromagnetic field given by Proposition \ref{decayF}. Moreover, by Proposition \ref{bilanG}, we have that
$$W_p = z^{\kappa} Y^{\beta} f , \hspace{5mm} \text{with} \hspace{5mm} |\xi|+|\beta| \leq N \hspace{5mm} \text{and} \hspace{5mm} |\kappa| \leq N+1-\beta_P-\xi_P.$$
Suppose first that $|\kappa| \leq 2N-1-\beta_P-2\xi_P$. Then, since $|\Phi|^r | \nabla_{Z^{\gamma}} F| \lesssim \sqrt{\epsilon} \tau_+^{-\frac{3}{4}}\tau_-^{-1}$ and $1 \lesssim \sqrt{v^0 v^{\underline{L}}}$, we get
$$\left| \overline{B}^j_i K^j_i W_p \right| \lesssim \sqrt{\epsilon} \left( \frac{v^0}{\tau_+^{\frac{5}{4}}}+\frac{v^{\underline{L}}}{\tau_+^{\frac{1}{4}} \tau_-^2} \right) \left( \left| z^{\kappa} P_{\xi}(\Phi)^2 Y^{\beta} f \right|+\left|K^j_i \right|^2 \left| W_p \right| \right).$$
Hence, we can obtain $I_{\overline{B}} \lesssim \epsilon^{\frac{3}{2}}$ by following the computations of Proposition \ref{M2}, as, by the bootstrap assumptions on $\overline{\E}_N[f]$ and $\E_L$,
$$\E[ z^{\kappa} P_{\xi}(\Phi)^2 Y^{\beta} f ](t)+\E \left[ \left| K^j_i \right|^2 W_p \right](t) \lesssim \epsilon (1+t)^{\frac{1}{8}}.$$
Otherwise, $|\kappa| = 2N-\beta_P-2\xi_P$ so that $\xi_P=N-1$, $|\beta| \leq 1$ and $|\kappa| = 2-\beta_P$. We can then write $z^{\kappa}=z z^{\kappa_0}$ and find $q \in \llbracket 1,N_1 \rrbracket$ such that $W_q = z^2 z^{\kappa_0} Y^{\beta}f$. It remains to follow the previous case after noticing that 
$$\left| \overline{B}^j_i K^j_i W_p \right| \lesssim \sqrt{\epsilon} \left( \frac{v^0}{\tau_+^{\frac{5}{4}}}+\frac{v^{\underline{L}}}{\tau_+^{\frac{1}{4}} \tau_-^2} \right) \left( \left| z^{\kappa_0} P_{\xi}(\Phi)^2 Y^{\beta} f \right|+\left|K^j_i \right|^2 \left| W_q \right| \right) \hspace{5mm} \text{and} \hspace{5mm} |\kappa_0| \leq 2N-1-2\xi_P-\beta_P.$$
\end{proof}

\subsection{$L^2$ estimates on the velocity averages of $f$}

We finally end this section by proving several $L^2$ estimates. The first one is clearly not sharp but is sufficient for us to close the energy estimates for the electromagnetic field.
\begin{Pro}\label{estiL2}
Let $z \in \V$, $p \leq 3N$, $|k| \leq N-1$ and $\beta$ such that $|k|+|\beta| \leq N$. Then, for all $t \in [0,T[$,
$$ \left\| \frac{1}{\sqrt{\tau_+}} \int_v \left| zP_{k,p}(\Phi) Y^{\beta}f \right| dv \right\|_{L^2 (\Sigma_t)} \lesssim \frac{1}{1+t} \left\| \sqrt{\tau_+} \int_v \left| zP_{k,p}(\Phi) Y^{\beta}f \right| dv \right\|_{L^2 (\Sigma_t)} \lesssim \frac{\epsilon}{(1+t)^{\frac{5}{4}}}$$
\end{Pro}
\begin{proof}
The first inequality ensues from $1+t \leq \tau_+$ on $\Sigma_t$. For the other one, we start by the case $|\beta| \leq N-3$. Write $P_{k,p}(\Phi)=\Phi^n P_{\xi}(\Phi)$ and notice that $|\Phi|^n \lesssim \log^{2p}(1+\tau_+)$. Then, using the bootstrap assumption \eqref{bootf3} and Proposition \ref{Xdecay},
\begin{eqnarray}
\nonumber \left\| \sqrt{\tau_+} \int_v \left|z P_{k,p}(\Phi) Y^{\beta}f \right| dv \right\|^2_{L^2 (\Sigma_t)} & \lesssim & \left\| \tau_+ \log^{4p}(1+\tau_+) \int_v \left| P_{\xi}(\Phi)^2 Y^{\beta}f \right| dv \int_v \left|z^2 Y^{\beta}f \right| dv \right\|_{L^{1} (\Sigma_t)} \\ \nonumber
& \lesssim & \left\| \tau_+ \log^{4p}(1+\tau_+) \int_v \left|z^2 Y^{\beta}f \right| dv \right\|_{L^{\infty} (\Sigma_t)} \overline{\E}_N[f](t) \\ \nonumber
& \lesssim & \epsilon \frac{\log^{4p+6}(3+t)}{1+t} (1+t)^{\eta} \hspace{2mm} \lesssim \hspace{2mm} \frac{\epsilon^2}{(1+t)^{\frac{3}{4}}}.
\end{eqnarray}
Otherwise, $|\beta| \geq N-2$ so that $|k| \leq 2$ and, according to Remark \ref{estiPkp}, $P_{k,p}(\Phi) \lesssim \tau_+^{\frac{1}{8}}$. Moreover, as there exists $i \in \llbracket 1, |I_1| \rrbracket$ such that $\beta=\beta^1_i$, we obtain
$$ \left\|\tau_+^{\frac{1}{2}} \int_v \left| zP_{k,p}(\Phi) Y^{\beta} \right| dv \right\|_{L^2 (\Sigma_t)} \lesssim  \left\| \tau_+^{\frac{5}{8}} \int_v \left| z H_i \right| dv \right\|_{L^2 (\Sigma_t)}+\left\|  \tau_+^{\frac{5}{8}} \int_v \left| z G_i \right| dv \right\|_{L^2 (\Sigma_t)}.$$
Applying Proposition \ref{estidecayH}, one has
$$\left\| \tau_+^{\frac{5}{8}} \int_v \left| z H_i \right| dv \right\|^2_{L^2 (\Sigma_t)} \lesssim \left\| \tau_+^{\frac{5}{4}} \int_v \left| z^2 H_i \right| dv \right\|_{L^{\infty} (\Sigma_t)}\left\| \int_v \left|  H_i \right| dv \right\|_{L^1 (\Sigma_t)} \lesssim \frac{\epsilon^2}{(1+t)^{\frac{1}{2}}}.$$
As there exists $q \in \llbracket 1, |I| \rrbracket$ such that $ G_i=L_q=K_q^j W_j$, we have, using this time Proposition \ref{Inho1} and the decay estimate on $\int_v |z^2 W| dv$ given in Proposition \ref{bilanG},
\begin{eqnarray}
\nonumber \left\| \tau_+^{\frac{5}{8}} \int_v \left| z G_i \right| dv \right\|^2_{L^2 (\Sigma_t)} & = & \left\| \tau_+^{\frac{5}{8}} \int_v \left| z K_q^j W_j \right| dv \right\|^2_{L^2 (\Sigma_t)} \\ \nonumber
& \lesssim & \sum_{j=0}^{N_1} \left\| \tau_+^{\frac{5}{4}} \int_v \left| z^2 W_j \right| dv \right\|_{L^{\infty} (\Sigma_t)} \left\|  \int_v \left| K_q^j \right|^2 \left| W_j \right| dv \right\|_{L^1 (\Sigma_t)} \\ \nonumber
& \lesssim & \epsilon \frac{\log^{3N+M_1}(3+t)}{(1+t)^{\frac{3}{4}}}\log^{4[q]}(3+t) \E_L(t) \hspace{2mm} \lesssim \hspace{2mm} \frac{\epsilon^2}{(1+t)^{\frac{1}{2}}}.
\end{eqnarray}
\end{proof}
This proposition allows us to improve the bootstrap assumption \eqref{bootL2} if $\epsilon$ is small enough. More precisely, the following result holds.
\begin{Cor}
For all $t \in [0,T[$, we have $\sum_{|\beta| \leq N-2} \left\| r^{\frac{3}{2}} \int_v \frac{v^A}{v^0} \widehat{Z}^{\beta} f dv \right\|_{L^2(\Sigma_t)}  \lesssim  \epsilon$.
\end{Cor}
\begin{proof}
Let $t \in [0,T[$. Using $\tau_+|v^A| \lesssim v^0 \sum_{z \in \V} |z|$ and rewritting $\widehat{Z}^{\beta}$ in terms of modified vector fields through the identity \eqref{lifttomodif}, one has
$$ \sum_{|\beta| \leq N-2} \left\| r^{\frac{3}{2}} \int_v \frac{v^A}{v^0} \widehat{Z}^{\beta} f dv \right\|_{L^2(\Sigma_t)} \hspace{2mm} \lesssim \hspace{2mm} \sum_{z \in \V} \sum_{p \leq N-2} \sum_{\begin{subarray}{} |q|+|\kappa| \leq N-2 \\ \hspace{2mm} |q| \leq N-3  \end{subarray}} \left\| \sqrt{r} \int_v  \left| P_{q,p}(\Phi) Y^{\kappa} f \right| dv \right\|_{L^2(\Sigma_t)}.$$
It then only remains to apply the previous proposition.
\end{proof}
The two following estimates are crucial as a weaker decay rate would prevent us to improve the bootstrap assumptions.

\begin{Pro}\label{crucial1}
Let $\beta$ and $\xi$ such that $|\xi|+|\beta| \leq N-1$. Then, for all $t \in [0,T[$,
\begin{eqnarray}
\nonumber \left\| \sqrt{\tau_-} \int_v \left| P^X_{\xi}(\Phi) Y^{\beta} f \right| dv \right\|_{L^2(\Sigma_t)} & \lesssim & \epsilon \frac{1}{1+t} \hspace{5mm} \text{if} \hspace{5mm} |\beta| \leq N-3 \\ \nonumber
& \lesssim & \epsilon \frac{\log^M(3+t)}{1+t} \hspace{5mm} \text{otherwise}.
\end{eqnarray}
\end{Pro}
\begin{proof}
Suppose first that $|\beta| \leq N-3$. Then, by Proposition \ref{Xdecay},
\begin{eqnarray}
\nonumber \left\| \sqrt{\tau_-} \int_v \left| P^X_{\xi}(\Phi) Y^{\beta} f \right| dv \right\|^2_{L^2(\Sigma_t)} & \lesssim & \left\| \tau_- \int_v \left| Y^{\beta} f \right| dv \right\|_{L^{\infty}(\Sigma_t)} \left\|  \int_v \left| P^X_{\xi}(\Phi)^2 Y^{\beta} f \right| dv \right\|_{L^1(\Sigma_t)} \\ \nonumber
& \lesssim & \frac{\epsilon}{(1+t)^2}  \E^X_{N-1}[f](t) \hspace{2mm} \lesssim \hspace{2mm} \left| \frac{\epsilon}{1+t} \right|^2 .
\end{eqnarray}
Otherwise,
\begin{itemize}
\item $|\beta| \geq N-2$, so $|\xi| \leq 1$ and then $|P^X_{\xi}(\Phi)| \lesssim \log^{\frac{3}{2}} (1+\tau_+)$ by Proposition \ref{Phi1}.
\item There exists $i \in \llbracket 1, |I_1| \rrbracket$ and $q \in \llbracket 1, |I| \rrbracket$ such that $Y^{\beta} f = H_i+G_i=H_i+L_q$.
\end{itemize}
Using Proposition \ref{estidecayH} (for the first estimate) and Propositions \ref{bilanG}, \ref{Inho1} (for the second one), we obtain
\begin{eqnarray}
\nonumber \nonumber \left\| \sqrt{\tau_-} \int_v \left| P^X_{\xi}(\Phi) H_i \right| dv \right\|^2_{L^2(\Sigma_t)} & \lesssim & \left\| \tau_- \log^{3}(1+\tau_+) \int_v \left| H_i \right| dv \right\|_{L^{\infty}(\Sigma_t)} \left\|  \int_v \left| H_i  \right| dv \right\|_{L^1(\Sigma_t)} \\ \nonumber
& \lesssim & \left\| \epsilon \frac{\tau_- \log^{3+M_1}(1+\tau_+)}{\tau_+^2 \tau_-}  \right\|_{L^{\infty}(\Sigma_t)} \E[H_i](t) \hspace{2mm} \lesssim \hspace{2mm} \epsilon^2 \frac{\log^{3+M_1}(3+t)}{(1+t)^2}, \\ \nonumber
\left\| \sqrt{\tau_-}  \int_v \left| P^X_{\xi}(\Phi) L_q \right| dv \right\|^2_{L^2(\Sigma_t)} & = & \left\| \sqrt{\tau_-}  \int_v \left| P^X_{\xi}(\Phi) K_q^j W_j \right| dv \right\|^2_{L^2(\Sigma_t)} \\ \nonumber
& \lesssim & \sum_{j=0}^{N_1} \left\| \tau_- \log^{3}(1+\tau_+) \int_v \left| W_j \right| dv \right\|_{L^{\infty}(\Sigma_t)} \left\|  \int_v \left|K_q^j  \right|^2 |W_j| dv \right\|_{L^1(\Sigma_t)} \\ \nonumber
& \lesssim & \epsilon \frac{\log^{3+3N+M_1}(3+t)}{(1+t)^2} \epsilon \log^{M_0+4[q]}(3+t) \hspace{2mm} \lesssim \hspace{2mm} \epsilon^2 \frac{\log^{M_0+M_1+3N+3}(3+t)}{(1+t)^2},
\end{eqnarray}
since $[q]=0$. This concludes the proof if $M$ is choosen such that\footnote{Recall from Remark \ref{estiPkp} that $M_1$ is independent of $M$.} $2M \geq M_0+M_1+3N+3$.
\end{proof}
The following estimates will be needed for the top order energy norm. As it will be used combined with Proposition \ref{CommuFsimple}, the quantity $P_{q,p}(\Phi)$ will contain $\Y_X$ derivatives of $\Phi$.
\begin{Pro}\label{crucial2}
Let $\beta$, $q$ and $p$ be such as $|q|+|\beta| \leq N$, $|q| \leq N-1$ and $p \leq q_X+\beta_T$. Then, for all $t \in [0,T[$,
$$\left\| \sqrt{\tau_-} \int_v \left| P_{q,p}(\Phi) Y^{\beta} f \right| dv \right\|_{L^2(\Sig^0_t)} \lesssim \frac{\epsilon}{(1+t)^{1-\frac{\eta}{2}}}.$$
\end{Pro}
\begin{proof}
We consider various cases and, except for the last one, the estimates are clearly not sharp. Let us suppose first that $|\beta| \geq N-2$. Then $|q| \leq 2$ and $|P_{q,p}(\Phi)| \lesssim \log^{M_1}(3+t)$ on $\Sig^0_t$ by Remark \ref{estiPkp}, so that, using Proposition \ref{crucial1},
$$ \left\| \sqrt{\tau_-} \int_v \left| P_{k,p}(\Phi) Y^{\beta} f \right| dv \right\|_{L^2(\Sig^0_t)} \hspace{1mm} \lesssim \hspace{1mm} \log^{M_1}(3+t) \left\| \sqrt{\tau_-} \int_v \left| Y^{\beta} f \right| dv \right\|_{L^2(\Sig^0_t)} \hspace{1mm}  \lesssim \hspace{1mm} \epsilon \frac{\log^{M+M_1}(3+t)}{1+t}.$$
Let us write $P_{q,p}(\Phi)=\Phi^r P_{\xi}(\Phi)$ with $r \leq p$ and $(\xi_T,\xi_P,\xi_X)=(q_T,q_P,q_X)$. If $|\beta| \leq N-3$ and $|q| \leq N-2$, then by the Cauchy-Schwarz inequality (in $v$), \eqref{Auxenergy} as well as Propositions \ref{Phi1} and \ref{Xdecay},
\begin{eqnarray}
\nonumber \left\| \sqrt{\tau_-} \int_v \left| P_{k,p}(\Phi) Y^{\beta} f \right| dv \right\|^2_{L^2(\Sigma_t)} & \lesssim & \left\| \tau_- \int_v \left| \Phi^{2r}  Y^{\beta} f \right| dv \right\|_{L^{\infty}(\Sigma_t)} \left\| \int_v \left| P_{\xi}(\Phi)^2 Y^{\beta} f \right| dv \right\|_{L^1(\Sigma_t)} \\ \nonumber
& \lesssim & \left\|\tau_-\frac{\epsilon \log^{4r}(1+\tau_+)}{\tau_+^2\tau_-} \right\|_{L^{\infty}(\Sigma_t)}\mathbb{A}[f](t) \hspace{2mm} \lesssim \hspace{2mm} \epsilon^2 \frac{\log^{8N}(3+t)}{(1+t)^{2-\frac{3}{4}\eta}}.
\end{eqnarray}
The remaining case is the one where $|q|=N-1$ and $|\beta| \leq 1$. Hence, $p \leq k_X+1$. 
\begin{itemize}
\item If $p \geq 2$, we have $k_X \geq 1$ and then, schematically, $P_{\xi}(\Phi)=P^X_{\xi^1}(\Phi) P_{\xi^2}(\Phi)$, with $|\xi^1| \geq 1$ and $|\xi^1|+|\xi^2| = N-1$. If $|\xi^2| \geq 1$, we have $\min(|\xi^1|, |\xi^2|) \leq \frac{N-1}{2} \leq N-6$ and one of the two factor can be estimated pointwise, which put us in the context of the case $|k| \leq N-2$ and $|\beta| \leq N-3$. Otherwise, $P_{k,p}(\Phi)= \Phi^r P^X_{\xi^1}(\Phi)$ and, using again \eqref{Auxenergy},
\begin{eqnarray}
\nonumber \left\| \sqrt{\tau_-} \int_v \left| P_{k,p}(\Phi) Y^{\beta} f \right| dv \right\|^2_{L^2(\Sigma_t)} & \lesssim & \left\| \tau_- \int_v \left| \Phi^{2r}  Y^{\beta} f \right| dv \right\|_{L^{\infty}(\Sigma_t)} \left\| \int_v \left| P^X_{\xi^1}(\Phi)^2 Y^{\beta} f \right| dv \right\|_{L^1(\Sigma_t)} \\ \nonumber
& \lesssim & \left\|\tau_-\frac{\epsilon \log^{4r}(1+\tau_+)}{\tau_+^2\tau_-} \right\|_{L^{\infty}(\Sigma_t)} \mathbb{A}[f](t) \hspace{2mm} \lesssim \hspace{2mm} \epsilon^2 \frac{\log^{8N}(3+t)}{(1+t)^{2-\frac{3}{4} \eta} }.
\end{eqnarray}
\item If $p=1$, we have $P_{k,p}(\Phi)=Y^{\kappa} \Phi$ and, using $\overline{\E}_N[f](t) \leq 4 \epsilon (1+s)^{\eta}$,
\begin{eqnarray}
\nonumber \left\| \sqrt{\tau_-} \int_v \left| P_{k,p}(\Phi) Y^{\beta} f \right| dv \right\|^2_{L^2(\Sigma_t)} & \lesssim & \left\| \tau_- \int_v \left| Y^{\beta} f \right| dv \right\|_{L^{\infty}(\Sigma_t)} \left\| \int_v \left| Y^{\kappa} \Phi \right|^2 \left| Y^{\beta} f \right| dv \right\|_{L^1(\Sigma_t)} \\ \nonumber
& \lesssim & \left\|\tau_-\frac{\epsilon}{\tau_+^2\tau_-} \right\|_{L^{\infty}(\Sigma_t)} \E \left[ \left| Y^{\kappa} \Phi \right|^2 Y^{\beta} f \right](t) \hspace{2mm} \lesssim \hspace{2mm} \epsilon^2 \frac{(1+t)^{\eta}}{(1+t)^2}.
\end{eqnarray}
\end{itemize}
\end{proof}

\section{Improvement of the energy estimates of the electromagnetic field}\label{sec12}
In order to take advantage of the null structure of the system, we start this section by a preparatory lemma.
\begin{Lem}\label{calculGFener}
Let $G$ be a $2$-form and $g$ a function, both sufficiently regular and recall that $J(g)^{\nu}= \int_v \frac{v^{\nu}}{v^0} g dv$, $\left| \overline{S}^{L} \right|  \lesssim \tau_+$ and $\left| \overline{S}^{\underline{L}} \right| \lesssim \tau_-$. Then, using several times Lemma \ref{weights1} and Remark \ref{rqweights1},
\begin{eqnarray}
\nonumber \left| G_{ 0 \nu} J(g)^{\nu} \right| & \lesssim & |\rho|\int_v |g|dv+(|\alpha_A|+|\underline{\alpha}_A| )\int_v \frac{|v^A|}{v^0}|g|dv \hspace{1.5mm} \lesssim \hspace{1.5mm} |\rho|\int_v |g|dv+\frac{1}{\tau_+}\sum_{w \in \V}(|\alpha|+|\underline{\alpha}| )\int_v |wg|dv, \\
\nonumber \left|\overline{S}^{\mu}  G_{ \mu \nu} J(g)^{\nu} \right| & \lesssim & \tau_+ |\rho| \int_v \frac{v^{\underline{L}}}{v^0} |g|dv+\tau_- |\rho| \int_v \frac{v^L}{v^0} |g|dv+\tau_+|\alpha| \int_v \frac{|v^A|}{v^0} |g| dv+\tau_-|\underline{\alpha}| \int_v \frac{|v^A|}{v^0}|g| dv \\ \nonumber
& \lesssim & \left( |\alpha|+|\rho|+ \frac{\tau_-}{\tau_+}|\underline{\alpha}| \right) \sum_{z \in \V} \int_v  |z g | dv \hspace{1cm} \text{if $|x| \geq t$},\\ \nonumber
& \lesssim & |\rho| \int_v \left( \tau_-+\sum_{z \in \V} |z| \right) |g| dv+\left( |\alpha|+ \frac{\tau_-}{\tau_+}|\underline{\alpha}| \right) \int_v \sum_{z \in \V} |z g | dv \hspace{1cm} \text{otherwise}.
\end{eqnarray}
\end{Lem}
We are now ready to improve the bootstrap assumptions concerning the electromagnetic field.
\subsection{For $\mathcal{E}^0_N[F]$}
Using Proposition \ref{energyMax1} and commutation formula of Proposition \ref{CommuFsimple}, we have, for all $t \in [0,T]$,
\begin{equation}\label{28:eq}
\mathcal{E}^0_N[F](t)-2\mathcal{E}^0_N[F](0) \lesssim  \sum_{|\gamma| \leq N} \sum_{\begin{subarray}{l}  p \leq |k|+|\beta| \leq N \\ \hspace{3mm} |k| \leq N-1 \end{subarray}} \int_0^t \int_{\Sigma_s} |\mathcal{L}_{Z^{\gamma}}(F)_{\mu 0} J(P_{k,p}(\Phi) Y^{\beta} f)^{\mu} | dx ds.
\end{equation}
We fix $|k|+|\beta| \leq N$, $p \leq N$ and $|\gamma| \leq N$. Denoting the null decomposition of $\mathcal{L}_{Z^{\gamma}}(F)$ by $(\alpha, \underline{\alpha}, \rho, \sigma)$, $P_{k,p}(\Phi) Y^{\beta} f$ by $g$ and applying Lemma \ref{calculGFener}, one has
$$ \int_0^t  \int_{\Sigma_s} |\mathcal{L}_{Z^{\gamma}}(F)_{\mu 0} J(P_{k,p}(\Phi) Y^{\beta} f)^{\mu} | dx ds \lesssim  \int_0^t   \int_{\Sigma_s} |\rho|  \int_v |g| dv +\left( |\alpha|+|\underline{\alpha}| \right) \sum_{w \in \V} \frac{1}{\tau_+} \int_v  \left|w g \right|dv dx ds.$$
On the one hand, using Proposition \ref{estiL2},
\begin{eqnarray}
\nonumber \sum_{w \in \V} \int_0^t \int_{\Sigma_s} \left( |\alpha|+|\underline{\alpha}| \right) \int_v \frac{1}{\tau_+} \left|w g \right|dv dx ds & \lesssim & \sum_{w \in \V} \int_0^t \sqrt{\mathcal{E}^0_N[F](s)} \left\| \frac{1}{\tau_+} \int_v |wg| dv \right\|_{L^2(\Sigma_s)} ds \\ \nonumber
& \lesssim & \epsilon^{\frac{3}{2}}.
\end{eqnarray}
On the other hand, as $\rho = \rho (\mathcal{L}_{Z^{\gamma}}(\F)) +\rho ( \mathcal{L}_{Z^{\gamma}}(\Ff))$
and $\rho ( \mathcal{L}_{Z^{\gamma}}(\Ff)) \lesssim \epsilon \tau_+^{-2}$ , we have, using Proposition \ref{estiL2} and the bootstrap assumptions \eqref{bootf3}, \eqref{bootext} and \eqref{bootF4},
\begin{eqnarray}
\nonumber \int_0^t   \int_{\Sigma_s} |\rho|  \int_v |g| dv dx ds & \lesssim & \int_0^t \left\| \sqrt{\tau_+} \rho ( \mathcal{L}_{Z^{\gamma}}(\F)) \right\|_{L^2(\Sigma_s)} \left\| \frac{1}{\sqrt{\tau_+}} \int_v |g| dv \right\|_{L^2(\Sigma_s)}+\int_{\Sigma_s} \rho ( \mathcal{L}_{Z^{\gamma}}(\Ff)) \int_v |g| dv dx  ds \\ \nonumber
& \lesssim & \int_0^t \sqrt{\mathcal{E}^{Ext}_N[\F](s)+\mathcal{E}_N[F](s)} \frac{\epsilon }{(1+s)^{\frac{5}{4}}}ds+ \int_0^t \frac{\epsilon}{(1+s)^2} \E[g](s) ds \hspace{2mm} \lesssim \hspace{2mm} \epsilon^{\frac{3}{2}}.
\end{eqnarray}
The right-hand side of \eqref{28:eq} is then bounded by $\epsilon^{\frac{3}{2}}$, implying that $\mathcal{E}^0_N[f] \leq 3\epsilon$ on $[0,T[$ if $\epsilon$ is small enough.

\subsection{The weighted norm for the exterior region}

Applying Proposition \ref{energyMax1} and using $\mathcal{E}^{Ext}_N[\F](0) \leq \epsilon$ as well as $\F=F-\Ff$, we have, for all $t \in [0,T[$,
$$  \mathcal{E}_N^{Ext}[\F](t) \leq 6\epsilon+\sum_{|\gamma| \leq N} \int_0^t \int_{\Si^0_s} \left| \overline{S}^{\mu}  \mathcal{L}_{Z^{\gamma}}(\F)_{\mu \nu} \nabla^{\lambda} {\mathcal{L}_{Z^{\gamma}}(F)_{\lambda} }^{ \nu} \right| dx ds+\int_0^t \int_{\Si^0_s} \left| \overline{S}^{\mu}  \mathcal{L}_{Z^{\gamma}}(\F)_{\mu \nu} \nabla^{\lambda} {\mathcal{L}_{Z^{\gamma}}(\Ff)_{\lambda} }^{ \nu} \right| dx ds .$$
Let us fix $|\gamma| \leq N$ and denote the null decomposition of $\mathcal{L}_{Z^{\gamma}}(\F)$ by $(\alpha, \underline{\alpha}, \rho , \sigma)$. As previously, using Proposition \ref{CommuFsimple},
$$\int_0^t \int_{\Si^0_s} \left| \overline{S}^{\mu}  \mathcal{L}_{Z^{\gamma}}(\F)_{\mu \nu} \nabla^{\lambda} {\mathcal{L}_{Z^{\gamma}}(F)_{\lambda} }^{ \nu} \right| dx ds \lesssim   \sum_{\begin{subarray}{l}  p \leq |k|+|\beta| \leq |\gamma| \\ \hspace{3mm} |k| \leq |\gamma|-1 \end{subarray}} \int_0^t \int_{\Si^0_s} |\overline{S}^{\mu} \mathcal{L}_{Z^{\gamma}}(\F)_{\mu \nu} J(P_{k,p}(\Phi) Y^{\beta} f)^{\nu} | dx ds.$$
We fix $|k|+|\beta| \leq N$, $p \leq N$ and $|\gamma| \leq N$ and we denote again $P_{k,p}(\Phi) Y^{\beta} f$ by $g$. Using successively Lemma \ref{calculGFener}, the Cauchy-Schwarz inequality, the bootstrap assumption \eqref{bootext} and Proposition \ref{estiL2}, we obtain
\begin{eqnarray}
\nonumber \int_0^t  \int_{\Si^0_s} |\overline{S}^{\mu} \mathcal{L}_{Z^{\gamma}}(\F)_{\mu \nu} J(P_{k,p}(\Phi) Y^{\beta} f)^{\nu} | dx ds & \lesssim & \int_0^t   \int_{\Sigma_s} \left(|\rho|+ |\alpha|+\frac{\sqrt{\tau_-}}{\sqrt{\tau_+}}|\underline{\alpha}| \right) \sum_{w \in \V}  \int_v  \left|w g \right|dv dx ds. \\ \nonumber
& \lesssim & \sum_{w \in V} \int_0^t \sqrt{\mathcal{E}_N^{Ext}[F](s)} \left\| \frac{1}{\sqrt{\tau+}} \int_v  \left|w g \right|dv \right\|_{L^2(\Sigma_s)} ds \\ \nonumber
& \lesssim &  \epsilon^{\frac{3}{2}} \int_0^{+\infty} \frac{ds}{(1+s)^{\frac{5}{4}}} \hspace{2mm} \lesssim \hspace{2mm} \epsilon^{\frac{3}{2}}.
\end{eqnarray}
 Using Proposition \ref{propcharge} and iterating commutation formula of Proposition \ref{basiccom}, we have,
$$\tau_+^2\left| \nabla^{\mu} {\mathcal{L}_{Z^{\gamma}}(\Ff)_{\mu} }^{ L} \right|(t,x) +\tau_+^4\left| \nabla^{\mu} {\mathcal{L}_{Z^{\gamma}}(\Ff)_{\mu} }^{ \underline{L}} \right|(t,x) + \tau_+^3 \left| \nabla^{\mu} {\mathcal{L}_{Z^{\gamma}}(\Ff)_{\mu} }^{ A} \right|(t,x) \lesssim |Q(F)| \mathds{1}_{-2 \leq t-|x| \leq -1}(t,x).$$
Consequently, as $|Q(F)| \leq \| f_0 \|_{L^1_{x,v}} \leq  \epsilon$, $\left| \overline{S}^L \right| \lesssim \tau_+$ and $\left| \overline{S}^{\underline{L}} \right| \lesssim \tau_-$, 
 $$|\overline{S}^{\mu}  \mathcal{L}_{Z^{\gamma}}(\F)_{\mu \nu} \nabla^{\lambda} {\mathcal{L}_{Z^{\gamma}}(\Ff)_{\lambda} }^{\nu}| \lesssim \left( \tau_+ |\rho |\frac{\epsilon}{\tau_+^4}+\tau_- |\rho |\frac{\epsilon}{\tau_+^2}+\tau_+ |\alpha |\frac{\epsilon}{\tau_+^3}+\tau_- |\underline{\alpha} | \frac{\epsilon}{\tau_+^3} \right) \mathds{1}_{-2 \leq t-|x| \leq -1}(t,x).$$
Note now that $\tau_- \mathds{1}_{-2 \leq t-|x| \leq -1} \leq \sqrt{5}$, so that, using the bootstrap assumption \eqref{bootext} and the Cauchy-Schwarz inequality,
\begin{eqnarray}
\nonumber \int_0^t \int_{\Si^0_s} \left| S^{\mu}  \mathcal{L}_{Z^{\gamma}}(\F)_{\mu \nu} \nabla^{\lambda} {\mathcal{L}_{Z^{\gamma}}(\Ff)_{\lambda} }^{ \nu} \right| dx ds  & \lesssim &  \int_0^t \frac{\epsilon}{(1+s)^{\frac{5}{2}}} \int_{s+1 \leq |x| \leq s+2} \sqrt{\tau_+}|\rho|+\sqrt{\tau_+}|\alpha|+|\underline{\alpha}| dx ds \\ \nonumber
& \lesssim & \int_0^t \frac{\epsilon}{(1+s)^{\frac{5}{2}}} \sqrt{\mathcal{E}_N^{Ext}[\F](s)} \sqrt{s^2+1} ds \hspace{2mm} \lesssim \hspace{2mm} \epsilon^{\frac{3}{2}}.
\end{eqnarray}
Thus, if $\epsilon$ is small enough, we obtain $\mathcal{E}^{Ext}_N[\F] \leq 7 \epsilon$ on $[0,T[$ which improves the bootstrap assumption \eqref{bootext}.

\subsection{The weighted norms for the interior region}

Recall from Proposition \ref{energyMax1} that we have, for $Q \in \{ N-3, N-1, N \}$ and $t \in [0,T[$,
\begin{equation}\label{ensource}
 \mathcal{E}_Q[F](t) \hspace{2mm} \leq \hspace{2mm} 24\epsilon +\sum_{|\gamma| \leq Q} \int_0^t \int_{\Sig^0_s} \left| \overline{S}^{\mu} \mathcal{L}_{Z^{\gamma}}(F)_{\mu \nu}  \nabla^{\lambda} {\mathcal{L}_{Z^{\gamma}}(F)_{\lambda} }^{ \nu} \right| dx ds, 
\end{equation}
since $\mathcal{E}^{Ext}_{N}[\F] \leq 8\epsilon$ on $[0,T[$ by the bootstrap assumption \eqref{bootext}). The remainder of this subsection is divided in two parts. We consider first $Q \in \{ N-3, N-1 \}$ and we end with $Q=N$ as we need to use in that case a worst commutation formula in order to avoid derivatives of $\Phi$ of order $N$, which is the reason of the stronger loss on the top order energy norm.

\subsubsection{The lower order energy norms}
Let $Q \in \{ N-3, N-1 \}$. According to commutation formula of Proposition \ref{ComuMaxN}, we can bound the last term of \eqref{ensource} by a linear combination of the following ones.
\begin{eqnarray}
\hspace{-4mm} \mathcal{I}_1 \hspace{-1mm} & := & \hspace{-1mm} \int_0^t \int_{|x| \leq s} \left| \overline{S}^{\mu}  \mathcal{L}_{ Z^{\gamma}}(F)_{\mu \nu} \int_v  \frac{v^{\nu}}{v^0} P^X_{\xi}(\Phi) Y^{\beta} f dv \right|  dx ds, \hspace{5mm} \text{with} \hspace{5mm} |\gamma| , \hspace{1mm} |\xi|+|\beta| \leq Q, \label{pevar} \\ 
\hspace{-4mm} \mathcal{I}_2 \hspace{-1mm}  & := & \hspace{-1mm} \int_0^t \int_{|x| \leq s}  \left| \overline{S}^{\mu}  \mathcal{L}_{ Z^{\gamma}}(F)_{\mu \nu}  \int_v \frac{z}{\tau_+} P_{k,p}( \Phi) Y^{\beta} f dv  \right| dxds, \hspace{5mm} \text{with} \hspace{5mm} |\gamma| , \hspace{1mm} |k|+|\beta| \leq Q, \hspace{3mm} z \in \V, \label{pevarbis}
\end{eqnarray}
$0 \leq \nu \leq 3$ and $ p \leq 3N$. Fix $|\gamma| \leq Q$ and denote the null decomposition of $\mathcal{L}_{Z^{\gamma}}(F)$ by $(\alpha, \underline{\alpha}, \rho, \sigma)$. We start by \eqref{pevarbis}, which can be estimated independently of $Q$. Recall that $\left| \overline{S}^{L} \right|  \lesssim \tau_+$ and $\left| \overline{S}^{\underline{L}} \right| \lesssim \tau_-$, so that, using Proposition \ref{estiL2} and the bootstrap assumption \eqref{bootF3},
\begin{eqnarray}
\nonumber \mathcal{I}_2 & \lesssim & \int_0^t \int_{\Sig^0_s} \left( \tau_+ |\rho|+\tau_+ |\alpha|+\tau_- |\underline{\alpha}| \right) \int_v \left| \frac{z}{\tau_+} P_{k,p}( \Phi) Y^{\beta} f  \right| dvdxds \\ \nonumber
& \lesssim & \int_0^t \sqrt{\mathcal{E}_{N-1}[F](s)} \left\| \frac{1}{\sqrt{\tau_+}} \int_v \left| z P_{k,p}( \Phi) Y^{\beta} f  \right| dv \right\|_{L^2(\Sigma_s)} ds \\ \nonumber
& \lesssim & \epsilon^{\frac{3}{2}} \int_0^t \frac{\log^M(3+s)}{(1+s)^{\frac{5}{4}}} ds \hspace{2mm} \lesssim \hspace{2mm} \epsilon^{\frac{3}{2}}.
\end{eqnarray}
We now turn on \eqref{pevar} and we then consider $|\xi|+|\beta| \leq Q$. Start by noticing that, by Lemma \ref{calculGFener},
$$ \left| \overline{S}^{\mu}  \mathcal{L}_{ Z^{\gamma}}(F)_{\mu \nu} \hspace{-0.3mm} \int_v  \frac{v^{\nu}}{v^0} P^X_{\xi}(\Phi) Y^{\beta} f \right| dv \lesssim \tau_-|\rho| \int_v \left| P^X_{\xi}(\Phi) Y^{\beta} f \right| dv+\left(|\rho|+ |\alpha|+\frac{\tau_-}{\tau_+} |\underline{\alpha}| \right) \hspace{-1mm} \sum_{w \in \V} \hspace{-0.5mm} \int_v \left| w P^X_{\xi}(\Phi) Y^{\beta} f \right| dv .$$
Consequently, by the bootstrap assumption \eqref{bootF3} and Proposition \ref{estiL2},
\begin{eqnarray}
\nonumber \mathcal{I}_1 & \lesssim & \int_0^t \sqrt{\mathcal{E}_{Q}[F](s)} \left( \left\| \sqrt{\tau_-} \int_v \left| P^X_{\xi}(\Phi) Y^{\beta} f \right| dv \right\|_{L^2(\Sigma_s)}+ \sum_{w \in \V} \left\| \frac{1}{\sqrt{\tau_+}} \int_v \left|w P^X_{\xi}(\Phi) Y^{\beta} f \right| dv \right\|_{L^2(\Sigma_s)} \right) ds \\ \nonumber
& \lesssim & \epsilon^{\frac{3}{2}} + \int_0^t \sqrt{\mathcal{E}_{Q}[F](s)} \left\| \sqrt{\tau_-} \int_v \left| P^X_{\xi}(\Phi) Y^{\beta} f \right| dv \right\|_{L^2(\Sigma_s)}ds.
\end{eqnarray}
The last integral to estimate is the source of the small growth of $\mathcal{E}_Q[F]$. We can bound it, using again the bootstrap assumptions \eqref{bootF2}, \eqref{bootF3} and Proposition \ref{crucial1}, by
\begin{itemize}
\item $\epsilon^{\frac{3}{2}} \log^2(3+t)$ if $Q=N-3$ and
\item $\epsilon^{\frac{3}{2}} \log^{2M}(3+t)$ otherwise. 
\end{itemize} 
Hence, combining this with \eqref{ensource} we obtain, for $\epsilon$ small enough, that
\begin{itemize}
\item $\mathcal{E}_{N-3}[F](t) \leq 25 \epsilon \log^2(3+t)$ for all $t \in [0,T[$ and
\item $\mathcal{E}_{N-1}[F](t) \leq 25 \epsilon \log^{2M}(3+t)$ for all $t \in [0,T[$. 
\end{itemize} 

\subsubsection{The top order energy norm}
We consider here the case $Q=N$ and we then apply this time the commutation formula of Proposition \ref{CommuFsimple}, so that the last term of \eqref{ensource} can be bounded by a linear combination of terms of the form
$$ \mathcal{I}:= \int_0^t \int_{\Sig^{0}_s} \left| \overline{S}^{\mu} \mathcal{L}_{Z^{\gamma}}(F)_{\mu \nu}  \int_v  \frac{v^{\mu}}{v^0} P_{q,p}(\Phi) Y^{\beta} f  dv \right| dx ds,$$
with $|\gamma| \leq N$, $|q|+|\beta| \leq N$,  $|q| \leq N-1$ and $ p \leq q_X+\beta_T$. Let us fix such parameters. Following the computations made previously to estimate $\mathcal{I}_1$ and using $\mathcal{E}_{N}[F](s) \lesssim \sqrt{\epsilon} (1+s)^{\eta} \lesssim \sqrt{\epsilon} (1+s)^{\frac{1}{8}}$, we get
\begin{eqnarray}
\nonumber \mathcal{I}_1 & \lesssim & \int_0^t \sqrt{\mathcal{E}_{N}[F](s)} \left( \left\| \sqrt{\tau_-} \int_v \left| P_{q,p}(\Phi) Y^{\beta} f \right| dv \right\|_{L^2(\Sig^0_s)}+ \sum_{w \in \V} \left\| \frac{1}{\sqrt{\tau_+}} \int_v \left| w P_{q,p}(\Phi) Y^{\beta} f \right| dv \right\|_{L^2(\Sigma_s)} \right) ds \\ 
& \lesssim & \epsilon^{\frac{3}{2}} + \sqrt{\epsilon} \int_0^t (1+s)^{\frac{\eta}{2}} \left\| \sqrt{\tau_-} \int_v \left| P_{q,p}(\Phi) Y^{\beta} f \right| dv \right\|_{L^2(\Sig^0_s)}ds. \label{lafin}
\end{eqnarray}
Applying now Proposition \ref{crucial2}, we can bound \eqref{lafin} by $\epsilon^{\frac{3}{2}}(1+t)^{\eta}$. Thus, if $\epsilon$ is small enough, we obtain $\mathcal{E}_N[F](t) \leq 25 \epsilon (1+t)^{\eta}$ for all $t \in [0,T[$, which concludes the improvement of the bootstrap assumption \eqref{bootF4} and then the proof.

\section*{Acknowledgements}

I am very grateful towards Jacques Smulevici, my Ph.D. advisor, for his support and for giving me precious advice. Part of this work was funded by the European Research Council under the European Union's Horizon 2020 research and innovation program (project GEOWAKI, grant agreement 714408).

\renewcommand{\refname}{References}
\bibliographystyle{abbrv}
\bibliography{biblio}

\end{document}